\documentclass[11pt]{amsart}
\usepackage{amsmath}
\usepackage{amsfonts}
\usepackage{amsthm}
\usepackage{amssymb}
\usepackage{mathrsfs}
\usepackage{latexsym}
\usepackage{verbatim}
\usepackage{diagrams}
\usepackage{a4wide}
\usepackage{enumerate}
\usepackage{hyperref}
\usepackage{bookmark}
\usepackage{amscd}
\usepackage[all,cmtip,line]{xy}
\usepackage{url}

\def\EE{\E}
\def\E{\mathscr{E}}
\def\GG{\mathscr{G}}

\def\flatt{b}
\newtheorem{theoremA}{Theorem}

\renewcommand{\thetheoremName}

\newtheorem{conjectureA}[theoremA]{Conjecture}

\def\fppf{\mathrm{fppf}}
\def\LL{\mathscr{L}}

\def\Rmint{R}
\def\Rmin{R^{\mathrm{min}}}

\def\gr{\mathrm{gr}}
\def\coker{\mathrm{coker}}
\def\nind{\mathrm{n\text{-}ind}}
\def\WD{\mathrm{WD}}
\def\Cat{\mathscr{C}}
\def\wchi{\widetilde{\chi}}

\def\mhat{\widehat{m}}

\def\Aut{\mathrm{Aut}}
\def\res{\mathrm{res}}
\def\length{\mathrm{length}}

\def\Gf{\mathscr{G}}

\def\Sym{\mathrm{Sym}}

\def\Sp{\mathrm{Sp}}

\def\PSL{\mathrm{PSL}}
\def\r{r}

\def\Rbox{R^{\square}}

\def\pssi{\phi}

\def\Ee{\mathscr{E}}
\def\rbar{\overline{r}}
\def\CO{\mathcal{O}}

\def\Art{\mathrm{Art}}
\def\ab{\mathrm{ab}}
\def\projdim{\mathrm{proj.dim}}
\def\Ass{\mathrm{Ass}}
\def\codim{\mathrm{codim}}
\def\Cc{C}
\def\Dc{D}
\def\DP{\mathcal{D}}
\def\Kc{K}
\def\Kcp{L}
\def\n{\mathfrak{n}}
\def\Ff{\mathscr{F}}
\def\big{big}

\def\sbar{\overline{s}}
\def\psibar{\overline{\psi}}
\def\CS{\mathcal{S}}
\def\Iw{\mathrm{Iw}}
\def\tY{\widetilde{Y}}
\def\Art{\mathrm{Art}}
\def\ab{\mathrm{ab}}

\def\SO{\mathrm{SO}}
\def\SU{\mathrm{SU}}
\def\Res{\mathrm{Res}}
\def\Trace{\mathrm{Trace}}
\def\Ad{\mathrm{ad}}
\def\nmf{\mathbf{N}_{F/\Q}}

\def\H{H_{\psi}}
\def\Hom{\mathrm{Hom}}
\def\SHom{\mathcal{H}\kern -.5pt om}

\def\L{\mathcal{L}}
\def\T{\mathbf{T}}
\def\Z{\mathbf{Z}}
\def\G{\mathbf{G}}
\def\A{\mathbf{A}}
\def\OL{\mathcal{O}}
\def\L{\mathscr{L}}
\def\Q{\mathbf{Q}}
\def\F{\mathbf{F}}
\def\R{\mathbf{R}}
\def\C{\mathbf{C}}

\def\CC{\mathcal{C}}
\def\CE{\mathcal{E}}
\def\CL{\mathcal{L}}
\def\Cl{\mathrm{Cl}}
\def\RCl{\mathrm{RCl}}

\def\CF{\mathcal{F}}
\def\CV{\mathcal{V}}
\def\a{\mathfrak{a}}
\def\m{\mathfrak{m}}
\def\p{\mathfrak{p}}
\def\b{\mathfrak{b}}
\def\d{\mathfrak{d}}

\def\WT{\widetilde{\T}}
\def\Qbar{\overline{\Q}}

\def\rhobar{\overline{\rho}}

\def\Tor{\mathrm{Tor}}
\def\Ind{\mathrm{Ind}}
\def\PGL{\mathrm{PGL}}
\def\Ext{\mathrm{Ext}}
\def\Hom{\mathrm{Hom}}
\def\End{\mathrm{End}}
\def\Ann{\mathrm{Ann}}
\def\GSp{\mathrm{GSp}}
\def\GL{\mathrm{GL}}
\def\rank{\mathrm{rank}}
\def\Gal{\mathrm{Gal}}
\def\ad{\mathrm{ad}}
\def\SL{\mathrm{SL}}
\def\Frob{\mathrm{Frob}}
\def\Spec{\mathrm{Spec}}

\newarrow{Equals}{=}{=}{=}{=}{=}
\newarrow{pots}{.}{.}{.}{.}{>}
\def\ra{\rightarrow}
\def\lra{\longrightarrow}
\def\onto{\twoheadrightarrow}
\def\iso{\stackrel{\sim}{\rightarrow}}
\def\onto{\twoheadrightarrow}
\def\into{\hookrightarrow}
\def\liso{\stackrel{\sim}{\longrightarrow}}

\def\ss{\mathrm{ss}}
\def\ff{\mathrm{f}}
\def\cusps{\infty}

\def\Twist{\Phi}
\def\twist{\phi}
\def\eps{\epsilon}
\def\NF{N_{F/\Q}}
\def\wrho{\widetilde{\rho}}
\def\walpha{\widetilde{\alpha}}
\def\wbeta{\widetilde{\beta}}
\def\wtm{\widetilde{\m}}
\def\hatm{\widetilde{\m}}
\def\Tan{\T^{\mathrm{an}}}
\def\TE{\T_{\emptyset}}

\def\mE{\mathfrak{m}_{\emptyset}}

\def\tr{\mathrm{Trace}}
\def\wt{\widetilde}

\def\univ{\mathrm{univ}}
\def\II{\mathscr{I}}
\def\JJ{\mathscr{J}}
\def\JJg{\JJ^{\mathrm{glob}}}
\def\Ra{\widetilde{R}}
\def\Rv{R^{\mathrm{univ}}}
\def\Rva{\Ra^{\mathrm{univ}}}
\def\Rd{R^{\dagger}}
\def\Rda{\Ra^{\dagger}}
\def\Ru{R^{\mathrm{unr}}}
\def\Rua{\Ra^{\mathrm{unr}}}
\def\Rloc{R_{\mathrm{loc}}}

\DeclareMathOperator\depth{depth}
\DeclareMathOperator\Frac{Frac}
\DeclareMathOperator\Tr{tr}
\DeclareMathOperator\im{Im}
\DeclareMathOperator\diag{diag}

\DeclareMathOperator\red{red}
\DeclareMathOperator\ord{ord}
\DeclareMathOperator\sub{sub}

\newtheorem{theorem}{Theorem}[section]

\newtheorem{df}[theorem]{Definition}
\newtheorem{lemma}[theorem]{Lemma}
\newtheorem{prop}[theorem]{Proposition}
\newtheorem{corr}[theorem]{Corollary}

\newtheorem{remark}[theorem]{Remark}
\newtheorem{assumption}[theorem]{Assumption}

\begin{document}

\title{Modularity Lifting 
beyond the Taylor--Wiles Method}
\author{Frank Calegari \and David Geraghty} 
\thanks{The first author was supported in part by NSF Career Grant
  DMS-0846285, NSF Grant DMS-1404620, NSF Grant DMS-1648702, NSF Grant DMS-1701703,
      and a Sloan Foundation Fellowship; the second author was
  supported in part by NSF Grant DMS-1440703;
both authors were supported in part by NSF Grant DMS-0841491.}

\begin{abstract}
We prove new modularity lifting theorems for $p$-adic Galois
  representations in situations where the methods of Wiles and
  Taylor--Wiles do not apply. Previous generalizations of these
  methods have been restricted to situations where the automorphic
  forms in question contribute to a single degree of cohomology. In
  practice, this imposes several restrictions -- one must be in a
  Shimura variety setting and the automorphic forms must be of regular
  weight at infinity. In this paper, we essentially show how to remove
  these restrictions. 

  Our most general result is a modularity lifting theorem which, on the
  automorphic side applies to automorphic forms on the group $\GL(n)$
  over a general number field; it is contingent on a conjecture which, in particular,
  predicts the existence of Galois representations associated to
  torsion classes in the cohomology of the associated locally
  symmetric space. We show that if this conjecture holds, then our
  main theorem implies the following: if $E$ is an elliptic curve over
  an arbitrary number field, then $E$ is potentially automorphic and
  satisfies the Sato--Tate conjecture. 

  In addition, we also prove some unconditional results. For example,
  in the setting of $\GL(2)$ over $\Q$, we identify certain minimal
  global deformation rings with the Hecke algebras acting on spaces of
  $p$-adic Katz modular forms of weight 1. Such algebras may well
  contain $p$-torsion. Moreover, we also completely solved the problem
  (for $p$ odd) of determining the multiplicity of an irreducible
  modular representation $\rhobar$ in the Jacobian $J_1(N)$, where $N$ is the
  minimal level such that $\rhobar$ arises in weight two.  
\end{abstract}

\maketitle

\newpage

{\tiny
 \tableofcontents
}

\section{Introduction}

In this paper, we prove a new kind of modularity lifting theorem for $p$-adic
Galois representations. Previous generalizations
of the work of Wiles~\cite{W} and Taylor--Wiles~\cite{TW} have (essentially) been
restricted to circumstances where the automorphic forms in question arise from
the middle degree cohomology of Shimura varieties. In particular, such
approaches
ultimately rely on a  ``numerical coincidence''
(see the introduction to~\cite{CHT}) which does not hold in general, and does not
hold in particular for $\GL(2)/F$ if $F$ is not totally real. A second requirement
of these generalizations
 is that the Galois representations in question are \emph{regular} at $\infty$,
that is, have distinct Hodge--Tate weights for all $v|p$. 
Our approach, in contrast, does not \emph{a priori} require either such assumption.

When considering questions of modularity in more general contexts, there are two issues
that need to be overcome. The first is that there do not seem to be ``enough'' automorphic
forms to account for all the Galois representations. In~\cite{CM,CV,CE}, the suggestion
is made that one should instead consider  \emph{integral} cohomology,
 and that the torsion occurring in these cohomology groups
may account for the missing automorphic forms. 
In order to make this approach work,
one needs to show that there is ``enough'' torsion. This is the problem that we solve
in some cases.
A second problem is the lack of Galois representations attached to 
these integral cohomology classes.
In particular, our methods require Galois representations associated 
to  torsion classes which do not necessarily lift
to characteristic zero, where one might hope to apply
the recent results of~\cite{HLTT}.
We do not resolve the problems of constructing Galois representations in this paper, and instead,
our results are contingent on  a conjecture which predicts that there exists a map
from a suitable deformation ring $R^{\min}$ to a Hecke algebra $\T$.
In a recent preprint, Scholze~\cite{Scholze} has constructed Galois representations associated
to certain torsion classes. If one can show that these Galois representations satisfy certain local-global
compatibility conditions (including showing that the Galois representations associated to
cohomology classes on which $U_v$ for $v|p$ is invertible are reducible after restriction to the decomposition
group at $v$), then our modularity lifting
theorems for imaginary quadratic fields would be unconditional.
There \emph{are} contexts, however, in which  
the
existence of Galois representations is known; in these cases we can produce unconditional
results.
In principle, our method currently applies in two contexts: 
\begin{enumerate}
\item[\bf{Betti}\rm] To Galois representations conjecturally arising from 
tempered $\pi$ of cohomological type associated to $G$, where
$G$ is reductive with a maximal compact $K$, maximal $\Q$-split torus $A$, and
$l_0 = \rank(G) - \rank(K) - \rank(A)$ is arbitrary,
\item[\bf{Coherent}] To Galois representations conjecturally arising from
tempered $\pi$ associated to $G$, where
$(G,X)$  is a Shimura variety over a totally real field $F$, and
such that $\pi_v$  is a non-degenerate limit of discrete series at~$\ell_0$ infinite places
and a discrete series at all other infinite places.
\end{enumerate}
 In practice, however, what we really need is  that
(after localizing at a suitable maximal ideal $\m$ of the Hecke algebra $\T$)  the cohomology
is concentrated in $l_0 + 1$ consecutive degrees.  (This is certainly true of the tempered representations
which occur in Betti cohomology. According to~\cite{BW}, the range of cohomological
degrees to which they occur has length $l_0 + 1$. In the coherent
case, the value of $l_0$ will depend on the infinity components
$\pi_v$ allowed. That tempered representations occur in a range of
length $l_0$, then follows from~\cite[Theorems 3.4 \&
3.5]{harris-ann-arb} together with knowledge of L-packets at infinite primes.)
The specialization of our approach to the case~$\ell_0 = 0$ exactly recovers the
usual  Taylor--Wiles method.

\medskip

The following results are a sample of what can be shown by these methods
in case {\bf Betti}\rm, assuming (Conjecture~\ref{conj:AA} of ~\S~\ref{conjectures}) the existence of Galois
representations in appropriate degrees satisfying the expected properties.

\begin{theorem}  \label{theorem:ST} Assume Conjecture~\ref{conj:AA}. Let $F$ be any number
field, and let $E$ be an elliptic curve over $F$.
Then the following hold:
\begin{enumerate}
\item $E$ is potentially modular.
\item The Sato--Tate conjecture is true for $E$.
\end{enumerate}
\end{theorem}

 The proof of Theorem~\ref{theorem:ST} relies on the following ingredients. The first ingredient consists of  the usual
 techniques in modularity lifting (the Taylor--Wiles--Kisin method) as augmented by  Taylor's
 Ihara's Lemma avoidance trick~\cite{Taylor}. The second ingredient is to observe that
 these arguments \emph{continue to hold} in a more general situation, 
 \emph{provided} that one can show that there is ``enough'' cohomology. Ultimately,
 this amounts to giving a lower
 bound on the depth of certain patched Hecke modules.
Finally, one can obtain such a lower bound by a commutative algebra argument, 
\emph{assuming} that the relevant cohomology occurs only in a certain range of length $l_0$.
Conjecture~\ref{conj:AA} amounts to assuming both the existence of Galois
representations together with the vanishing of cohomology (localized at an appropriate $\m$) outside a given range.
We deduce Theorem~\ref{theorem:ST} from a more general modularity lifting theorem,
see Theorem~\ref{theorem:modularity}.

\medskip

The following result is a sample of what can be shown by these methods
in case {\bf Betti}\rm \  assuming only Conjecture~\ref{conj:A} concerning the existence of Galois
representations for arithmetic lattices in $\GL_2(\OL_F)$ for an imaginary
quadratic field $F$. Unlike Conjecture~\ref{conj:AA}, it appears that Conjecture~\ref{conj:A} may well be
quite tractable in light of the work of~\cite{Scholze}.
 Let $\OL$ denote the ring of integers in a finite extension of $\Q_p$,
let $\varpi$ be a uniformizer of $\OL$, and let $\OL/\varpi = k$ be the residue field.
Say that a representation $\rho: G_F \rightarrow \GL_2(\OL)$ is semi-stable if
$\rho | I_v$ is unipotent for all finite $v$ not dividing $p$, and semi-stable in the sense of Fontaine~\cite{Fontaine}
if~$v|p$.  Furthermore, for $v|p$, we say that $\rho |D_v$ is finite flat if 
 for all $n \ge 1$
 each finite quotient $\rho | D_v \mod \varpi^n$ is the generic fiber of a finite flat $\OL$-group
scheme, and ordinary if  $\rho | D_v$ is conjugate to a representation
of the form
\[ \begin{pmatrix}
\epsilon\chi_1 & * \\
0 & \chi_2 
\end{pmatrix} \]
where $\chi_1$ and $\chi_2$ are unramified and $\epsilon$ is the
cyclotomic character.

\begin{theorem} \label{theorem:modularityimaginary}
Assume Conjecture~\ref{conj:A}. Let $F/\Q$ be an imaginary quadratic field.
Let $p \ge 3$ be unramified in $F$. Let
$$\rho: G_{F} \rightarrow \GL_2(\OL)$$
be a continuous semi-stable Galois representation with cyclotomic determinant unramified outside finitely many primes.
Let $\rhobar: G_{F} \rightarrow \GL_2(k)$ denote the mod-$\varpi$ reduction of $\rho$.
Suppose that
\begin{enumerate}
\item If $v|p$, the representation $\rho |D_v$ is either finite flat
  or ordinary.
  \item The restriction of $\rhobar$
to 
$\displaystyle{G_{F \kern-0.05em{\left(\zeta_p \right)}}}$
 is absolutely irreducible.
\item $\rhobar$ is modular of level $N(\rhobar)$, where $N(\rhobar)$ is the product of the usual
prime-to-$p$ Artin conductor and the primes $v|p$ where $\rhobar$ is not finite-flat. 
\item $\rho$ is minimally ramified.
\end{enumerate}
Then $\rho$ is modular, that is, there exists a regular algebraic cusp form $\pi$ for $\GL(2)/F$
such that $L(\rho,s) = L(\pi,s)$.
\end{theorem}

It is important to note that the condition $(3)$  is only a statement about the
existence of a mod-$p$ cohomology class of level $N(\rhobar)$, not the existence of a
characteristic zero lift. This condition is the natural  generalization of Serre's
conjecture.

\medskip

It turns out that --- even assuming Conjecture~\ref{conj:A} --- this is not enough to
prove that all minimal semi-stable elliptic curves over $F$ are modular. Even though
the Artin conjecture for finite two-dimensional solvable representations  of $G_F$ is known,
there are no
obvious congruences between eigenforms arising from Artin $L$-functions and cohomology classes
over $F$. (Over $\Q$, this arose from the happy accident that classical weight one forms could be interpreted
via coherent cohomology.) 
One class of mod-$p$ Galois representations known to
satisfy $(3)$ are the restrictions of odd Galois representations $\rhobar:G_{\Q} \rightarrow \GL_2(k)$
to $G_F$. 
One might imagine that the minimality condition is a result
of the lack of Ihara's lemma; however, Ihara's lemma
and level raising are known for $\GL(2)/F$ (see~\cite{CV}). The issue
arises because there is no analogue of Wiles' numerical
criterion for Gorenstein rings of dimension zero.

We deduce Theorem~\ref{theorem:modularityimaginary} from
the following more general result.
\begin{theorem} \label{theorem:RequalsT} Assume
conjecture~\ref{conj:A}. Let $F/\Q$ be an imaginary quadratic field.
Let $p \ge 3$ be unramified in $F$. Let
$$\rhobar: G_{F} \rightarrow \GL_2(k)$$
be a continuous  representation with cyclotomic determinant, and suppose that:
\begin{enumerate}
\item If $v|p$, the representation $\rhobar |D_v$ is either finite flat or ordinary.
\item $\rhobar$ is modular of level $N = N(\rhobar)$.
\item  $\rhobar| G_{F(\zeta_p)}$ is absolutely irreducible.
\item \label{cond:nonvexing}  If $\rhobar$ is  ramified at  $x$ where ${\NF}(x) \equiv -1 \mod p$,
then either $\rhobar | D_x$ is reducible or $\rhobar | I_x$ is absolutely irreducible.
\end{enumerate}
Let $R^{\min}$ denote the minimal finite flat (respectively, ordinary) deformation ring of $\rhobar$ with cyclotomic determinant.
Let $\T_{\m}$ be the algebra of Hecke operators acting on $H_1(Y_0(N),\OL)$ localized at the maximal
ideal corresponding to $\rhobar$. Then there is an isomorphism:
$$R^{\min} \stackrel{\sim}{\longrightarrow}\T_{\m},$$
and  there exists an integer
$\mu \ge 1$ such that $H_1(Y_0(N),\OL)_{\m}$ is free
of rank $\mu$ as a $\T_{\m}$-module.
If $H_1(Y_0(N),\OL)_{\m} \otimes \Q \ne 0$, then $\mu = 1$.
If $\dim(\T_{\m}) = 0$, then $\T_{\m}$ is a complete intersection.
\end{theorem}

Note that condition~\ref{cond:nonvexing} --- the non-existence of ``vexing primes''  $x$ such that
$\NF(x) \equiv -1 \mod p$ ---
is already a condition that arises in the original paper of Wiles~\cite{W}.
It could presumably be
removed by making the appropriate modifications as in either~\cite{DiamondVexingTwo, DiamondVexing}
or~\cite{CDT} and making the corresponding modifications to Conjecture~\ref{conj:A}.

\medskip

Our results  are obtained by applying a modification of
Taylor--Wiles to the Betti cohomology of arithmetic manifolds.
In such a context, it seems difficult to construct Galois representations whenever $l_0 \ne 0$.
Following~\cite{Pilloni} and~\cite{HarrisCoherent}, however, we may  also apply our methods to the
\emph{coherent} cohomology of Shimura varieties, where Galois representations are more readily available.
 In contexts where the underlying automorphic forms
$\pi$ are discrete series at infinity, one expects (and in many cases can prove, see~\cite{LanSuh}) that
the integral coherent cohomology localized  at  a suitably generic maximal ideal $\m$ of $\T$ vanishes
outside the middle degree. If $\pi_{\infty}$ is a \emph{limit} of discrete series, however, 
(so that we are in case {\bf Coherent}\rm) then the cohomology
of the associated coherent sheaf can sometimes be shown to be non-zero 
in exactly in the expected number of degrees, in which case our methods apply. 
In particular, \emph{a priori}, the {\bf Coherent} case appears more tractable,
since there are available methods for constructing Galois representations to coherent cohomology classes
in low degree~\cite{DeligneSerre,wushi}. However, our methods \emph{still} lead to open conjectures concerning the existence of Galois representations, since the usual methods for constructing
representations on torsion classes (using congruences) only work with Hecke actions on $H^0(X,\Ee)$ rather than $H^i(X,\Ee)$ for $i > 0$,
and we require Galois representations coming from the latter groups.

In this paper, we confine our discussion of the general
{\bf Coherent\rm} case to addressing the problem of constructing suitable complexes (see Section~\ref{section:nakajima}).
We expect, however, that our methods may be successfully applied to prove 
unconditional modularity lifting theorems in a number of interesting cases in small rank.
The most well known example of such a situation is the case of classical modular forms of weight $1$.
Such modular forms contribute to the  cohomology of
$H^0(X_1(N),\omega)$ \emph{and} $H^1(X_1(N),\omega)$ in characteristic zero, where $X_1(N)$ is a modular curve, and
$\omega$ is the usual pushforward $\pi_{*} \omega_{\mathcal{E}/X_1(N)}$ of the relative dualizing sheaf
along the universal generalized elliptic curve. Working over $\Z_p$
for some prime $p\nmid N$, the group $H^0(X_1(N),\omega)$ is torsion free, but $H^1(X_1(N),\omega)$ is not torsion free in general,
as predicted by Serre and confirmed by Mestre (Appendix~A of~\cite{Bas}).
In order to deal with vexing primes, we introduce
a 
vector bundle 
$\CL_\sigma$ which plays the role
of the locally constant sheaf
$\CF_{M}$ of section~\S~6 of~\cite{CDT}  --- see \S~\ref{section:vexing} for details.
We also introduce a curve $X_U$ which sits in the sequence
$X_1(M) \rightarrow X_U \rightarrow X_0(M)$ for some $M$ dividing the Serre conductor
of $\rhobar$ and such that the first map has
$p$-power degree.
Note that if $\rhobar$ has no vexing primes, then  $\CL_{\sigma}$  is
trivial of rank 1,
and if $\rhobar$ is not ramified at any primes congruent to $1 \mod p$, then $X_U = X_1(N)$.
In this context, we prove the following result:
\begin{theorem} \label{weightone} Suppose that $p \ge 3$. Let
$\rhobar: G_{\Q} \rightarrow \GL_2(k)$
be an odd continuous  irreducible Galois representation of
Serre level $N$.
Assume that $\rhobar$ is unramified at $p$.
Let $R^{\min}$ denote the universal minimal unramified-at-$p$ deformation ring of $\rhobar$.
Then there exists a quotient
$X_U$ of $X_1(N)$ and a vector bundle $\CL_{\sigma}$ on $X_U$ such
that  if $\T$ denotes the Hecke algebra of $H^1(X_U,\omega \otimes \CL_\sigma)$, there is an isomorphism
$$R^{\min} \stackrel{\sim}{\longrightarrow}\T_{\m}$$
where $\m$ is the maximal ideal of $\T$ corresponding to $\rhobar$. Moreover,
$H^1(X_U,\omega \otimes  \CL_\sigma)_{\m}$ is free as a $\T_{\m}$-module.
\end{theorem}

Note that even the fact  that there \emph{exists} a surjective map
from $R^{\min}$ to $\T_{\m}$ is non-trivial, and requires us to prove
a local--global compatibility result for Galois representations associated
to Katz modular forms of weight one over any $\Z_p$-algebra (see Theorem~\ref{theorem:galoisweight1}).
  We immediately deduce  from Theorem~\ref{weightone} the following:

\begin{corr}
Suppose that $p \ge 3$. Suppose also that
$\rho: G_{\Q} \rightarrow \GL_2(\OL)$ is a continuous
representation satisfying the following conditions. 
\begin{enumerate}
\item For all primes $v$, either  $\rho(I_v)  \iso \rhobar(I_v)$ or
$\dim(\rho^{I_v}) = \dim(\rhobar^{I_v}) = 1$.
\item $\rhobar$ is odd and irreducible.
\item $\rho$ is unramified at $p$. 
\end{enumerate}
Then $\rho$ is modular of weight one.
\end{corr}

It is instructive to compare this theorem and
the corollary to the main theorem of Buzzard--Taylor~\cite{BuzzT} (see also~\cite{BuzzWild}).
Note that the hypothesis in that paper
that $\rhobar$ is modular is no longer necessary, following the proof of Serre's conjecture~\cite{KhareW}.
In both cases, if $\rho$ is a deformation of $\rhobar$ to
a field of characteristic zero, we deduce that $\rho$
is modular of weight one, and hence has finite image.
 The method of~\cite{BuzzT}  applies in a non-minimal situation,
 but it requires the hypothesis that
 $\rhobar(\Frob_p)$ has distinct eigenvalues.
 Moreover, it has the disadvantage that it only gives an identification of 
reduced points on the generic fibre (equivalently, that $R^{\min}[1/p]^{\mathrm{red}} =  \T_{\m}[1/p]$,
although from this by class field theory --- see Lemma~\ref{lemma:forbrian} and the subsequent remarks after the proof --- one may deduce that $R^{\min}[1/p] = \T_{\m}[1/p]$),
 and says nothing about the torsion structure of
$H^1(X_1(N),\omega)$.  
 Contrastingly, we may deduce the following result:

\begin{corr} Suppose that $p\ge 3$. Let $\rhobar:G_{\Q} \rightarrow
  \GL_2(k)$ be odd, continuous, irreducible,
and unramified at $p$. 
Let $(A,\m)$ denote a complete local Noetherian $\OL$-algebra with residue field $k$ and $\rho:G_{\Q} \rightarrow \GL_2(A)$ a minimal deformation of $\rhobar$.
Then $\rho$ has finite image.
\end{corr}

This gives the first results towards Boston's strengthening of the Fontaine--Mazur conjecture for representations
unramified at $p$ (See~\cite{BostonII}, Conjecture~2).

\medskip

It is natural to ask whether our results can be modified using Kisin's method to yield modularity
lifting results in non-minimal level. Although the formalism of this method can be adapted to our context,
there is a genuine difficultly in proving that the support of
$\Spec(\T_{\infty}[1/p])$ hits each of the components of $\Spec(R_{\infty}[1/p])$  whenever the latter has
more than one component. In certain situations, we may apply Taylor's trick~\cite{Taylor}, but this can not
be made to work in general.  However, suppose one replaces the ``minimal'' condition away from $p$ with 
the following condition:
\begin{itemize}
\item If $\rho$ is special at $x \nmid p$, and $\rhobar$ is unramified at $x$, then $x \equiv 1 \mod p$.
\end{itemize}
In this context our methods should yield that the deformation ring $R$ acts nearly faithfully on
$H^1(X_1(M),\omega)_{\m}$ for an appropriate $M$. This is sufficient for applications to the
conjectures of Fontaine--Mazur and Boston.

\medskip

In the process of proving our main result, we 
also completely solved the problem (for $p$ odd)   of determining the multiplicity of  an
irreducible modular representation $\rhobar$
in  the Jacobian $J_1(N^*)[\m]$, where $N^*$ is the minimal level such that
$\rhobar$ arises in weight two.  In particular, we prove that when $\rhobar$
is unramified at $p$ and $\rhobar(\Frob_p)$ is a scalar, then the multiplicity of $\rhobar$
is two (see Theorem~\ref{theorem:multiplicitytwo}). (In all other cases, the multiplicity was
already 
known to be one --- and in the exceptional cases we consider, the multiplicity was 
also known to be $\ge 2$.)

\medskip

Finally, we outline here the structure of the paper, which has two parts. In
Part~\ref{part:1}, we treat the case where $l_0=1$ in two specific
instances -- namely, the case of classical modular forms of weight 1 (Section~\ref{sec:weight-one}) and the case of automorphic forms on $\GL(2)$
over a quadratic imaginary field that contribute to the Betti
cohomology (Section~\ref{sec:imag-quadr-fields}). The
ideas from commutative algebra and the abstract Taylor--Wiles patching
method necessary to treat these two situations are developed in
Section~\ref{sec:com-alg-I}. The `multiplicity two' result mentioned
above is proved in Section~\ref{sec:complements}.

In Part~\ref{part:2}, we treat the case of general $l_0$. In contrast
to Part~\ref{part:1}, we only treat the Betti case in detail (more
specifically, we consider the Betti cohomology of the locally
symmetric spaces associated to $\GL(n)$ over a general number field).
Section~\ref{section:commutetwo} contains the results from commutative
algebra and the abstract Taylor--Wiles style patching result that
underlie our approach to the case of general $l_0$. These techniques
are more `derived' in nature than the techniques that treat $l_0=0$ or
$1$, and in particular rely on the existence of complexes which compute
cohomology and satisfy various desirable properties. The existence of
such complexes is proved in Section~\ref{sec:complexes} (in both
contexts --- Betti cohomology and coherent cohomology). In
Section~\ref{sec:galois-deformations}, we consider the Galois
deformation side of our arguments. In
Section~\ref{sec:hom-arithm-quot}, we consider cohomology and Hecke
algebras. This section contains Conjecture~\ref{conj:AA} on the
existence of Galois representations as well our main modularity
lifting theorem. Finally, Section~\ref{sec:proof-of-ST} contains the
proof of Theorem~\ref{theorem:ST} above.

\section*{Notation}

In this paper, we fix a prime $p\geq 3$ and let $\OL$ denote the ring of integers in a finite
extension $K$ of $\Q_p$. We let $\varpi$ denote a uniformizer in $\OL$
and let $k = \OL/\varpi$ be the residue field. We denote by $\CC_\OL$
the category of complete Noetherian local $\OL$-algebras with residue
field $k$. The homomorphisms in $\CC_\OL$ are the continuous $\OL$-algebra homomorphisms. If $G$ is a group and $\chi : G \to k^\times$ is a
character, we denote by $\langle \chi\rangle : G \to \OL^\times$
the Teichm\"uller lift of $\chi$.

If $F$ is a field, we let $G_F$ denote the Galois group
$\mathrm{Gal}(\overline F/F)$ for some choice of algebraic closure $\overline
F/F$. We let $\epsilon : G_F \to \Z_p^\times$ denote the $p$-adic
cyclotomic character. If $F$ is a number field and $v$ is a prime of $F$, we  let
$\OL_v$ denote the ring of integers in the completion of $F$ at $v$
and we let $\pi_v$ denote a uniformizer in $\OL_{v}$. We denote
$G_{F_v}$ by $G_v$ and let $I_v \subset G_v$ be the inertia group. We
also let $\Frob_v \in G_v/I_v$ denote the \emph{arithmetic}
Frobenius. We let $\Art : F_v^\times \to W_{F_v}^{\ab}$ denote the
local Artin map, normalized to send uniformizers to geometric
Frobenius lifts. We will also sometimes denote the decomposition group
at $v$ by $D_v$. If $R$ is a topological ring and $\alpha\in R^\times$, we let
$\lambda(\alpha) : G_v \to R^\times$ denoted the continuous unramified character
which sends $\Frob_v$ to $\alpha$, when such a character exists. We let $\A_F$ and
$\A_F^\infty$ denote the adeles and finite adeles of $F$
respectively. If $F=\Q$, we simply write $\A$ and $\A^\infty$.

If $P$ is a bounded complex of $S$-modules for some ring $S$, then we
let $H^*(P)=\oplus_i H^i(P)$. Any map $H^*(P)\to H^*(P)$ will be
assumed to be degree preserving. If $R$ is a ring, by a
\emph{perfect complex of $R$-modules} we mean a bounded complex of finitely generated projective $R$-modules. 

If $R$ is a local ring, we will
sometimes denote the maximal ideal of $R$ by $\m_R$.

\part{\texorpdfstring{$l_0$}{l0} equals \texorpdfstring{$1$}{1}.}
\label{part:1}

\section{Some Commutative Algebra I}
\label{sec:com-alg-I}

This section contains one of  the main new technical innovations of this paper.
The issue, as mentioned in the introduction, is to show that there are \emph{enough}
modular Galois representations. 
This involves showing that certain modules $H_N$ (consisting of
modular forms) for the group rings $S_N:=\OL[(\Z/p^N\Z)^q]$
compile, in a  Taylor--Wiles patching process, to form a module of
codimension one over the completed group ring $S_\infty:=
\OL[[(\Z_p)^q]]$. The problem then becomes to  find a suitable
notion of ``codimension one'' for modules over a local ring that
\begin{enumerate}
\item is well behaved for non-reduced quotients of power series rings
over $\OL$ (like $S_N$),
\item can be established for the spaces $H_N$ in question,
\item compiles well in a Taylor--Wiles system.
\end{enumerate}

It turns out that the correct notion is that of being ``balanced'', a notion defined below.
When  $l_0 > 1$, we shall ultimately be required to patch more information than simply
the modules $H_N$; rather, we shall patch entire complexes (see~\S~\ref{section:commutetwo}).

\subsection{Balanced Modules}

Let $S$ be a Noetherian local ring with residue field $k$ and let $M$ be a
finitely generated $S$-module.

\begin{df}
  \label{defn:defect}
We define the \emph{defect} $d_S(M)$ of $M$ to be 
\[ d_S(M)= \dim_k \Tor^0_S(M,k)-\dim_k \Tor^1_S(M,k)=\dim_k M/\m_S M -
\dim_k \Tor^1_S(M,k).\]
\end{df}

Let
\[ \dots \ra P_i \ra \dots \ra P_1 \ra P_0 \ra M \ra 0\]
be a (possibly infinite) resolution of $M$ by finite free
$S$-modules. Assume that the image of $P_i$ in $P_{i-1}$ is
contained in $\m_S P_{i-1}$ for each $i\geq 1$. (Such resolutions
always exist and are often
called `minimal'.) Let $r_i$ denote the rank of $P_i$. Tensoring
the resolution over $S$ with $k$ we see that $P_i/\m_S P_i \cong
\Tor^i_S(M,k)$ and hence that $r_i = \dim_k \Tor^i_S(M,k)$.

\begin{df}
  \label{defn:balanced}
We say that $M$ is \emph{balanced} if $d_S(M)\geq 0$.
\end{df}

If $M$ is balanced, then we see that it admits a presentation 
\[ S^d \ra S^d \ra M \ra 0 \]
with $d = \dim_k M/\m_S M$.

\subsection{Patching}
\label{sec:patchingimaginary}

We establish in this section an abstract  Taylor--Wiles style patching
result which may be viewed as an analogue of Theorem 2.1
of~\cite{DiamondMult}. This result will be one of the key ingredients in the proofs of our main theorems.

\begin{prop}
\label{prop:patchingimaginary}
Suppose that
\begin{enumerate}[\ \ \ \ (1)]
\item $R$ is an object of $\CC_\OL$ and $H$ is a finite $R$-module
  which is also finite over $\OL$;
\item $q \geq 1$ is an integer, 
                                and for each integer $N\geq 1$,  
$S_N:=\OL[\Delta_N]$ with $\Delta_N:=(\Z/p^N\Z)^q$;
\item  $R_\infty:= \OL[[x_1,\dots,x_{q-1}]]$;
\item \label{cond4} for each $N\geq 1$, $\phi_N: R_\infty \onto R$ is a surjection
  in $\CC_\OL$ and $H_N$ is an $R_\infty\otimes_\OL
  S_N$-module.
\item \label{cond5}
For each $N\geq 1$ the following conditions are satisfied
\begin{enumerate}[\ \ \ \ (a)]
 \item\label{cond-image} the image of $S_N$ in $\End_\OL(H_N)$ is contained in the image
   of $R_\infty$ and moreover, the image of the augmentation ideal of
   $S_N$ in $\End_{\OL}(H_N)$ is contained in the image of $\ker(\phi_N)$;
 \item\label{cond-coninvts} there is an isomorphism $\psi_N: (H_N)_{\Delta_N} \iso H$ of
   $R_\infty$-modules (where $R_\infty$ acts on $H$ via $\phi_N$);
 \item\label{cond-balanced} $H_N$ is finite and balanced over $S_N$ (see Definition \ref{defn:balanced}).
\end{enumerate}
\end{enumerate}
Then $H$ is a free $R$-module.
\end{prop}

\begin{proof}
  Let $S_\infty=\OL[[(\Z_p)^q]]$ and let $\a$ denote
  the augmentation ideal of $S_\infty$ (that is, the kernel of the
  homomorphism $S_\infty \onto \OL$ which sends each element of
  $(\Z_p)^q$ to 1). For each $N\geq 1$, let
  $\a_N$ denote the kernel of the natural
  surjection $S_\infty \onto S_N$ and let $\b_N$ denote the open ideal
  of $S_\infty$ generated by $\varpi^N$ and $\a_N$.  Let $d = \dim_k
  (H/\varpi H)$. We may assume that $d>0$ since otherwise $H=\{0\}$
  and the result is trivially true.  Choose a
  sequence of open ideals $(\d_N)_{N\ge 1}$ of $R$ such that
  \begin{itemize}
  \item $\d_N\supset \d_{N+1}$ for all $N\ge 1$;
  \item $\cap_{N\ge 1}\d_N = (0)$;
  \item $\varpi^NR \subset \d_N \subset \varpi^NR+\Ann_R(H)$ for all $N$.
  \end{itemize}
(For example, one can take $\d_N$ to be the ideal generated by
$\varpi^N$ and $\Ann_R(H)^N$. These are open ideals since
$R/\Ann_R(H)\subset \End_{\OL}(H)$
is finite as an $\OL$-module.)

Define a \emph{patching datum of level $N$} to be a 4-tuple $(\phi,X,\psi,P)$
where
\begin{itemize}
\item $\phi : R_\infty \onto R/\d_N$ is a surjection in $\CC_\OL$;
\item $X$ is an $R_\infty\widehat\otimes_\OL S_\infty$-module such that the action of $S_\infty$ on $X$ factors through
    $S_\infty/\b_N$ and $X$ is finite over $S_\infty$;
\item $\psi : X/\a X \iso H/\varpi^NH$ is an isomorphism of $R_\infty$
  modules (where $R_\infty$ acts on $H/\varpi^N H$ via $\phi$);
\item $P$ is a presentation
\[ (S_\infty/\b_N)^d \to (S_\infty/\b_N)^d \to X \to 0. \]
\end{itemize}
We say that two such 4-tuples $(\phi,X,\psi,P)$ and
$(\phi',X',\psi',P')$ are isomorphic if 
\begin{itemize}
\item $\phi = \phi'$;
\item there is an isomorphism $X \iso X'$ of $R_\infty\widehat\otimes_\OL
  S_\infty$ modules compatible with $\psi$ and $\psi'$, and with the
  presentations $P$ and $P'$.
\end{itemize}
We note that there are only finitely many isomorphism classes of patching
data of level $N$. (This follows from the fact that $R_\infty$ and
$S_\infty$ are topologically finitely generated.) If $D$ is a patching
datum of level $N$ and $1\le N' \le N$, then $D$ gives rise to
patching datum of level $N'$ in an obvious fashion. We denote this
datum by $D \bmod N'$.

For each pair of integers $(M,N)$ with $M\ge N\ge 1$, we define a
patching datum $D_{M,N}$ of level $N$ as follows: the statement of the
proposition gives a homomorphism $\phi_M:R_\infty \onto R$ and an
$R_\infty\otimes_\OL S_M$-module $H_M$. We take
\begin{itemize}
\item $\phi$ to be the composition $R_\infty\onto R \onto R/\d_N$;
\item $X$ to be $H_M/\b_N$;
\item $\psi : X/\a X \iso H/\b_N$ to be the reduction modulo $\varpi^N$
  of the given isomorphism $\psi_M:H_M/ \a H_M \iso H$;
\item $P$ to be any choice of presentation
\[ (S_\infty/\b_N)^d \to (S_\infty/\b_N)^d \to X \to 0. \]
(The facts that $H_M/\a H_M \iso H$ and $d_{S_M}(H_M)\geq 0$ imply
that such a presentation exists.)
\end{itemize}

Since there are finitely many patching data of each level $N\ge 1$, up
to isomorphism, we can find a sequence of pairs $(M_i,N_i)_{i\geq 1}$
such that
\begin{itemize}
\item $M_i \ge N_i$, $M_{i+1}> M_i$, and $N_{i+1}> N_i$ for all $i$;
\item $D_{M_{i+1},N_{i+1}} \bmod N_{i}$ is isomorphic to
  $D_{M_i,N_i}$ for all $i\ge 1$.
\end{itemize}
For each $i\ge 1$, we write $D_{M_i,N_i}=(\phi_i,X_i,\psi_i,P_i)$ and
we fix an isomorphism
between the modules $X_{i+1}/\b_{N_i}X_{i+1}$ and $X_i$ giving rise to an
isomorphism between $D_{M_{i+1},N_{i+1}}\bmod N_i$ and $D_{M_i,N_i}$. We define
\begin{itemize}
\item $\phi_\infty : R_\infty \onto R$ to be the inverse limit of the $\phi_i$;
\item $X_\infty := \varprojlim_i X_i$ where the map $X_{i+1} \to X_i$
  is the composition $X_{i+1}\onto X_{i+1}/\b_{N_i}\iso X_i$;
\item $\psi_\infty$ to be the isomorphism of $R_\infty$-modules $X_\infty/\a X_\infty \iso H$
  (where $R_\infty$ acts on $H$ via $\phi_\infty$) arising from the isomorphisms $\psi_i$;
\item $P_\infty$ to be the presentation
\[ S_\infty^d \to S_\infty^d \to X_\infty \to 0 \]
obtained from the $P_i$. (Exactness follows from the Mittag--Leffler condition.)
\end{itemize}

Then $X_\infty$ is an $R_\infty\widehat\otimes_\OL S_\infty$-module,
and the image of $S_\infty$ in $\End_\OL(X_\infty)$ is contained in
the image of $R_\infty$. (By condition~\ref{cond-image}, the image of
$S_\infty$ in each $\End_\OL(X_i)$ is contained in the image of
$R_\infty$. The same containment of images then holds in each
$\Hom_{\OL}(X_\infty,X_i)$ and hence in
$\End_{\OL}(X_\infty)=\varprojlim_i \Hom_{\OL}(X_\infty,X_i)$.)
 It follows that $X_\infty$ is a finite
$R_\infty$-module. Since $S_\infty$ is formally smooth over $\OL$, we
can and do choose a homomorphism $\imath:S_\infty \to R_\infty$ in $\CC_\OL$,
compatible with the actions of $S_\infty$ and $R_\infty$ on
$X_\infty$. 

Since $\dim_{S_\infty}(X_\infty) = \dim_{R_\infty}(X_\infty)$ and
$\dim R_\infty < \dim S_\infty$, we deduce that
$\dim_{S_\infty}(X_\infty)<\dim S_\infty$. It follows that the first
map $S_\infty^d \to S_\infty^d$ in the presentation $P_\infty$ is
injective. (Denote the kernel by $K$. If $K\neq (0)$, then
$K\otimes_{S_\infty}\Frac(S_\infty) \neq (0)$ and hence
$X_\infty\otimes_{S_\infty}\Frac(S_\infty)\neq (0)$, which is
impossible.) We see that $P_\infty$ is a minimal projective resolution
of $X_\infty$, and by the Auslander--Buchsbaum formula, we deduce that
$\depth_{S_\infty}(X_\infty) = \dim(S_\infty)-1$. Since
$\depth_{R_\infty}(X_\infty)=\depth_{S_\infty}(X_\infty)$, it follows that
$\depth_{R_\infty}(X_\infty)=\dim(R_\infty)$, and applying the
Auslander--Buchsbaum formula again, we deduce that $X_\infty$ is free
over $R_\infty$. Using this and the second part of
condition~\eqref{cond-image}, we also deduce that
$\imath(\a)\subset \ker(\phi_\infty)$.

Finally, the existence of the isomorphism $\psi_\infty:X_\infty/\a
X_\infty \iso H$ tells us that $H$ is free over
$R_\infty/\imath(\a)R_\infty$. However, since the action of $R_\infty$
on $H$ also factors through the quotient
$R_\infty/\ker(\phi_\infty)=R$ and since
$\imath(a)\subset\ker(\phi_\infty)$, we deduce that
 $R_\infty/\imath(\a)R_\infty \cong R$
and that $R$ acts freely on $H$.
\end{proof}

\section{Weight One Forms}
\label{sec:weight-one}

\subsection{Deformations of Galois Representations}
\label{sec:deform-galo-repr-w1}

Let
$$\rhobar: G_{\Q} \rightarrow \GL_2(k)$$
be a continuous, odd, absolutely irreducible  Galois representation.
Let us suppose that $\rhobar | G_p$ is unramified; this implies that
$\rhobar$ remains absolutely irreducible when restricted to $G_{\Q(\zeta_p)}$.
Let $S(\rhobar)$ denote the set of primes of $\Q$ at which $\rhobar$ is
ramified and let $T(\rhobar)\subset S(\rhobar)$ be the subset consisting of those primes
$x$ such that $x\equiv -1 \mod p$, $\rhobar|G_x$ is irreducible and
$\rhobar|I_x$ is reducible. Following Diamond, we call the primes in
$T(\rhobar)$ \emph{vexing}. We further assume that if $x\in
S(\rhobar)$ and $\rhobar|G_x$ is reducible, then $\rhobar^{I_x}\neq
(0)$. Note that this last condition is always satisfied by a twist of
$\rhobar$ by a character unramified outside of $S(\rhobar)$.

Let $Q$  denote a finite set of primes of $\Q$ disjoint from $S(\rhobar)\cup\{p\}$.
(By abuse of notation, we sometimes use $Q$ to denote the product of primes in $Q$.)
For objects $R$ in $\CC_\OL$, a \emph{deformation} of $\rhobar$ to
$R$ is a $\ker(\GL_2(R)\to \GL_2(k))$-conjugacy class of continuous
lifts $\rho : G_\Q\to\GL_2(R)$ of $\rhobar$. We will often refer to
the deformation containing a lift $\rho$ simply by $\rho$. 

\begin{df}
  \label{defn:minimal}
We say that
a deformation $\rho:G_\Q\to \GL_2(R)$ of $\rhobar$ is \emph{minimal
  outside $Q$}
if it satisfies the following properties:
\begin{enumerate}
\item\label{det} The determinant $\det(\rho)$ is equal to the Teichm\"uller lift of $\det(\rhobar)$.
\item\label{outside-NQ} If $x\not\in Q\cup S(\rhobar)$ is a prime of $\Q$, then $\rho|G_x$ is unramified.
\item\label{at-vexing} If $x\in T(\rhobar)$, then $\rho(I_x)\iso \rhobar(I_x)$.
\item\label{at-principal} If $x \in S(\rhobar)- T(\rhobar)$ and $\rhobar|G_x$ is
  reducible, then $\rho^{I_x}$ is a rank one direct summand of $\rho$
  as an $R$-module.
\end{enumerate}
If $Q$ is empty, we will refer to such deformations simply as being
\emph{minimal}.
\end{df}

Note that condition~\ref{outside-NQ} implies that $\rho$ is unramified at $p$. 
 The functor that associates to each object $R$ of
$\CC_\OL$ the set of deformations of $\rhobar$ to $R$ which are
minimal outside $Q$ is represented by a complete Noetherian local
$\OL$-algebra $R_Q$. This follows from the proof
of Theorem 2.41 of ~\cite{DDT}. If $Q=\emptyset$, we will
sometimes denote $R_Q$ by $R^{\min}$. Let $H^1_{Q}(\Q,\ad^0 \rhobar)$
denote the Selmer group defined as the kernel of the map
\[ H^1(\Q,\ad^0 \rhobar) \lra \bigoplus_{x} H^1(\Q_x,\ad^0\rhobar)/L_{Q,x}\]
where $x$ runs over all primes of $\Q$ and 
\begin{itemize}
\item $L_{Q,x} =
  H^1(G_x/I_x,(\ad^0\rhobar)^{I_x})$ if $x
  \not \in Q$;
\item $L_{Q,x} =
  H^1(\Q_x,\ad^0\rhobar)$ if $x \in Q$.
\end{itemize}
Let $H^1_{Q}(\Q,\ad^0 \rhobar(1))$ denote the corresponding
dual Selmer group.

\begin{prop}
\label{prop:tangent-space-w1}
 The reduced tangent space $\Hom(R_Q/\m_\OL,k[\epsilon]/\epsilon^2)$ of $R_{Q}$ has
  dimension  
$$\dim_k H^1_{Q}(\Q,\ad^0 \rhobar(1)) - 1 + 
\sum_{x \in Q} \dim_k H^0(\Q_x,
\ad^0 \rhobar(1)).$$
\end{prop}

\begin{proof} The argument is very similar to that of
Corollary~2.43 of~\cite{DDT}. The reduced tangent space has dimension
$\dim_k H^1_Q(\Q,\ad^0\rhobar)$. By Theorem 2.18 of
\emph{op.\ cit.}\ this is equal to 
$$\begin{aligned}  \dim_k H^1_Q(\Q,\ad^0\rhobar(1)) + \dim_k
 & H^0(\Q,\ad^0\rhobar)-\dim_k H^0(\Q,\ad^0\rhobar(1)) \\ 
+ \sum_{x} (\dim_k L_{Q,x}-\dim_k &
H^0(\Q_x,\ad^0\rhobar)) -1,
\end{aligned}$$
where $x$ runs over all finite places of $\Q$. The final term is
the contribution at the infinite place. The second and third terms
vanish by the absolute irreducibility of $\rhobar$ and the fact that
$\rhobar|G_p$ is unramified. Finally, as in the proof of Corollary 2.43
of \emph{loc.\ cit.}\ we see that the contribution at the prime
$x$ vanishes if $x\not\in Q$, and equals $\dim_k
H^0(\Q_x,\ad^0\rhobar(1))$ if $x \in Q$.
\end{proof}

Suppose that $x \equiv 1\mod p$ and $\rhobar(\Frob_x)$ has distinct eigenvalues for
each $x \in Q$. Then $H^0(\Q_x,\ad^0\rhobar(1))$ is one
dimensional for $x\in Q$ and the preceding proposition
shows that the reduced tangent space of $R_Q$ has dimension
\[ \dim_k H^1_{Q}(\Q,\ad^0 \rhobar(1)) - 1 + \# Q.\]
Using this fact and the argument of Theorem 2.49 of \cite{DDT}, we
deduce the following result. (We remind the reader that $\rhobar|G_{\Q(\zeta_p)}$ is absolutely
  irreducible, by assumption.)

\begin{prop} 
\label{prop:tw-primes-w1}
Let $q =\dim_k H^1_{\emptyset}(\Q,\ad^0
  \rhobar(1))$. Then $q\geq 1$ and for any integer $N\geq 1$ we can find a set $Q_N$
  of primes of $\Q$ such that 
\begin{enumerate}
\item $\# Q_N =q$.
\item $x \equiv 1 \mod p^N$ for each $x\in Q_N$.
\item For each $x\in Q_N$, $\rhobar$ is unramified at $x$ and $\rhobar(\Frob_x)$ has distinct
eigenvalues.
\item $H^1_{Q_N}(\Q,\ad^0\rhobar(1))=(0)$.
\end{enumerate}
 In particular, the reduced tangent space of $R_{Q_N}$ has dimension
  $q-1$ and $R_{Q_N}$ is a quotient of a power series
  ring over $\OL$ in $q-1$ variables.
\end{prop}

We note that the calculations on the Galois side are virtually identical to those that occur
in Wiles' original paper,
with the caveat that the tangent space is of dimension ``one less'' in our case.
On the automorphic side, this $-1$ will be a reflection of the fact that
the Hecke algebras will not (in general) be flat over $\OL$ and the
modular forms we are interested in will contribute to one extra degree
of cohomology.

\subsection{Cohomology of Modular Curves}
\label{sec:cohom-modul-curv}

\subsubsection{Modular Curves} 
\label{sec:mod-curves}
 We begin by recalling some classical facts regarding modular curves.
  Fix an integer $N\geq 5$ such that $(N,p) = 1$, and fix a
 squarefree integer $Q$ with $(Q,Np) = 1$.  Let $X_1(N)$, $X_1(N;Q)$, and
 $X_1(NQ)$ denote the modular curves of level $\Gamma_1(N)$,
 $\Gamma_1(N) \cap \Gamma_0(Q)$, and $\Gamma_1(N) \cap \Gamma_1(Q)$
 respectively as smooth proper schemes over $\Spec(\OL)$.
 To be
 precise, we take $X_1(N)$ and $X_1(NQ)$ to be the base change to
 $\Spec(\OL)$ of the curves denoted by the same symbols
 in \cite[Proposition 2.1]{Gross}. Thus, $X_1(NQ)$
 represents the functor that assigns to each
 $\OL$-scheme $S$ the set of isomorphism classes of triples
 $(E,\alpha_{NQ})$ where $E/S$ is a generalized elliptic curve and $\alpha_{NQ}
 : \mu_{NQ} \into E[NQ]$ is an embedding of group schemes whose image
 meets every irreducible component in each geometric fibre. Given such a
 triple, we can naturally decompose $\alpha_{NQ} = \alpha_N \times
 \alpha_Q$ into its $N$ and $Q$-parts. The group
 $(\Z/N Q\Z)^\times$ acts on $X_1(NQ)$ in the following fashion:
 $a\in(\Z/N Q\Z)^\times$ sends a pair $(E,\alpha_{NQ})$ to
 $\langle a \rangle (E,\alpha_{NQ}):=(E,a \circ \alpha_{NQ})$.
 We let
 $X_1(N;Q)$ be the smooth proper curve over $\Spec(\OL)$ classifying
 triples $(E,\alpha_N,C_Q)$ where $E$ is a generalized elliptic curve,
 $\alpha_N : \mu_N \into E[N]$ is an embedding of group schemes and
 $C_Q \subset E[Q]$ is a subgroup \'{e}tale locally isomorphic to
 $\Z/Q\Z$ and such that the subgroup $Q+\alpha_N(\mu_N)$ of
 $E$ meets every irreducible component of every geometric fibre of $E$.
 Then $X_1(N,Q)$ is the quotient of $X_1(NQ)$
 by the action of $(\Z/Q\Z)^\times \subset (\Z/N Q\Z)^\times$ and the map
 $X_1(NQ)\to X_1(N;Q)$ is \'{e}tale; on points it sends $(E,\alpha_{NQ})$
 to $(E,\alpha_N,\alpha_Q(\mu_Q))$.
 
 For any modular curve $X$ over $\OL$, let $Y \subset X$ denote the
 corresponding open modular curve parametrizing genuine elliptic
 curves.  Let $\pi: \mathcal{E} \rightarrow X$ denote the universal
 generalized elliptic curve, and let $\omega:= \pi_*
 \omega_{\mathcal{E}/X}$, where $\omega_{\mathcal{E}/X}$ is the
 relative dualizing sheaf.  Then the Kodaira--Spencer map
 (see~\cite{Katz}, A1.3.17) induces an isomorphism $\omega^{\otimes 2}
 \simeq \Omega^{1}_{Y/\OL}$ over $Y$, which extends to an isomorphism
 $\omega^{\otimes 2} \simeq \Omega^{1}_{X/\OL}(\cusps)$, where
 $\cusps$ is the reduced divisor supported on the cusps. If $R$ is an
 $\OL$-algebra, we let $X_R= X\times_{\Spec \OL}\Spec R$. If $M$ is an
 $\OL$-module and $\L$ is a coherent sheaf on $X$, we let $\L_M$
 denote $\L\otimes_{\OL}M$.

 We now fix a subgroup $H$ of $(\Z/N\Z)^\times$. We let $X$ (resp.\
 $X_1(Q)$, resp.\ $X_0(Q)$) denote\footnote{We apologize in advance
 that this is not entirely consistent with the usual notation for modular curves. The alternative
 was to adorn the object~$X$ with the (fixed throughout) level structure at~$N$ coming from~$\rhobar$, which the
 first author felt too notationally cumbersome.} the quotient of $X_1(N)$ (resp.\
 $X_1(NQ)$, resp.\ $X_1(N;Q)$) by the action of $H$.
  Note that each of
 these curves carries an action of $(\Z/N\Z)^\times/H$. We assume that
 $H$ is chosen so that $X$ is the moduli space  (rather than the coarse moduli space) of generalized elliptic
 curves with $\Gamma_H(N):=\left\{
 \begin{pmatrix}
   a & b \\ c & d
 \end{pmatrix}
\in \Gamma_0(N) : d \mod N \in H\right\}$-level structure.

\medskip

\subsubsection{Modular forms with coefficients}
The map $j: X_{\OL/\varpi^m} \rightarrow X
$ is a closed immersion.
If $\L$ is any $\OL$-flat sheaf of $\OL_X$-modules on $X$, this
 allows us to identify $H^0(X_{\OL/\varpi^m},j^* \L)$
 with $H^0(X,\L_{\OL/\varpi^m})$. 
 For
 such a sheaf $\L$, we may identify $\L_{K/\OL}$ with the direct limit $\displaystyle{\lim_{\rightarrow} \L_{\OL/\varpi^m}}$.

\subsubsection{Hecke operators} \label{sec:hecke-ops}

Let $\T^{\univ}$ denote the commutative polynomial algebra over the
group ring $\OL[(\Z/NQ\Z)^\times]$ generated by
indeterminates $T_x$, $U_y$ for $x\nmid pNQ$
prime and $y|Q$ prime. If $a \in (\Z/NQ\Z)^\times$, we let $\langle a
\rangle$ denote the corresponding element of $\T^{\univ}$.
We recall in this section how the Hecke algebra $\T^{\univ}$ acts on coherent
cohomology groups.

We have an \'{e}tale covering map $X_1(Q) \rightarrow X_0(Q)$
with Galois group
$(\Z/Q\Z)^\times$.
Let $\Delta$ be a quotient of $\Delta_Q:=(\Z/Q\Z)^\times$ and let $X_\Delta(Q)\ra
X_0(Q)$ be the corresponding cover. We will define an
action of $\T^{\univ}$ on the groups $H^i(X_{\Delta}(Q),\CL_A)$ for
$A$ an $\CO$-module, 
$i=0,1$ and $\CL$ equal to the bundle  $\omega^{\otimes n}$ or $\omega^{\otimes n}(-\infty)$. If $\CC_{\Delta}(Q)\subset X_{\Delta}(Q)$ denotes the divisor of
cusps, we will also define an action of $\T^{\univ}$ on
$H^0(\CC_{\Delta}(Q),\omega^{\otimes n}_A)$.

First of all, if $a \in (\Z/NQ\Z)^\times$ and
$\CL$ denotes either $\omega^{\otimes n}$ or $\omega^{\otimes n}(-\infty)$ on
$X_{\Delta}(Q)$, then we have a natural isomorphism $\langle a
\rangle^* \CL \iso \CL$. We may thus define the operator $\langle a
\rangle$ on $H^i(X_{\Delta}(Q),\CL_A)$ as the pull back 
\[ H^i(X_{\Delta}(Q),\CL_A) \stackrel{\langle a \rangle^*}{\lra} H^i(X_{\Delta}(Q),\langle a \rangle^*
\CL_A) = H^i(X_{\Delta}(Q),\CL_A).\]
When $i=0$, this is just the usual action of the diamond operators (as in
\cite[\S 3]{Gross}, for instance). We define the action of $\langle a
\rangle$ on $H^0(\CC_{\Delta},\omega^{\otimes n}_A)$ in the same way, using the fact
that $\langle a \rangle$ preserves $\CC_{\Delta}(Q)\subset X_{\Delta}(Q)$

Now, let $x$ be a prime number which does not divide
$pNQ$. We let $X_{\Delta}(Q;x)$ denote the modular curve over $\CO$
obtained by adding $\Gamma_0(x)$-level structure to
$X_{\Delta}(Q)$ (or equivalently, by taking the quotient of $X_1(NQ;x)$ by
the appropriate subgroup of $(\Z/NQ\Z)^\times$).
We have two finite flat
projection maps
\[ \pi_i : X_{\Delta}(Q;x) \rightarrow X_{\Delta}(Q) \]
for $i=1,2$. The map $\pi_1$ corresponds to the
natural forgetful map on open modular curves, extended by
`contraction' to the compactifications. The map $\pi_2$ is
defined on the open modular curves $Y_{\Delta}(Q;x) \to
Y_{\Delta}(Q)$ by sending a tuple $(E,\alpha_{NQ},C_x)$
to the tuple $(E':=E/C_x,
\alpha_{NQ}')$ where $\alpha'_{NQ}$ is the level structure
on $E'$ obtained from $\alpha_{NQ}$ by composing with the natural
isomorphism $E[NQ] \iso E'[NQ]$.
The fact that the $\pi_i$ extend
to the compactifications is ensured by \cite[Prop.\
4.4.3]{Brian}. We
also have the `Fricke involution' $w_x: X_{\Delta}(Q;x)\to
X_{\Delta}(Q;x)$ which is defined on the open modular curve
$Y_{\Delta}(X;x)$ by sending a tuple $(E,\alpha_{NQ},C_x)$ as above to
$(E':=E/C_x,\alpha'_{NQ},E[x]/C_x)$. Note that this is not
really an involution, since $w_x^2(E,\alpha_{NQ},C_x) = (E,x\circ \alpha_{NQ}, C_x)
=  \langle x \rangle
(E,\alpha_{NQ},C_x)$. We have $\pi_2 = \pi_1 \circ
w_x$, and hence $\pi_2^* \omega = (w_x^*\circ \pi_1^*) \omega = w_x^*
\omega$.
Let $\CE$ denote the universal elliptic curve over
$Y_{\Delta}(Q;x)$ and let and let $\CC_x\subset \CE$ denote the
universal subgroup of order $x$. Let $\phi$ denote the quotient map
\[ \CE \lra \CE/\CC_x.\]
Then pull back of differentials along $\phi$ defines a map of sheaves
\[ \phi_{12}: \pi_2^{*} \omega \lra \pi_1^* \omega \]
over $Y_{\Delta}(Q;x)$. This map extends over $X_{\Delta}(Q;x)$
by \cite[Prop.\ 4.4.3]{Brian}.

We define the operator $W_x$ on $H^0(X_{\Delta}(Q;x),\omega^n_A)$ by
setting $W_x$ to be the composite
\[  H^0(X_{\Delta}(Q;x),\omega^n_A) \stackrel{w_x^*}{\lra} H^0(X_{\Delta}(Q,x),\pi_2^*\omega_A^n)
\stackrel{\phi_{12}^{\otimes n}}{\lra} H^i(X_{\Delta}(Q;x),\omega_A^n). \] 
Explicitly, we have:
\[ W_xf(E,\alpha_{NQ},C_x) =  \phi^*f(E/C_x,\alpha'_{NQ},E[x]/C_x).\]
Thus, using the identification $x : E/E[x] \iso E$ together with the fact that we have an equality
$w_x^2(E,\alpha_{NQ},C_x) = \langle x \rangle (E,\alpha_{NQ},C_x)$, we
see that:
\[ W_x^2 = x \langle x \rangle. \]
We will use this fact below

We define Hecke operators $T_x$ on
$H^i(X_{\Delta}(Q),\CL_A)$, for $\CL=\omega^n$, by setting $x T_x$ to be the composition
\[ H^i(X_{\Delta}(Q),\CL_A) \to H^i(X_{\Delta}(Q,x),\pi_2^*\CL_A)
\stackrel{\phi_{12}^{\otimes n}}{\to} H^i(X_{\Delta}(Q;x),\pi_1^* \CL_A)
\stackrel{\mathrm{tr}(\pi_1)}{\to} H^i(X_{\Delta}(Q),\CL_A).\]
This is the same definition as in \cite[p.\ 586]{Edixhoven} and
recovers the usual definition when $i=0$.
We define an action of $T_x$ on the cohomology of
$\CL=\omega^n(-\infty)$, in a similar fashion: the operator $xT_x$ is
the composition
\begin{align*} H^i(X_{\Delta}(Q),\omega^{\otimes n}(-\infty)_A) &\to
  H^i(X_{\Delta}(Q,x),(\pi_2^*\omega^{\otimes n})(-\infty)_A) \\
& \stackrel{\phi_{12}^{\otimes n}}{\to} H^i(X_{\Delta}(Q;x),(\pi_1^* \omega^{\otimes n})(-\infty)_A)
\stackrel{\mathrm{tr}(\pi_1)}{\to} H^i(X_{\Delta}(Q),\omega^{\otimes
  n}(-\infty)_A).
\end{align*} 
(In the first map, we use that $\pi_2^*(\omega^{\otimes
  n}(-\infty))\subset (\pi_2^*\omega^{\otimes n})(-\infty)$ and in
the last map we use that the trace maps sections which vanish at the
cusps to sections which vanish at the cusps.)
Let $\CC_{\Delta}(Q) \subset X_{\Delta}(Q)$ be the divisor of cusps,
as above, and
define $\CC_{\Delta}(Q;x)\subset X_{\Delta}(Q;x)$ similarly. Then we
define an action of $T_x$ on $H^0(\CC_{\Delta}(Q),\omega^n_{A})$ by
setting $xT_x$ equal to the composition:
 \[ H^0(\CC_{\Delta}(Q),\omega^{\otimes n}_A) \to
 H^0(\CC_{\Delta}(Q,x),\pi_2^*\omega^{\otimes n}_A)
\stackrel{\phi_{12}^{\otimes n}}{\to}
H^0(\CC_{\Delta}(Q;x),\pi_1^* \omega^{\otimes n}_A)
\stackrel{\mathrm{tr}(\pi_1)}{\to} H^0(\CC_{\Delta}(Q),\omega^{\otimes
  n}_A).\]
(In this line, both $\phi_{12}$ and the trace map are
obtained from the previous ones by passing to the appropriate quotients.)
In
this way the long exact sequence
\[ \dots \lra H^i(X_{\Delta}(Q),\omega^{\otimes n}(-\infty)_A ) \lra
H^i(X_{\Delta}(Q),\omega^{\otimes n}_A ) \lra
H^i(\CC_{\Delta}(Q),\omega^{\otimes n}_A ) \to \dots \]
is $T_x$-equivariant.

\begin{remark}
  \label{rem:boundary-eis}
  \emph{We note that the action of the operators $T_x$ on
    $H^0(\CC_{\Delta}(Q),\omega^{\otimes n}_A)$ is given by
    specializing $q$ equal to 0 in the usual $q$-expansion formula for
    the action of $T_x$ (as in \cite[\S 3]{Gross}, for instance). This
    expansion is with respect to a local parameter $q$ at the cusp
    `$\infty$', but note that the group $(\Z/NQ\Z)^{\times}$ acts
    transitively on the set of cusps in $\CC_{\Delta}(Q)$. In
    particular, suppose
    $f \in H^0(\CC_{\Delta}(Q),\omega^{\otimes n}_k)$ is a non-zero
    mod $p$ eigenform for all $T_x$ (with $x$ prime to $pNQ$) and is
    of character $\chi : (\Z/NQ\Z)^\times \to k^\times$ in the sense
    that $\langle a \rangle f = \chi(a) f$ for all
    $a\in (\Z/NQ\Z)^\times$. Then $f$ cannot vanish at any cusp and
    the formula \cite[(3.5)]{Gross} implies that
    for each $x$ prime to $pNQ$, we have
    $T_x(f) = (1+\chi(x)x^{n-1})f$. Thus, the semisimple Galois
    representation naturally associated to $f$ (in the sense that for
    $x\nmid pNQ$, the representation is unramified at $x$ with
    characteristic polynomial $X^2-T_xX+\langle x\rangle x^{n-1}$) is
    $1\oplus \chi \epsilon^{n-1}$ (where we also think of $\chi$ as a
    character of $G_{\Q}$ via class field theory).}
\end{remark}

For $x$ a prime dividing $Q$, the action of $U_x$ on
$H^i(X_{\Delta}(Q),\CL_A)$ (for $\CL=\omega^{\otimes n}$, or
$\omega^{\otimes n}(-\infty)$) and on
$H^0(\CC_{\Delta}(Q),\omega^{\otimes n}_A)$ is defined similarly. In
this case, we let $X_1(NQ;x)$ denote the smooth $\CO$-curve
parametrizing tuples $(E,\alpha_{NQ},C_x)$ where $E$ is a generalized
elliptic curve, $\alpha_{NQ} : \mu_{NQ} \into E[NQ]$ is an embedding and
$C_x \subset E[x]$ is a subgroup \'{e}tale locally isomorphic to
$\Z/x\Z$ such that $\alpha_{NQ}(\mu_{NQ})+C_x$ meets every irreducible
component in every geometric fiber and
$\alpha_{NQ}(\mu_x)+C_x = E^{\mathrm{sm}}[x]$. We then let $X_{\Delta}(Q;x)$
be the quotient of $X_1(NQ;x)$ by the same subgroup used to define
$X_{\Delta}(Q)$ as a quotient of $X_1(NQ)$. We have two projection
maps $\pi_i : X_{\Delta}(Q;x) \to X_{\Delta}(Q)$ where $\pi_1$ is the
forgetful map and $\pi_2$ sends $(E,\alpha_{NQ},C_x)$ to $(E'=E/C_x,\alpha'_{NQ})$,
as above. By \cite[Prop.\ 4.4.3]{Brian}, we have a map
$\phi_{12} : \pi_{2}^* \omega \to \pi_1^* \omega$ which allow us to
define $U_x$ by exactly the same formulas as above.

\subsubsection{Properties of cohomology groups} \label{specify}

Define the Hecke algebra
$$\Tan  \subset \End_{\OL} H^0(X_1(Q),{\omega}_{K/\OL})$$
to be the subring of endomorphisms generated over $\OL$ by the Hecke operators $T_n$ with
$(n,p N Q) = 1$ together with all the diamond operators $\langle a
\rangle$ for $(a,NQ)=1$. (Here ``an'' denotes anaemic.) Let $\T$ denote the $\OL$-algebra generated by
these same operators together with $U_x$ for $x$ dividing $Q$.
If $Q = 1$, we let $\TE=\Tan_{\emptyset}$  denote $\T$. The ring $\Tan$ is
a finite $\OL$-algebra and hence decomposes as a direct product over
its maximal ideals $\Tan=\prod_\m \Tan_{\m}$. We have  natural
homomorphisms
$$ \Tan \ra  \Tan_{\emptyset} = \TE, \qquad \Tan \hookrightarrow \T $$
where the first is induced by the map $H^0(X,{\omega}_{K/\OL})\into H^0(X_1(Q),{\omega}_{K/\OL})$ and the
second is the obvious inclusion.

For each maximal ideal $\mE$ of $\TE$, there is a finite extension $k'$
of $\TE/\mE$ and a continuous semisimple representation $G_{\Q} \to
\GL_2(k')$ characterized by the fact that for each prime $x \nmid Np$,
the representation is unramified at $x$ and $\Frob_x$ has
characteristic polynomial $X^2 - T_x X + \langle x \rangle$. (To see
this, choose an extension $k'$ and a normalized eigenform $f \in H^0(X,\omega_{k'})$ such
that the action of $\TE$ on $f$ factors through $\mE$. Then apply
\cite[Prop.\ 11.1]{Gross}.) The representation may in fact be defined
over the residue field $\TE/\mE$.  The maximal ideal $\mE$ is said to be
\emph{Eisenstein} if the associated Galois representation is reducible
and \emph{non-Eisenstein} otherwise. If $\mE$ is non-Eisenstein, then
there is a continuous representation $G_{\Q} \to \GL_2(\T_{\emptyset,\mE})$ which
is unramified away from $pN$ and characterized by the same condition on characteristic polynomials.

Let $\mE$ be a non-Eisenstein maximal ideal of $\TE$.  By a slight abuse of notation, we
also denote the preimage of $\mE$ in $\Tan$ by $\mE$. Note that
the resulting ideal $\mE \subset \Tan$ is maximal. 
The localization $\T_{\mE}$ is a direct factor of $\T$
whose maximal ideals correspond (after possibly extending $\OL$) to the 
$U_x$-eigenvalues on $H^0(X_1(Q),\omega_k)[\mE]$. There is a
continuous representation $G_{\Q}\to \GL_2(\Tan_{\mE})$ which is
unramified away from $pNQ$ and satisfies the same condition on
characteristic polynomials as above for $x\nmid pNQ$.
The following lemma is essentially well known in the construction of Taylor--Wiles systems, we give a detailed
proof just to show that the usual arguments apply equally well in weight one.

\begin{lemma} \label{lemma:matt-w1} 
Suppose that for each $x | Q$ we have that $x\equiv 1
\mod p$ and that the polynomial $X^2 - T_{x} X +  \langle x \rangle \in \TE[X]$ has distinct eigenvalues modulo $\mE$.
Let $\m$ denote the maximal ideal of $\T$ containing $\mE$ and
$U_x - \alpha_x$ for some choice of root $\alpha_x$ of $X^2 - T_x X + \langle x \rangle \mod \mE$ for each $x|Q$.
Then there is a $\Tan_{\mE}$-isomorphism
$$  H^0(X,\omega_{K/\OL})_{\mE}\liso
H^0(X_0(Q),\omega_{K/\OL})_{\m}.$$
\end{lemma}

\begin{proof}  We first of all remark that the localization
  $H^0(X,\omega_{K/\OL})_{\mE}$ is independent of whether we consider
  $\mE$ as an ideal of $\TE$ or $\Tan$. To see this, it suffices to
  note that $\Tan_{\mE} \to \T_{\emptyset,\mE}$ is surjective. This in
  turn follows from the fact that the Galois representation $G_{\Q}
  \to \GL_2( \T_{\emptyset,\mE})$ can be defined over the image of
  $\Tan_{\mE} \to \T_{\emptyset,\mE}$, and moreover that for $x|Q$,
  the operators $T_x$ and
  $\langle x \rangle$ are given by the trace and determinant of $\Frob_x$.

By induction, we reduce immediately to the case when $Q
  = x$ is prime. 
Let $\pi_1,\pi_2 : X_0(x) \rightarrow X$ denote the natural
projection maps and let $\phi_{12} : \pi_2^{*}\omega \to
\pi_1^* \omega$ be the map described in Section~\ref{sec:hecke-ops}.
We define $\psi:=(\pi^*_1,\phi_{12}\circ \pi_2^*)$ and 
 $\psi^{\vee}:=\frac{1}{x} 
 \binom{\Tr(\pi_1)}{\Tr(\pi_1) \circ  W_x }$
     where $W_x$ is the operator defined in
 Section~\ref{sec:hecke-ops} (with $NQ$ there playing the role of $N$
 here). These
give a sequence of $\Tan$-linear morphisms
$$H^0(X,\omega_{K/\OL})^2 \stackrel{\psi}{\rightarrow} H^0(X_0(x),\omega_{K/\OL}) \stackrel{\psi^\vee}{\rightarrow} 
H^0(X,\omega_{K/\OL})^2,$$
such that the composite map $\psi^{\vee} \circ \psi$ is given by
$$\left(\begin{matrix} x^{-1}\Tr(\pi_1)\circ \pi_1^*  &
    x^{-1}\Tr(\pi_1)\circ \phi_{12}\circ \pi_2^* \\ x^{-1}
    \Tr(\pi_1)\circ W_x \circ \pi_1^* 
     & x^{-1}\Tr(\pi_1)\circ W_x \circ \phi_{12} \circ \pi_2^* \end{matrix} \right)
= \left(\begin{matrix} x^{-1}(x + 1) &  T_x \\ T_x & \langle x \rangle
    (x + 1) \end{matrix} \right).$$
(On the first row, this follows from the definition of $T_x$ and the
fact that $\pi_1$ has degree $x+1$; in the lower left corner we use
the definition of $T_x$ and the fact that $W_x \circ \pi_1^* =
\phi_{12}\circ w_x^* \circ \pi_1^* = \phi_{12}\circ \pi_2^*$; in the
lower right corner we use the facts that $W_x\circ \phi_{12}\circ
\pi_2^* = W_x^2 \circ \pi_1^*$ and $W_x^2 = x \langle x \rangle$.)

If $\alpha_x$ and $\beta_x$ are the roots of
$X^2 - T_x X + \langle x \rangle \mod \mE$, then $T_x \equiv \alpha_x + \beta_x \mod \mE$ and
$\langle x \rangle \equiv \alpha_x \beta_x  \mod \mE$.
Since $x \equiv 1 \mod p$, we have
\begin{align*}
\det(\psi^{\vee} \circ \psi) = x^{-1}(x+1)^2\langle x \rangle -  T^2_x
\equiv 4\langle x \rangle - T^2_x &\equiv 4 \alpha_x \beta_x - (\alpha_x + \beta_x)^2\\
&\equiv -(\alpha_x - \beta_x)^2 \mod \mE.
\end{align*}
By assumption, $\alpha_x \not\equiv \beta_x$, and thus, after localizing at $\mE$,
the composite map $\psi^{\vee} \circ \psi$ is an isomorphism.

We deduce that
$H^0(X,{\omega}_{K/\OL})^2_{\mE}$ is a direct factor of  $H^0(X_0(x),{\omega}_{K/\OL})_{\mE}$ as a
$\Tan_{\mE}$-module. 
Consider the action of $U_x$ on the image of $H^0(X,{\omega}_{K/\OL})^2$.
Let $V_x$ denote the second component $\phi_{12}\circ \pi_2^*$ of
the degeneracy map $\psi$. Then we have equalities of maps
$H^0(X,\omega_{K/\CO})\to H^0(X_0(x),\omega_{K/\CO})$
\[ \pi_1^*\circ T_x = U_x \circ \pi_1^* + \frac{1}{x} V_x \space
\hbox{\;\;\; and \;\;\; } U_x\circ V_x = \pi_1^*\circ \langle x
\rangle .\]
To see that the first of these holds, note that:
\[ (\pi_1^* \circ T_x)(f)(E,\alpha_N,C_x) = \frac{1}{x}\sum_{D \subset E[x]}
  \phi_D^* f(E/D,\phi_D \circ \alpha_N), \] where the sum is over all
order $x$ subgroups $D \subset E[x]$, and $\phi_D$ denotes the
quotient map $E \to E/D$. (This formula holds after base-change to an
\'{e}tale extension over which the subgroups $D$ are defined.)
Restricting the sum to all $D \neq C_x$ gives $(U_x \circ
\pi_1^*)(f)(E,\alpha_N,C_x)$, while the remaining term is
$\frac{1}{x}(\phi_{12}\circ \pi_2^*)(f)(E,\alpha_N,C_x)$. For the second equality, we have:
\[ (U_x \circ V_x)(f)(E,\alpha_N,C_x) = \frac{1}{x}\sum_{D \neq C_x}
  \phi_{D+C_x}^* f(E/(D+C_x),\phi_{D+C_x} \circ \alpha_N), \]
where $D$ is as above and $\phi_{D+C_x}$ is the quotient $E \to
E/(D+C_x)$. Since $D+C_x = E[x]$ for all $D \neq C_x$, each term
$\phi_{D+C_x}^* f(E/(D+C_x),\phi_{D+C_x} \circ \alpha_N)$ can be
identified with $x f(E, x \circ \alpha_N) = x \langle x \rangle
f(E,\alpha_N)$. Since there exactly $x$ subgroups $D\neq C_x$, and we divide by
$x$, we deduce that $U_x \circ V_x = \pi_1^* \langle x \rangle$.

 It follows that the action of $U_x$ on $H^0(X_0(x),{\omega}_{K/\OL})^2_{\mE}$
is given by the matrix
$$ A=\left( \begin{matrix} T_x & x\langle x\rangle \\ -\frac{1}{x} & 0 \end{matrix} \right).$$
There is an identity $(A - \alpha_x)(A - \beta_x) \equiv 0 \mod \mE$ in $M_2(\T_{\emptyset,\mE})$.
Since $\alpha_x \not\equiv \beta_x$,
by Hensel's Lemma, there exist $\walpha_x$ and $\wbeta_x$ in $\T_{\emptyset,\mE}^\times$ such that
$(U_x - \walpha_x)(U_x - \wbeta_x) = 0$ on $(\im \psi)_{\mE}$.  It follows that $U_x - \wbeta_x$ is
a projector (up to a unit) from $(\im \psi)_{\mE}$ to $(\im \psi)_{\m}$.
We claim that there is an isomorphism of $\Tan_{\mE}$-modules
$$H^0(X,{\omega}_{K/\OL})_{\mE} \simeq (\im\psi)_\m= (H^0(X,{\omega}_{K/\OL})^2)_{\m}.$$
It suffices to show that there is a $\Tan_{\mE}$-equivariant injection
from $H^0(X,{\omega}_{K/\OL})_{\mE}$ to  the module 
$(\im\psi)_{\mE}$ such that the image
has trivial intersection with the kernel of $U_x - \wbeta_x$: if there is such an
injection, then, by symmetry, there is also an injection from
$H^0(X,{\omega}_{K/\OL})_{\mE}$ to 
$(\im \psi)_{\mE}$ whose image intersects the
kernel of $U_x-\walpha_x$ trivially; by length considerations both injections are forced to
be isomorphisms. We claim that  the natural inclusion $\pi_1^*$ composed with $U_x-\wbeta_x$ is such a map: from the
computation of the matrix above, it follows that
$$(U_x - \wbeta_x)
\left( \begin{matrix} f \\ 0 \end{matrix} \right) = 
\left(\begin{matrix} T_x f - \wbeta_x f \\ -  f \end{matrix} \right),$$
which is non-zero whenever $f$ is by examining the second coordinate.

We deduce that there is a decomposition of $\Tan_{\mE}$-modules
$$H^0(X_0(x),{\omega}_{K/\OL})_{\m} = H^0(X,{\omega}_{K/\OL})_{\mE}
\oplus V$$
where $V$ is the kernel of $\psi^{\vee}$. It suffices to show that
$V[\m]$ is zero. Let $f \in V[\m]$. We may
regard $f$ as an element of $H^0(X_0(x),\omega_{k})$.
It satisfies the following properties:
\begin{enumerate}
\item $U_x f = \alpha_x f$,
\item $\langle x \rangle f = \alpha_x \beta_x f$,
\item $(\Tr(\pi_1)\circ W_x)f = 0$.
  \end{enumerate}
The first two properties follow from the fact that $f$ is killed by $\m$, and the last
follows from the fact that $f$ lies in the kernel of
$\psi^{\vee}$.

We claim that
\[ \frac{1}{x}(\pi_1^{*}\circ \Tr(\pi_1)\circ W_x) =  U_x + \frac{1}{x}W_x.\]
To see this, we will rewrite the relation $\pi_1^* \circ
T_x = U_x \circ \pi_1^* + \frac{1}{x}V_x$ that we established earlier.
Since $\pi_2 = \pi_1 \circ w_x$, we have that $T_x = \frac{1}{x} \Tr(\pi_1) \circ W_x \circ \pi_1^*$ and
$V_x = W_x \circ \pi_1^*$. The earlier relation can thus be written:
\[ \frac{1}{x}\pi_1^* \circ \Tr(\pi_1) \circ W_x \circ \pi_1^* = U_x \circ \pi_1^* +
  \frac{1}{x} W_x\circ \pi_1^*.\]
Since $\pi_1^*$ is fully faithful, the claim follows.

Now, property (3) above tells us that $ - x U_x f = W_x f$. By
(1), both sides are $k$-multiples of $f$ and applying~$- x U_x f = W_x
f$ once more, we see that $x^2 U_x^2 f = W_x^2f$. Since $W_x^2 =
x\langle x \rangle$, we deduce that $x U_x^2
f = \langle x \rangle f$. Thus, by (1) and (2):
\[ x \alpha_x^2 f = \alpha_x \beta_x f.\]
Since $x\equiv 1 \mod p$ and $\alpha_x \neq \beta_x$, we deduce that $f=0$, as
required.
\end{proof}

We have the following mild generalization of Lemma~\ref{lemma:matt-w1}:

\begin{lemma} \label{lemma:stacky}
Suppose $Q=x$ is prime not dividing~$Np$ and the eigenvalues of the polynomial $X^2 - T_{x} X +  \langle x \rangle \in \TE/\mE[X]$ 
do not have ratio $x$ or $x^{-1}$.
Let $\m$ denote the maximal ideal of $\T$ containing $\mE$ and
$U_x - \alpha_x$ for some choice of root $\alpha_x$ of $X^2 - T_x X + \langle x \rangle \mod \mE$.
If the roots  $\alpha_x$ and $\beta_x$ are distinct, then
there is a $\Tan_{\mE}$-isomorphism
$$  H^0(X,\omega_{K/\OL})_{\mE}\liso
H^0(X_0(x),\omega_{K/\OL})_{\m}.$$
If the roots $\alpha_x$ and $\beta_x$ are equal, then there is
a 
$\Tan_{\mE}$-isomorphism
$$  H^0(X,\omega_{K/\OL})^2_{\mE}\liso
H^0(X_0(x),\omega_{K/\OL})_{\m}.$$
\end{lemma}

\begin{proof} The proof is essentially identical to Lemma~\ref{lemma:matt-w1}.
The only calculations which are different are the following:
If $\alpha_x$ and $\beta_x$ are the roots of
$X^2 - T_x X + \langle x \rangle \mod \mE$, then $T_x \equiv \alpha_x + \beta_x \mod \mE$ and
$\langle x \rangle \equiv \alpha_x \beta_x  \mod \mE$. Hence
$$\begin{aligned}
\det(\psi^{\vee} \circ \psi) =  & \  x^{-1}(x+1)^2 \langle x \rangle -
T^2_x\\
\equiv & \  x^{-1}((x+1)^2 \alpha_x \beta_x - x (\alpha_x + \beta_x)^2)\\
\equiv& \  x^{-1}(\alpha_x - x \beta_x)(\beta_x - x \alpha_x) \mod \mE.
\end{aligned}
$$
This is non-zero under our assumptions.
 If the eigenvalues are
distinct, we proceed as before; the proof that the summand $V$ of
$H^0(X_0(Q),\omega_{K/\OL})_{\m}$ is zero is also the same, since the
final conclusion is that $\beta_x \equiv x \alpha_x \mod \mE$, a contradiction. If the eigenvalues are equal, note that 
\[ \psi(H^0(X,\omega_{K/\OL})^2_{\mE}) =
\psi(H^0(X,\omega_{K/\OL})^2_{\mE})_{\m} \]
since $\T$ acts on this space via the quotient
$\T_{\emptyset,\mE}[X]/(X^2-T_xX+\langle x \rangle)$ (with $X$ corresponding to
$U_x$), which is a local ring, by assumption. Thus, we can decompose
\[ H^0(X_0(x),\omega_{K/\OL})_{\m} = \psi(H^0(X,\omega_{K/\OL})_{\mE})^2
\oplus V, \]
as a direct sum of $\Tan$-modules. The same proof as above shows that $V$ is zero.

\end{proof}

\medskip

As in Section~\ref{sec:hecke-ops}, let $\Delta$ be a quotient of $\Delta_Q:=(\Z/Q\Z)^\times$ and let $X_\Delta(Q)\ra
X_0(Q)$ be the corresponding cover.
If $A$ is an $\OL$-module, we have defined an action of the universal polynomial algebra
$\T^{\univ}$ on the cohomology groups $H^i(X_\Delta(Q), \CL_A)$ for $\CL = \omega^{\otimes n}$ or $\omega^{\otimes
  n}(-\infty)$. 
The ideal $\mE$ gives rise to a maximal ideal $\m$ of $\T^{\univ}$
after a choice of eigenvalue mod $\mE$ for $U_x$ for all $x$ dividing
$Q$. Extending $\OL$ if necessary, we may assume that $\T^{\univ}/\m \cong k$.

\medskip

Let $M\mapsto M^\vee:=\Hom_\OL(M,K/\OL)$ denote the Pontryagin duality functor.

\begin{lemma} 
\label{lem:cohom-basic-props}
Let $\Delta$ be a quotient of $\Delta_Q$. Then:
\begin{enumerate}
\item\label{torsion-free} $H^1(X_\Delta(Q),\CL_{K/\OL})^{\vee}$
  is $p$-torsion free for $\mathcal L$ a vector bundle on $X_\Delta(Q)$.
\item\label{vanish-cusps} For $i=0,1$, we have an isomorphism
\[ H^i(X_\Delta(Q),\omega(-\infty)_{K/\OL})_\m
\liso H^i(X_\Delta(Q),\omega_{K/\OL})_\m.\]
\end{enumerate}
\end{lemma}

\begin{proof}
The first claim is equivalent to the divisibility of
$H^1(X_{\Delta}(Q), \CL_{K/\OL})$. Since $X_{\Delta}(Q)$
is flat over $\OL$, there is an exact sequence
$$0 \rightarrow \mathcal L_{k} \rightarrow \mathcal L_{K/\OL} \stackrel{\varpi}{\ra} \mathcal L_{K/\OL} \rightarrow 0.$$
Taking cohomology and using the fact that $X_{\Delta}(Q)$ is a curve
and hence $H^2(X_\Delta(Q),\mathcal L_{k})$ vanishes,
we deduce that  $H^1(X_{\Delta}(Q), \mathcal L_{K/\OL})/\varpi = 0$, from which divisibility follows.

For the second claim, note that 
there is an exact sequence:
$$0 \rightarrow \omega(-\cusps) \rightarrow \omega \rightarrow \omega|_{\CC_{\Delta}(Q)}  \rightarrow 0,$$
where $\CC_{\Delta}(Q)$ denote the divisor of cusps.
The cohomology of $\CC_{\Delta}(Q)$ is concentrated in degree 0.
 Yet the action of Hecke on 
 $H^0(\CC_{\Delta}(Q),\omega)$ is
 Eisenstein (by Remark~\ref{rem:boundary-eis}), thus the lemma. 
\end{proof}

If $\CL$ is a vector bundle on $X_\Delta(Q)$, we define
\[ H_i(X_\Delta(Q),\mathcal L):=
H^i(X_\Delta(Q),(\Omega^1\otimes\CL^*)_{K/\OL})^{\vee} \]
for $i=0,1$, where $\CL^*$ is the dual bundle and
$\Omega^1=\Omega^1_{X_\Delta(Q)/\OL}$ can be identified with
$\omega^{\otimes 2}(-\infty)$ via the Kodaira--Spencer isomorphism. 
If
$\Delta \onto \Delta'$ are two quotients of $(\Z/Q\Z)^\times$ giving rise to
a Galois cover $\pi:X_\Delta(Q)\to X_{\Delta'}(Q)$ and $\CL$
is vector bundle on $X_{\Delta'}(Q)$, then there is a
natural map $\pi_*: H_i(X_{\Delta}(Q),\pi^*\CL)\to
H_i(X_{\Delta'}(Q),\CL)$ coming from the dual of the pullback $\pi^*$ on cohomology. Verdier duality
(\cite[Corollary~11.2(f)]{Verdier}) gives an isomorphism
\[ D: H_i(X_\Delta(Q),\CL)\liso H^{1-i}(X_\Delta(Q),\CL) \]
under which $\pi_*$ corresponds to the trace map $\Tr(\pi):H^{1-i}(X_\Delta(Q),\pi^*\CL)\to
H^{1-i}(X_{\Delta'}(Q),\CL)$.

We endow $H_i(X_{\Delta}(Q),\omega^{\otimes n})$
with a Hecke action
by first identifying $\Omega^1$ with $\omega^2(-\infty)$ and then taking the Pontryagin dual of the Hecke action on
$H^i(X_{\Delta}(Q),\omega^{2-n}(-\infty)_{K/\CO})$.
However, we note that for $\CL =
\omega^n$, the isomorphism $D$ is not
Hecke equivariant: extend $\CO$ if necessary so that it contains a
primitive $NQ$-th root of unity $\zeta$.
Let $w^*$ be the operator on
$H^0(X_{\Delta}(Q),\omega^{\otimes n}_{K/\CO})$ associated to $\zeta$ as in \cite[\S
7.1]{Edixhoven}. Let $\Phi$ denote the composition of isomorphisms:
\begin{eqnarray*}
 H^1(X_{\Delta}(Q),\omega^{2-n}(-\infty)_{K/\CO}) 
  &\stackrel{D}{\lra}& H^0(X_{\Delta}(Q),\Omega\otimes \omega^{n-2}(\infty))^{\vee} \\
  &\stackrel{KS^{\vee}}{\lra}& H^0(X_{\Delta}(Q),\omega^{n})^{\vee} \\
  &\stackrel{(w^*)^{\vee}}{\lra}& H^0(X_{\Delta}(Q),\omega^{n})^{\vee},
\end{eqnarray*}
where $D$ is Verdier duality, and $KS$ is the Kodaira--Spencer
isomorphism. Then by the proof of \cite[Prop.\ 7.3]{Edixhoven}, we
have:
\[ \Phi \circ T_x = x^{1-n} T_x^{\vee} \circ \Phi, \]
with the same relation holding for $U_x$. We also have $\Phi\circ
\langle x \rangle = \langle x \rangle \circ \Phi$.
We let $\Psi : = \Phi^{\vee}$ be the dual isomorphism
\[ \Psi: H^0(X_{\Delta}(Q),\omega^n) \lra
  H_1(X_{\Delta}(Q),\omega^n). \]
When $n = 1$, $\Psi$ is $\T$-linear.

\begin{prop}
\label{prop:balanced-homology-w1} 
Let $\Delta$ be a quotient of $(\Z/Q\Z)^\times$ of $p$-power order. Then the
$\OL[\Delta]$-module $H_0(X_\Delta(Q),\omega)_\m$ is balanced (in
the sense of Definition \ref{defn:balanced}).
\end{prop}

\begin{proof}
Let $M = H_0(X_\Delta(Q),\omega \otimes \CL )_\m$ and
$S=\OL[\Delta]$, where $\CL = \OL_X$.
\footnote{We present the proof writing~$\CL$ instead of~$\OL_X$ since we will
use the same argument in the proof of Theorem~\ref{thm:main-thm-w1-v2}
 with~$\CL$ as a more general vector bundle on~$X$.}
Consider the exact sequence of $S$-modules (with trivial $\Delta$-action):
$$0 \rightarrow \OL \stackrel{\varpi}{\rightarrow} \OL \rightarrow
k\rightarrow 0.$$
Tensoring this exact sequence over $S$ with $M$, we obtain an exact sequence:
$$ 0 \rightarrow \Tor^S_1(M,\OL)/\varpi  \rightarrow \Tor^S_1(M,k) \rightarrow M_\Delta  \rightarrow M_\Delta 
\rightarrow M \otimes_S k \rightarrow 0.$$
Let $r$ denote the $\OL$-rank of $M_\Delta$. Then this exact sequence
tells us that
\[ d_S(M):=\dim_k M\otimes_S k - \dim_k \Tor_1^S(M,k) = r - \dim_k \Tor^S_1(M,\OL)/\varpi.\]
We have a second quadrant Hochschild--Serre spectral sequence (\cite[Theorem III.2.20, Remark III.3.8]{milne})
\[ H^i(\Delta,H^j(X_\Delta(Q),(\Omega^1\otimes \omega^{-1} \otimes \CL^*)_{K/\OL}))
\implies H^{i+j}(X_0(Q),(\Omega^1\otimes \omega^{-1} \otimes \CL^*)_{K/\OL}).\] 
Applying Pontryagin
duality, we obtain a third quadrant spectral sequence
\[ H_i(\Delta,H_j(X_\Delta(Q),\omega \otimes \CL)) =
\Tor_i^S(H_j(X_\Delta(Q),\omega \otimes \CL),\OL) \implies
H_{i+j}(X_0(Q),\omega \otimes \CL) .\]
We claim that the differentials in the spectral sequence commute with
the action of $\T^{\univ}$ on the individual terms. This follows from
the fact that the Hochschild--Serre spectral sequence, for a finite
\'{e}tale Galois cover $\pi:X\to Y$ with group $G$, 
\[ H^i(G,H^j(X,\pi^*\CF)) \implies H^{i+j}(Y,\CF) \]
is functorial in $\CF$. Thus for example, to see that the
differentials commute with \[T_x = \frac{1}{x}\Tr(\pi_1)\circ
\phi_{12}\circ \pi_2^*,\]
(where $\pi_1,\pi_2 : X_{*}(Q;x)\to X_{*}(Q)$ are the two projection
maps and $*\in\{0,\Delta\}$), we use the canonical isomorphisms
$H^i(X_{*}(Q;x),\pi_j^*\omega)=H^i(X_*(Q),\pi_{j,*}\pi_j^*\omega)$ and
successively apply the functoriality of the spectral sequence
with $(X,Y,\CF\to \CF')$ taken to equal 
$(X_{\Delta}(Q),X_0(Q),\omega \to \pi_{2,*}\pi_2^* \omega)$,
$(X_{\Delta}(Q;x),X_{0}(Q;x),\pi_2^*(\omega)\stackrel{\phi_{12}}{\ra}\pi_1^*(\omega))$
and
$(X_{\Delta}(Q),X_0(Q),\pi_{1,*}\pi_1^{*}(\omega)\stackrel{\Tr}{\ra}\omega)$.

Localizing at ${\m}$, we obtain an isomorphism $M_\Delta \cong
H_0(X_0(Q),\omega \otimes \CL)_{\m}$ and an exact sequence
\[ (H_1(X_\Delta(Q),\omega \otimes \CL)_{\m})_\Delta \to H_1(X_0(Q),\omega \otimes \CL)_{\m} \to
\Tor^S_1(M,\OL) \to 0 .\]
To show that $d_S(M)\geq 0$, we see that it suffices to show that
$H_1(X_0(Q),\omega \otimes \CL)_{\m}$ is free of rank $r$ as an
$\OL$-module. The module $H_1(X_0(Q),\omega \otimes \CL)$
is $p$-torsion free by
Lemma~\ref{lem:cohom-basic-props}(\ref{torsion-free}).
It therefore suffices to show that 
$\dim_K
H_1(X_0(Q),\omega \otimes \CL)_\m\otimes K=r$. In other words, by the
definition of $r$, it suffices to show that
\[ \dim_K H_0(X_0(Q)_K,\omega \otimes \CL)_\m =
  \dim_K H_1(X_0(Q),\omega \otimes \CL)_\m. \]
(Here we use the slight abuse of notation $H^i(X_K, *)_{\m} = H^i(X,*)_{\m}\otimes_{\CO}K$.)

At this point we will specialize to $\CL = \CO_X$. By definition of
$H_0$ and its Hecke action, the left hand side above is:
\[ \dim_K H^0(X_0(Q)_K,\omega(-\infty))_\m. \]
Using the isomorphism $\Psi$, we see that the right hand side is equal to:
\[ H^0(X_{0}(Q)_K, \omega )_{\m} \]
We are therefore reduced to showing that 
\[ \dim_K H^0(X_0(Q)_K,\omega(-\infty))_\m =  \dim_K
  H^0(X_0(Q)_K,\omega)_{\m} .\]
The result thus follows from
Lemma~\ref{lem:cohom-basic-props}(\ref{vanish-cusps}).
\end{proof}

\subsection{Galois Representations}
\label{sec:gal-rep}
Let $N = N(\rhobar)$ where $N(\rhobar)$ is the Serre level of
$\rhobar$. 
We let
$H$ denote the $p$-part of $(\Z/N(\rhobar)\Z)^\times$. Having
fixed $N$ and $H$, we let $X$ denote the modular curve defined at the
beginning of \S~\ref{sec:cohom-modul-curv}. (We note that $N\geq 5$ by
Serre's conjecture, and also that $\rhobar$ is thus modular of the appropriate level, by 
Theorem~4.5 of~\cite{Edixhoven}.) 
We add, for now, the following assumption:
\begin{assumption} Assume that: \label{novexing}
\begin{enumerate}
\item  The set $T(\rhobar)$ is empty.
\end{enumerate}
\end{assumption}
We address how to remove this assumption in~\S~\ref{section:vexing}.
The only point at which this assumption is employed is in~\S~\ref{section:interlude}.

\begin{remark} {\bf A digression on representability and stacks\rm}.
\emph{Several of the modular  curves we consider do not represent the corresponding
moduli problems for elliptic curves with level structure (due to automorphisms). 
In such cases, the object $X_H(N)$ still exists as a smooth proper Deligne--Mumford stack over $\Spec(\OL)$,
and the sheaf $\omega$ descends to a sheaf on $X_H(N)$ (if not always to the corresponding associated scheme).
For these stacks $X$, one can still make sense of the cohomology groups $H^0(X,\omega)$ and show that
they satisfy many of the required
properties. For example,
suppose that $Y \rightarrow X$ is a finite \'{e}tale morphism of modular
stacks with Galois group $G$, and that
$\LL = \omega^k$ for some $k$. Then there is an isomorphism
$$H^0(X,\LL) \simeq H^0(Y,\LL)^{G}.$$
Taking $Y$ to be representable (which is always possible for the $X$ we consider),
and letting $R$ be an $\OL$-algebra, one may identify $H^0(X_R,\omega^k)$ with the ring of Katz modular forms of weight~$k$ 
over~$R$ as defined in~\cite{Katz}.
\footnote{Here is a somewhat different example: let $H \subset (\Z/13 \Z)^{\times}$ denote the group of squares; there
is a finite  \'{e}tale map $X_1(13) \rightarrow X_H(13)$ with Galois group $\Z/3\Z$ (viewing $X_H(13)$ as a stack).
The underlying scheme of $X_H(13)$ is isomorphic over $\OL$ (for $p \ne 13$) to the projective line,
and hence, na\"{\i}vely, one would expect $H^0(X_H(13),\omega^2)$ to vanish. However, as noted
by Serre~\cite{RibetLetter,RibetComponent}, there exists a Galois  representation $\rhobar: G_{\Q} \rightarrow \GL_2(\F_3)$ with $N(\rhobar) = 13$,  $k(\rhobar) = 2$, and $\eps(\rhobar)$ quadratic.
(The representation $\rhobar$ is induced from $\Q(\sqrt{-3})$). The original conjecture 
$3.2.4_{?}$ of~\cite{Serre}  asserts that $\rhobar$ gives rise to a mod-$3$ modular form on $X_H(13)$.
However, considering $X_H(13)$ as a stack, one finds, for $p = 3$, that the group 
$$H^0(X_H(13)_{\F_3},\omega^2) = H^0(X_1(13)_{\F_3},\omega^2)^{\Z/3\Z}$$
\emph{is} indeed non-zero, even though the scheme underlying $X_H(13)$ has genus zero. 
This is in accordance with Edixhoven's reformulation of Serre's conjecture (Conjecture~4.2 of~\cite{Edixhoven}).}}

\medskip

\emph{Instead of trying to adapt our arguments (when necessary) to the context of stacks, we introduce the following
fix. Choose any prime $q \not\equiv 1 \mod p$  with $q \ge 5$ such that
$\rhobar$ is unramified at $q$ and such that  the ratio of
the eigenvalues of $\rhobar(\Frob_q)$ is neither  $q$ nor $q^{-1}$ (we allow the possibility that the
eigenvalues are the same). The assumption on the ratio of the eigenvalues ensures that $\rhobar$
admits no  deformations which are ramified and Steinberg (unipotent on inertia) at $q$; the assumption that $q \not\equiv 1 \mod p$ guarantees that
there are no other deformations of $\rhobar$ which are ramified at $q$.
First, we claim that the Chebotarev density theorem guarantees the existence of such primes.
Since $\det(\rhobar)$ is unramified at $p$, the fixed field of $\ker(\rhobar)$ does not contain $\Q(\zeta_p)$ (note that $p \ne 2$).
Hence, we may find infinitely many primes $q$ such that the fixed field of
$\ker(\rhobar)$ splits completely and $q \not\equiv 1 \mod p$; such primes satisfy the required hypothesis.
We then add $\Gamma_1(q)$ level structure and  the arguments proceed almost entirely unchanged (the
assumption on $q$ implies that any deformation of $\rhobar$ with fixed determinant
is unramified at $q$). The only difference is that the multiplicity of the corresponding Hecke modules will
either be the same or twice as expected --- depending on whether $\rhobar(\Frob_q)$ has distinct eigenvalues
or not --- by Lemma~\ref{lemma:stacky}. }
\end{remark}

\medskip

By Serre's conjecture~\cite{KhareW} and by the companion form result of Gross~\cite{Gross}
and Coleman--Voloch~\cite{CV2},
there exists a maximal ideal $\mE$ of $\TE$ corresponding to
$\rhobar$. The ideal $\mE$ is generated by $\varpi$, $T_x
-\tr(\rhobar(\Frob_x))$ for all primes $x$ with $(x,Np)=1$ and $\langle
x \rangle - \det(\rhobar(\Frob_x))$ for all $x$ with
$(x,N)=1$. Extending $\OL$ if necessary, we may assume $\TE/\mE=k$. Let
$Q$ be as in Section \ref{sec:cohom-modul-curv}. For each $x \in Q$,
assume that
 the polynomial $X^2 - T_xX +\langle
x \rangle$ has distinct roots in $\TE/\mE = k$ and choose a
root $\alpha_x\in k$ of this polynomial. Let $\m$ denote the maximal
ideal of $\T$ generated by $\mE$ and $U_x -\alpha_x$ for $x\in Q$.

\begin{theorem}[Local--Global Compatibility] \label{theorem:galoisweight1} There exists a
  deformation 
$$\rho_Q: G_{\Q} \rightarrow \GL_2(\T_{\m})$$
 of $\rhobar$ unramified outside $NQ$
and determined by the property that for all primes $x$ satisfying $(x,pNQ) = 1$,
$\tr(\rho_Q(\Frob_x)) = T_x$. 
Let $\rho'_{Q} = \rho_{Q} \otimes \eta$, where~$\eta$ is the unique~$p$-power order character
satisfying~$\eta^{2} =  \langle \det(\rhobar) \rangle \det(\rho_Q)^{-1}$. Then
$\rho'_{Q}$ is a deformation of $\rhobar$  minimal outside
$Q$.  For a prime~$x | Q$, the restriction of~$\rho'_{Q}$ to~$D_{x}$ is 
conjugate to a direct sum of residually unramified  characters~$\xi \oplus \langle \det(\rhobar)\rangle |_{D_x} \xi^{-1}$ where~$\overline{\xi}(\Frob_x) \equiv \alpha_x \mod \m$,
and such that the restriction of~$\xi$ to~$I_x$ and hence via local class field theory to~$\Z^{\times}_x \otimes \Z_p = \Delta_x$ is compatible 
in the usual way with the diamond
operators in~$\T_{\m}$.
\end{theorem}

\begin{remark}
\label{awayfromp}
\emph{
The existence of~$\rho_Q$ follows immediately by using congruences between weight one forms and higher weight (using powers of the Hasse invariant).
The assumptions on~$x|Q$ imply
that~$\rhobar|_{D_x}$ is~$p$-distinguished and, because there are no local extensions, totally split. Hence the required information is preserved under
congruences, and one is reduced once more to higher weight, where this statement is known. Hence the main difficulty in proving
Theorem~\ref{theorem:galoisweight1} is showing that the Galois representation is unramified at~$p$.
}
\end{remark}

\begin{remark}
\label{rem:distinct-evals}
\emph{Under the hypothesis that $\rhobar(\Frob_p)$ has distinct eigenvalues, Theorem~\ref{theorem:galoisweight1}
can be deduced using an argument similar to that of~\cite{Gross}.
Under the hypothesis that $\rhobar(\Frob_p)$ has repeated eigenvalues but is
\emph{not} scalar, we shall deduce this using an argument of Wiese~\cite{Wiese}
and Buzzard. When $\rhobar(\Frob_p)$ is trivial, however, we shall be forced
to find a new argument using properties of local deformation rings.
In the argument below,  we avoid using the fact that
the Hecke eigenvalues for all primes $l$  determine a modular eigenform completely.
One reason for doing this is that we would like to generalize our arguments
to situations in which this fact is no longer true; we apologize in advance that this
increases the difficulty of the argument slightly (specifically, we avoid using the
fact that $T_p$ in weight one can be shown to live inside the Hecke algebra $\T$, although this
will be a consequence of our results).}
\end{remark}

\begin{proof} 

For each $m >0$, we have
$H^0(X_1(Q),\omega_{\OL/\varpi^m})_{\m}\cong
H^0(X_1(Q),\omega_{K/\OL})_{\m}[\varpi^m]$, and we let $I_m$ denote
the annihilator of this space in $\T_{\m}$. Since
$\T_{\m}=\varprojlim_m \T_{\m}/I_m$, it suffices to
construct, for each $m>0$, a representation $\rho_{Q,m} : G_\Q \to
\GL_2(\T_{\m}/I_m)$ satisfying the conditions of the theorem.

Fix $m>0$ and let $A$ be a lift of (some power of) the Hasse invariant such that $A \equiv 1 \mod \varpi^m$; let $n-1$ denote the weight of $A$.
We may assume that $n-1$ is sufficiently divisible by powers of $p$ (and $(p-1)$) to
ensure that $\eps^{n-1} \equiv 1 \mod \varpi^m$.
Multiplication by $A$ induces a  map:
$$\begin{diagram}
 H^0(X_1(Q),\omega_{K/\OL})  & \rTo_{\phi} & H^0(X_1(Q),\omega^{n}_{K/\OL}) \\
 \dInto & & \dInto \\
K/\OL[[q]] & \rEquals & K/\OL[[q]] \end{diagram}.$$
 This map is Hecke equivariant away from $p$ on $\varpi^m$ torsion; indeed, the diagram is only
 commutative modulo $\varpi^m$.

Consider the map
$$ \begin{diagram}
\psi: H^0(X_1(Q),\omega_{K/\OL})^2[\varpi^m]  & \rTo & H^0(X_1(Q),\omega^{n}_{K/\OL})[\varpi^m] \\
\dEquals & & \dEquals \\
 \psi: H^0(X_1(Q),\omega_{\OL/\varpi^m})^2  & \rTo & H^0(X_1(Q),\omega^{n}_{\OL/\varpi^m}). \end{diagram}$$
 defined by $\psi = (\phi, \phi \circ T_p - U_p \circ \phi)$. (The operator~$T_p$ acts in this setting by the results of~\S4 of Gross~\cite{Gross}.)
 For ease of notation, we let $T_p$ (or $\phi \circ T_p$) exclusively refer to the Hecke operator in weight one,
 and let $U_p$ denote the corresponding Hecke operator in weight $n$. (The operator $U_p$ has the expected
 effect on $q$-expansions, since the weight $n$ is sufficiently large with respect to $m$.)
 On $q$-expansions modulo $\varpi^m$, we may compute that
 $\psi = (\phi, \langle p \rangle V_p)$ (see also~4.7 of~\cite{Gross}) .
 We claim that $\psi$ is injective. It suffices to check this on the $\OL$-socle, namely, on $\varpi$-torsion.
On $q$-expansions,  $\phi$ is the identity  and  $V_p(\sum
a_nq^n)=\sum a_nq^{np}$. Suppose we have an identity
$\langle p \rangle V_p  f = g$.
It follows that $\theta g = 0$ in $H^0(X_1(Q),\omega_{\OL/\varpi^m})[\varpi] = H^0(X_1(Q),\omega_{k})$.
 By a  result of Katz~\cite{Katz}, the $\theta$ map has
no kernel in weight $\le p-2$, and so in particular no kernel in weight $1$. Hence $\psi$ is injective.

The action of $U_p$  in weight $n$ on $H^0(X_1(Q),\omega_{\OL/\varpi^m})^2$ via $\psi^{-1}$ is given by
$$\left( \begin{matrix}
    T_p & 1 \\ -\langle p \rangle & 0 \end{matrix} \right),$$ where
here $T_p$ is acting in weight one (cf. Prop~4.1 of~\cite{Gross}), and
hence satisfies the quadratic relation $X^2 - T_p X + \langle p
\rangle = 0$. Note that the action of $U_p + \langle p \rangle U^{-1}_p$
on the image of $\psi$ is given by 
$$\left( \begin{matrix} T_p & 0 \\ 0 & T_p \end{matrix} \right).$$
By Proposition 12.1 and the remark before equation (4.7)
of \cite{Gross}, we see that $\langle p\rangle = \alpha\beta\mod\m$
and $(U_p -\widetilde\alpha)(U_p-\widetilde\beta)$ acts nilpotently on
$\psi(H^0(X_1(Q),\omega_{\OL/\varpi^m})^2_{\m})$,
where $\alpha$ and $\beta$ denote the (possibly non-distinct) eigenvalues of
$\rhobar(\Frob_p)$ and $\widetilde\alpha$, $\widetilde\beta$ are any
lifts of $\alpha$ and $\beta$ to $\OL$.
Explicitly, the only possible eigenvalues of $U_p$
modulo $\m$ in higher weight are determined by $\rhobar$, and are either equal to $\alpha$, $\beta$,
or $0$. Yet $U_p$ acts invertibly on $\psi(H^0(X_1(Q),\omega_{\OL/\varpi^m})^2_{\m})$, as can be
seen by considering the matrix description of $U_p$ (which has invertible determinant given by~$\langle p \rangle$).

Let $\Tan_{n}$ denote the subalgebra of
$\End_{\OL}(H^0(X_1(Q),\omega^n_{\OL}))$ generated over $\OL$ by the
operators $T_x$ for primes $(x,NQp)=1$ and diamond operators $\langle
a \rangle$ for $(a,NQ)=1$. Let $\T_{n}$ denote the subalgebra of
$\End_{\OL}(H^0(X_1(Q),\omega^n_{\OL}))$ 
generated by $\Tan_{n}$ and $U_x$ for $x$ dividing
$Q$, and let $\WT_{n}$ denote the subalgebra generated by $\T_{n}$ and
$U_p$ (recall that we are denoting $T_p$ in weight $n$ by $U_p$).
By a slight abuse of notation, let $\mE$ denote the
maximal ideal of $\Tan_{n}$
corresponding to $\rhobar$. Similarly, let $\m$ denote the maximal
ideal of $\T_{n}$ generated by $\mE$ and $U_x-\alpha_x$ for $x\in
Q$.

Let $\hatm_{\alpha}$ and
$\hatm_{\beta}$ denote the 
ideals of $\WT_{n}$ containing
$\m$ and
$U_p - \alpha$ or $U_p - \beta$ respectively. 
If $\alpha = \beta$, we simply write
$\hatm = \hatm_{\alpha} = \hatm_{\beta}$.
Note that since $n>1$, we have 
\[ H^0(X_1(Q),\omega_{\OL}^n) \otimes \OL/\varpi^m  \iso
H^0(X_1(Q),\omega_{\OL/\varpi^m}^n) \]
and hence we may regard the latter as a module for  $\WT_n$ (and its sub-algebras $\T_n$ and
$\Tan_n$).  The proof of Theorem \ref{theorem:galoisweight1} will be
completed in Sections~\ref{section:interlude}--\ref{section:case3}.

 \subsection{Interlude: Galois Representations in higher weight} \label{section:interlude}
In this section, we summarize some results about
Galois representations associated to ordinary
Hecke algebras in weight $n \ge 2$. As above, let $\alpha$ and $\beta$
be the eigenvalues
of $\rhobar(\Frob_p)$.
There is a natural map $\T_{n,\m} \rightarrow \WT_{n,\wtm_{\alpha}}$.
If $\alpha = \beta$ this map is injective, otherwise, write
$\T_{n,\m_{\alpha}}$ for the image.
 There are continuous Galois representations
\[ \rho_{n,\alpha} : G_\Q \to \GL_2(\T_{n,\m_{\alpha}}) \] 
\[ \wrho_{n,\alpha} : G_\Q \to \GL_2(\WT_{n,\wtm_{\alpha}}) \]
with the following properties:
\begin{enumerate}[\quad (a)]
\item The representation $\wrho_{n,\alpha}$  is obtained from $\rho_{n,\alpha}$
by composing $\rho_{n,\alpha}$ with the natural inclusion
map $\T_{n,\m_{\alpha}} \rightarrow
\WT_{n,\wtm_{\alpha}}$.
\item\label{char-pol} $\rho_{n,\alpha}$  and $\wrho_{n,\alpha}$ are unramified at all primes $(x,pNQ)=1$ and
  the characteristic polynomial of $\rho_{n,\alpha}(\Frob_x)$ for such $x$ is
\[ X^2-T_x X + x^{n-1}\langle x \rangle .\]
\item If $E$ is a field of characteristic zero, and
$\phi: \WT_{n,\wtm} \rightarrow E$ is a homomorphism, then
 $\phi \circ \wrho_{n,\alpha}|G_p$
 is equivalent to a representation of the
  form
\[ 
\begin{pmatrix}
  \epsilon^{n-1}\lambda(\phi(\langle p \rangle)/\phi(U_p)) & * \\
  0 & \lambda(\phi(U_p)) 
\end{pmatrix}
\]
where $\lambda(z)$ denotes the unramified character sending $\Frob_p$
to $z$.
\end{enumerate}

These results follow from standard facts about Galois
representations attached to classical ordinary modular forms, together
with the fact that there is an inclusion
$$\T_{n,\m_{\alpha}} \hookrightarrow \WT_{n,\wtm_{\alpha}}
\hookrightarrow \prod  E_i$$
with $E_i$ running over a finite collection of finite extensions of
$K$ corresponding to the ordinary eigenforms of weight $n$ and level $\Gamma_1(NQ)$.

\begin{theorem} \label{theorem:higherweight} Under Assumption~\ref{novexing},
$\rho_{n,\alpha}$ is a deformation of $\rhobar$ that satisfies
conditions~\eqref{outside-NQ} and~\eqref{at-principal} of Definition
\ref{defn:minimal}, with the exception of the condition of being unramified at~$p$.
\end{theorem}

This follows from the choice of $N$ and $H$ together with
the local Langlands correspondence and results of Diamond--Taylor and
Carayol (see \cite[Lemma 5.1.1]{CDT}). (The choice of $H$ ensures that
for each prime $x\neq p$, $\det(\rho_{n,\alpha})|I_x$ has order prime
to $p$.)
Note that without Assumption~\ref{novexing}, the representation $\rho_{n,\alpha}$
still satisfies condition~\eqref{outside-NQ} of Definition \ref{defn:minimal}; the issue is that $\rho_{n,\alpha}$
may have extra ramification at those primes not in $T(\rhobar)$.

We now fix one of the eigenvalues of $\rhobar(\Frob_p)$, $\alpha$
say, and write $\wtm = \wtm_{\alpha}$. 
The existence of $\wrho_{n,\alpha}$ gives $B:= \WT^2_{n,\wtm}$ the structure
of a $\WT_{n,\wtm}[G_{\Q}]$-module. Recall that $G_p$ is the decomposition group of
$G_{\Q}$ at $p$.

\begin{lemma} \label{lemma:ABC} Suppose that $\rhobar(\Frob_p)$ is not a scalar.
Then there exists an exact sequence of 
$\WT_{n,\wtm}[G_p]$-modules
$$0 \rightarrow A \rightarrow B \rightarrow C \rightarrow 0$$
such that:
\begin{enumerate}
\item $A$ and $C$ are free $\WT_{n,\wtm}$-modules of rank one.
\item The sequence splits $B \simeq A \oplus C$ as a sequence of 
$\WT_{n,\wtm}$-modules.
\item The action of $G_p$ on $C$ factors through $G_{\F_p} = G_p/I_p$,
and Frobenius acts via the operator $U_p \in \WT_{n,\wtm}$.
\item The action of $G_{p}$ on $A$ is unramified and is via the character $\eps^{n-1} \lambda(\langle p \rangle U^{-1}_p)$.
\end{enumerate}
\end{lemma}

\begin{proof} 
Let $C$ denote the maximal $\WT_{n,\wtm}$-quotient
on which $\Frob_p$ acts by $U_p$.
The construction of $C$ is given by taking  a quotient,
and thus its formulation is preserved under taking quotients of $B$.
Since $B$ is free of rank two, $B/\wtm$ has dimension $2$.
The action of $U_p$ on $C/\wtm$ is, by definition, given by the scalar $\alpha$.
Yet $B/\wtm$ as a $G_{p}$-representation is given by $\rhobar$,
and $\rhobar(\Frob_p)$  either
has distinct eigenvalues $\alpha$ and $\beta$ or is non-scalar by definition.
Hence $\dim C/\wtm = 1$, and by Nakayama's lemma, $C$ is cyclic.
On the other hand, by  well-known
properties of
the Galois actions arising from classical modular forms, we see that $C \otimes \Q$ has rank one, and thus $C$ is
free as a $\WT_{n,\wtm}$-module. It follows that $B \rightarrow C$
splits, and that  $A$ is also free of rank one.
Considering  once more the local Galois structure of
representations arising from  classical modular forms, it follows that
$G_{p}$ acts on $A \otimes \Q$ via $\eps^{n-1} \lambda(\langle p \rangle U^{-1}_p)$.
Since $\WT_{n,\wtm}$ is $\OL$-flat, the action of $G_{p}$ on $A$
itself is given by the same formula, proving the lemma.
\end{proof}

When $\rhobar(\Frob_p)$ is scalar, there does not exist such a
decomposition.
Instead,  in~\S~\ref{section:case3}, we shall study the local properties
of $\rho_{n,\alpha}$ and $\wrho_{n,\alpha}$ using finer properties of local
deformation rings.

 \subsection{Proof of Theorem~\ref{theorem:galoisweight1},
  Case \texorpdfstring{$1$: $\alpha \ne \beta$, $\rhobar(\Frob_p)$}{} has distinct eigenvalues}
 \label{section:case1}

We claim there is an isomorphism of $\T^{\univ}$-modules:
$$H^0(X_1(Q),\omega_{\OL/\varpi^m})_{\m} \simeq (\psi
H^0(X_1(Q),\omega_{\OL/\varpi^m})^2)_{\hatm_{\alpha}}$$
obtained by composing $\phi$ with the $\WT_{n}$-equivariant projection onto $(\psi H^0(X_1(Q),\omega_{\OL/\varpi^m})^2)_{\hatm_{\alpha}}$.
The argument is  similar to the proof of
Lemma~\ref{lemma:matt-w1}. 
We know that $U_p$ satisfies the equation $X^2-T_p X+\langle p\rangle
=0$ on the image of $\psi$ but we may not use Hensel's Lemma to deduce that
  there exist $\walpha$ and $\wbeta$ in $\T_{\m}/I_m$ such that
$(U_p - \walpha)(U_p - \wbeta) = 0$ on the $\m$-part of the image of
$\psi$, since we do not know \emph{a priori} that
$T_p$ lies in $\T_{\m}$.  Instead, we note the following.
Since $U_p$ acts invertibly on the image of $\psi$, we deduce from the equality
$T_p = U^{-1}_p \langle p \rangle + U_p$ that $T_p - \alpha - \beta$ lies in $\hatm_{\alpha}$
and $\hatm_{\beta}$, and thus acts nilpotently on $H^0(X_1(Q),\omega_{\OL/\varpi^m})_{\m}$.
It follows that $\T_{\m}/I_m[T_p]\subset \End(H^0(X_1(Q),\omega_{\OL/\varpi^m})_{\m})$ is a local ring with maximal ideal $\wtm$ which acts on
$H^0(X_1(Q),\omega_{\OL/\varpi^m})_{\m}$.
The operator $U_p$ \emph{does} satisfy the quadratic relation $X^2 - T_p X + \langle p \rangle = 0$
over $\T_{\m}/I_m[T_p]$, and hence by Hensel's Lemma there exists $\walpha$ and $\wbeta$
in $\T_{\m}/I_m[T_p]$ such that $(U_p-\walpha)(U_p-\wbeta)=0$ on the
$\m$-part of the image of $\psi$.
The argument then proceeds as in the proof of Lemma~\ref{lemma:matt-w1}, noting
(tautologically) that $H^0(X_1(Q),\omega_{\OL/\varpi^m})_{\wtm} = H^0(X_1(Q),\omega_{\OL/\varpi^m})_{\m}$.

It follows from the result just established
that $\T_{\m}/I_m[T_p][U_p]\subset \End(\im(\psi)_{\hatm_{\alpha}})$ is a quotient of $\WT_{n,\hatm_\alpha}$,
and $\T_{\m}/I_m$ is the corresponding quotient of $\T_{n,\m_{\alpha}}$.
Note that the trace of any lift of Frobenius on 
the corresponding quotient of $\rho_{n,\alpha}$  is equal to
$U_p + \langle p \rangle U^{-1}_p$, which is equal to $T_p$
in $\End(\mathrm{Im}(\psi)_{\hatm_{\alpha}})$. (We use here, as below, that $\eps^{n-1}$
is trivial modulo $\varpi^m$.)
In particular, this implies that $T_p \in \T_{\m}/I_m$.
We now define $\rho_{Q,m}$ to be the composition of
$\rho_{n,\alpha}$ with the surjection
$\T_{n,\m_\alpha}\onto \T_{\m}/I_m$, and
$\wrho_{Q,m}$ to be $\wrho_{n,\alpha}$ on the corresponding quotient
$\T_{\m}/I_m[U_p]$ of $\WT_{n,\wtm_{\alpha}}$. (Since $U_p = \walpha$ in $\T_{\m}/I_m[U_p]$,
the corresponding quotients of $\T_{n,\m_{\alpha}}$ and $\WT_{n,\wtm_{\alpha}}$ are the same.)
The character $\nu$ defined by the formula
$\nu:=\langle \det \rhobar
\rangle \det(\rho_{Q,m})^{-1}$ is thus unramified outside $Q$ and of $p$-power order. Since $p>2$, $\nu$ admits a square root $\eta$ (also
unramified outside $Q$).  We have established that $\rho'_{Q,m}:=
\rho_{Q,m}\otimes \eta$ satisfies all the
conditions of Definition \ref{defn:minimal}, except the condition that
it be unramified at $p$.
Equivalently, it suffices to show that $\wrho_{Q,m}$ is unramified at $p$.
By Lemma~\ref{lemma:ABC}, we may write
\[ \wrho_{Q,m}|G_p \cong
\begin{pmatrix}
  \lambda(\wbeta) & * \\
  0 & \lambda(\walpha) 
\end{pmatrix}.
\]
By symmetry, we could equally well have defined $\rho_{Q,m}$ 
by regarding $\T_{\m}/I_m$ as a quotient of
$\T_{n,\m_\beta}$. (Note that the Chebotarev density theorem and
\cite[Th\'eor\`eme 1]{carayol} imply that $\rho_{Q,m}$ is uniquely determined
by the condition that $\tr\rho_{Q,m}(\Frob_x)=T_x$ for all $(x,pNQ)=1$.) 
It follows that we also have
\[ \wrho_{Q,m}|G_p \cong
\begin{pmatrix}
  \lambda(\walpha) & * \\
  0 & \lambda(\wbeta) 
\end{pmatrix}.
\] 
Since $\alpha\neq\beta$, this forces $\rho_{Q,m}|G_p$ to split as a direct
sum of the unramified character $\lambda(\walpha)$ and
$\lambda(\wbeta)$. (Moreover, we see that 
$T_p = U_p^{-1}\langle p\rangle + U_p 
= \walpha + \wbeta = \tr(\rho_{Q,m}(\Frob_p)) \in \T_{\m}$.)

\subsection{Proof of Theorem~\ref{theorem:galoisweight1}, 
Case \texorpdfstring{$2$: $\alpha = \beta$, $\rhobar(\Frob_p)$}{} non-scalar}

We will assume below that~$\alpha = \beta$ is a generalized
eigenvalue of  $\rhobar(\Frob_p)$, and furthermore that $\rhobar(\Frob_p)$ is non-scalar.
However, we first prove the lemma below.

\begin{lemma} [Doubling] \label{lemma:doubling} Without any assumption
on $\rhobar(\Frob_p)$, the action of 
$\WT_{n,\wtm}$ on 
$\psi(H^0(X_1(Q),\omega_{\OL/\varpi^m})^2_{\m})$
factors through a quotient isomorphic to
$$\T_{\m}/I_m[T_p][X]/(X^2 - T_p X + \langle p \rangle),$$
where $U_p$ acts by $X$.
\end{lemma}

\begin{proof} The action of $\WT_{n,\wtm}$ certainly contains
$T_p:=U_p + \langle p \rangle U^{-1}_p$, and moreover $U_p$
also satisfies the indicated relation. Thus it suffices to show that $U_p$
does not satisfy any \emph{further} relation. Such a relation would be of the form
$A U_p + B = 0$ for operators $A$, $B$ in $\T_{\m}/I_m[T_p]$.
By considering the action of $U_p$ as a matrix on the image of $\psi$, however, this
would imply an identity:
$$A \left( \begin{matrix}  T_p & 1 \\ -\langle p \rangle & 0 \end{matrix} \right)
 + B \left(\begin{matrix} 1 & 0 \\ 0 & 1 \end{matrix} \right) = 
 \left( \begin{matrix} 0 & 0 \\ 0 & 0 \end{matrix} \right),$$
 from which one deduces that $A = B = 0$ (the fact that one can
 deduce the vanishing of the entries is equivalent to the injectivity of $\psi$).
\end{proof}

Following Wiese~\cite{Wiese}, we call this phenomenon ``doubling'', because the corresponding
quotient of $\WT_{n,\wtm}$ contains two copies of the image of $\T_{n,\m}$.
We show that this implies that the corresponding Galois representation is unramified.

Note that the trace under~$\rho_{n,\alpha}$ of any lift of Frobenius
is sent 
to
$U_p + \langle p \rangle U^{-1}_p 
=  T_p$ in~$\T_{\m}/I_m$, and so
this in particular implies that $T_p \in \T_{\m}/I_m$. The image of $\T_{n,\m}$ under the map
$$\T_{n,\m} \rightarrow \WT_{n,\wtm} \onto \T_{\m}/I_m[T_p][U_p]$$
is given by $\T_{\m}/I_m=\T_{\m}/I_m[T_p]$. 
We thus obtain a Galois representation
$$\rho_{Q,m}: G_{\Q} \rightarrow \GL_2(\T_{\m}/I_m).$$
As in~\S~\ref{section:case1},
it suffices to prove that $\rho_{Q,m}$ is unramified at $p$.
Consider the Galois representation
$\wrho_{Q,m}: G_{\Q} \rightarrow \GL_2(\T_{\m}/I_m[U_p])$ obtained
by tensoring over $\T_{n,\m}$ with $\WT_{n,\wtm}$.
By  Lemma~\ref{lemma:doubling}, there is an isomorphism
$$\T_{\m}/I_m[U_p] \simeq \T_{\m}/I_m \oplus \T_{\m}/I_m$$
as a $\T_{\m}/I_m$-module. Since $\wrho_{Q,m}$
is obtained from $\rho_{Q,m}$ by tensoring with a doubled
module, it follows that
 there is an isomorphism
$\wrho_{Q,m} \simeq \rho_{Q,m} \oplus \rho_{Q,m}$
as a  $\T_{\m}/I_m[G_p]$-module (or even $\T_{\m}/I_m[G_{\Q}]$-module).

\begin{lemma} Let  $(R,\m)$ be a local ring, and let $N$, $M$, and $L$ be $R[G_{p}]$-modules
which are free $R$-modules of rank two.
Suppose there is an exact sequence of $R[G_{p}]$-modules
$$0 \rightarrow N \rightarrow M \oplus M \rightarrow L \rightarrow 0$$
which is split as a sequence of $R$-modules. Suppose
that $L/\m$ is indecomposable as  a $R[G_{p}]$-module.
Then $N \simeq M \simeq L$ as $R[G_{p}]$-modules.
\end{lemma}

\begin{proof}
This is Proposition~4.4 of Wiese~\cite{Wiese} (we use $L$ here instead of $Q$ in~\cite{Wiese}
 to avoid
notational conflicts). Note that
the lemma is stated for $\F_p$-algebras $R$ and the stated condition is on the sub-module $N$
rather than the quotient $L$, but the proof is exactly the same.
\end{proof}

We apply this as follows. Consider the sequence of 
$\WT_{n,\wtm}$-modules considered in Lemma~\ref{lemma:ABC}.
If we tensor this sequence with the quotient of
$\WT_{n,\wtm}$ corresponding to the doubling isomorphism
$$\T_{\m}/I_m[U_p] = \T_{\m}/I_m \oplus \T_{\m}/I_m,$$
then the corresponding 
quotient of $B$ is  $\wrho_{Q,m}$,
which, by doubling, is
free of rank $4$ over $\T_{\m}/I_m$ and as a $\T_{\m}/I_m[G_{p}]$-module
is given by $\rho_{Q,m} \oplus \rho_{Q,m}$.
The corresponding quotients $A$ and $C$ are similarly free over
$\T_{\m}/I_m$ of rank $2$.
The action of $\Frob_p$ on $L/\m$ is given by $U_p$.
Since $U_p$ does not lie in $\m$
--- as this would contradict doubling --- it follows
that $(U_p - \alpha)$ acts nilpotently but non-trivially on $L/\m$,
and hence $L/\m$ is indecomposable (indeed, by construction,
$L/\m$ is free of rank one over $k[U_p]/(U_p - \alpha)^2$). Hence, by the lemma above,
there are isomorphisms $L \simeq \rho_{Q,m}$ as a $G_{p}$-module.
Yet $L$ is a quotient of $C$, which is  by construction unramified, and thus $\rho_{Q,m}$ is
also unramified.
Finally, we note that the trace of Frobenius
at $p$ is given by $U_p + \langle p \rangle U^{-1}_p = T_p$, so $T_p 
=\tr( \rho_{Q,m}(\Frob_p))$.

\subsection{Proof of Theorem~\ref{theorem:galoisweight1}, Case 
\texorpdfstring{$3$: $\alpha = \beta$, $\rhobar(\Frob_p)$}{} scalar} \label{section:case3}
The construction of the previous section 
gives a representation $\rho_{Q,m}$ which satisfies all the required deformation properties
with the possible exception of knowing that $\rho_{Q,m}$ is unramified at $p$.
In order to deal with the case when $\rhobar(\Frob_p)$ is scalar,
we shall have to undergo a closer study of local deformation rings.
Suppose that
$$\rhobar:G_{p} \rightarrow \GL_2(k)$$
is trivial. (If $\rhobar$ is scalar, it is trivial after twisting.)
 We introduce some local framed universal deformation rings associated
 to  $\rhobar$. Fix a lift $\phi_p \in G_p$ of $\Frob_p$.

In the definition below,   an \emph{eigenvalue} of a linear operator is defined to
be a root of
the corresponding characteristic polynomial.

\begin{df} For $A$ in $\CC_{\OL}$, let $D(A)$ denote the set of framed deformations of $\rhobar$ to
$A$, and let $\widetilde{D}(A)$ denote the framed
deformations together with an eigenvalue $\alpha$ of $\phi_p$.
Let these functors be represented by rings $\Rv$ and $\Rva$ respectively.
\end{df}

There is a natural inclusion $\Rv \rightarrow \Rva$ and $\Rva$ is
isomorphic to a quadratic extension of $\Rv$ (given by the
characteristic polynomial of $\phi_p$).
Kisin constructs certain quotients of $\Rv$ which capture characteristic zero quotients with good $p$-adic Hodge theoretic properties.
Let $\eps$ denote the cyclotomic character, let $\omega$ denote the
Teichm\"uller lift of the mod-$p$ reduction of $\eps$, and let $\chi = \eps \omega^{-1}$,
so $\chi \equiv 1 \mod \varpi$. We modify the choice of $\phi_p$ if
necessary so that $\chi(\phi_p)=1$. Let $R^{\univ,\chi^{n-1}}$ and
$\widetilde{R}^{\univ,\chi^{n-1}}$ denote the quotients of $\Rv$ and
$\Rva$ corresponding to deformations with determinant $\chi^{n-1}$.

\begin{theorem}
\label{thm:ord-def-ring}
 Fix an integer $n \ge 2$. 
\begin{enumerate}
\item There exists a unique
  reduced $\OL$-flat quotient
$\Rda$ of $\widetilde{R}^{\univ,\chi^{n-1}}$ such that points on the generic fiber of $\Rda$ correspond to
representations $\rho:G_{p} \rightarrow \GL_2(E)$ such that:
$$\rho \sim \left( \begin{matrix} \chi^{n-1} \lambda(\alpha^{-1}) & * \\ 0 &  \lambda(\alpha) \end{matrix}
\right)$$
\item The ring $\Rda$ is an integral domain which is normal,
  Cohen--Macaulay, and of relative dimension 4 over $\OL$.
\end{enumerate}
\end{theorem}

\begin{proof} Part $1$ of this theorem is due to Kisin. For part $2$,
  the fact that $\Rda$ is an integral domain follows from the proof of
  Lemma 3.4.3 of \cite{G}. The rest of part $2$ follows 
from the method and results of Snowden (\cite{Snowden}). 
More precisely, Snowden works over an arbitrary finite extension of $\Q_p$ containing
$\Q_p(\zeta_p)$, and assumes that $n = 2$, so $\chi = \eps \omega^{-1} = \eps$.
However, this is exactly the hardest case --- since for us $p \ne 2$, $\chi^{n-1} \ne \eps$,
 our deformation problem
 consists of a single potentially crystalline component. In particular, the arguments of~\cite{Snowden}
show that $\Rda \otimes k$ is an integral
normal Cohen--Macaulay ring of dimension four, which is not Gorenstein,
and is identified (in the notation of \emph{ibid}.) with  the completion of
$\mathcal{B}_1$ at $b = (1,1;0)$.
\end{proof}

Let $\Rd$ denote the image of $\Rv$ in $\Rda$.
We also define the following rings:

\begin{df} Let $\Ru$ denote the largest quotient of $\Rd$ corresponding to unramified
deformations of $\rhobar$. Let $\Rua$ denote the corresponding quotient of $\Rda$.
\end{df}

We are now in a position to define two ideals of $\Rv$.

\begin{df} The \emph{unramified ideal} $\II$ is the kernel of the map $\Rv \rightarrow \Ru$.
 The \emph{doubling ideal} $\JJ$ is the annihilator of $\Rda/\Rd$ as an $\Rv$-module.
 \end{df}

\begin{lemma} There is an equality $\JJ =  \II$. \label{lemma:IIinJJ}
\end{lemma}

\begin{proof}  We first prove the inclusion $\JJ \subseteq \II$. By definition, $\Rv/\JJ$ acts faithfully on $\Rda/\Rd$, 
and it is the largest such quotient. Hence it suffices to show that
$\Rv/\II$ acts faithfully on 
$$(\Rda/\Rd) \otimes \Rv/\II
\simeq (\Rda/\II)/(\Rd/\II).$$
Since
$\Rda/\II \simeq \Rua$ and $\Rd/\II \simeq \Ru$,  it suffices to show
that $\Rua/\Ru$ is a faithful $\Ru = \Rv/\II$-module. (It is not \emph{a priori} obvious that
the map $\Ru \rightarrow \Rua$ is injective, so the notation $\Rua/\Ru$ is slightly misleading; however,
we prove it is so by explicit computation below.)
Let $\varpi^m$ denote the greatest power of $\varpi$ dividing $(\chi^{n-1}(g) - 1) \OL$
for all $g$ in the decomposition group at $p$. 
By considering determinants, $\Ru$ (and $\Rua$) is annihilated by $\varpi^m$.
The moduli space of matrices $\phi$ in $\OL/\varpi^m$ which
are trivial modulo $\varpi$ and have determinant one, that is, with
$$\phi = \left( \begin{matrix} 1 + \phi_1 & \phi_2 \\ \phi_3 & 1 + \phi_4 \end{matrix} \right),$$
is represented by:
$$\OL/\varpi^m[[\phi_1,\phi_2,\phi_3,\phi_4]]/(\phi_1 + \phi_4 + \phi_1 \phi_4 - \phi_2 \phi_3).$$
We show shortly that this ring is isomorphic to $\Ru$; admit this for a moment. 
The corresponding moduli space $\Rua$
of such matrices together with an eigenvalue $\alpha = 1 + \beta$ is represented by
$$\Rua \simeq \Ru[\beta]/(\beta^2 - (\phi_1 + \phi_4) \beta  - (\phi_1 + \phi_4))
\simeq \Ru \oplus \Ru,$$
where the last isomorphism is as an $\Ru$-module. Clearly  $\Ru$
acts faithfully on $(\Ru \oplus \Ru)/\Ru \simeq \Ru$, proving the inclusion $\JJ \subseteq \II$.
We now prove the equality of rings above. It suffices to prove it for $\Rua$. By construction,
the ring above certainly surjects onto $\Rua$. Hence, it suffices to show that this ring is naturally
a quotient of $\Rda$.
As in Snowden, the ring $\Rda$   represents the functor given by deformations to $A$ with
eigenvalue $\alpha$ satisfying the following equations:

\begin{enumerate}
\item $\phi \in M_2(A)$ has determinant $1$.
\item $\alpha$ is a root of the characteristic polynomial of $\phi$.
\item $\tr(g) = \chi^{n-1}(g) + 1$ for $g \in I_p$.
\item $(g-1)(g'-1) = (\chi^{n-1}(g) - 1)(g' - 1)$ for $g, g' \in I_p$.
\item $(g-1)(\phi - \alpha) = (\chi^{n-1}(g) - 1)(\phi - \alpha)$ for $g \in I_p$.
\item $(\phi - \alpha)(g -1) = (\alpha^{-1} - \alpha)(g - 1)$ for $g \in I_p$.
\end{enumerate}
To understand where these equations come from, one should imagine writing down the following equations:
$$\rho(\phi_p) = \phi \approx
\left( \begin{matrix} \alpha^{-1} & * \\ 0 & \alpha  \end{matrix} \right),
\qquad \rho(g) \approx \left( \begin{matrix}\chi^{n-1}(g)
  & * \\ 0 & 1 \end{matrix} \right), \ g \in I_p.$$
  We caution, however, that although one can find such a basis for any representation when $A$
  is a field,
  we do not claim that there exists any universal such basis (indeed, we presume that there does not).
Returning to our argument, it is now  trivial to observe that the quotient of $\Rda$ in which $g \in I_p$ is the identity is equal to the
ring we asserted to be $\Rua$ above (and that the image of $\Rv$ in this ring is what we asserted to be $\Ru$).

\medskip

We now prove the opposite inclusion, namely that $\II \subseteq \JJ$.
Instead of writing down a presentation of $\Rda$, it will suffice to note the following, which follows
from the explicit description above:
The ring $\Rda$ is generated over $\OL$ by the following elements which
all lie in the maximal ideal:
\begin{enumerate} 
\item Parameters $\phi_i$ (for $i = 1$ to $4$) corresponding to the image of
$\rho(\phi_p) - 1$,
\item Parameters $x_{ij}$ for $i = 1$ to $4$ and a finite number of $j$
 corresponding to the image of inertial elements $m_j = \rho(g_j) - 1$.
 \item An element $\beta$, where $\alpha = 1+ \beta$ is an eigenvalue of
 $\rho(\phi_p)$.
 \end{enumerate}
 Moreover, $\Rd$ is generated as a sub-algebra by $\phi_i$ and
 the $x_{ij}$, and $\beta$ satisfies
 $$\beta^2 - (\phi_1 + \phi_4) \beta - (\phi_1 + \phi_4) = 0.$$
 Since the determinant of $\phi$ is one, it follows that
 $\alpha + \alpha^{-1} = 2 + \phi_1 + \phi_4$.
 By definition,
 there is a decomposition of $\Rd$-modules $\Rda/\JJ = \Rd/\JJ \oplus
 \beta \Rd/\JJ$ with each summand being free over $\Rd/\JJ$.
From the equality ($6$), we deduce that the relation
 $$(\phi - 1)m_j  - (\phi_1 + \phi_4) m_j = (\alpha^{-1} - 1 - \phi_1 - \phi_4) m_j = 
 -(\alpha - 1) m_j = -\beta m_j$$
 holds in $M_2(\Rda)$, and hence also in $M_2(\Rda/\JJ)$.
 Yet by assumption, over $\Rd/\JJ$, the modules $\Rd/\JJ$ and $\beta \Rd/\JJ$ have
 trivial intersection, from which it follows that $\beta m_j = 0$ in $M_2(\beta \Rd/\JJ)$.
 In particular, since the latter module is generated by $\beta$, we must have
 $x_{ij} \in \JJ$ for all $i$ and $j$. Since $\II$ is generated by $x_{ij}$,
 we deduce that $\II \subset \JJ$, and hence that
 $\II = \JJ$.
 \end{proof}

\begin{remark} \emph{Why might one expect an equality $\II = \JJ$?
One reason is as follows. The doubling ideal $\JJ$ represents the largest
quotient of $\Rda$ on which the eigenvalue of Frobenius $\alpha$
\emph{cannot be distinguished from its inverse $\alpha^{-1}$}. 
Slightly more precisely, it is the largest quotient for which there is an isomorphism
$\Rda/\JJ \rightarrow \Rda/\JJ$ fixing the image of $\Rd$ and sending $\alpha$ to~$\alpha^{-1}$.
It is clear that such an isomorphism exists for
 unramified representations. Similarly, for a ramified ordinary
quotient, one might expect that the $\alpha$ can be distinguished
from $\alpha^{-1}$ by looking at the ``unramified quotient line'' of the representation.
Indeed, for characteristic zero representations this is clear --- one even
has $\Rd[1/\varpi] \simeq \Rda[1/\varpi]$.}
\end{remark}

\begin{lemma} There is a surjection $\T_{n,\m} \otimes_{\Rd} \Rda \rightarrow \WT_{n,\wtm}$.
\end{lemma}

\begin{proof} Recall that $\WT_{n,\wtm} = \T_{n,\m}[U_p]$. Since $U_p$ is given as an eigenvalue
of Frobenius,  $\WT_{n,\wtm}$ naturally has the structure of a $\Rva$-algebra. We claim that
the map from $\Rva$ to $\WT_{n,\wtm}$ factors through $\Rda$. Since $\WT_{n,\wtm}$ acts
faithfully on a space of modular forms,
there is an injection:
$$\WT_{n,\wtm} \hookrightarrow \prod E_i$$
into a product of fields corresponding to the Galois representations associated
to the ordinary modular forms of weight $n$ and level $\Gamma_1(NQ)$.
By the construction of $\Rda$, it follows that the map from $\Rva$ to
this product factors through
via $\Rda$. This also implies that the map from  $\Rv$ to $\WT_{n,\wtm}$ 
(and hence 
to $\T_{n,\m}$) factors through
$\Rd$, and hence there exists a map
$$\T_{n,\m} \otimes_{\Rd} \Rda \rightarrow \WT_{n,\wtm},$$
sending $\alpha \in \Rda$ to $U_p$. Yet the image of this map contains
both $\T_{n,\m}$ and $U_p$, and is thus surjective.
\end{proof}

\begin{df} Let the \emph{global
doubling ideal} $\JJg$ be the annihilator of $\WT_{n,\wtm}/\T_{n,\m}$ as an $\Rd$-module.
\end{df}

Since there is a surjection
$\T_{n,\m} \otimes_{\Rd} \Rda \rightarrow \WT_{n,\wtm}$, it follows that
 $\WT_{n,\wtm}/\T_{n,\m}$
 is a quotient of
 $$(\T_{n,\m} \otimes_{\Rd} \Rda)/\T_{n,m}
 = (\T_{n,\m} \otimes_{\Rd} \Rda)/\T_{n,m} \otimes_{\Rd} \Rd
 \simeq \T_{n,\m} \otimes_{\Rd} \Rda/\Rd$$
  as an $\Rd$-module.
 In particular, by considering the action on the last factor, we deduce that $\JJ \subset \JJg$.
 In particular, $\II \subset \JJg$, or equivalently, on any quotient of $\T_{n,\m}$ on which
 the corresponding quotient of $\WT_{n,\wtm}$ is doubled
 (in the sense that the quotient of $\WT_{n,\wtm}$ is free of rank $2$ as a module
 for the image of $\T_{n,\m}$), the action of the Galois group
 at $p$ is unramified. 
 In particular, by Lemma~\ref{lemma:doubling}, this applies to
 the quotient of $\WT_{n,\wtm}$ given by $\T_{\m}/I_m[T_p][U_p]$.
 Specifically, as in  the previous sections, we obtain corresponding 
 Galois representations:
$$\rho_{Q,m}:G_{\Q} \rightarrow \GL_2(\T_{\m}/I_m), \qquad 
\wrho_{Q,m}: G_{\Q} \rightarrow \GL_2(\T_{\m}/I_m[U_p]).$$
(The trace of any lift of Frobenius on this quotient is equal to
$U_p + \langle p \rangle U^{-1}_p  =  T_p$, and so
$T_p \in \T_{\m}/I_m$.)
From the discussion above, we deduce that $\wrho_{Q,m}$ and
thus $\rho_{Q,m}$ is unramified at $p$, and that
$\mathrm{Trace}(\rho_{Q,m}(\Frob_p)) = T_p$. The rest of the argument
follows as in Case~\ref{section:case1}, and
this completes the proof of
Theorem~\ref{theorem:galoisweight1} 
 \end{proof}

\subsection{Modularity Lifting}
\label{sec:modularity-lifting-w1}
We now return to the situation of Section
\ref{sec:deform-galo-repr-w1}. 
Taking $Q=1$ in Theorem~\ref{theorem:galoisweight1}, we obtain a
minimal deformation $\rho':G_\Q \to \GL_2(\T_{\emptyset,\mE})$ of $\rhobar$ and hence a
homomorphism $\varphi:R^{\min} \to \T_{\emptyset,\mE}$ which is easily
seen to be surjective. Recall that Assumption \ref{novexing} is still
in force.

\begin{theorem}
  \label{thm:main-thm-w1}
The map $\varphi : R^{\min}\onto\T_{\emptyset,\mE}$ is an isomorphism and
$\T_{\emptyset,\mE}$ acts freely on $H_0(X,\omega)_{\mE}$.
\end{theorem}

\begin{proof}
 We view $H_0(X,\omega)_{\mE}$ as an $R^{\min}$-module via $\varphi$. Since $\varphi$ is surjective, to prove the theorem, it suffices to show
  that $H_0(X,\omega)_{\mE}$ is free over $R^{\min}$. To show this, we will apply Proposition
  \ref{prop:patchingimaginary}.

We set $R=R^{\min}$ and $H=H_0(X,\omega)_{\mE}$ and we define 
\[ q := \dim_k H^1_\emptyset(G_{\Q},\ad^0 \rhobar(1)) .\]
Note that $q \geq 1$ by Proposition \ref{prop:tw-primes-w1}. As in
Proposition \ref{prop:patchingimaginary}, we set $S_N=\OL[(\Z/p^N\Z)^q]$ for
each integer $N\geq 1$ and we let
$R_\infty$ denote the power series ring
$\OL[[x_1,\ldots,x_{q-1}]]$. 
For each integer $N\geq 1$, fix a set of primes $Q_N$ of $\Q$
satisfying the properties of Proposition \ref{prop:tw-primes-w1}. We can
and do fix a surjection $\wt\phi_N : R_\infty \onto R_{Q_N}$ for each
$N\geq 1$. We let $\phi_N$ denote the composition of $\wt\phi_N$ with
the natural surjection $R_{Q_N}\onto R^{\min}$. Let
\[ \Delta_{Q_N} = \prod_{x\in Q_N}(\Z/x)^\times \] and choose a
surjection $\Delta_{Q_N}\onto \Delta_N:= (\Z/p^N\Z)^q$. Let
$X_{\Delta_N}(Q_N)\to X_0(Q_N)$ denote the corresponding Galois
cover. For each $x \in Q_N$, choose an eigenvalue $\alpha_x$ of
$\rhobar(\Frob_x)$. We let $\T_{Q_N}$ denote the Hecke algebra denoted
$\T$ in Section~\ref{specify} with the $Q$ of that section taken to be
the current $Q_N$. We let $\m$ denote the maximal ideal of $\T_{Q_N}$
generated by $\mE$ and $U_x -\alpha_x$ for each $x\in Q_N$. We set
$H_N:= H_0(X_{\Delta_N}(Q_N),\omega_{\OL})_\m$. Then $H_N$ is
naturally an $\OL[\Delta_N]=S_N$-module. By Theorem~\ref{theorem:galoisweight1},
we deduce the
existence of a surjective homomorphism $R_{Q_N}\onto
\T_{Q_N,\m}$.  Since $\T_{Q_N,\m}$ acts on $H_N$, we get an induced
action of $R_\infty$ on $H_N$ (via $\wt\phi_N$ and the map
$R_{Q_N}\onto \T_{Q_N,\m}$). We can therefore view $H_N$ as a module
over $R_\infty\otimes_{\OL}S_N$. 

To apply Proposition
\ref{prop:patchingimaginary}, it remains to check points~(\ref{cond-image})--(\ref{cond-balanced}). We check these conditions one by one:
\begin{itemize}
\item[(a)] The image of $S_N$ in $\End_\OL(H_N)$ is contained in the
  image of $R_\infty$ by construction (see
  Theorem~\ref{theorem:galoisweight1}). The second part of
  condition~\eqref{cond-image} is a consequence of the following:
  for each $x \in Q_N$, the restriction to $G_x$ of the universal
  representation $G_{\Q} \to \GL_2(R_{Q_N})$ is of the form
  $\chi_{\alpha_x} \oplus \chi_{\beta_x}$ where each summand is of
  rank 1 over $R_{Q_N}$ and where $\chi_{\alpha_x}$ lifts
  $\lambda(\alpha_x)$. By restricting $\chi_{\alpha_x}$ to $I_x$ for
  each $x\in Q_N$, we obtain, by local class field theory, a map
  $\OL[\Delta_{Q_N}] \to R_{Q_N}$. The quotient of $R_{Q_N}$ by the
  image of the augmentation ideal of $\OL[\Delta_{Q_N}]$ is just $R^{\min}$.
\item[(b)] As in the proof of Proposition~\ref{prop:balanced-homology-w1},
we have a Hochschild-Serre spectral sequence
\[ \Tor_i^{S_N}(H_j(X_{\Delta_N}(Q_N),\omega)_\m,\OL)\implies
H_{i+j}(X_0(Q_N),\omega)_\m .\] We see that $(H_N)_{\Delta_N} \cong
H_0(X_0(Q_N),\omega)_\m$. Then, by Lemmas~\ref{lemma:matt-w1}
and~\ref{lem:cohom-basic-props} we obtain an isomorphism
$(H_N)_{\Delta_N} \cong H_0(X,\omega)_{\mE} = H$, as required.
\item[(c)] The module $H_N$ is finite over $\OL$ and hence over
  $S_N$. Proposition \ref{prop:balanced-homology-w1} implies that
  $d_{S_N}(H_N)\geq 0$.
\end{itemize}
We may therefore apply Proposition \ref{prop:patchingimaginary} to deduce that
$H$ is free over $R$, completing the proof.
\end{proof}

We now deduce Theorem~\ref{weightone}, under
Assumption~\ref{novexing}, from the previous result. In the statement
of Theorem~\ref{weightone}, we take $X_U=X=X_1(N)/H$ and
$\CL_\sigma=\OL_X$ and, as in the statement, we let $\T$ be the Hecke algebra of
$H^1(X,\omega)$ (generated by prime-to-$Np$ Hecke operators) and $\m$
the maximal ideal of $\T$ corresponding to $\rhobar$. Analogous to the
discussion preceding Prop.~\ref{prop:balanced-homology-w1}), we have
a Hecke equivariant isomorphism  $H_0(X,\omega)\iso H^1(X,\omega)$ which thus gives
rise to an isomorphism $\T_{\emptyset,\mE}\iso \T_\m$. 

We also show that
$H_0(X,\omega)_{\mE}$ has rank one as a $\T_{\emptyset,\mE}$-module: this follows by multiplicity one for $\GL(2)/\Q$ if $H_0(X,\omega_K)_{\mE}$
is non-zero. In the finite case, we argue as follows.
By Nakayama's lemma it suffices to show that
$H^0(X,\omega_{k}(-\infty))[\mE]$ has dimension one.
We claim that~$U_x \in \T_{\emptyset,\mE}$ for all~$x|N$. This is a consequence of the assumption that~$N(\rhobar) = N$ as we now explain.
Suppose that~$x \| N$. Then the~$\T^2_{\m}$-representation has a unique invariant~$\T_{\m}$-line on
which~$\Frob_x$ acts by~$U_x$, and so~$U_x \in \T_{\m}$. On the other hand, if~$x^2 | N$, then~$U_x = 0$ is also in~$\T_{\m}$.
Since we have also shown that $T_p \in \T_{\emptyset,\mE}$, we may
deduce this from the fact that $q$-expansion is completely
determined by the Hecke eigenvalues $T_x$ for all $(x,N) = 1$ and~$U_x$ for all~$x | N$.

\subsection{Vexing Primes}
\label{section:vexing}
\label{sec:mod-curves-vexing}

In this section, we detail the modifications to the previous arguments which
are required to deal with vexing primes. To recall the difficulty, recall
that a prime $x$ different from $p$ is {\bf vexing\rm} if:
\begin{enumerate}
\item $\rhobar | D_x$ is absolutely irreducible.
\item $\rhobar | I_x \simeq \xi \oplus \xi^c$ is reducible.
\item $x \equiv -1 \mod p$.
\end{enumerate}
The vexing nature of these primes can be described as follows:
in order to realize $\rhobar$ automorphically, one must work with
$\Gamma_1(x^n)$ structure where $x^n$ is the Artin conductor of $\rhobar | D_x$.
However, according to local Langlands, at such a level we also expect to see
\emph{non}-minimal deformations of $\rhobar$, namely, deformations with
$\rho | I_x \simeq \psi \langle \xi \rangle \oplus \psi^{-1}\langle \xi^c \rangle$, where $\psi$
is a character of $(\F_{x^2})^{\times}$ of $p$-power order.
Diamond~\cite{DiamondVexingTwo} was the first to address this problem by observing that one can
cut out a smaller space of modular forms by using the  local Langlands correspondence.
The version of this argument in~\cite{CDT} can be explained as follows.
By Shapiro's Lemma, working with trivial coefficients at level $\Gamma(x^n)$ is the same as working at level
prime to $x$ where one now replaces trivial coefficients $\Z$ by a local
system $\CF$ corresponding to the group ring of the corresponding geometric cover.
In order to avoid non-minimal lifts of $\rhobar$,
 one works with a smaller local system $\CF_{\sigma}$  cut out of
$\CF$ by a representation $\sigma$ of the Galois group of the cover to
capture exactly the minimal automorphic lifts of $\rhobar$. The
representation $\sigma$ corresponds to a fixed inertial type at $x$.
In our setting (coherent cohomology) we may carry out a completely analogous construction.
Thus, instead,
we shall construct a vector bundle $\CL_{\sigma}$ on $X$. 
We then replace
$H^*(X_1(N),\omega)$ by the groups $H^*(X_1(N),\omega \otimes  \CL_{\sigma})$.
The main points to check are as follows:
\begin{enumerate}
\item The spaces $H^0(X_1(N),\omega^{\otimes n} \otimes  \CL_{\sigma})$ 
for $n \geq 1$ do
indeed cut out the requisite spaces of automorphic forms.
\item This construction is sufficiently functorial so that all the associated cohomology
groups admit actions by Hecke operators.
\item These cohomology groups  inject into natural spaces of $q$-expansions.
\item  This construction is compatible with arguments involving the
Hochschild--Serre spectral sequence and Verdier duality.
\end{enumerate}

\medskip

We start by discussing some more refined properties of modular curves,
in the spirit of~\S~\ref{sec:mod-curves}.
Let $S(\rhobar)$, $T(\rhobar)$ and $Q$ be as in
Section~\ref{sec:deform-galo-repr-w1}. Let $P(\rhobar)$ denote the set
of $x\in S(\rhobar)-T(\rhobar)$ where $\rhobar$ is ramified and reducible.
 
We will now introduce compact open subgroups $V\vartriangleleft U\subset
\GL_2(\A^\infty)$ and later we will fix a representation $\sigma$ of $U/V$ on a
finite free $\OL$ module $W_\sigma$. (In applications, $U$, $V$ and $\sigma$ will be chosen to capture all
minimal modular lifts of $\rhobar$. If the set of vexing primes $T(\rhobar)$ is
empty, then $U=V$ and all minimal lifts of $\rhobar$ will appear in $H^0(X_U,\omega)$. As indicated above, there is a
complication if $T(\rhobar)$ is non-empty. In this case, minimal modular lifts of $\rhobar$
will appear in the $\sigma^*:=\Hom(\sigma,\OL)$-isotypical part of $H^0(X_V,\omega)$.)

For each prime $x\in S(\rhobar)$, let $c_x$ denote the Artin conductor
of $\rhobar|G_x$. Note that $c_x$ is even when $x\in T(\rhobar)$. For $x\in S(\rhobar)$, we define subgroups $V_x
\subset U_x \subset \GL_2(\Z_x)$ as follows:
\begin{itemize}
\item If $x \in P(\rhobar)$, we let 
\[ U_x=V_x = \left\{ g\in \GL_2(\Z_x) : g\equiv
  \begin{pmatrix}
    * & * \\
    0 & d 
  \end{pmatrix}
\mod x^{c_x} \text{,\  $d\in(\Z/x^{c_x})^\times$ has $p$-power order}
\right\} .\]
\item  If $x \in T(\rhobar)$, then let $U_x = \GL_2(\Z_x)$ and 
\[ V_x = \ker \left( \GL_2(\Z_x) \lra \GL_2(\Z/x^{c_x/2})\right). \]
\item If $x \in S(\rhobar)- (T(\rhobar)\cup P(\rhobar))$, 
\[ U_x = V_x = \left\{ g\in \GL_2(\Z_x) : g\equiv
  \begin{pmatrix}
    * & * \\
    0 & 1 
  \end{pmatrix}
\mod x^{c_x}\right\}. \]
\end{itemize}

 For $x$ a prime not in $S(\rhobar)$, we let 
\[ U_x=V_x=\GL_2(\Z_x)\]
Finally, if $x$ is any rational prime, we define subgroups $U_{1,x}\subset U_{0,x} \subset
\GL_2(\Z_x)$ by:
\begin{eqnarray*} U_{0,x} & = & \left\{ g\in \GL_2(\Z_x) : g\equiv
  \begin{pmatrix}
    * & * \\
    0 & * 
  \end{pmatrix}
\mod x \right\} \\  
U_{1,x} & =& \left\{ g\in \GL_2(\Z_x) : g\equiv
  \begin{pmatrix}
    * & * \\
    0 & 1 
  \end{pmatrix}
\mod x 
\right\}. 
\end{eqnarray*}

We now set
\begin{align*}
  U=\prod_x U_x, \qquad U_i(Q)=\prod_{x\not
  \in Q}U_x\times \prod_{x\in Q} U_{i,x} \\ 
  V=\prod_x V_x, \qquad V_i(Q)=\prod_{x\not
  \in Q}V_x\times \prod_{x\in Q} U_{i,x},
\end{align*}
for $i=0,1$. For $W$ equal to one of $U$, $V$, $U_i(Q)$ or $V_i(Q)$, we have a smooth projective modular curve
$X_W$ over $\Spec(\OL)$ which is a moduli space of generalized
elliptic curves with $W$-level structure\footnote{Again, in order to
  obtain a representable moduli problem, we may need
  to introduce auxiliary level structure at a prime $q$ as in
  Section~\ref{sec:gal-rep}. This would be necessary if every prime in
$S(\rhobar)$ were vexing, for example.}.
Let $Y_W \subset X_W$ be the open curve
parametrizing genuine elliptic curves and let $j: Y_W \into X_W$
denote the inclusion. As in Section \ref{sec:cohom-modul-curv}, we let
$\pi: \mathcal{E} \rightarrow X_W$ denote the universal generalized
elliptic curve, we let
$\omega:= \pi_* \omega_{\mathcal{E}/X_W}$ and we let $\cusps$
denote the reduced divisor supported
on the cusps. If $M$ is an $\OL$-module and $\mathcal L$ is a sheaf of
$\OL$-modules on $X_M$, then we denote by $\mathcal L_M$ the sheaf
$\mathcal L\otimes_\OL M$ on $X_W$. If $R$ is an $\OL$-algebra, we
will sometimes denote $X_W \times_{\Spec(\OL)}\Spec(R)$ by $X_{W,R}$.

There is a natural right action of $U/V$ on $X_V$ coming from the
description of $X_V$ as a moduli space of generalized elliptic curves
with level structure (\cite[\S IV]{deligne-rapoport}). It follows from
\cite[IV 3.10]{deligne-rapoport} that we have
$X_V/(U/V) \iso X_U$. Away from the cusps, the map $Y_V\to Y_U$ is \'{e}tale and Galois
with Galois group $U/ V$ and the map $X_V \to X_U$ is tamely ramified. Similar remarks apply to the maps
$X_{V_i(Q)}\to X_{U_i(Q)}$ for $i=0,1$.

The natural map $X_{U_1(Q)}\to X_{U_0(Q)}$ is \'{e}tale and Galois with
Galois group 
$$\Delta_Q:= \prod_{x\in Q} U_{0,x}/U_{1,x}\cong\prod_{x\in Q}(\Z/x)^{\times}.$$

\subsubsection{ Cutting out spaces of modular forms}
Let $G=U/V=\prod_{x\in T}\GL_2(\Z/x^{c_x/2})$ and let $\sigma$
denote a representation of $G$ on a finite free $\OL$-module
$W_\sigma$. We will now proceed to define a vector bundle $\CL_\sigma$
on $X$ such that
\[ H^0(X_U,\omega^{\otimes n}\otimes_{\OL_X}\CL_\sigma) \liso
(H^0(X_V,\omega^{\otimes n})\otimes_{\OL}W_\sigma)^{G} = \Hom_{\OL[G]}(W_{\sigma}^*,H^0(X_V,\omega^{\otimes n})),\]
where $W_\sigma^*$ is the $\OL$-dual of $W_\sigma$. The sheaf
$\CL_\sigma$ will thus allow us to extract the $W_\sigma^*$-part of the
space of modular forms at level $V$. We shall also define a cuspidal
version $\CL_{\sigma}^{\sub}\subset \CL_{\sigma}$ which extracts the
$W_\sigma^*$-part of the space of cusp forms at level $V$:
\[ H^0(X_U,\omega^{\otimes n}\otimes_{\OL_X}\CL_\sigma^{\sub}) \liso
(H^0(X_V,\omega^{\otimes n}(-\infty))\otimes_{\OL}W_\sigma)^{G} = \Hom_{\OL[G]}(W_{\sigma}^*,H^0(X_V,\omega^{\otimes n}(-\infty))).\]

The definitions are as follows. Let $f$ denote
the natural map $X_V \to X_U$ and define 
\begin{align*}
 \CL_{\sigma} & := (f_* (\CO_{X_V} \otimes_{\CO}W_{\sigma}))^G \\ 
 \CL_{\sigma}^{\sub} & := (f_* (\CO_{X_V}(-\infty) \otimes_{\CO}W_{\sigma}))^G, 
\end{align*}
where $G$ acts diagonally in both cases. Note that
\[ \CL_{\sigma} = (f_* f^*(\CO_{X_U}\otimes_{\CO}W_\sigma))^G =
((f_*\CO_{X_V})\otimes_{\CO}W_\sigma)^G \] by the projection
formula. Similarly, by arguing locally, we see that
\[ \CL_{\sigma}^{\sub} = (f_*(\CO_{X_V}(-\infty))\otimes_{\CO}W_\sigma)^G. \]

For $i=0,1$ we denote the pull back of $\CL_\sigma$ to $X_{U_i(Q)}$
also by $\CL_\sigma$. This notation is justified since, by flatness of
the map $X_{U_i(Q)}\to X_U$, the pullback is isomorphic to
$(f_*(\CO_{X_{V_i(Q)}}\otimes_{\CO}W_\sigma))^G$, where we continue to
denote by $f$ the natural map $X_{V_i(Q)} \to X_{U_i(Q)}$. When we use
the same notation for sheaves on different spaces, the underlying
spaces will always be clear from the context.

On $X_{U_i(Q)}$ we reserve the notation $\CL_{\sigma}^{\sub}$ for $
(f_*(\CO_{X_{V_i(Q)}}(-\infty))\otimes_{\CO}W_\sigma)^G$. The pull back of
$\CL_{\sigma}^{\sub}$ on $X_{U}$ to $X_{U_i(Q)}$ is a sub-sheaf
of $\CL_{\sigma}^{\sub}$ on $X_{U_i(Q)}$ (the quotient being supported
at ramified cusps).

We now discuss Hecke actions on cohomology. 
 Let $X$ denote $X_{U_i(Q)}$ and let $X(V)$ denote $X_{V_i(Q)}$ for
 some choice of $i=0$ or $1$. Let $f$ denote the map $X(V)\to X$. (Note that if
 $Q$ is empty, then we recover $f: X_V \to X_U$.)
 Let $x
\not \in S(\rhobar)\cup \{ p\}$. As in Section~\ref{sec:hecke-ops}, we have a modular curve $X_0(x)$,
obtained from $X$ by the addition of an appropriate level structure
at $x$,
together with degeneracy maps $\pi_1,\pi_2 : X_0(x) \to X$. (The level
structure at $x$ depends on whether or not $x \in Q$.) We define
$X_0(V;x)$ similarly, starting from $X(V)$.
The natural map $X_0(V;x)\to X_0(x)$ is again denoted $f$. Then note that we have a natural isomorphism
\[ \phi(\sigma)_{12}: \pi_2^* \CL_\sigma \liso \pi_1^* \CL_\sigma \]
of sheaves on $X_0(x)$. Indeed for $i=0,1$, by flatness of the map $\pi_i : X_0(x) \to
X$, the pullback $\pi_i^*\CL_\sigma$ is canonically isomorphic to
\[ (f_*(\CO_{X_0(V;x)}\otimes_\CO W_\sigma))^G, \]
independently of $i$.
(The only point to note is that the morphism $X_0(x) \times_{\pi_i,
  X}X(V) \to X_0(x)$ is canonically isomorphic to $X_0(V;x)\to X_0$.) 
Similarly, if $a \in \Z$ is coprime to the elements of $S(\rhobar)\cup
Q$, we have a morphism $\langle a \rangle : X \to X$ which corresponds
to multiplication by $a$ on the level structure. Then $\langle a
\rangle^* \CL_{\sigma}$ is canonically isomorphic to $\CL_\sigma$.

Let $M$ denote an $\CO$-module and let $n$ be an integer. Then using
the isomorphisms $\pi_2^* \CL_\sigma \iso \pi_1^*\CL_\sigma$ of the
previous paragraph, and following the definitions of
Section~\ref{sec:hecke-ops}, we can define Hecke operators on the
cohomology of $\omega^n\otimes \CL_{\sigma}\otimes_{\CO}M$. For example,
$xT_x$ is defined as the composite (taking $M = \CO$ for simplicity):
\[ H^i(X_U, \omega^n\otimes\CL_\sigma)
  \stackrel{\pi_2^*}{\lra}
  H^i(X_0(U;x), \pi_2^* \omega^n\otimes\CL_\sigma)
  \stackrel{\phi_{12}\otimes \phi_{12}(\sigma)}{\lra}
  H^i(X_0(U;x), \pi_1^* \omega^n\otimes\CL_\sigma)
  \stackrel{\Tr(\pi_1)}{\lra}
  H^i(X_U,\omega^n\otimes\CL_\sigma).\]

Let $\CL_{\sigma}^{\sub}$ denote the sheaf
\[ (f_*(\CO_{X_0(V;x)}\otimes_\CO W_\sigma))^G, \] on $X_0(U;x)$. (In
what follows we will be using $\CL_{\sigma}^{\sub}$ and $\CL_{\sigma}$
to denote sheaves on both $X_U$ and $X_0(U;x)$, but the underlying
space will be clear in each instance). We then have canonical
inclusions
\[ \pi_1^*(\CL_{\sigma}^{\sub}), \pi_2^*(\CL_{\sigma}^{\sub}) \subset
  \CL_{\sigma}^{\sub} \]
of sheaves on $X_0(U;x)$. Note also that the composition of morphisms
of sheaves on $X_U$
\[ \pi_{1,*}(\CL_{\sigma}^{\sub}) \hookrightarrow \pi_{1,*}(\CL_{\sigma}) =
  \pi_{1,*}(\pi_1^* \CL_{\sigma}) \stackrel{\Tr(\pi_1)}{\lra}
  \CL_{\sigma} \] factors through the sheaf
$\CL_{\sigma}^{\sub}$. This then allows us to define Hecke operators
on the cohomology of
$\omega^n\otimes \CL_{\sigma}^{\sub}\otimes_{\CO}M$.

In summary, we have operators:
\begin{itemize}
\item $T_x$ and $\langle a \rangle$ on
\[ H^j(X_U,
  \omega^n\otimes_{\CO_{X_U}}\CL_\sigma^{\sub}\otimes_{\CO}M) \text{\ \ and\
    \ } H^j(X_U,
  \omega^n\otimes_{\CO_{X_U}}\CL_\sigma\otimes_{\CO}M)\] for all $x \not\in
  S(\rhobar)\cup\{p\}$ and $a$ coprime to the elements of
  $S(\rhobar)$, and
\item  $T_x,U_y,\langle a \rangle$ on \[ H^j(X_{U_i(Q)},
  \omega^n\otimes_{\CO_{X_U}}\CL_\sigma^{\sub}\otimes_{\CO}M) \text{\ \ and\
    \ }H^j(X_{U_i(Q)},
  \omega^n\otimes_{\CO_{X_U}}\CL_\sigma\otimes_{\CO}M)\] for all $x \not\in
  S(\rhobar)\cup Q\cup\{p\}$, $y \in Q$ and $a$ coprime to the
  elements of $S(\rhobar)\cup Q$.
\end{itemize}

Part~\eqref{H0} of the following lemma shows that $\CL_\sigma$ and $\CL_{\sigma}^{\sub}$ do
indeed allow us to extract the $W_{\sigma}^*$-part of the space of
modular forms at level $V$.

\begin{lemma}
  \label{lem:properties-of-L}
Let $X$ denote $X_U$ (resp.\ $X_{U_i(Q)}$ for $i=0$ or 1), let $X(V)$ denote $X_V$
(resp.\ $X_{V_i(Q)}$) and let $f$ denote the map $X(V)\to X$.  
Then 
  \begin{enumerate}
  \item\label{bundle} The sheaves $\CL_{\sigma}^{\sub}$ and $\CL_\sigma$ are locally free of finite rank on $X$.
\item\label{H0} If $A$ is an $\OL$-algebra and $\CV$ is a coherent locally free sheaf of $\OL_{X_A}$-modules, then 
  \begin{align*}
 H^0(X_A ,\CV \otimes_{\CO_{X_A}} (\CL_\sigma)_A) &\liso
(H^0(X(V)_A,f^*\CV)\otimes_{\OL}W_\sigma)^{G} \\    
H^0(X_A ,\CV \otimes_{\CO_{X_A}} (\CL_\sigma^{\sub})_A) &\liso
(H^0(X(V)_A,(f^*\CV)(-\infty))\otimes_{\OL}W_\sigma)^{G}.
  \end{align*}

\end{enumerate}
\end{lemma}

\begin{proof}
We give the proof for $\CL_{\sigma}$;  the case of $\CL_{\sigma}^{\sub}$ is
treated in exactly the same way.
Let $Y$
(resp.\ $Y(V)$) denote the non-cuspidal open subscheme of $X$ (resp.\ $X(V)$).
We have $\CL_\sigma|_{Y} = f_*(\OL_{Y(V)}\otimes_{\OL}W_\sigma)^{G}$
 since the inclusion $Y\to X$ is flat.
Since the map $Y(V)\to Y$ is \'{e}tale, it follows from \cite[\S III.12 Theorem 1
(B)]{Mumford} (and its proof) that $\CL_\sigma|_{Y}\iso \OL_{Y}\otimes_{\OL}W_\sigma$.
To show that $\CL_\sigma$ is locally free of finite rank on $X$, it remains
to check that its stalks at points of $X-Y$ are free. Let $x$ be a
point of $X-Y$. We can and do assume that for each point $x'$ of $X(V)$
lying above $x$, the natural map on residue fields is an
isomorphism. We have
\[ \CL_{\sigma,x} = \left( \bigoplus_{x' \mapsto x} \OL_{X_V,x'}\otimes W_\sigma\right)^{G}.\]
Choose some point $x'\mapsto x$ and let $I(x'/x)\subset G$ be the
inertia group of $x'$. Then projection onto the $x'$-component defines
an isomorphism
\[ \left( \bigoplus_{x'' \mapsto x} \OL_{X(V),x''}\otimes
  W_\sigma\right)^{G} \liso \left( \OL_{X(V),x'}\otimes
  W_\sigma\right)^{I(x'/x)}.\]
Now $I(x'/x)$ is abelian of order prime to $p$ (see \cite[\S VI.5]{deligne-rapoport}). Extending $\OL$, we
may assume that each character $\chi$ of $I(x'/x)$ is defined over
$\OL$. Let $W_{\sigma,\chi}$ and $\OL_{X_V,x',\chi}$ denote the
$\chi$-parts of $W_\sigma$ and $\OL_{X_V,x'}$. Then $W_\sigma\cong
\oplus_\chi W_{\sigma,\chi}$ and similarly
$\OL_{X(V),x'}\cong \oplus_\chi \OL_{X(V),x',\chi}$. Each $W_{\sigma,\chi}$
(resp.\ $\OL_{X(V),x',\chi}$) is free over $\OL$ (resp.\  $\OL_{X,x}$), being a summand of a
free module. (Note that $f$ is finite flat.) We now have
\[ \CL_{\sigma,x} \liso \left( \OL_{X(V),x'}\otimes
  W_\sigma\right)^{I(x'/x)} \liso \bigoplus_{\chi}
W_{\sigma,\chi}\otimes_{\OL} \OL_{X(V),x',\chi^{-1}},\]
which is free over $\OL_{X,x}$. This establishes part
\eqref{bundle}. 

We now turn to part \eqref{H0}. We first of all note that the proof of
the previous part shows that
\[ (\CL_{\sigma})_A \liso (f_*
  (\CO_{X(V)_A}\otimes_{\CO}W_\sigma))^{G},\] as sheaves on $X_A$.
Now, let $\CV$ be as in the statement of the lemma. Then,
\begin{align*}
 \CV\otimes_{\CO_{X_A}}(\CL_{\sigma})_A & \cong  \CV\otimes_{\CO_{X_A}}   (f_*
  (\CO_{X(V)_A}\otimes_{\CO}W_\sigma))^{G} \\
& \cong (\CV\otimes_{\CO_{X_A}}f_*
  (\CO_{X(V)_A}\otimes_{\CO}W_\sigma))^{G} \\
& \cong (f_*f^*(\CV\otimes_{\CO}W_\sigma))^G.
\end{align*}
Here $G$ acts trivially on $\CV$ and the third isomorphism follows
from the projection formula. Taking global sections we obtain,
\begin{align*}
 H^0(X_A,\CV\otimes(\CL_{\sigma})_A) &=
 (H^0(X(V)_A,f^*(\CV\otimes_{\CO}W_{\sigma}))^G \\
& =  (H^0(X(V)_A,f^*(\CV))\otimes_{\CO}W_\sigma)^G, 
\end{align*}
as required. 
\end{proof}

Let $X$ and $X(V)$ be as in the statement of the previous lemma. Let $\sigma^* =
\SHom_{\CO}(W_\sigma,\CO)$ be the dual of the representation
$\sigma$.
We now consider the dual vector bundle $\CL_{\sigma}^* =
\SHom_{\CO_{X}}(\CL_\sigma,\CO_X)$ and its relation to
$\CL_{\sigma^*}$. In addition, we let $A$ denote an $\CO$-algebra and
we consider the situation base changed to $\Spec A$. First of all, note that we
have an isomorphism
\[ X(V)_A/G
 \iso X_A \]
and in particular, $\CO_{X_A}\iso f_*(\CO_{X(V)_A})^G$. (When $A=\OL$, this follows from \cite[\S IV Prop.\
 3.10]{deligne-rapoport}. The same argument works when $A=k$. These two
 cases, and the flatness of $X(V)$ over $\OL$, imply the result when
 $A=\OL/\varpi^n$. The general result follows from this by
 \cite[Prop.\ A7.1.4]{katz-mazur}. Alternatively, as pointed out to us
 by the referee, one can see directly that $X(V)_A/G=X_A$ by applying
 the argument of the proof of
 Lemma~\ref{lem:properties-of-L}~\eqref{bundle}.) We have shown
 in the proof of Lemma~\ref{lem:properties-of-L}~\eqref{H0} that
\[ (\CL_{\sigma})_A \cong  (f_*(\CO_{X(V)_A}\otimes_{\CO}W_\sigma))^{G},\] 
as sheaves on $X_A$. By the projection formula, we therefore also have 
\[ (\CL_{\sigma})_A \cong
(f_*(\CO_{X(V)_A})\otimes_{\CO}W_\sigma)^{G}.\] 
Applying this with $\sigma^*$ in place of
$\sigma$, we see that
\begin{align*}
  (\CL_{\sigma^*})_A & \cong
  (f_*(\CO_{X(V)_A})\otimes_{\CO}\SHom_{\CO}(W_\sigma,\CO))^G \\
& \cong \left(\SHom_{f_*(\CO_{X(V)_A})}(f_*(\CO_{X(V)_A})\otimes_{\CO}W_\sigma,f_*(\CO_{X(V)_A}))\right)^G 
\end{align*}
Since $\CO_{X_A}\iso f_*(\CO_{X(V)_A})^G$, we have a map
\[ (\CL_{\sigma^*})_A \lra
\SHom_{\CO_{X_A}}((\CL_{\sigma})_A,\CO_{X_A}) = (\CL_{\sigma})_A^*\]
given by restriction to $G$-invariants. This map is induced from the
corresponding map $\CL_{\sigma^*}\to \CL_{\sigma}^*$ over $\Spec
\CO$. In addition, when restricted to $Y_A$, this map is an
isomorphism since the equivalence of \cite[\S III.12 Theorem 1
(B)]{Mumford} for locally free sheaves is compatible with taking
duals. In particular, the map $\CL_{\sigma^*} \to \CL_\sigma^*$ is
injective and remains injective after base change to $\Spec A$, for
all $A$. 

\begin{lemma}
  \label{lem:dual-of-L}
  The injection $\CL_{\sigma^*}\into \CL_{\sigma}^{*}$ restricts to an isomorphism
\[ \CL_{\sigma^*}^{\sub} \liso (\CL_{\sigma}^*)(-\infty).\]
Similarly, we have
\[ (\CL_{\sigma}^{\sub})^*(-\infty) \cong \CL_{\sigma^*}. \]
\end{lemma}

\begin{proof}
The second statement follows immediately from the first by reversing the roles of $\sigma$ and
$\sigma^*$. Thus, we consider the first statement.
  Away from the cusps, all three inclusions 
\[ \CL_{\sigma^*}^{\sub} \into \CL_{\sigma^*}\into \CL_\sigma^*
\text{\ \ and\ \ } (\CL_\sigma^*)(-\infty) \into \CL_\sigma^* \] 
 are isomorphisms. It therefore suffices to show
  that at each closed point $x$ of $\CC$, the natural map gives rise to an isomorphism
\[ (\CL_{\sigma^*}^{\sub})^{\wedge}_x \liso
(\CL_{\sigma}^{^*})(-\infty)^{\wedge}_{x} \]
along the formal completions at $x$. Extending $\CO$ if necessary, we
may assume that all cusps of $X$ and $X(V)$ are defined over $\CO$ and
for each point $x'$ of $X(V)$ lying over $x$, there is a uniformizer
$q$ at $x'$ so the map
\[ \CO_{X,x}^{\wedge} \to \CO_{X(V),x'}^{\wedge}  \]
is isomorphic to 
\[ \CO[[q^e]] \to \CO[[q]] .\]
Here, $e=\#I(x'/x)$ and we may assume that $\CO$ contains the
primitive $e$-th roots of unity and the inertia group $I(x'/x)$ is
isomorphic to $\mu_e \subset \CO^\times$ via $\sigma \mapsto
\frac{\sigma(q)}{q}$. Choose a primitive $e$-th root of unity $\zeta$
and for $i=0,\dots,e-1$, let $\chi_i : I(x'/x) \cong \mu_e \to \CO^{\times}$ be
the character which sends $\zeta$ to $\zeta^i$. Then, the
$\chi_i$-part of $\CO_{X(V),x'}^{\wedge}$ is given by
\[ \CO_{X(V),x',\chi_i}^{\wedge} = q^i\CO[[q^e]] \subset \CO[[q]].\]
Thus, as in the proof of Lemma~\ref{lem:properties-of-L}, we have
\[ (\CL_{\sigma})^{\wedge}_{x} \cong \bigoplus_{i=0}^{e-1}
q^i\CO[[q^e]]\otimes_{\CO}W_{\sigma,\chi_i^{-1}}.\]
Taking duals over $\CO_{X,x}^{\wedge}\cong \CO[[q^e]]$, we obtain
\[ (\CL_{\sigma}^*)_x^{\wedge} \cong \bigoplus_{i=0}^{e-1}
q^{-i}\CO[[q^e]]\otimes_{\CO}(\Hom_{\CO}(W_{\sigma,\chi_i^{-1}},\CO)). \]
Note that $\Hom_{\CO}(W_{\sigma,\chi_i^{-1}},\CO)=W_{\sigma^*,\chi_i}$ 
and since $q^e$ is a uniformizer at $x$, we obtain,
under the natural map, identifications
\begin{align*}
 (\CL_{\sigma}^*)(-\infty)^{\wedge}_x & \cong \bigoplus_{i=0}^{e-1}
 q^{e-i}\CO[[q^e]]\otimes_{\CO}W_{\sigma^*,\chi_i}\\
 & \cong \bigoplus_{i=1}^{e}q^i \CO[[q^e]]\otimes_\CO W_{\sigma^*,\chi_i^{-1}}. 
\end{align*}
This is precisely $(\CL_{\sigma^*}^{\sub})^{\wedge}_x$ by the proof of Lemma~\ref{lem:properties-of-L}~\eqref{bundle}.
\end{proof}

We deduce the following.

\begin{corr}
  \label{cor:vanishing-H1}
If $n>1$, then
\[ H^1(X , \omega^{\otimes n}\otimes\CL_\sigma) = \{0\} \]
 and hence 
 \[ H^0(X , \omega^{\otimes n}\otimes \CL_\sigma\otimes_{\OL} \OL/\varpi^m) = (H^0(X(V), \omega^{\otimes n})\otimes_{\OL}W_\sigma)^{G}
 \otimes_{\OL} \OL/\varpi^m.\]
Moreover, the analogous result holds for $n>2$ if we replace $\CL_{\sigma}$ by
$\CL_{\sigma}^{\sub}$.
\end{corr}

\begin{proof}

The second statement follows immediately from
the first and from Lemma~\ref{lem:properties-of-L}\eqref{H0} by considering the long exact
sequence in cohomology associated to the short exact sequence
\[ 0 \lra \omega^{\otimes n}\otimes_{\OL_X}\CL_\sigma \stackrel{\varpi^m}{\lra}
\omega^{\otimes n}\otimes_{\OL_X}\CL_\sigma  \lra (\omega^{\otimes n}\otimes_{\OL_X}\CL_\sigma
)/\varpi^m \lra 0.\]
To prove the first statement, it suffices to show that
$H^1(X,\omega^{\otimes
  n}\otimes_{\OL_X}\CL_\sigma\otimes_{\OL}k)=\{0\}$. By Serre duality,
this is equivalent to the vanishing of
$H^0(X_k,\omega^{\otimes (2-n)}\otimes_{\OL_{X}}
\CL_\sigma^*(-\infty))$. By Lemma~\ref{lem:dual-of-L}, we are
therefore reduced to showing $H^0(X_k,\omega^{\otimes (2-n)}\otimes_{\OL_{X}}
\CL_{\sigma^*}^{\sub})=0$. However, by
Lemma~\ref{lem:properties-of-L}~\eqref{H0} again, we have
\begin{align*}
 H^0(X_k,\omega^{\otimes (2-n)}\otimes_{\OL_{X}}
 \CL_{\sigma^*}^{\sub}) \cong (H^0(X(V)_k,\omega^{\otimes
   (2-n)}(-\infty))\otimes_{\OL} W_{\sigma^*})^G 
\end{align*}
which vanishes since $n>1$. The case of $\CL_{\sigma}^{\sub}$ is
proved in the same way. 
\end{proof}

\subsubsection{The proof of Theorem~\ref{weightone} in the presence
of vexing primes}

To complete the proof of Theorem~\ref{weightone}, it suffices
to note the various modifications which must be made to the argument.
For vexing primes $x$, let $c_x$ denote the conductor of $\rhobar$ (which is necessarily
even).
We define a $\OL$-representation $W_{\sigma_x}$ of $\GL_2(\Z/x^{c_x/2}\Z)$ to be the
representation $\sigma_x$ as in \S~5 of~\cite{CDT}. The collection
$\sigma=(\sigma_x)_{x\in T(\rhobar)}$ gives rise to a sheaf
$\CL_{\sigma}$ on $X_U$ as above.
Let $\T_{\emptyset}$ denote the ring 
of Hecke operators acting on $H_0(X_U,\omega \otimes \CL_{\sigma})$
generated by Hecke operators away from $S(\rhobar)\cup\{p\}$. 
The analogue of 
Theorem~\ref{thm:main-thm-w1} is as follows:

\begin{theorem}
  \label{thm:main-thm-w1-v2}
The map $\varphi : R^{\min} \onto\T_{\emptyset,\mE}$ is an isomorphism and
$\T_{\emptyset,\mE}$ acts freely on $H_0(X_U,\omega \otimes \CL_{\sigma})_{\mE}$.
\end{theorem}

\begin{proof} The proof is the same as the proof of Theorem~\ref{thm:main-thm-w1}; we indicate below
the modifications that need to be made.

\begin{enumerate}
\item (Lemma~\ref{lemma:matt-w1}): Exactly the same argument shows
  that there is an isomorphism of Hecke modules:
\[ H^0(X_U,(\omega\otimes\CL_\sigma)_{K/\OL})_{\mE} \liso
H^0(X_{U_0(Q)},(\omega\otimes\CL_\sigma)_{K/\OL})_\m.\]
The only point to note is that there is an operator
\[ W_x : H^0(X_{U_0(x)},(\omega\otimes\CL_{\sigma})_{K/\OL}) \to
H^0(X_{U_0(x)},(\omega\otimes\CL_{\sigma})_{K/\OL}) \]
such that $W_x^2 = x\langle x \rangle $ and
\[ \frac{1}{x}\pi_1^*\circ\Tr(\pi_1)\circ W_x = U_x + \frac{1}{x}W_x. \]
To see this, one can note that the corresponding operator $W_x$ on
$H^0(X_{V_0(x)},\omega_{K/\CO})$ (defined in Section~\ref{sec:hecke-ops}) commutes with the action of $G=U/V$, and
hence induces the desired operator $W_x$ on
\[  H^0(X_{U_0(x)},(\omega\otimes\CL_{\sigma})_{K/\OL}) =
(H^0(X_{V_0(x)},\omega_{_{K/\OL}})\otimes_{\CO}W_\sigma)^G .\]

\item (Proposition~\ref{lem:cohom-basic-props}~\eqref{vanish-cusps}):
  The corresponding statement holds: namely, the natural map
\[ H^i(X_{U_{\Delta}(Q)},(\omega\otimes \CL_{\sigma}^{\sub})_{K/\CO})_{\m}
\to H^i(X_{U_{\Delta}(Q)},(\omega\otimes \CL_{\sigma})_{K/\CO})_{\m} \]
is an isomorphism for $i=0,1$ whenever $\m$ is
non-Eistenstein. Indeed, if $\CC_{V_{\Delta}(Q)}\subset
X_{V_{\Delta}(Q)}$ is the cuspidal subscheme, then we have an exact
sequence of sheaves on $X_{U_{\Delta}(Q)}$:
\[ 0 \lra \omega\otimes\CL_{\sigma}^{\sub} \lra \omega\otimes
\CL_{\sigma} \lra
(f_*((\omega\otimes_{\CO}W_{\sigma})|_{\CC_{V_{\Delta}(Q)}}))^G, \]
and it suffices to show that the cohomology of the last term (which is
concentrated in degree 0) is Eisenstein. However, the argument of
Remark~\ref{rem:boundary-eis} (noting that  the group $\prod_{x\in
  T(\rhobar)}\GL_2(\Z_x)\times \prod_{x \in P(\rhobar)} \Z_x^{\times}$
acts transitively on the set of cusps in $X_{V_{\Delta}(Q)}$) shows that
\[ H^0(X_{U_{\Delta}(Q)},(f_*((\omega\otimes
W_{\sigma})|_{\CC_{V_{\Delta}(Q)}}))^G) =
(H^0(\CC_{V_{\Delta}(Q)},\omega)\otimes_{\CO} W_{\sigma})^G \]
is Eisenstein, which gives the desired result.

\item (Proposition~\ref{prop:balanced-homology-w1}): We need to show that the
$\OL[\Delta]$-module $M = H_0(X_{\Delta}(Q),\omega \otimes 
\ \CL_{\sigma})_{\m}$ is balanced. First of all note that 
$\Omega^1_{X_{\Delta}(Q)/\CO}\otimes\CL_{\sigma}^* =
\omega^2(-\infty)\otimes \CL_{\sigma}^* = \omega^2\otimes
\CL_{\sigma^*}^{\sub}$ by Lemma~\ref{lem:dual-of-L}, and hence 
\[  H_i(X_{U_{\Delta}(Q)},\omega^n \otimes \CL_\sigma)=  H^i(X_{U_{\Delta}(Q)},(\omega^{2-n} \otimes
\CL_{\sigma^*}^{\sub})_{K/\CO})^{\vee}. \]
We use this to endow the left hand side with a Hecke action.

We now modify definitions of $\Phi$ and $\Psi$ from 
Section~\ref{specify}: let $\Phi$ denote the composition of isomorphisms:
\begin{eqnarray*}
 H^1(X_{U_{\Delta}(Q)},(\omega^{2-n}\otimes\CL^{\sub}_{\sigma^*})_{K/\CO}) 
  &\stackrel{D}{\lra}& H^0(X_{U_{\Delta}(Q)},\Omega\otimes \omega^{n-2}\otimes\CL_{\sigma}(\infty))^{\vee} \\
  &\stackrel{KS^{\vee}}{\lra}& H^0(X_{U_{\Delta}(Q)},\omega^{n}\otimes\CL_{\sigma})^{\vee},
\end{eqnarray*}
where $D$ is Verdier duality, and $KS$ is the Kodaira--Spencer
isomorphism, and we have used Lemma~\ref{lem:dual-of-L}. Then by the proof of \cite[Prop.\ 7.3]{Edixhoven}, we
have:
\[ \Phi \circ T_x = x^{1-n} T_x^{t,\vee} \circ \Phi, \] for all $x$
prime to $NQ$, and the same relation holds for the operators $U_y$
with $y|Q$. We also have
$\Phi \circ \langle a \rangle = \langle a^{-1} \rangle \circ \Phi$ for
$x | NQ$ because $D$ switches $\langle a\rangle^* $ with
$\langle a \rangle_*^{\vee} = \langle a^{-1} \rangle^{*,\vee}$. The
transposed operators $T_x^t$ and $U_x^t$ are defined in
\cite{Edixhoven}.  We have:
\[ (T_x^t f)(E,\alpha_{NQ}) = \sum_{\phi : (E',\alpha_{NQ}') \to (E,
    \alpha_{NQ})} \phi^{t,*} f(E',\alpha'_{NQ}) \]
where the sum is over all $x$-isogenies $\phi$ such that $\phi \circ
\alpha_{NQ}' = \alpha_{NQ}$. In the definition of $T_x^t$, we can replace the
sum over $\phi$ by the sum over the corresponding dual isogenies $E
\to E'$. However, note that if $\phi : (E',\alpha_{NQ}') \to
(E,\alpha_{NQ})$ is compatible with level structures at $NQ$, then so
is $\phi^t : (E, \alpha_{NQ}) \to (E', x \circ \alpha_{NQ}')$. In this
way we see that
\[ T_x^t = \langle x^{-1} \rangle T_x\] on
$H^0(X_{U_{\Delta}(Q)}, \omega^n)$.
The Pontryagin dual
$\Psi := \Phi^{\vee}$ is thus an isomorphism
\[ \Psi:  H^0(X_{U_{\Delta}(Q)},\omega^{n} \otimes \CL_{\sigma})  \lra
  H_1(X_{U_{\Delta}(Q)},\omega^n\otimes\CL_{\sigma}) \]
such that
\begin{eqnarray*}
\Psi \circ (x^{1-n} \langle x^{-1} \rangle T_x ) &=& T_x \circ
  \Psi \\
\Psi \circ (y^{1-n} U_y^t) &=& U_y \circ \Psi \\
  \Psi \circ \langle a^{-1} \rangle &=& \langle a \rangle \circ \Psi.
\end{eqnarray*}

Now, with $\CL =  \CL_{\sigma}$, the proof of Proposition~\ref{prop:balanced-homology-w1} proceeds in exactly the same manner, up to the point where
it suffices to show that
\[ \dim_K H_0(X_{U_0(Q),K},\omega \otimes \CL_{\sigma})_\m =  \dim_K
  H_1(X_{U_0(Q),K},\omega \otimes \CL_{\sigma})_{\m}.\]
As before, by definition, the left hand side of this is equal to
\[ \dim_K H^0(X_{U_0(Q),K},\omega \otimes
  \CL_{\sigma^*}^{\sub})_\m, \]
which in turn, by point (2) above, is equal to:
\[ \dim_K H^0(X_{U_0(Q),K},\omega \otimes
  \CL_{\sigma^*})_\m, \]
On the other hand, using the isomorphism $\Psi$, we see that the right
hand side is equal to:
\[ \dim_K H^0(X_{U_{0}(Q)},\omega^{n} \otimes \CL_{\sigma})_{\m^*} \]
where $\m^*$ is a maximal ideal of the polynomial ring $R$ generated over
$\CO$ by the operators $T_x,U_y^t,\langle a \rangle$. Specifically,
let $\alpha : \T \to \T/\m = k$ be the reduction map. Then $\m^*$ is
the kernel of the map $\beta :
R \to k$ defined by: $\beta(T_x) = \alpha(\langle x \rangle T_x)$,
$\beta(U^t_y) = \alpha(U_y)$, and $\beta(\langle a \rangle) =
\alpha(\langle a^{-1}\rangle)$.

Thus, we need to see that $H^0(X_{U_0(Q),K},\omega \otimes
\CL_{\sigma})_{\m^*}$ and $H^0(X_{U_0(Q),K},\omega \otimes
\CL_{\sigma^*})_\m$ have the same dimension. One way to see this is as
follows: after choosing an embedding $K\into \C$, we can identify both
sides in terms of automorphic representations of $\GL_2/\Q$ which are
limits of discrete series at $\infty$, unramified outside
$S(\rhobar)\cup Q$ and satisfy appropriate local
conditions at the primes in $S(\rhobar)\cup Q$. The operation which
sends each such automorphic representation $\pi$ to its contragredient
then interchanges 
\[ H^0(X_{U_0(Q),K},\omega \otimes
\CL_{\sigma})_{\m^*}\otimes_{K}\C \text{\ and\ }  H^0(X_{U_0(Q),K},\omega \otimes
\CL_{\sigma^*})_\m\otimes_K \C, \]
from which the result follows.

\item (Theorem~\ref{theorem:galoisweight1}): The analogue of this
  theorem is true. Namely, let $\T$ denote the subalgebra of
  endomorphisms of
\[ H^0(X_{U_{1}(Q)},(\omega\otimes\CL_{\sigma})_{K/\CO}) \]
generated by the operators $T_x$, $U_y$ and $\langle a \rangle$. For
each $x \in Q$, we assume that the Hecke polynomial $X^2 - T_x X
+\langle x \rangle$ has distinct roots in $\TE/\mE$ and we let
$\alpha_x$ be one of these roots. Let $\m$ be the ideal of $\T$
generated by $\mE$ and $U_x - \alpha_x$ for $x\in Q$. Then there is a
Galois representation 
\[ \rho_Q : G_{\Q} \to \GL_2(\T_\m) \]
deforming $\rhobar$, unramified away from $S(\rhobar)\cup\{p\}$ and
such that $\Frob_x$ has trace $T_x$ for all $x\not \in
S(\rhobar)\cup\{p\}$. Moreover $\rho'_Q := \rho_Q\otimes \eta$, where
$\eta$ is defined as before, is a deformation of $\rhobar$ minimal
outside $Q$.

This is proved as follows: as before it suffices to fix an $m \geq 1$ and work with the
quotient $\T_\m/J_m$ of $\T_\m$ acting faithfully on
$H^0(X_{U_{1}(Q)},(\omega\otimes\CL_{\sigma})_{\CO/\varpi^m})_\m$. Then
we have
\[ H^0(X_{U_{1}(Q)},(\omega\otimes\CL_{\sigma})_{\CO/\varpi^m})_\m
\subset H^0(X_{V_{1}(Q)},\omega_{\CO/\varpi^m})_\m \otimes_{\CO}
W_\sigma. \] From this inclusion and the arguments of
Sections~\ref{sec:gal-rep}--\ref{section:case3}, we immediately
deduce the existence of $\rho_Q$ over $\T_\m/J_m$ such that $\rho_Q'$
satisfies all the conditions of Definition~\ref{defn:minimal}, except
possibly for condition~\eqref{at-vexing}. More precisely, we
construct a deformation over the Hecke algebra of
\[  H^0(X_{V_{1}(Q)},\omega_{\CO/\varpi^m})_\m \]
which satisfies these properties exactly as we did in Sections~\ref{sec:gal-rep}--\ref{section:case3}. (The new level structures at the primes in
$T(\rhobar)$ do not affect the arguments; the only essential difference is that
the modular curves are no longer geometrically connected. Thus, in any
argument involving $q$-expansions, one needs
to consider $q$-expansions at a cusp on each connected component
instead of at the single cusp $\infty$. Note, however, that the use of~$q$-expansions
was only used for the following two facts:
the identity~$\phi \circ T_p - U_p \circ \phi  = \langle p \rangle V_p$ and
the claim that~$\theta V_p = 0$, which was used to show that~$(\phi,\phi \circ T_p - U_p \circ \phi)$
was injective. On the other hand, the group~$\prod_{x \in T(\rhobar)} \GL_2(\Z_x)$ acts invertibly on~$X_{V_{1}(Q)}$
and hence also the cohomology group above, acts transitively on the set of connected components,
and commutes with the Hecke operators at~$p$. Hence it suffices to check these identities
on the component at~$\infty$, where the required conclusions follow from our previous computation.)
We then use the above inclusion of Hecke
modules to deduce the result over the algebra $\T_\m/J_m$.

It remains to show that condition~\eqref{at-vexing} of
Definition~\ref{defn:minimal} holds. For this, we use that fact that
multiplication by a high power of a lift of the Hasse invariant of
level $X_{V_1(Q)}$ realizes  
\[ H^0(X_{U_{1}(Q)},(\omega\otimes\CL_{\sigma})_{\CO/\varpi^m})_\m \]
as a Hecke equivariant subquotient of 
\[ (H^0(X_{V_{1}(Q)},\omega^n )_\m \otimes_{\CO}W_\sigma)^G \]
for some sufficiently large $n$. (This follows from
Corollary~\ref{cor:vanishing-H1}.) It therefore suffices to show that
the deformation of $\rhobar$ over the Hecke algebra of 
\[ (H^0(X_{V_{1}(Q)},\omega^n )_\m \otimes_{\CO}W_\sigma)^G \]
satisfies condition~\eqref{at-vexing} of
Definition~\ref{defn:minimal}. However, this is precisely the point of
the representation $W_\sigma$: it cuts out the automorphic
representations giving rise to minimal deformations of $\rhobar$ at
the primes in $T(\rhobar)$ (see \cite[Lemma 5.1.1]{CDT}). (Note that, since $n$ is
large, the space $H^0(X_{V_{1}(Q)},\omega^n )_\m$ is torsion free.)
This completes the proof.

 \end{enumerate}
\end{proof}

Theorem~\ref{weightone} follows from the previous result and Verdier duality as in
Section~\ref{sec:modularity-lifting-w1}. We remark that 
$H_0(X,\omega\otimes\CL_\sigma)_{\mE}$ is of rank one over $\T_{\emptyset,\mE}$ when
$H_0(X,\omega_K\otimes\CL_\sigma)_{\mE}$ is non-zero. This
follows from multiplicity one  for $\GL(2)$ and
 \cite[Lemma 4.2.4(3)]{CDT}.

\section{Complements}
\label{sec:complements}

 \subsection{Multiplicity Two} \label{section:multiplicitytwo}
  Although this is not needed for our main results, we deduce in this section some
  facts about global multiplicity of Galois representations in modular Jacobians.
  Recall that $k$ denotes a finite field of odd characteristic, $\OL$
  denotes the ring of integers
of some finite extension $K$ of $\Q_p$ with uniformizer $\varpi$ and $\OL/\varpi = k$.

We recall some standard facts about Cohen--Macaulay rings
from ~\cite{Eisenbud}, \S~21.3 (see also~\cite{Kunz}).
Let $(A,\m,k)$ be a  complete local Cohen--Macaulay ring of
dimension $n$.  Then $A$ admits a canonical module $\omega_A$.
 Moreover, if $(x_1, \ldots x_m)$ is a regular sequence for $A$,
and $B = A/(x_1, \ldots, x_m)$, then 
$$\omega_B:= \omega_A \otimes_A B$$
is a canonical module for $B$. It follows that
$$\omega_A \otimes_A A/\m = \omega_A \otimes_A (B \otimes_B B/\m)
= \omega_B \otimes_B B/\m.$$
If $m = n$, so $B$ is of dimension zero, then
$\Hom(*,\omega_B)$ is a dualizing functor, and so
$$\dim_k B[\m] = \dim_k \omega_B \otimes B/\m = \dim_k \omega_A \otimes A/\m.$$
Moreover, we have the following:

\begin{lemma} \label{lemma:dual} Let $A$ be a finite flat local $\Z_p$-algebra. Suppose that
$A$ is Cohen--Macaulay. Then 
$\Hom_{\Z_p}(A,\Z_p)$ is a canonical module for $A$.
\end{lemma}

\begin{proof} More generally, if $A$ is a module-finite
extension of
a regular (or Gorenstein) local ring $R$, then (by Theorem~21.15 of~\cite{Eisenbud}) $\Hom_{R}(A,R)$
is a canonical module for $A$. 
\end{proof}

Finally, we note the following:

\begin{lemma}   \label{lemma:vexCM} If $B$ is a complete local Cohen--Macaulay $\OL$-algebra
and admits a dualizing module
$\omega_B$ with $\mu$ generators, then the same is true for the power series
ring $A = B[[T_1,\ldots, T_n]]$. Moreover, the same is also
true for $B \widehat\otimes_{\OL} C$, for any  complete local $\OL$-algebra  $C$ which is 
a complete intersection.
\end{lemma}

\begin{proof} For power series rings this is a special case of the discussion above.
Consider now the case of $B \widehat\otimes_{\OL} C$. By assumption, 
$C$ is a quotient of $\OL[[T_1, \ldots, T_n]]$ by a regular sequence.
Hence $B \widehat\otimes_{\OL} C$ is a quotient of $B[[T_1, \ldots, T_n]]$ by a regular sequence,
and the result follows 
from the discussion above applied to the maps $B[[T_1, \ldots, T_n]] \rightarrow B$ and
$B[[T_1, \ldots, T_n]] \rightarrow B \widehat\otimes_{\OL} C$ respectively.
\end{proof}

As an example, this applies to $B[\Delta]$ for any finite abelian group $\Delta$ of $p$-power
order, since $\OL[\Delta]$ is a complete intersection.
 						
\medskip

Let $\rhobar : G_p \to \GL_2(k)$ be unramified with $\rhobar(\Frob_p)$
scalar. Let $\Rda$ denote the framed deformation ring of ordinary representations  of weight $n$ over $\OL$-algebras
with fixed determinant 
(together with a Frobenius eigenvalue $\alpha$ acting on an ``unramified quotient'')
as in Theorem \ref{thm:ord-def-ring}.

\begin{theorem} $\Rda$ is a complete normal local Cohen--Macaulay ring
of relative dimension~$4$ over~$\OL$.
 \label{theorem:three} Let $\omega_{\Rda}$ denote the canonical module
of $\Rda$. Then $\dim_k \omega_{\Rda}/\m = 3$.
\end{theorem}

\begin{proof} Following the previous discussion, 
to determine $ \dim_k \omega_{\Rda}/\m$, it suffices to find a regular sequence of length $5$ ( $= \dim \Rda$),
take the quotient $C$, and compute $\dim_k C[\m]$.
Since $\Rda$ is $\OL$-flat, $\varpi$ is regular, and thus we may choose $\varpi$ as the
first term of our regular sequence.
Yet the method of Snowden shows that
$\Rda \otimes k$ is given  by the following relations (the completion of $\mathcal{B}_1$ at $b = (1,1;0)$ in the notation of~\cite{Snowden}):
$$m = \left( \begin{matrix} a & b \\ c & -a \end{matrix}\right), n = \phi - \mathrm{id}  = \left( \begin{matrix} \phi_1 & \phi_2 \\ \phi_3 & \phi_4
\end{matrix} \right), \beta = \alpha -  1,
m n = \beta m, P_{\phi}(\alpha) = 0, m^2 = 0, \det(\phi)=1$$
Explicitly, in terms of equations, this is given by the quotient $A$ of
$$ k[[a,b,c,\phi_1,\phi_2,\phi_3,\phi_4,\beta]]$$ by the following relations:
$$\phi_1 + \phi_4 + \phi_1 \phi_4 - \phi_2 \phi_3 = 0,
\qquad
\beta^2 - (\phi_1 + \phi_4) \beta - (\phi_1 + \phi_4) = 0,$$
$$a \phi_1 + b \phi_3 = a \beta,
a \phi_2 + b \phi_4 = b \beta,
-a \phi_3 + c\phi_1 = c \beta,
a \phi_4 - c \phi_2 = a \beta,
a^2 + b c =  0.$$
 
 \medskip
 
 For a complete local $k$-algebra $(R,\m)$ with residue field $k$, let
$$H_R(t) = \sum_{n=0}^{\infty} \dim_k(\m^n/\m^{n+1}) t^n \in \Z[[t]]$$
denote the corresponding Hilbert series.
We define a partial  ordering of elements of $\Z[[t]]$ as follows: say that
$$\sum_{n=0}^{\infty} a_n t^n \ge \sum_{n=0}^{\infty} b_n t^n$$
whenever $a_n \ge b_n$ for all $n$. 
\begin{lemma} \label{lemma:hilbert} Let $x \in \m^d$. 
We have
$$\frac{H_{R/x}(t)}{1-t} \ge H_R(t) \cdot \frac{1-t^d}{1 - t},$$ 
and equality holds if and only if $x$ is a regular element. 
Moreover, if there is an isomorphism 
$$R \simeq \mathrm{gr}(R) = 
\bigoplus \m^n/\m^{n+1},$$
and $x$ is pure of degree $d$, 
 then equality holds if and only if $x$ is a regular element.
  \end{lemma}

\begin{proof}
There is an exact sequence as follows:
$$R/\m^{n} \rightarrow R/\m^{n} \rightarrow R/(x,\m^{n}) \rightarrow 0.$$
The kernel of the first map certainly contains $\m^{n-d}/\m^{n}$.
If $H_R(t) = \sum a_n t^n$ and $H_{R/x}(t) = \sum b_n t^n$, it follows that
$$\begin{aligned}
\text{coefficient of $t^{m+d-1}$ in $\displaystyle{\frac{H_{R/x}(t)}{(1-t)}}$} \ = & \  \sum_{n=0}^{m+d-1} b_n = \dim R/(x,\m^{m+d}) \\
= & \ \dim \coker(R/\m^{m+d} \rightarrow R/\m^{m+d} )\\
= & \ \dim \ker(R/\m^{m+d} \rightarrow R/\m^{m+d} )\\
\ge & \ \dim \m^m/\m^{m+d} \\
= & \ a_m + a_{m+1} + \ldots + a_{m+d-1} \\
= & \  \text{coefficient of $t^{m+d-1}$ in} \  H_R(t)(1 + t + \ldots + t^{d-1}).
\end{aligned}$$
This proves the inequality. (Note that the coefficients of~$t^n$ for~$n < d$ are automatically the same.) On the other hand, assume that
$x$ is not a regular element. By assumption, there exists a non-zero element $y \in R$ such that
$xy = 0$. By Krull's intersection theorem, there exists an $m$ such that
$y \notin \m^m$. For such an $m$, it follows that the
kernel of $R/\m^{m+d} \rightarrow R/\m^{m+d}$ is strictly
bigger than $\m^m/\m^{m+d}$, and the  inequality above is strict. 
Finally, assume that $x$ is a regular element, and that $R \simeq \mathrm{gr}(R)$. 
Then the kernel of the map with $n = m+d$ above is precisely  $\m^{m}/\m^{m+d}$, and
we have equality.
\end{proof}

Note that, in the non-graded case, the converse is not true, namely, $x$ may be regular of degree one and yet
the equality $H_{R/x}(t) = H_R(t)(1-t)$ fails; as an example
one may take  $R = k[[\phi^6,\phi^7,\phi^{15}]]$ and $x = \phi^6$. Then $x$ is regular,
but
$$\frac{H_{R/x}(t)}{1-t} = \frac{1 + 2t + 2t^2 + t^3}{1-t} > \frac{1+2t+t^2+t^3+t^5}{1-t}
= H_{R}(t).$$
The point is that $x$ is no longer regular in $\mathrm{gr}(R)$. (In fact, $R$ is a Cohen--Macaulay
domain, but the depth of $\mathrm{gr}(R)$ is zero; this example was taken
from~\cite{Rossi}).

\medskip

Let $B \simeq A/\beta$. The equation $\beta^2 - (\phi_1 + \phi_4) \beta - (\phi_1 + \phi_4) = 0$ in~$A$
becomes $\phi_1 + \phi_4 = 0$ in~$B$, and hence $B$ is the quotient of
$$ k[[a,b,c,\phi_1,\phi_2,\phi_3]]$$ by the following relations:
$$-\phi^2_1 - \phi_2 \phi_3 = 0,$$
$$a \phi_1 + b \phi_3 = 0,
a \phi_2 - b \phi_1 = 0,
-a \phi_3 + c\phi_1 = 0,
-a \phi_1 - c \phi_2 = 0,
a^2 + b c =  0.$$
All the relations in $B$ are pure of degree two, and hence there is an isomorphism
$B \simeq \gr(B)$.

\begin{lemma} The first few terms of $H_B(t)$ are
$$H_B(t) = 1 + 6t + 15 t^2 + \ldots $$
\end{lemma}

\begin{proof} Clearly $\dim(B/\m) = 1$. $B$ is a quotient of a power
series ring  $S = k[[a,b,c,\phi_1,\phi_2,\phi_3]]$.
Moreover, since all the relations are quadratic, we have
$$\dim \m/\m^2 = \dim \m_S/\m^2_S =  6.$$
The six generators of $S$ give rise,  \emph{a priori}, to 
$$\dim \m^2_S/\m^3_S = \displaystyle{\binom{7}{2} = 21}$$
 generators of $\m^2/\m^3$. 
Note, however, that we have $6$ quadratic relations. In order to prove that $\dim \m^2/\m^3 = 21 - 6 = 15$, it
suffices to show that these six relations are linearly independent.
Choose a basis of $\m^2_S/\m^3_S$ coming from the lexiographic ordering $a > b > c > \phi_1 > \phi_2 > \phi_3$.
With respect to this basis, the matrix of relations is as follows:
$$\left(
\begin{array}{ccccccc}
a^2 		&   0 & 0 & 0 & 0 & 0 & 1 \\
ab 		&   0 & 0 & 0 & 0 & 0 & 0 \\
ac 		&   0 & 0 & 0 & 0 & 0 & 0 \\
a\phi_1 	&   0 & 1 & 0 & 0 & -1 & 0 \\
a\phi_2 	&   0 & 0 & 1 & 0 & 0 & 0 \\
a\phi_3 	&   0 & 0 & 0 & -1 & 0 & 0 \\
b^2 		&   0 & 0 & 0 & 0 & 0 & 0 \\
bc 		&   0 & 0 & 0 & 0 & 0 & 1 \\
b\phi_1 	&   0 & 0 & -1 & 0 & 0 & 0 \\
b\phi_2 	&   0 & 0 & 0 & 0 & 0 & 0 \\
b\phi_3 	&   0 & 1 & 0 & 0 & 0 & 0 \\
c^2 		&   0 & 0 & 0 & 0 & 0 & 0 \\
c\phi_1 	&   0 & 0 & 0 & 1 & 0 & 0 \\
c\phi_2 	&   0 & 0 & 0 & 0 & -1 & 0 \\
c\phi_3 	&   0 & 0 & 0 & 0 & 0 & 0 \\
\phi^2_1 	&   -1 & 0 & 0 & 0 & 0 & 0 \\
\phi_1\phi_2  &   0 & 0 & 0 & 0 & 0 & 0 \\
\phi_1\phi_3  &   0 & 0 & 0 & 0 & 0 & 0 \\
\phi^2_2  &   0 & 0 & 0 & 0 & 0 & 0 \\
\phi_2\phi_3  &   -1 & 0 & 0 & 0 & 0 & 0 \\
\phi^2_3  &   0 & 0 & 0 & 0 & 0 & 0 \\
\end{array}
\right)
$$
The minor consisting of rows
$1$,  $4$,  $6$, $9$, $14$, and $16$ has determinant $1$,
and hence the result follows.
\end{proof}

Recall (Theorem~3.4.1 of~\cite{Snowden}) that $A$
is, in addition to being  Cohen--Macaulay, also a domain.   Since
$\beta \ne 0$ (it is non-zero in $\m/\m^2$), it follows that $\beta$ is a regular element,
and hence
$B = A/\beta$ is also Cohen--Macaulay.

\begin{lemma} \label{lemma:newhilbert} If $I \subset B$ is an ideal generated by a regular sequence of elements
of pure degree one of length $3$, then
$$H_{B/I}(t) = 1 + 3t.$$
Moreover, if $I$ is any ideal generated by three pure  elements of degree one such that $H_{B/I}(t) = 1 + 3 t$, then
the generators of $I$ consist of a regular sequence.
\end{lemma}

\begin{proof}
Let $R$ be a complete local Cohen--Macaulay Noetherian graded $k$-algebra with residue field $k$.
 Replacing $R$ by $R \otimes_k \overline{k}$ does not effect the Hilbert series of $R$.
 Assume that $\dim(R) \ge 1$, so that $\m$ is not an associated prime.
We claim that $R \otimes_k \overline{k}$ admits
a regular 
element $x \in \m$ of  pure degree one. Without loss of generality, we assume that
$k = \overline{k}$. The set of zero divisors is the union of the associated
primes. By assumption, $\m$ is not  one of the associated primes.
Hence, for every associated prime $\p$, the image of $\p$ in $\m/\m^2$ is proper (since otherwise
$\p = \m$ by Nakayama's Lemma).  Because $R$ is Noetherian, there exist only finitely many associated primes.
Hence the union of the images of all such $\p$ cut out  a finite number of proper linear subspaces of $\m/\m^2$. Since
$k$ is infinite, such a union misses an infinite number of points, and hence there exists an $x \in \m \setminus \m^2$
which is not a zero-divisor. By induction, there exists a regular sequence of length $\dim(R)$ generated
by pure degree one elements. It follows that, after a finite extension, $B$ admits a regular sequence of length $3$
generated by pure degree one elements. By Lemma~\ref{lemma:hilbert} (in the graded case), if $I$ is the corresponding ideal, then
$$H_{B/I}(t) = H_B(t)(1-t)^3 = (1 + 6t + 15 t^2 + \ldots)(1-t)^3 = 1 + 3t + O(t^3).$$
If $\m$ is the maximal ideal of $B/I$, we deduce that $\m^2/\m^3 = 0$, and thus by Nakayama's Lemma that $\m^2 = 0$,
and $H_{B/I}(t) = 1 + 3t$.
Conversely, if~$I$ is any ideal generated by three pure elements such that~$H_{B/I}(t) = 1+3t$, then by  Lemma~\ref{lemma:hilbert}, we deduce that
the three generators of $I$ consist of a regular sequence.
\end{proof}

\begin{lemma} $\{\beta,a,\phi_2 + \phi_3,b+c+\phi_1  \}$ is a regular sequence in $A$.
\end{lemma}

\begin{proof} It suffices to show that $\{a,\phi_2 + \phi_3,b+c+\phi_1 \}$ is regular in $B = A/\beta$.
By Lemma~\ref{lemma:newhilbert}, it suffices to show that the Hilbert series of $B/I$ with
$I = (a,\phi_2 + \phi_3,b+c+\phi_1  )$ is $1+3t$. If $C = B/I$, then we compute that
$C$ is given by the quotient of
$$ k[[b,c,\phi_2]]$$ by the following relations:
$$- (b + c)^2 + \phi^2_2  = 0,$$
$$- b \phi_2 = 0,
b(b+c) = 0,
(b+c)c = 0,
- c \phi_2 = 0,
 b c =  0.$$
Let $x = b$, $y = c$, and $z = \phi_2$. Then, from the second, fifth, and sixth relations, we deduce
that
$$xz = yz = xy = 0.$$
Combining this with the third and fourth equations yields:
$$x^2 = x^2 + xy = 0, \ y^2 = xy + y^2 = 0.$$
The first equation yields
$$z^2 = -x^2 - y^2 - 2xy + z^2 = 0.$$
It follows that $C$ is a quotient of 
$$k[x,y,z]/(x^2,y^2,z^2,xy,xz,yz).$$
On the other hand, since all the relations are trivial in $\m^2$, we have $\dim(\m/\m^2) = 3$.
Hence the Hilbert polynomial of $C$ is $1 + 3t$, and the sequence is regular
in $A$.
\end{proof}
Since $\dim C[\m] = 3$, this completes the proof.
\end{proof}

From now until the end of Section \ref{section:multiplicitytwo}, we let $\rhobar:G_{\Q} \rightarrow \GL_2(k)$ be  an absolutely irreducible 
modular ($=$ odd) representation of Serre conductor $N =
N(\rhobar)$ and Serre weight $k(\rhobar)$ with $p+1 \ge k(\rhobar) \ge 2$. This is an abuse of notation as we have already fixed a
representation $\rhobar$ in Section \ref{sec:deform-galo-repr-w1} but
we hope it will not lead to confusion.
Assume that $\rhobar$ has minimal
conductor amongst all its twists at all other primes (one can always twist  $\rhobar$ to satisfy
these condition.)
One knows that $\rhobar$ occurs as  the mod-$p$ reduction
of a modular form of weight $2$ and level $N^*$,
where $N^* = N$ if $k = 2$ and $Np$ otherwise.
Let $\T$ denote the ring of endomorphisms of $J_1(N^*)/\Q$ generated
by the Hecke operators $T_l$ for \emph{all} primes~$l$ (including $p$), and let $\m$
denote the maximal ideal of $\T$ corresponding to $\rhobar$.
Assume that $p \ge 3$.

\begin{theorem}[Multiplicity one or two]  \label{theorem:multiplicitytwo} 
If $\rhobar$ is either ramified at $p$ or
unramified at $p$ and $\rhobar(\Frob_p)$ is non-scalar,
then $J_1(N^*)[\m] \simeq \rhobar$, that is, $\m$ has multiplicity one.
If $\rhobar$ is unramified at $p$ and $\rhobar(\Frob_p)$ is scalar, then
$J_1(N^*)[\m] \simeq \rhobar \oplus \rhobar$, that is, $\m$ has multiplicity two.
\end{theorem}

\begin{remark} \emph{ 
By results of
Mazur~\cite{MazurEisenstein} (Prop.~14.2), Mazur--Ribet~\cite{MazurRibet} (Main Theorem), 
Gross~\cite{Gross} (Prop.~12.10), Edixhoven~\cite{Edixhoven} (Thm.~9.2), Buzzard~\cite{Buzzard},
and Wiese~\cite{WieseMult} Cor.~4.2,
the theorem is known except in the case when $\rhobar$ is
unramified at $p$ and $\rhobar(\Frob_p)$ is scalar.
In this case, Wiese~\cite{WieseMult} has shown that the multiplicity is always \emph{at least}
two. Thus our contribution to this result is to show that the multiplicity
is \emph{exactly} two in the scalar case. }
\end{remark}

\begin{remark} \emph{It was historically the case that multiplicity one was an
\emph{ingredient} in modularity lifting theorems, e.g., Theorem~2.1 of~\cite{W}.
It followed that the methods used to prove such theorems required a careful
study of
the geometry of $J_1(N^*)$. However, a refinement of the Taylor--Wiles method
due to Diamond showed that one could \emph{deduce} multiplicity one in certain
circumstances  while simultaneously proving a modularity theorem (see~\cite{DiamondMult}).
Our argument is in the spirit of Diamond, where it is the geometry of a local deformation ring
rather than $J_1(N^*)$ that is the crux of the matter.}
\end{remark}

\begin{proof} 
Let $G$ denote the part of the $p$-divisible
group of $J_1(N^*)$ which is associated to $\m$. 
By~\cite{Gross}, Prop~12.9, as well as the proof
of Prop~12.10, recall there is an exact sequence of groups
$$0 \rightarrow T_p G^0 \rightarrow T_p G \rightarrow T_p G^e \rightarrow 0$$
which is stable under $\T_{\m}$. Moreover,  $T_p G^0$ is free of
rank one over $\T_{\m}$, and $T_p G^e = \Hom(T_p G^0,\Z_p)$.

\medskip

We may assume that $\rhobar$ is unramified at $p$ and
$\rhobar(\Frob_p)$ is scalar. Thus $N^*=pN$ (since $p$ is odd).  Let $M$ denote the
largest factor of $N$ which is only divisible by
the so called  ``harmless'' primes, that is, the primes $v$ such that~$v \equiv 1 \mod p$
and such that $\rhobar|G_v$ is absolutely irreducible.
Define the group $\Twist$  as follows:
$$\Twist:=  \Z_p \otimes \prod_{x | M} (\Z/x\Z)^{\times}$$
$\Twist$ measures the group of Dirichlet characters  (equivalently,
characters of $G_{\Q}$) congruent to $1 \mod \varpi$
 which preserve the set of lifts of $\rhobar$
 of minimal conductor under twisting (by assumption, $\rhobar$
 has minimal conductor amongst its twists, so an easy exercise
 shows that these are the only twists with this property). Extending
 $\OL$ if necessary, we may assume that each character $\twist \in
 \widehat{\Twist} :=\Hom(\Twist,\overline{\Q}_p^\times)$ is valued in $\OL^\times$.
For~$\twist \in \widehat{\Twist}$, and let $\chi_\twist$ denote the
 character $\eps \cdot \langle \rhobar \eps^{-1} \rangle \twist$ of
 $G_\Q$.
 For $v|N^*$ we define a quotient $R_v=R_{v,\twist}$ of the universal framed
 deformation ring with determinant $\chi_\twist$ of $\rhobar|G_v$ as follows:
\begin{enumerate}
\item  When $v = p$, $R_v=R_{v,\twist}$ is
the ordinary framed deformation ring $\Rd$ of Section
\ref{section:case3} (with $n=2$).
\item When $v\ne p$, $R_v=R_{v,\twist}$ is the unrestricted framed deformation
  ring with determinant~$\chi_\twist | G_v$.
\end{enumerate}
The isomorphism types of these deformation rings do not depend on $\twist$.
Let $R=R_{\twist}$ denote the (global) universal deformation
ring of $\rhobar$ corresponding to
deformations with determinant $\chi_\phi$ which are unramified outside $N^*$ and which are
classified (after a choice of framing) by $R_v$ for each $v|N^*$. 
Let $R^{\square}$ denote the framed version of $R$, with framings at
each place $v|N^*$.
Let $\Tan_{\twist}$ be the anaemic weight $2$, level $\Gamma_1(N^*)$ ordinary Hecke algebra (so it does
not contain $U_p$) which acts on 
$$S_{\twist}:=\bigoplus_{\chi} S_2^{\ord}(\Gamma_1(N^*),\chi,\OL),$$
where $\chi$ runs over all the characters of $(\Z/Np\Z)^{\times}$ with
$\chi|\Twist = \phi$.
Note that 
$$S_2^{\ord}(\Gamma_1(N^*),\OL) \otimes \Q = \bigoplus_{\widehat{\Twist}} S_{\twist} \otimes \Q.$$
The reason for dealing with the harmless primes in this manner is as follows.
For all \emph{non}-harmless $v \ne p$, the conductor-minimal
deformation ring of $\rhobar|G_v$ (which classifies deformations which appear at
level $\Gamma_1(N)$) is isomorphic to the ring $R_{v,\phi}$ (up to
unramified twists). 
Equivalently, for non-harmless primes, any characteristic zero lift~$\rho$ of~$\rhobar$ is uniquely twist
equivalent to a lift of minimal conductor.
The only time this is not true is when~$v \equiv 1 \mod p$ and~$\rhobar$ has trivial invariants under~$I_v$.
Since we are assuming that~$\rhobar |_{G_v}$ is minimally ramified amongst all its twists, this
only happens when~$\rhobar |_{G_v}$ is absolutely irreducible and~$v \equiv 1 \mod p$.
However, the  na\"{\i}ve conductor-minimal deformation deformation
ring at a harmless prime is equal to the unrestricted deformation ring
and does not have fixed inertial determinant, and one needs the determinant
to be fixed for the Taylor--Wiles method to work correctly.

Consider the Galois representation $\rho : G_\Q \to
\GL_2(\Tan_{\phi,\m})$ associated to eigenforms in $S_{\twist}$. The
character $\epsilon^{-1}\det \rho$ can be regarded as a character
$\chi:(\Z/Np\Z)^{\times} \to (\Tan_{\phi,\m})^{\times}$ with $\chi|\Twist
=\phi$. Let $\psi$ denote the restriction of $\chi$ to $\Z_p\otimes
\prod_{x \nmid M}(\Z/x\Z)^\times$, which we may regard as a character
of $G_\Q$. After twisting $\rho$ by $\psi^{-1/2}$, we obtain a
Galois representation
$$G_{\Q} \rightarrow \GL_2(\Tan_{\phi,\m})$$
with determinant $\chi_\phi=\eps \cdot \langle \rhobar \eps^{-1} \rangle \twist$ which is
classified by $R=R_\twist$. Our hypotheses on $\rhobar$  (that $\rhobar$ is absolutely
irreducible and unramified at $p$) imply
that $\rhobar|G_{\Q(\zeta_p)}$ is absolutely irreducible.
Kisin's improvement of the Taylor-Wiles method yields an isomorphism
$R_{\twist}[1/p]\simeq \Tan_{\twist,\m}[1/p]$. (Here we apply the
Taylor-Wiles type patching results 
Prop.\ 3.3.1 and Lemma 3.3.4 of \cite{kisin-moduli} --- as in the proof
of Theorem~3.4.11 of \emph{ibid.} --- except that the rings denoted $B$
and $D$ in the statements of these results may no longer be integral domains in our
situation (though their generic fibres will be formally smooth over
$K$ by Lemma \ref{lem:local-rings} below). This is
due to the fact that the rings $R_{v,\twist}$ defined above may have multiple irreducible
components for certain $v\neq p$. On the other hand, the only place in \cite{kisin-moduli} where the assumption
that $B$ and $D$ be integral domains is used is in the first paragraph of the proof of Lemma
3.3.4. In our case, it will suffice to show that each irreducible
component of $R_{v,\phi}$ is in the support of $S_{\phi}$. This
follows from now standard results on the existence of modular
deformations with prescribed local inertial types.)

Now, $R^{\square}$ is (non-canonically) a power series ring over $R$, and is realized as a quotient of 
$$\left( \Rd \widehat\otimes \widehat\bigotimes_{v|N} R_{v} \right)[[x_1,\ldots,x_n]] \rightarrow R^{\square}$$
by a sequence of elements that can be extended to a system of
parameters; this last fact follows from the proof of
Prop 5.1.1 of \cite{Snowden} and the fact that $R$ is finite over
$\OL$ (for the most general results concerning the finiteness of deformation rings over $\OL$,
see Theorem~10.2 of~\cite{thorne}). 
As a variant of this, we may consider deformations of $\rhobar$ together with an eigenvalue
$\alpha$ of Frobenius at $p$. Globally, this now corresponds to a modified global
deformation ring $\Ra = \Ra_{\twist}$ and the corresponding framed version $\Ra^{\square}$, where
we now map to the full Hecke algebra $\T_{\m}$.
There are surjections:
$$\Rloc[[x_1,\ldots,x_n]]:= \left( \Rda \widehat\otimes \widehat\bigotimes_{v|N} R_{v} \right)[[x_1,\ldots,x_n]] \rightarrow \Ra^{\square}
\rightarrow \Ra \rightarrow \Ra/\varpi.$$
Since $\Ra$ is finite over $\OL$, it follows that $\Ra/\varpi$ is Artinian.
Again, as in the proof of Prop~5.1.1 of~\cite{Snowden} (see also Prop.~4.1.5 of~\cite{KisinCDM}),
the kernel of the composition of these maps
is given by a system of parameters, one of which is $\varpi$. 
On the other hand, we have:
\begin{lemma}
\label{lem:local-rings} The rings $R_v=R_{v,\phi}$ for $v \ne p$ are complete
  intersections. Moreover, their generic fibres $R_v[1/p]$ are
  formally smooth over $K$.
\end{lemma}
\begin{proof}  There are three cases in which $R_v$ is not smooth.
In two of these cases,  we shall prove that $R_v$ is a power series ring over $\OL[\Delta]$ for
some finite cyclic abelian $p$-group $\Delta$. Since $\OL[\Delta]$ is manifestly
a complete intersection with formally smooth generic fibre, this suffices to prove the lemma in these
cases. In the other case, we
will show that $R_v$ is a quotient of a power series ring by a single
relation. This shows that it is a complete intersection. 
The three situations in which $R_v$ is not smooth correspond to primes $v$ such that:
\begin{enumerate}
\item $v \equiv 1 \mod p$, $\rhobar |G_v$ is reducible, and 
$\rhobar |I_v \simeq \chi \oplus 1$ for some ramified $\chi$.
\item $v \equiv -1 \mod p$, $\rhobar |G_v$ is absolutely irreducible and induced
from a character $\xi$. 
\item $v \equiv 1 \mod p$, $\rhobar^{I_v}$ is 1-dimensional and
  $\rhobar^{\ss}|G_v$ is unramified.
\end{enumerate}
Suppose that $v$ is a vexing prime (the second case).
Any conductor-minimal
deformation of $\rhobar$ is induced  from a character of the form
$\langle\xi \rangle \psi$ over the quadratic unramified extension of $\Q_v$, where $\psi \mod \varpi$ is trivial. It follows that $\psi$
is tamely ramified, and in particular, up to unramified  twist, it may be identified with a character of 
$\F^{\times}_{v^2}$ of $p$-power order. We may therefore write down the universal
deformation explicitly, which identifies $R_v$ with a power series ring over $\OL[\Delta]$,
where $\Delta$ is the maximal $p$-quotient of $\F_{v^2}^{\times}$.

Suppose that we are in the first case, and so, after an unramified
twist, $\rhobar|G_v \cong \chi \oplus 1$. 
All $R_{v}$-deformations of $\rhobar$ are
of the form $(\langle \chi \rangle \psi \oplus  \psi^{-1})\otimes
(\chi_\phi \langle\chi^{-1}\rangle)^{1/2}$, where $\psi \equiv 1 \mod \varpi$.
It follows that $\psi$ is tamely ramified, and in particular, decomposes as an unramified character
and a character of $\F^{\times}_{v}$ of $p$-power order.  We may therefore write down the universal
framed deformation explicitly, which identifies $R_v$ with a power series ring over $\OL[\Delta]$,
where $\Delta$ is the maximal $p$-quotient of $\F_{v}^{\times}$.

In the third case, the deformation rings are not quite as easy to describe explicitly, so we
use a more general argument. As noted by the referee, the following argument may also
be easily modified to deal with the first two cases.
We first note that $R_{v,\phi}$ is a quotient of a power
series ring over $\OL$ in $\dim Z^1(G_v,\ad^0\rhobar)= 4$ variables
by at most $\dim H^2(G_v,\ad^0\rhobar)=1$ relation.  Closed points on the
generic fiber of $R_{v,\phi}$ correspond to lifts $\rho$ of
$\rhobar|G_v$ which are either unramified twists of the Steinberg
representation or lifts which decompose (after inverting $p$) into a
sum $\chi_\phi \psi \oplus \psi^{-1}$ with $\psi|I_v$ of $p$-power
order. The completion of $R_v[1/p]$ at such a point is the
corresponding characteristic 0 deformation ring of the lift $\rho$ (see Prop.~2.3.5 of~\cite{kisin-moduli}). In
each case, we have $\dim H^2(G_v,\ad^0\rho)=0$ and hence this ring is
a power series ring (over the residue field at the point) in $\dim
Z^1(G_v,\ad^0 \rho) = 3$ variables. It follows that $R_v \cong
\OL[[x_1,x_2,x_3,x_4]]/(r)$ for some $r\neq 0$ and $R_v[1/p]$ is
formally smooth over $K$\footnote{One may take $r$ to be $C(T)-T$,
where $C$ is the  Chebyshev-type polynomial determined by the
relation $C(t+t^{-1}) = t^v+ t^{-v}$,
and $T$ is the trace of a generator of tame inertia (note that $T-2 \in \m_{R_v}$).
The generic fibre of $R_v$ has $(q+1)/2$ geometric components,
where $q$ is the largest power of $p$ dividing $v-1$ (see also Theorem~1.0.1(A2.2) of~\cite{Reduzzi}).
One component corresponds to lifts of $\rhobar$ on which
inertia is nilpotent, and in particular has trace $T = 2$.
The remaining $(q-1)/2$ components correspond to
representations which are finitely ramified of order dividing $q$, on which
$T = \zeta + \zeta^{-1}$ for   some primitive $q$-th root of unity $\zeta \ne 1$.}. This concludes the proof of the lemma.
\end{proof}
 By Lemma~\ref{lemma:vexCM}, it follows
 that $\Rloc[[x_1,\ldots,x_n]]$  is Cohen--Macaulay,
and  hence  the sequence of parameters giving rise to the quotient $\Ra/\varpi$
 is a regular sequence. In particular,
$\Ra$ is Cohen--Macaulay and $\varpi$-torsion free.
Moreover, again by Lemma~\ref{lemma:vexCM}, the number of generators of the canonical module of 
of $\Rloc[[x_1,\ldots,x_n]]$ (and hence of $\Ra$) is equal to the number of generators
of the canonical module of $\Rda$, which is $3$, by Theorem~\ref{theorem:three}.
Since patching arguments may also be applied to the adorned Hecke algebras $\T_{\m}$,
The method of Kisin yields  an isomorphism $\Ra[1/p] = \T_{\twist,\m}[1/p]$ (note that $\Rda$ is a domain, and
$\Rda[1/p]$ is formally smooth). Since (as proven above) $\Ra$ is
$\OL$-flat, it follows that $\Ra \simeq \T_{\twist,\m}$.
In particular, we deduce that $\T_{\twist,\m}$ is Cohen--Macaulay,
and 
 that $\dim \omega_{\T_{\twist,\m}}/\m = 3$.
There is an isomorphism as follows:
 $$S_2(\Gamma_1(N^*),\OL)_{\m}  \otimes K = \bigoplus_{\widehat{\Twist}} S_{\twist^2,\m} \otimes K,$$
 where, since~$p$ is odd, we write every element  of~$\widehat{\Twist}$ uniquely as a square.
 If~$\T_{\twist,\m}$ denotes the Hecke action on~$S_{\twist,\m} \otimes K$, then twisting by~$\twist$ induces an isomorphism
 $\T_{\twist^2,\m} \simeq \T_{1,\m} \otimes_{\OL} \OL(\twist),$ since this is precisely the effect
 twisting has on the action of the diamond operators. (Here~$\OL(\twist) = \OL$ with the~$\Phi$ action twisted by~$\twist$.)
   If $\T_{\m}$ is the Hecke ring at full level $\Gamma_1(N^*)$, then the restriction map induces an inclusion map
   $$\T_{\m} \hookrightarrow \bigoplus \T_{\twist^2,\m}  \simeq \T_{1,\m} \otimes \bigoplus \OL(\twist).$$
  Since~$\T_{\m}$ is local, the image lands inside~$\T_{1,\m} \otimes \OL[\Twist]$.  Because the map 
  above is an isomorphism
  after tensoring with~$K$, and since all the  relevant spaces of modular forms are $\varpi$-torsion free,
  we have an isomorphism~$\T_{\m} \simeq \T_{1,\m} \otimes_{\OL} \OL[\Twist]$.
   Hence, applying  Lemma~\ref{lemma:vexCM} once more, we deduce that $\T_{\m}$
 is Cohen--Macaulay and  $\dim \omega_{\T_{\m}}/\m = 3$
  
Since $\T_{\m}$ 
is finite over $\Z_p$, we deduce by 
Lemma~\ref{lemma:dual} 
that $\Hom(\T_{\m},\Z_p)$
is the canonical module of $\T_{\m}$, and thus $\Hom(\T_{\m},\Z_p)/\m$ also has dimension three.
Yet we have identified  $\Hom(\T_{\m},\Z_p)$ with $T_p G^e$, and it follows that $\dim G^0[\m]  =
\dim T_p G^e/\m =  3$, and hence
$$\dim J_1(N^*)[\m] = \frac{1}{2} \dim G[\m] = \frac{1}{2}(\dim G^0[\m] + \dim G^e[\m])
= \frac{3 + 1}{2} = 2.$$
\end{proof}

\begin{remark} \emph{If  $X_H(N^*) = X_1(N^*)/H$ is the smallest
quotient of $X_1(N^*)$ where one might
expect $\rhobar$ to occur, a similar argument shows that
$J_H(N^*)[\m]$ has multiplicity two if $\rhobar$ is unramified and scalar at $p$,
and has multiplicity one otherwise,
 providing that $p \ne 3$ and $\rhobar$ is not induced from
a character of $\Q(\sqrt{-3})$. The only extra ingredient required is
the result  of Carayol (see~\cite{carayolred}, Prop.~3 and
also~\cite{edixhoven-fermat}, Prop.~1.10).}
\end{remark}

\begin{remark} \emph{We expect that these arguments should also apply
in principle when $p = 2$; the key point is that one should
instead use the quotient $\Rda_3$ of $\Rda$ (in the notation of~\cite{Snowden},~\S~4),
 corresponding to \emph{crystalline}
ordinary deformations. The special fibre of $\Rda_3$ is (in this case)
also given by $\mathcal{B}_1$, and thus one would deduce that
the multiplicity of $\rhobar$ is 
two when $\rhobar(\Frob_2)$ is scalar, assuming that $\rhobar$ is
not induced from a quadratic
extension. The key point to check is that the arguments above are compatible
with the modifications to the $R = \T$ method for $p = 2$ developed by
Khare--Wintenberger and Kisin (in particular, this will require that $\rhobar$
is not dihedral.)}
\end{remark}

\subsection{Finiteness of deformation rings}

\begin{lemma} \label{lemma:forbrian} Let $F/\Q$ be a number field, let $k$ be a finite field, and let $S$ denote a finite
set of places not containing any $v|p$. Let $G_{F,S}$ denote the Galois group of the maximal
extension of $F$ unramified outside $S$. 
Let
$$\rhobar: G_{F,S} \rightarrow \GL_2(k)$$ 
be a continuous absolutely irreducible representation, and
let $\Rmint$ denote the universal deformation ring of $\rhobar$. Suppose that the Galois
representation associated to any~$\Qbar_p$-point of~$\Rmint$ has finite image, and suppose that there
are only finitely many~$\Qbar_p$-points of~$\Rmint$. Then~$\Rmint[1/p]$ is reduced; equivalently, 
$\Rmint[1/p]^{\red} = \Rmint[1/p]$.
\end{lemma}

\begin{proof}  Because $S$ is a finite set of primes, it follows
from the discussion in~\S1 p.387 of~\cite{Mazdef} that $\Rmint$ is a complete local Noetherian~$W(k)$-algebra.
 The assumption that~$\Rmint$ has only finitely many~$\Qbar_p$-points implies 
 that~$\Rmint[1/p]^{\red}$ is isomorphic to a product of  finitely many fields indexed by prime ideals $\p^{\red}$ of 
$\Rmint[1/p]^{\red}$. Since $\Spec(\Rmint[1/p]^{\red})$ and $\Spec(\Rmint[1/p])$ are naturally isomorphic as
sets, there is a bijection between primes $\p$ of $\Rmint[1/p]$ and $\p^{\red}$ of $\Rmint[1/p]^{\red}$.
Hence
$\Rmint[1/p]$ is a Noetherian semi-local ring, which therefore decomposes as a direct sum of its localizations
over all finite ideals $\p$. It suffices to show that the localizations of $\Rmint[1/p]$ and $\Rmint[1/p]^{\red}$
at every prime $\p$ are isomorphic. Denote this localization of $\Rmint[1/p]$ by $(A,\m)$. Note that 
$A/\m = A/\p = \Rmint[1/p]^{\red}/\p^{\red} \simeq E$ for some finite extension $E$ of $\Q_p$.
We have Galois representations as follows:
$$G_{F,S} \rightarrow \GL_2(\Rmint) \rightarrow \GL_2(\Rmint[1/p]) \rightarrow \GL_2(A) \rightarrow \GL_2(A/\m^2) \rightarrow \GL_2(E)$$
To show that $A = E$, it suffices, by Nakayama's Lemma, to show that
$A/\m^2 = A/\m$.  Because $E$ is of characteristic zero, the map
$A/\m^2 \rightarrow E$ splits, and $A/\m^2$ has the structure of an $E$-algebra.
If $A/\m^2 \ne A/\m$, then the map $A/\m^2 \rightarrow A/\m$  factors through a surjection
 $A/\m^2 \rightarrow E[\eps]/\eps^2$. 
Because $\rhobar$ is absolutely irreducible, the ring $\Rmint$ is generated by the traces of the images of
elements of $G_{F,S}$ (Proposition~4, ~S1.8 of~\cite{Mazdef}.) 
It follows that the traces of the elements of $G_{F,S}$ generate
$\Rmint[1/p]$ and all its quotients  over $W(k) \otimes \Q$.
 It thus suffices to show that the images of the elements
 of $G_{F,S}$ in $\GL_2(E[\eps]/\eps^2)$ all have traces in $E$. 
Consider the corresponding Galois representation
 $$\rho: G_{F,S} \rightarrow \GL_2(E[\eps]/\eps^2).$$
The composite to $\GL_2(E)$ has finite image by assumption.
 Denote the corresponding finite image Galois representation over $E$ by $V$.
Hence $\rho$ arises from some extension
$$0 \rightarrow V \rightarrow W \rightarrow V \rightarrow 0.$$
Consider the restriction of this representation to a finite extension $L/F$ such that $G_{L,S}$ acts trivially on $V$. Then the action of
$G_L$ on $W$ factors through a $\Z_p$-extension which is unramified outside primes outside those
above $S$, and is  in particular unramified at all primes $v|p$. Such extensions are trivial by class field theory.
Hence the extension splits over $G_{L,S}$. However, because $G_{L,S}$ has finite image in $G_{F,S}$, the extension
also splits over $G_{F,S}$, because the inflation map is injective (the kernel is computed by~$H^1$ of a finite group acting in characteristic zero).
It follows that the extension is trivial over $G_{F,S}$. Yet this implies that the trace of the image of any element lies in $E$,
which completes the proof.
\end{proof}

If~$\rhobar:G_{Q,S} \rightarrow \GL_2(k)$ as above is modular,
then one can often deduce the assumptions (and hence the conclusions)
of Lemma~\ref{lemma:forbrian} from work of Buzzard--Taylor and Buzzard~\cite{BuzzT,BuzzWild}.

\section{Imaginary Quadratic Fields}
\label{sec:imag-quadr-fields}

In this section, we apply our methods to Galois representations of
regular weight over imaginary quadratic fields. The argument, formally, is very
similar to what happens to weight one Galois representations over $G_{\Q}$.
The most important difference is that we are not able to prove the existence of
Galois representations associated to torsion classes in cohomology, and
so our results are predicated on a conjecture that suitable Galois representations
exist (Conjecture~\ref{conj:A}).

\subsection{Deformations of Galois Representations}
\label{sec:deform-galo-repr-imaginary}

Let $F$ be an imaginary quadratic field, and let $p \ge 3$ be a prime that
is unramified in $F$.  
Suppose that $v|p$ is a place of $F$ and $A$ is an
Artinian local $\OL$-algebra. We say that a continuous representation
$\rho:G_v \to \GL_2(A)$ is \emph{finite flat} if there is a finite flat
group scheme $\CF/\OL_{F_v}$ such that $\rho \cong
\CF(\overline{F}_v)$ as $\Z_p[G_v]$-modules, and $\det(\rho | I_v)$ is the cyclotomic character.
 We say that $\rho$ is
\emph{ordinary} if $\rho$ is conjugate in $\GL_2(A)$ to a
representation of the form 
\[ \begin{pmatrix}
\epsilon\chi_1 & * \\
0 & \chi_2 
\end{pmatrix} \]
where $\chi_1$ and $\chi_2$ are unramified.

Let
$$\rhobar: G_{F} \rightarrow \GL_2(k)$$
be a continuous Galois representation such that the restriction
$$\rhobar: G_{F(\zeta_p)} \rightarrow \GL_2(k)$$
is  absolutely irreducible.
Let $S(\rhobar)$ denote the set of primes not dividing $p$ where $\rhobar$ is ramified.
We assume the following:
\begin{enumerate}
\item $\det(\rhobar)$ is the mod-$p$ reduction of the cyclotomic character.
\item $\rhobar$ is either ordinary or finite flat at $v|p$.
\item If $x \in S(\rhobar)$, then either:
\begin{enumerate}
\item $\rhobar | I_x$ is irreducible.
\item  $\rhobar | I_x$ is unipotent.
\item $\rhobar | D_x$ is reducible, and $\rho | I_x$ is of the form $\psi \oplus \psi^{-1}$.
\item If $\rhobar | D_x$ is irreducible and $\rhobar | I_x$ is reducible, then ${\NF}(x)
\not\equiv -1 \mod p$.
\end{enumerate}
\end{enumerate}

Let $Q$  denote a finite set of primes in $\OL_F$ not containing
any primes above $p$ and not containing any primes at which $\rhobar$
is ramified.  For objects $R$ in $\CC_\OL$, we say that a representation
$\rho: G \rightarrow \GL_2(R)$ is unipotent if, after some change of basis, the
image of $\rho$ is a subgroup of the matrices of the form
$\displaystyle{\left(\begin{matrix} 1 & * \\ 0 & 1 \end{matrix} \right)}$.
For objects $R$ in $\CC_\OL$, we may consider lifts $\rho : G_F \to \GL_2(R)$ of $\rhobar$ with the following properties:
\begin{enumerate}
\item $\det(\rho) = \eps$. 
\item If $v|p$, then $(\rho \otimes_R (R/\m_R^n))|G_v$ is finite flat or ordinary for all $n\geq 1$.
\item If $v|p$ and $\rhobar|G_v$ is finite flat, then
$(\rho\otimes_R(R/\m_R^n))|G_v$ is finite flat for all $n\geq 1$.
\item If $x \notin Q \cup S(\rhobar) \cup \{v|p\}$, then
$\rho|G_{x}$ is unramified.
\item {\bf S \rm}: \label{unipotent}  If  $x  \in S(\rhobar)$, and $\rhobar | I_x$ is unipotent,
then $\rho | I_x$ is unipotent. 
\item  {\bf P \rm}:  \label{gamma1} If $x \in S(\rhobar)$, and $\rhobar | I_x \simeq  \psi  \oplus \psi^{-1}$, then
$\rho | I_x \simeq  \langle \psi \rangle \oplus \langle \psi \rangle^{-1}$.
\item {\bf M \rm}: \label{irreducible} If $x \in S(\rhobar)$, $\rhobar | D_x$ is irreducible, and $\rhobar |I_x = \psi_1 \oplus \psi_2$ is reducible, then
$\rho | I_x = \langle \psi_1 \rangle \oplus \langle \psi_2 \rangle$.
\item {\bf H}: \label{double} If $\rhobar | I_x$ is irreducible, then
  $\rho(I_x)\iso \rhobar(I_x)$. (This also follows automatically from the
determinant condition.)
\end{enumerate}
In cases $6$ and $7$ (and $8$),  there is an isomorphism
$\rho(I_x) \iso \rhobar(I_x)$.
For $x \in S(\rhobar)$, we say that $\rhobar|D_x$ is of type {\bf S\rm}pecial, {\bf P\rm}rincipal,
{\bf M\rm}ixed, 
or {\bf H\rm}armless respectively if is of the type indicated above.
Note that primes of type ${\bf M\rm}$ are called vexing by~\cite{DiamondVexing}, but we have eliminated
the most troublesome of the vexing primes, namely those $x$ with ${\NF}(x) \equiv -1 \mod p$.
The corresponding deformation functor is represented by a complete
Noetherian local $\OL$-algebra $R_Q$ (this follows from the proof
of Theorem 2.41 of ~\cite{DDT}). If $Q=\emptyset$, we will
sometimes denote $R_Q$ by $R^{\min}$. Let $H^1_{Q}(F,\ad^0 \rhobar)$
denote the Selmer group defined as the kernel of the map
\[ H^1(F,\ad^0 \rhobar) \lra \bigoplus_{x} H^1(F_x,\ad^0\rhobar)/L_{Q,x}\]
where $x$ runs over all primes of $F$ and 
\begin{itemize}
\item $L_{Q,x} =
  H^1(G_x/I_x,(\ad^0\rhobar)^{I_{x}})$ if $x
  \not \in Q\cup \{v|p\}$;
\item $L_{Q,x} =
  H^1(F_x,\ad^0\rhobar)$ if $x \in Q$ and
  $x\nmid p$;
\item $L_{Q,v} =
  H^1_{\ff}(F_v,\ad^0\rhobar)$ if $v|p$ and $v\not\in Q$;
\end{itemize}
(The group $H^1_{\ff}(F_v,\ad^0\rhobar)$ 
is defined as in~\S2.4 of~\cite{DDT}.) Let $H^1_{Q}(F,\ad^0 \rhobar(1))$ denote the corresponding
dual Selmer group.

\begin{prop}
\label{prop:tangent-space}
 The reduced tangent space $\Hom(R_Q/\m_\OL,k[\epsilon]/\epsilon^2)$ of $R_{Q}$ has
  dimension at most 
$$\dim_k H^1_{Q}(F,\ad^0 \rhobar(1)) - 1 + 
\sum_{x \in Q} \dim_k H^0(F_{x},
\ad^0 \rhobar(1)).$$
\end{prop}

\begin{proof} The argument follows along the exact lines of
Corollary~2.43 of~\cite{DDT}. The only difference in the calculation occurs 
at $v|p$ and at $v = \infty$. Specifically,  when $v|p$
and $p$ splits, the contribution to the Euler characteristic formula
(Theorem~2.19 of~\cite{DDT}) is
$$\sum_{v|p} (\dim_k H^1_{\ff}(F_v,\ad^0 \rhobar) - \dim_k H^0(F_v,\ad^0 \rhobar)),$$
which, by Proposition~2.27 of~\cite{DDT}, is at most $2$. However, the contribution
at the prime at $\infty$ is $- \dim_k H^0(\C, \ad^0 \rhobar) = -3$.
When $p$ is inert, the contribution at $p$ is 
$$  \dim_k H^1_{\ff}(F_p,\ad^0 \rhobar) - \dim_k H^0(F_p,\ad^0 \rhobar)$$
which is also at most 2 (see, for instance, Corollary~2.4.3 of~\cite{CHT} and note that  
there is an inclusion $H^1(G_{F_p}/I_{F_p},k)\subset
H^1_{\ff}(F_p,\ad\rhobar)\cap H^1(F_p,k)$ where we view $k$ as the
scalar matrices in $\ad \rhobar$). 
\end{proof}

Suppose that ${\NF}(x)\equiv 1\mod p$ and $\rhobar(\Frob_x)$ has distinct eigenvalues for
each $x\in Q$. Then $H^0(F_x,\ad^0\rhobar)$ is one
dimensional for $x\in Q$ and the preceding proposition
shows that the reduced tangent space of $R_Q$ has dimension at most
\[ \dim_k H^1_{Q}(F,\ad^0 \rhobar(1)) - 1 + \# Q.\]

We now show that one may choose a judicious set of primes (colloquially
referred to as Taylor--Wiles primes) to annihilate the dual Selmer group.

\begin{prop} 
\label{prop:tw-primes-imaginary}
Let $q =\dim_k H^1_{\emptyset}(F,\ad^0
  \rhobar(1))$ and suppose that $\rhobar|G_{F(\zeta_p)}$ is absolutely
  irreducible. Then $q\geq 1$ and for any integer $N\geq 1$ we can find a set $Q_N$
  of primes of $F$ such that 
\begin{enumerate}
\item $\# Q_N =q$.
\item ${\NF}(x) \equiv 1 \mod p^N$ for each $x\in Q_N$.
\item For each $x\in Q_N$, $\rhobar$ is unramified at $x$ and $\rhobar(\Frob_x)$ has distinct
eigenvalues.
\item $H^1_{Q_N}(F,\ad^0\rhobar(1))=(0)$.
\end{enumerate}
 In particular, the reduced tangent space of $R_{Q_N}$ has dimension at most
  $q-1$ and $R_{Q_N}$ is a quotient of a power series
  ring over $\OL$ in $q-1$ variables.
\end{prop}

\begin{proof}
That $q \geq 1$ follows immediately from Proposition \ref{prop:tangent-space}.
Now suppose that $Q$ is a finite set of primes of $F$ containing no primes
  dividing $p$ and no primes where $\rhobar$ is ramified. Suppose that
  $\rhobar(\Frob_x)$ has distinct eigenvalues and
  ${\NF}(x)\equiv 1 \mod p$ for each $x\in Q$. Then we have
  an exact sequence
\[ 0 \lra H^1_{Q}(F,\ad^0\rhobar(1)) \lra H^1_\emptyset(F,\ad^0\rhobar(1))
\lra \bigoplus_{x\in Q} H^1(G_x/I_x,\ad^0\rhobar(1)). \]
Moreover, for each $x\in Q$, the space $H^1(G_x/I_x,\ad^0\rhobar(1))$ is
one-dimensional over $k$ and is isomorphic to
$\ad^0\rhobar/(\rhobar(\Frob_x)-1)(\ad^0\rhobar)$ via the map which sends a
class $[\gamma]$ to $\gamma(\Frob_x)$. It follows that we may ignore
condition (1): if we can find a set $\tilde{Q}_N$ satisfying
conditions (2), (3) and (4), then $\# \tilde{Q}_N \geq q$ and by 
removing elements of $\tilde{Q}_N$ if necessary, we can obtain a set
$Q_N$ satisfying (1)--(4).

By the Chebotarev density theorem, it therefore suffices to show that for each non-zero class
$[\gamma]\in H^1_\emptyset(F,\ad^0\rhobar(1))$, we can find an element
$\sigma \in G_F$ such that 
\begin{itemize}
\item $\sigma | G_{F(\zeta_{p^N})} = 1$;
\item $\rhobar(\sigma)$ has distinct eigenvalues;
\item $\gamma(\sigma)\not\in (\rhobar(\sigma)-1)(\ad^0\rhobar)$. 
\end{itemize}
The existence of such a $\sigma$ can be established exactly as in the
proof of  Theorem~2.49 of~\cite{DDT}. 
\end{proof}

\subsection{Homology of Arithmetic Quotients}
\label{sec:hom-arithm-quot-imaginary}

 Let $\A$ denote the adeles of $\Q$, and
$\A^\infty$ the finite  adeles. Similarly, let $\A_F$ and
$\A_F^\infty$ denote the adeles and finite adeles of $F$. Let $\G = \mathrm{Res}_{F/\Q} \PGL(2)$, and write $G_{\infty} = 
\G(\R) = \PGL_2(\C)$.
Let $K_{\infty}$ denote a maximal compact of $G_{\infty}$ with
connected component $K^{0}_{\infty}$.
For any compact open subgroup $K$ of
$\G(\A^\infty)$, we may define an arithmetic orbifold $Y(K)$ as follows:
$$Y(K):= \G(\Q) \backslash \G(\A)/K^0_{\infty} K.$$

\begin{remark} \label{remark:orbifold}
\emph{ If $K$ is a sufficiently small (neat) compact subgroup, then $Y(K)$ is a manifold.
Moreover, it will also be a a (disjoint union of) $K(\pi,1)$ spaces, since each component is the
quotient 
of a contractible space.  Recall that for a $K(\pi,1)$-manifold $M$, there is a functorial isomorphism
$$H^n(\pi_1(M),\star) \simeq H^n(M,\star)$$
for all $n$. 
For orbifolds $M = \Gamma \backslash \mathbf{H}$ of a similar shape (with contractible  $\mathbf{H}$),
the cohomology of $M$ \emph{as an orbifold} satisfies the same formula. Note that the cohomology in this sense
may differ from the cohomology of the underlying space. (For example, the underlying manifold of $\PSL_2(\Z)$ is
the punctured sphere which is contractible, whereas the underlying orbifold has interesting cohomology.)
We take the convention that, for any $K$, the cohomology of $Y(K)$ is understood to be the cohomology
in the orbifold sense, namely, that the cohomology of each component is the cohomology of the corresponding
arithmetic lattice. The main advantage of this approach is that, for any finite index normal subgroup $K' \unlhd K$,
the corresponding map of orbifolds
$$Y(K') \rightarrow Y(K)$$
is a covering map with Galois group $K/K'$. This approach is the analogue (in the world of PEL Shimura varieties)
of working with stacks rather than the underlying schemes at non-representable level.}
\end{remark}

\medskip

We will specifically be interested in the following $K$.  Let
$S(\rhobar)$ and $Q$ be as above.

\subsubsection{Arithmetic Quotients} \label{section:ar1}
If $v$ is a place of $F$ and $c\geq 1$ is an integer, we define
\begin{eqnarray*}
  \Gamma_0(v^c)&= &\left\{ g \in \PGL_2(\OL_{v}) \ | \
g \equiv 
\left(\begin{matrix} 1 & * \\ 0 & * \end{matrix} \right) \mod \pi_{v}^c
\right\} \\
  \Gamma_1(v^c)&= &\left\{ g \in \PGL_2(\OL_{v}) \ | \
g \equiv 
\left(\begin{matrix} 1 & * \\ 0 & 1 \end{matrix} \right) \mod \pi_{v}^c
\right\} \\
\Gamma_p(v^c) & = & \left\{ g \in \PGL_2(\OL_{v}) \ | \
g \equiv 
\left(\begin{matrix} 1 & * \\ 0 & d \end{matrix} \right) \mod \pi_{v}^c
\textrm{, $d$ has $p$-power order} \right\} \\
\end{eqnarray*}
Let $K_{Q}=\prod_{v}K_{Q,v}$ and
$L_{Q}=\prod_vL_{Q,v}$ denote the open compact
subgroups of $\G(\A)$ such that:
\begin{enumerate}
\item  If $v \in Q$,  $K_{Q,v}=\Gamma_1(v)$.
\item  If $v \in Q$,  $L_{Q,v}=\Gamma_0(v)$.
\item If $v$ is not in $S(\rhobar)\cup\{v|p\} \cup Q$, then
$K_{Q,v} = L_{Q,v} = \PGL_2(\OL_v)$.
\item If $v|p$, then $K_{Q,v}=L_{Q,v}=\GL_2(\OL_v)$ if $\rhobar|D_v$
  is finite flat. Otherwise, $K_{Q,v}=L_{Q,v}=\Gamma_0(v)$.
\item  If $v \in S(\rhobar)$,  $K_{Q,v} = L_{Q,v}$ is defined as follows:
\begin{enumerate}
\item If $\rhobar$ is of type~{\bf S\rm} at $v$,  then  $K_{Q,v}=\Gamma_0(v)$.

\item If $\rhobar$ is of type~{\bf P\rm}, {\bf M\rm} or {\bf H\rm} at $v$, then
  $K_{Q,v}=\Gamma_p(v^c)$, where $c$ is the conductor of $\rhobar|D_v$.
\end{enumerate}
\end{enumerate}

 We define the arithmetic quotients $Y_0(Q)$  and $Y_1(Q)$
to be $Y(L_{Q})$ and $Y(K_{Q})$ respectively.
These spaces are the analogues
of the modular curves corresponding to the congruence
subgroups  consisting of $\Gamma_0(Q)$ and $\Gamma_1(Q)$ intersected
with a level specifically tailored to the ramification structure of $\rhobar$.
 Topologically,
they are a finite disconnected union of finite volume arithmetic hyperbolic
$3$-orbifolds.

\subsubsection{Hecke Operators}
We recall the construction of the Hecke operators. \label{subsection:heckeoperators-imaginary}
Let $g \in \G(\A^\infty)$ be an invertible matrix.
For $K\subset \G(\A^\infty)$ a compact open subgroup, the Hecke
operator $T(g)$ is defined on the homology modules $H_\bullet(Y(K),\OL)$ by considering the composition:
$$H_{\bullet}(Y(K),\OL) \rightarrow
H_{\bullet}(Y(g K g^{-1} \cap K),\OL)\ra H_\bullet(Y(K\cap g^{-1}Kg),\OL)  \rightarrow
H_{\bullet}(Y(K),\OL),$$
the first map coming from the corestriction ($=$ transfer) map, the second coming
from the map $Y(gKg^{-1}\cap K,\OL) \ra Y(K\cap g^{-1} Kg,\OL)$
induced by right multiplication by $g$ on $\G(\A)$ and the third
coming from the natural map on homology. (We recall that, since we are viewing these spaces as orbifolds,
the map $Y(g K g^{-1} \cap K) \rightarrow Y(K)$ is always a covering map.)
The Hecke operators act on $H_{\bullet}(Y(K),\OL)$ but
do not preserve the homology of the connected components. 
The group of components is isomorphic, via the determinant map, to
$$F^{\times} \backslash \A^{\infty,\times}_F/\A^{\infty,\times 2}_F \det(K).$$
This is the mod-$2$ reduction of a ray class group.
For $\alpha \in \A_F^{\infty,\times}$, we define the
Hecke operator $T_{\alpha}$ by taking 
$$g = \left(\begin{matrix} \alpha & 0 \\ 0 & 1 \end{matrix} \right).$$
If $\alpha \in \A^{\infty,\times}_F$ is a unit at all finite places, we denote
the corresponding operator by $\langle \alpha \rangle$ and refer to it as a diamond operator; it
acts as an automorphism on $Y(K)$ for all the $K$ considered above.

\medskip

\begin{df}
Let $\Tan_{Q}$ denote the sub-$\OL$-algebra of $\End_{\OL} \ H_{1}(Y_1(Q),\OL)$
generated by Hecke endomorphisms 
$T_{\alpha}$ for
all $\alpha$ which are trivial at primes in $Q\cup S(\rhobar)\cup \{v|p\}$.
Let $\T_Q$ denote the $\OL$-algebra generated by
the same operators together with $T_{\alpha}$ for
$\alpha$ non-trivial at places in $Q$. 
If $Q = \emptyset$, we write $\T=\TE$ for $\T_Q$.
\end{df}
These rings are commutative. 
If $\eps \in \OL^{\times}_F$ is a global unit, then $T_{\eps}$ acts by
the identity. If $\a \subseteq \OL_F$ is an ideal prime to the level, we may define the Hecke operator
$T_{\a}$  
as~$T_{\alpha}$
 where
$\alpha \in \A^{\times,\infty}_F$ is any element which represents the ideal $\a$ and such that
$\alpha$ is $1$ for each component dividing the level. In particular, if  $\a = x$ is prime, then $T_x$
is uniquely defined when $x$ is prime to the level but not when $x$ divides the level.

\subsection{Conjectures on Existence of Galois Representations}

Let $\m$ denote a maximal ideal of $\T_Q$, and let $\T_{Q,\m}$ denote the
completion. It is a local ring which is finite (but not necessarily flat) over
$\OL$. 

\begin{df}
\label{defn:non-eis}
We say that $\m$ is \emph{Eisenstein} if $T_\lambda-2\in\m$ for all but
finitely primes $\lambda$ which split completely in some fixed abelian extension
of $F$. We say that $\m$ is \emph{non-Eisenstein} if it is not Eisenstein.
\end{df}

We say that $\m$ is \emph{associated to $\rhobar$} if for each
$\lambda\not\in S(\rhobar)\cup Q\cup \{ v|p\}$, we have an inclusion $T_\lambda -
\tr(\rhobar(\Frob_v))\in \m$. 

\begin{conjectureA} 
\label{conj:A} Suppose that~$\m$ is non-Eisenstein and is associated to $\rhobar$,  and that~$Q$ is a set of  primes~$v$
 such that~$N(v) \equiv 1 \mod p$, $\rhobar$ is unramified at~$v$, and~$\rhobar(\Frob_v)$ has distinct eigenvalues.
 Then there exists a continuous Galois representation
$\rho = \rho_{\m}:G_F \rightarrow \GL_2(\T_{Q,\m})$ with the following properties:
\begin{enumerate}
\item\label{char-poly} If $\lambda\not\in S(\rhobar)\cup Q\cup\{v|p\}$ is a prime of $F$, then $\rho$ is unramified
at $\lambda$, and the characteristic polynomial of $\rho(\Frob_{\lambda})$ is
$$X^2 - T_{\lambda} X + {\NF}(\lambda) \in \T_{Q,\m}[X].$$
\item If $v \in S(\rhobar)$, then:
\begin{enumerate}
\item If $\rhobar|D_v$ is of type {\bf S\rm}, then $\rho | I_v$ is unipotent.
\item If $\rhobar|D_v$ is of type {\bf P\rm}, so that $\rhobar | I_v
  \cong \psi \oplus \psi^{-1}$, then $\rho|I_v \cong \langle \psi \rangle \oplus \langle \psi \rangle^{-1}$.
\end{enumerate}
\item\label{lgc-Q} If $v\in Q$,  the operators $T_{\alpha}$ for $\alpha \in F^{\times}_v \subset \A^{\infty,\times}_F$
are invertible. Let $\phi$ denote the character of $D_v = \Gal(\overline{F}_v/F_v)$ which, by class field theory,
is associated to the resulting homomorphism:
$$F^{\times}_v \rightarrow \T^{\times}_{Q,\m}$$
given by sending $x$ to $T_x$. By assumption, the image of $\phi \mod \m$ is unramified, and so
factors through $F^{\times}_v/\OL^{\times}_v \simeq \Z$, and so $\phi(\Frob_v) \mod \m$ is
well defined;  assume that $\phi(\Frob_v) \not\equiv \pm 1 \mod \m$.
Then $\rho|D_v \sim  \phi \eps \oplus \phi^{-1}$.
\item If $v | p$, then~$\rhobar|D_v$ is finite flat, and if~$\rhobar|D_v$ is ordinary, then~$\rho|D_v$ is ordinary.
\end{enumerate}
\end{conjectureA}

Some form of this conjecture has been suspected to be true at least as far back
as the investigations of F.~Grunewald in the early 70's (see~\cite{Grunewald,GHM}). Related conjectures about the existence
of $\rhobar_{\m}$ 
were made for $\GL(n)/\Q$ 
by Ash~\cite{Ash}, and for $\GL(2)/F$ by Figueiredo~\cite{Fig}.
Say that a deformation of~$\rho_Q$ of~$\rhobar$ is minimal outside~$Q$ if it arises
from a quotient of the ring~$R_{Q}$ of
\S~\ref{sec:deform-galo-repr-imaginary}.

\begin{lemma}  Assume Conjecture~\ref{conj:A}.
Assume that there exists a maximal ideal
$\m$ of $\T_{Q}$  associated to $\rhobar$. Suppose that $Q$ consists entirely of  Taylor--Wiles
primes. 
Then there exists a representation:
$\rho_Q: G_F \rightarrow \GL_2(\T_{Q,\m})$
 whose traces generate $\T_{Q,\m}$ and
 such that
$\rho_Q$ is a minimal deformation of $\rhobar$ outside $Q$ with  cyclotomic determinant.
\end{lemma}

\begin{proof} By Conjecture~\ref{conj:A}, the representation $\rho_{Q}:=\rho_{\m}$ to $\T_{Q,\m}$
is such a representation. Moreover, assumption~\ref{lgc-Q} above guarantees (by Hensel's Lemma)
that the~$T_{\alpha}$ for~$\alpha | Q$ lie in the~$\OL$-subalgebra generated by traces.
\end{proof}

\subsubsection{Properties of homology groups}
Let $\mE$ denote a  non-Eisenstein maximal ideal of $\TE$. We have  natural
homomorphisms
$$ \Tan_Q \ra \Tan  = \TE, \qquad \Tan_Q \hookrightarrow \T_Q  $$
induced by the map $H_1(Y_1(Q),\OL)\ra H_1(Y,\OL)$ and by the natural inclusion.
(The surjectivity of this map  is an immediate consequence of the interpretation
of these groups in terms of group cohomology and the fact that the abelianization of $\PSL_2(\F_x)$
is trivial for $N(x) > 3$.)
 The ideal $\mE$
of $\TE$
pulls back to an ideal of $\Tan_Q$ which we also denote by $\mE$ in a slight
abuse of notation. The ideal 
$\mE$ may give rise to multiple maximal ideals $\m$ of $\T_Q$.

\begin{remark}  \label{remark:defineU} \emph{If $x\not\in Q\cup S(\rhobar)\cup \{v|p\}$ is
    prime, then there is an operator $T_x \in \Tan_Q$. If
    $x \in Q$,
     then we let $U_x$ denote the operator
$U_x:=T_{\pi_x}$, where $\pi_x$, by abuse of notation, is the adele which is trivial
away from $x$ and the uniformizer $\pi_x$ at $x$. 
However, this operator  is only well defined up to
a diamond operator $\langle \alpha \rangle$, where $\alpha \in \OL^{\times}_x \subset \A^{\infty,\times}_F$.
On the other hand, by Conjecture~\ref{conj:A}, the image of $U_x$ modulo $\m$ is well defined,
because the associated character $\phi$ is unramified.}
\end{remark}

 If $x\not\in S(\rhobar)\cup \{v|p\}$ is a prime of
$F$  such that ${\NF}(x) \equiv 1 \mod p$ and
$\rhobar(\Frob_x)$ has distinct eigenvalues, then the representation
$\rhobar|G_x$ does not admit ramified semistable deformations.
The following lemma is the homological manifestation of this fact.

\begin{lemma} \label{lemma:matt} 
Suppose that for each $x \in Q$ we have that $N_{F/\Q}(x)\equiv 1
\mod p$ and that the polynomial $X^2 - T_{x} X +  N_{F/\Q}(x) \in \TE[X]$ has distinct eigenvalues modulo 
$\mE$.
Let $\m$ denote the maximal ideal of $\T_Q$ containing $\mE$ and
$U_x - \alpha_x$ for some choice of root $\alpha_x$ of $X^2 - T_x X + 1 \mod \m$ for each $x\in Q$.
Then there is an isomorphism of $\Tan_{Q,\mE}$-modules
$$  H_1(Y_0(Q),\OL)_{\m}\iso H_1(Y,\OL)_{\mE}.$$
\end{lemma}

\begin{proof}  Note that, by the universal coefficient theorem, we have~$H^1(Y,K/\OL) = H_1(Y,\OL)^{\vee}$
(and similarly for~$Y_0(Q)$). We proceed  as
in the proof
of Lemma~\ref{lemma:matt-w1} 
to deduce that there is an isomorphism
$$H^1(Y_0(x),K/\OL)_{\m} = H^1(Y,K/\OL)_{\mE} \oplus V.$$
 In light of the universal coefficient theorem, it suffices to show that~$V = 0$.
The remainder of the proof now proceeds as in Lemma~\ref{lemma:matt-w1}.
\end{proof}

\medskip

There is a natural covering map $Y_1(Q) \rightarrow Y_0(Q)$
with Galois group
$$\Delta_Q := \prod_{x \in Q} (\OL_F/x)^{\times}.$$
 If $\mu$ is a finitely generated $\OL[\Delta_Q]$-module, it
gives rise to a local system on $Y_0(Q)$. 
Let~$\T^{\univ}$ be the polynomial algebra generated analogously to the 
one in Section~\ref{sec:hecke-ops} 
 by Hecke endomorphisms 
$T_{\alpha}$ for
all $\alpha$ which are trivial at primes in $Q\cup S(\rhobar)\cup \{v|p\}$ and by~$U_x$ for~$x \in Q$
(see Remark~\ref{remark:defineU}).
We have an action of~$\T^{\univ}$
 on the homology groups $H_i(Y_0(Q),\mu)$
and the Borel-Moore homology groups $H_i^{BM}(Y_0(Q),\mu)$.
The ideal $\mE$ gives rise to a maximal ideal $\m$ of $\T^{\univ}$
after a choice of eigenvalue mod $\m$ for $U_x$ for all $x$ dividing $Q$.

\medskip

We let $\Delta$ denote a quotient of
$\Delta_Q$ and $Y_\Delta(Q)\ra Y_0(Q)$ the corresponding Galois
cover. 
Further suppose that $\Delta$ is a $p$-power order quotient of
$\Delta_Q$. Then $\OL[\Delta]$ is a local ring.
Note that by Shapiro's Lemma there is an isomorphism $H_1(Y_0(Q),\OL[\Delta]) \cong H_1(Y_\Delta(Q),\OL)$.

\begin{lemma} 
\label{lem:hom-basic-props}
Let $\mu$ be a finitely generated $\OL[\Delta]$-module. Then:
\begin{enumerate}
\item\label{vanish-0-3} $H_i(Y_0(Q),\mu)_{\m}=(0)$ for $i=0,3$.
\item\label{h2-torsion-free} If $\mu$ is $p$-torsion free, then $H_2(Y_0(Q),\mu)_{\m}$ is $p$-torsion free.
\item\label{vanish-boundary} For all $i$, we have an isomorphism
\[ H_i(Y_0(Q),\mu)_\m \iso H_i^{BM}(Y_0(Q),\mu)_\m.\]
\end{enumerate}
\end{lemma}

\begin{proof} Consider part $(1)$. By Nakayama's Lemma, we reduce to the case when $\mu = k$. Yet $H_3(Y_0(Q),k) = 0$
and the action of Hecke operators on $H_0(Y_0(Q),k)$ (which preserve the connected components) is via the degree map, 
and this action is Eisenstein (in the sense that the only $\m$ in the support of $H_0$ are Eisenstein).
For part $(2)$,
since $\mu$ is $\OL$-flat (by assumption), there is an exact sequence
$$0 \rightarrow \mu \rightarrow \mu \rightarrow \mu/\varpi \rightarrow 0.$$
Taking cohomology, localizing at $\m$, and using the vanishing of $H_3(Y_0(Q),\mu)_{\m}$ from
part $(1)$, we deduce that $H^2(Y_0(Q),\mu)_{\m}[\varpi] = 0$, hence the result.
For part $(3)$, there is a long exact sequence
$$ \ldots \rightarrow 
H_i(\partial Y_0(Q),\mu) \rightarrow H_i(Y_0(Q),\mu) \rightarrow H^{BM}_i(Y_0(Q),\mu) \rightarrow H_{i-1}(\partial Y_0(Q),\mu) \rightarrow \ldots $$
from which we observe that it suffices to show that $H_i(\partial Y_0(Q),\mu)_{\m}$ vanishes for all $i$.
(The action of Hecke operators on the boundary is the obvious one coming from topological considerations.
For an explicit exposition of the relevant details, see p.107 of~\cite{taylorthesis}.)
By Nakayama's Lemma, we once more reduce to the case when $\mu = k$. The cusps are given by
tori (specifically, elliptic curves with CM by some order in $\OL_{F}$), and since the cohomology with constant
coefficients of tori is torsion free, the case when $\mu = k$ reduces to the case when $\mu = \OL$ and then 
$\mu = K$.
We claim that 
the action of $\T^{\univ}$ on the homology of the  cusps in characteristic zero given by a
sum of algebraic Grossencharacters for the field $F$; such a representation
is Eisenstein by class field theory.
This follows from the work of~\cite{HarderGL2};  an explicit reference
 is~\S2.10 of~\cite{Berger}.
\end{proof}

\begin{prop}
\label{prop:balanced-homology} 
The $\OL[\Delta]$-module $H_1(Y_0(Q),\OL[\Delta])_{\m} \cong H_1(Y_\Delta(Q),\OL)_{\m}$ is balanced (in
the sense of Definition \ref{defn:balanced}).
\end{prop}

\begin{proof} The argument is almost identical to the proof of Proposition~\ref{prop:balanced-homology-w1}. 
Let $M$ denote the module $H_1(Y_0(Q),\OL[\Delta])_{\m}$ and $S=\OL[\Delta]$. Consider the exact sequence of $S$-modules (with trivial $\Delta$-action):
$$0 \rightarrow \OL \stackrel{\varpi}{\rightarrow} \OL \rightarrow
k\rightarrow 0$$
where $\varpi$ denotes a uniformizer in $\OL$.
Tensoring this exact sequence over $S$ with $M$, we obtain an exact sequence:
$$ 0 \rightarrow \Tor^S_1(M,\OL)/\varpi  \rightarrow \Tor^S_1(M,k) \rightarrow M_\Delta  \rightarrow M_\Delta 
\rightarrow M \otimes_S k \rightarrow 0.$$
Let $r$ denote the $\OL$-rank of $M_\Delta$. Then this exact sequence
tells us that
\[ d_S(M) =\dim_k M\otimes_S k - \dim_k \Tor_1^S(M,k) = r - \dim_k \Tor^S_1(M,\OL)/\varpi.\]
We have a Hochschild--Serre spectral sequence
\[ H_i(\Delta,H_j(Y_0(Q),S))=\Tor^S_i(H_j(Y_0(Q),S),\OL)
\implies H_{i+j}(Y_0(Q),\OL).\] 
We obtain an action of $\T^{\univ}$ on the spectral sequence by
essentially the same argument as that of Proposition~\ref{prop:balanced-homology-w1}.
Localizing at ${\m}$, and using the fact that
$H_i(Y_0(Q),S)_{\m}=(0)$ for $i=0,3$ by Lemma \ref{lem:hom-basic-props}(\ref{vanish-0-3}), we obtain an exact sequence
\[ (H_2(Y_0(Q),S)_{\m})_\Delta \to H_2(Y_0(Q),\OL)_{\m} \to
\Tor^S_1(M,\OL) \to 0 .\]
To show that $d_S(M)\geq 0$, we see that it suffices to show that
$H_2(Y_0(Q),\OL)_{\m}$ is free of rank $r$ as an
$\OL$-module. By Lemma \ref{lem:hom-basic-props}(\ref{h2-torsion-free}), it then
suffices to show that $\dim_K H_2(Y_0(Q),K)_{\m}=r$.
Inverting $\varpi$ and applying Hochschild--Serre again, we obtain
isomorphisms
\[ (H_i(Y_0(Q),S \otimes_\OL K)_{\m})_\Delta \iso H_i(Y_0(Q),K)_{\m} \]
for $i=1,2$. It follows that $r=\dim_K H_1(Y_0(Q),K)_{\m}$. By Poincar\'e duality, we have
\[ \dim_K H_2(Y_0(Q),K)_{\m} = \dim_K H_1^{BM}(Y_0(Q),K)_{\m}.\]
(Because we are working with $\PGL$, the dual maximal ideal $\m^*$ is identified with $\m$.)
Finally, by Lemma \ref{lem:hom-basic-props}(\ref{vanish-boundary}),we have
\[ \dim_K H_1(Y_0(Q),K)_{\m} =\dim_K H_1^{BM}(Y_0(Q),K)_{\m},\]
as required.
\end{proof}

\subsection{Modularity Lifting}
\label{sec:modularity-lifting-imaginary}

We now associate to $\rhobar$ the ideal $\mE$ of $\TE$ which is
generated by $(\varpi, T_\lambda - \tr(\rhobar(\Frob_\lambda)))$ where
$\lambda$ ranges over all primes $\lambda \not \in S(\rhobar)\cup\{v|p\}$ of
$F$. We make the hypothesis
that $\mE$ is a \emph{proper} ideal of $\TE$. In other words,
we are assuming that $\rhobar$ is `modular' of minimal level and
trivial weight. Since $\TE/\mE \hookrightarrow k$  it follows that
$\mE$ is maximal. Since $\rhobar$ is absolutely irreducible, it follows
by Chebotarev density that $\mE$ is non-Eisenstein.

We now assume that Conjecture \ref{conj:A} holds for $\mE$. In other
words, there is a continuous Galois representation 
\[ \rho_\m : G_F \ra \GL_2(\T_{\mE}) \]
satisfying the properties of Conjecture \ref{conj:A}. The definition
of $\mE$ and the Chebotarev density theorem imply that $\rho_{\mE} \mod \mE$
is isomorphic to $\rhobar$.
Properties (1)--(4) of Conjecture
\ref{conj:A} then imply that $\rho_{\mE}$
 gives rise to a homomorphism
\[ \varphi : R^{\min} \to \T_{\mE} \]
such that the universal deformation pushes
forward to $\rho_{\mE}$. The following is the main result of this
section.

\begin{theorem}
  \label{thm:main-thm-imaginary}
If we make the following assumptions:
\begin{enumerate}
\item the ideal $\mE$ is a proper ideal of $\TE$, and
\item Conjecture \ref{conj:A} holds for all $Q$,
\end{enumerate}
then the map $\varphi : R^{\min}\to\T_{\mE}$ is an isomorphism and
$\T_{\emptyset,\mE}$ acts freely on $H_1(Y,\OL)_{\mE}$.
\end{theorem}

\begin{proof}
  By property (1) of Conjecture \ref{conj:A}, the map $\varphi : R^{\min}
  \to \T_{\emptyset,\mE}$ is surjective. To prove the theorem, it therefore suffices to show
  that $H_1(Y,\OL)_{\mE}$ is free over $R^{\min}$ (where we view $H_1(Y,\OL)_{\mE}$ as an
  $R^{\min}$-module via $\varphi$). To show this, we will apply
  Proposition~\ref{prop:patchingimaginary}.

  We set $R=R^{\min}$ and $H=H_1(Y,\OL)_{\mE}$ and we define 
\[ q := \dim_k H^1_\emptyset(G_F,\ad^0 \rhobar) .\]
Note that $q \geq 1$ by Proposition \ref{prop:tw-primes-imaginary}. As in
Proposition \ref{prop:patchingimaginary}, we set $\Delta_\infty = \Z_p^q$ and
let $\Delta_N = (\Z/p^N\Z)^q$ for each integer $N\geq 1$. We also let
$R_\infty$ denote the power series ring
$\OL[[x_1,\ldots,x_{q-1}]]$. It remains to show that conditions~\ref{cond4} and~\ref{cond5} of
Proposition \ref{prop:patchingimaginary} are satisfied. For this we will use the
existence of  Taylor--Wiles primes together with the results
established in Section ~\ref{sec:hom-arithm-quot-imaginary}.

For each integer $N\geq 1$, fix a set of primes $Q_N$ of $F$
satisfying the properties of Proposition \ref{prop:tw-primes-imaginary}. We can
and do fix a surjection $\wt\phi_N : R_\infty \onto R_{Q_N}$ for each
$N\geq 1$. We let $\phi_N$ denote the composition of $\wt\phi_N$ with
the natural surjection $R_{Q_N}\onto R^{\min}$. Let
\[ \Delta_{Q_N} = \prod_{x\in Q_N}(\OL_F/x)^\times \]
and choose a surjection $\Delta_{Q_N}\onto \Delta_N$. Let
$Y_{\Delta_N}(Q_N)\to Y_0(Q_N)$ denote the corresponding Galois
cover. We set $H_N:= H_1(Y_{\Delta_N}(Q_N),\OL)_{\m}$ where $\m$ is
the ideal of $\T_{Q_N}$ which contains $\mE$ and $U_x-\alpha_x$ for
each $x\in Q$, for some choice of $\alpha_x$. Then $H_N$ is
naturally an $\OL[\Delta_N]=S_N$-module. Applying Conjecture
\ref{conj:A} to $\T_{Q_N,\m}$, 
we deduce the
existence of a surjective homomorphism $R_{Q_N}\onto
\T_{Q_N,\m}$. Since $\T_{Q_N,\m}$ acts on $H_N$, we get an induced
action of $R_\infty$ on $H_N$ (via $\wt\phi_N$ and the map
$R_{Q_N}\onto \T_{Q_N,\m}$). We can therefore view $H_N$ as a module
over $R_\infty\otimes_{\OL}S_N$. To apply Proposition
\ref{prop:patchingimaginary}, it remains to check points \eqref{cond-image}--\eqref{cond-balanced}.
 We check these conditions one by one:
\begin{itemize}
\item[(a)] The image of $S_N$ in $\End_\OL(H_N)$ is contained in the
  image of $R_\infty$ by Conjecture \ref{conj:A}, because it is
  given by the image of the diamond operators. The second part of
  condition~\eqref{cond-image} follows from Conjecture~\ref{conj:A}
  part~\eqref{lgc-Q}
  (exactly as in the proof of Theorem~\ref{thm:main-thm-w1}).
 
\item[(b)] We have a Hochschild-Serre spectral sequence
\[ \Tor_i^{S_N}(H_j(Y_{\Delta_N}(Q_N),\OL)_{\m})\implies
H_{i+j}(Y_0(Q_N),\OL)_{\m} .\]
Applying part (1) of Lemma \ref{lem:hom-basic-props}, we see that
$(H_N)_{\Delta_N} \cong H_1(Y_0(Q_N),\OL)_{\m}$. Then, by Lemma
\ref{lemma:matt} we see that $(H_N)_{\Delta_N} \cong H_1(Y,\OL)_{\mE} =
H$, as required.
\item[(c)]  $H_N$ is finite over $\OL$ and hence over $S_N$.
Proposition \ref{prop:balanced-homology} implies that
  $d_{S_N}(H_N)\geq 0$.
\end{itemize}
We may therefore apply Proposition \ref{prop:patchingimaginary} to deduce that
$H$ is free over $R$ and the theorem follows.
\end{proof}

\medskip

If $H_1(Y,\OL)_{\mE} \otimes \Q \ne 0$, then we may deduce that the multiplicity $\mu$ for
$H$ as a $\T_{\emptyset,\mE}$-module is one by multiplicity one for $\PGL(2)/F$. The proof also exhibits
$\T_{\emptyset,\mE}$ as a quotient of a power series ring in $q-1$ variables by $q$ elements. In particular, if 
$\dim(\T_{\emptyset,\mE}) = 0$, then $\T_{\emptyset,\mE}$ is a complete intersection. From these remarks we see that
Theorem~\ref{theorem:RequalsT} follows from Theorem~\ref{thm:main-thm-imaginary}.

\subsection{The distinction between \texorpdfstring{$\GL$}{GL} and \texorpdfstring{$\PGL$}{PGL}}

The reader may wonder why, when considering Galois
representations over imaginary quadratic fields,
 we consider the group $\G = \PGL$ rather than $\GL$. When $F = \Q$ or an imaginary quadratic field, the associated locally symmetric spaces are very similar (the
same up to components), and working with $\PGL$ has the disadvantage of forcing the determinant to be cyclotomic
rather than cyclotomic up to finite twist.  The main reason we use $\PGL$ is related to an 
issue which arises (and was pointed out to us by the referee) when
the class number of $F$ is divisible by $p$.
 Suppose that $\rhobar$ is a modular
 representation of level one, and suppose that the minimal fixed determinant deformation ring is $\OL$. Then, if the class group
 of $\OL_F$ is $\Z/p\Z$, the minimal Hecke ring $\T_{\m}$ is expected to be of the form $\OL[\Z/p\Z]$ rather than $\OL$, 
 and the map
 $\Rmin \rightarrow \T_{\m}$ will not be surjective. The issue is that the Hecke algebra even at minimal level sees the
 twists of the corresponding automorphic form by characters of the class group. If these characters have
 $p$-power order, they contribute to the localization of $\T$ at any maximal ideal $\m$. 
 This is analogous to what might happen classically if one has a representation $\rhobar$ over $\Q$ of tame level $N$;
 one has to be careful in choosing a minimal level, since the Hecke algebra of $X_1(N)$ will contain spurious
 twists if $N-1$ is divisible by $p$. The latter issue is easily resolved by a careful choice of level structure at $N$,
 namely, replacing $X_1(N)$ by $X_H(N)$ which is the quotient of $X_1(N)$ by the $p$-Sylow subgroup 
 of the group $(\Z/N\Z)^{\times}$ of diamond operators. 
 However, it is not possible to avoid the \emph{class group} in this way by choosing appropriate level structure, because
 the level structure  only sees ramification.  One fix is to work with $\PGL$, but there is another fix for imaginary quadratic fields $F$ which we sketch now.
 The natural approach is to replace the spaces $Y$, $Y_0(Q)$ and $Y_1(Q)$ by their quotients by the group
 $\Cl_p(\OL_F):=\Cl(\OL_F) \otimes \Z_p$. For example, the natural level structure at $Y$ admits a ring of diamond operators which act via
 an extension of $\Cl(\OL_F)$ by a group of order prime to $p$, and hence there is a canonical splitting and thus a canonical
 quotient $Y/\Cl_p(\OL_F)$ which gives the ``correct'' space. Note that, for $p$ odd, the group $\Cl_p(\OL_F)$ acts
 freely on the components, so this quotient is given explicitly by a subset of the connected components of $Y$.
 In the example above, the natural ring of Hecke operators $\T_{\m}$
 acting on $Y$ (now generated by $T_{\alpha}$ such that the image of $(\alpha)$ in $\Cl(\OL_F)$ has order prime to $p$)
 will be isomorphic to $\OL$.
 This construction, however, is not as canonical as one would like. For example,
 the ring of diamond operators on $Y_1(Q)$ naturally
 acts through a group whose $p$-Sylow subgroup is $\RCl_p(Q) = \RCl(Q) \otimes \Z_p$, the ($p$-part of the) ray
 class group of conductor $Q$. This group surjects onto $\Cl_p(\OL_F)$, but there is no natural section. It seems that
 the Taylor--Wiles method still applies as long as one restricts the set of Taylor--Wiles primes to $x \in Q$ such that
 the map
 $$\RCl_p(Q) \rightarrow \Cl_p(\OL_F)$$
 splits. This imposes a further Chebotarev condition on the Taylor--Wiles primes $x \in Q$ which corresponds to 
 $x$ splitting completely
 in a metabelian extension of $F$. Explicitly, if $\a \in \Cl_p(\OL_F)$ has $p$-power order $h$, let
 $\a^h = (\alpha)$. The necessary condition on $x$ is that (assuming $p$ is prime to the order of the unit group
 of $F$) that
 $$\alpha^{\frac{N(x) - 1}{p}} \equiv 1 \mod x,$$
 or equivalently that $x$ splits completely in $F(\alpha^{1/p},\zeta_p)$. 
 We ultimately decided, however, to impose the simplifying assumption that $\det(\rho)$ is cyclotomic,
 in part because  the main example of interest concerns elliptic curves over $F$ which naturally have cyclotomic
determinant.

\begin{remark} \emph{
A different approach to modularity lifting for $\GL$ is to allow the determinant to vary,
specifically, to fix the determinant only up to a character which is ramified only at the
Taylor--Wiles primes in~$Q$ (and so not at $v|p$). This is possibly the most general
way to proceed, although it requires working with $\ell_0 > 0$ even for $\GL(2)/F$ for totally
real fields of degree $[F:\Q] >1$.  We give some indication of this method
by considering the case of $\GL(1)$ in section~\ref{section:GL1}. The general
case of $\GL(n)$ is then a fibre product of this argument with the fixed determinant
arguments for $\PGL(n)$.
}
\end{remark}

\begin{remark}
\emph{ Our methods may easily be modified to prove an $R^{\min}  = \T_{\m}$
theorem for ordinary representations in weights other than weight zero (given the
appropriate modification of Conjecture~\ref{conj:A}). In weights which are not invariant
under the Cartan involution (complex  conjugation), one knows \emph{a priori}
for non-Eisenstein ideals $\m$ that $\T_{\m}$ is finite. Note that in this case it is sometimes
possible to prove unconditionally that $R^{\min}[1/p] = 0$, see Theorem~1.4 of~\cite{CEven}.
}
\end{remark}

\begin{remark}
\emph{
One technical tool that is conspicuously absent when $l_0 = 1$  is the technique of solvable
base change. When proving modularity results for $\GL(2)$ over totally real fields, for example,
one may pass to a finite solvable extension to avoid various technical issues, such as level
lowering (see~\cite{SW}). However, if $F$ is an imaginary quadratic field, then every non-trivial
extension $H/F$ has at least two pairs of complex places, and the corresponding invariant
$l_0 =   \rank(G) - \rank(K)$ for $\PGL(2)/H$ is at least $2$ (more precisely, it is equal to the number
of complex places of $H$). This means that when $l_0 = 1$, our techniques are mostly confined
to the approach used originally by Wiles, Taylor--Wiles, and Diamond
(\cite{W,TW,DiamondMult}).
}
\end{remark}

\begin{remark}
\emph{
Our techniques also apply to some other situations in which $l_0 = 1$ (the {\bf Betti}\rm \ case).
One may, for example, consider $2$-dimensional representations over a field $F$ with
one complex place. If $[F:\Q]$ is even, there exists an inner form for $\GL(2)/F$ which
is compact at all real places of $F$, and the corresponding arithmetic quotient is
a finite volume arithmetic hyperbolic manifold which is compact if $[F:\Q] > 2$.
(If $[F:\Q]$ is odd, one would have to require that $\rhobar$ be ramified with
semi-stable reduction at at least one prime $\lambda \nmid p$.)
Nonetheless, we obtain minimal lifting theorems in these cases,
modulo an analogue of
Conjecture~\ref{conj:A}. Similarly, our methods immediately produce minimal
lifting theorems for $\GL(3)/\Q$, modulo an appropriate version of Conjecture~\ref{conj:A}.
Similarly, our methods should also apply to other situations in which $\pi_{\infty}$ is
a holomorphic limit of discrete series (the {\bf Coherent}\rm \ case). 
One case to consider would be odd ordinary irreducible Galois representations
$\rho: G_{F} \rightarrow \GL_2(\Qbar_p)$ of a totally real field $F$ which
conjecturally arise from Hilbert modular forms exactly one of whose weights is one.
Other examples of particular interest 
include the case  in which $\rho: G_{\Q} \rightarrow \GSp_4(\Q_p)$ is the Galois
representation associated to an abelian surface $A/\Q$, or 
$\rho: G_{E} \rightarrow \GL_3(\Q_p)$ is the Galois representation associated
to a Picard Curve (see the appendix to~\cite{Blasius}). We hope to return to these examples
in future work.
}
\end{remark}

\part{\texorpdfstring{$l_0$}{l0} arbitrary.}
\label{part:2}

In this second part of the paper, our main result is a conditional
modularity lifting theorem for $n$-dimensional $p$-adic
representations of the Galois group of an arbitrary number
field. In this generality, we are forced to work in a situation
where the automorphic forms in question occur in a range of
cohomological degrees of arbitrary length $l_0$. We could have
presented our arguments in both the coherent cohomology setting and
the Betti cohomology setting, but for concreteness, we have decided to
treat only the latter case in detail.

 We now state our main (conditional) modularity lifting theorem; it
 will be used in Section~\ref{sec:proof-of-ST} to prove
 Theorem~\ref{theorem:ST}.
 Let $\OL$ denote the ring of integers in a finite extension of $\Q_p$,
let $\varpi$ be a uniformizer of $\OL$, and let $\OL/\varpi = k$ be the residue field.
Recall that a representation $\Gal(\C/\R)  \rightarrow \GL_n(\OL)$ is \emph{odd} if the image of $c$ has trace in $\{-1,0,1\}$,
and let $\eps$ denote the cyclotomic character.

\begin{theorem} \label{theorem:modularity}
Assume Conjecture~\ref{conj:AA}. Let $F/\Q$ be an arbitrary number field, and $n$ a positive integer.
Let $p > n$  be unramified in $F$. Let
$$r: G_{F} \rightarrow \GL_n(\OL)$$
be a continuous  Galois representation unramified outside a finite set of primes. 
Denote   the mod-$\varpi$ reduction of $r$ by
 $\rbar: G_{F} \rightarrow \GL_n(k)$.
Suppose that
\begin{enumerate}
\item If $v|p$, the representation $r |D_v$ is crystalline.
\item If $v|p$, then $ \mathrm{gr}^i(r \otimes_{\Z_p} B_{\mathrm{DR}})^{D_v} = 0$ unless
$i \in \{0,1,\ldots, n-1\}$, in which case it is free of rank 1 over
$\CO \otimes_{\Z_p} F_v$. 
  \item The restriction of $\rbar$ to
$\displaystyle{F \kern-0.05em{\left(\zeta_p \right)}}$
 is absolutely irreducible, and the field $F(\ad^0(\rbar))$ does not contain $F(\zeta_p)$.
 \item \label{big} In the terminology
 of~\cite{CHT}, Definition~2.5.1,  $\rbar$ is \big.
 \item \label{odd} If $v|\infty$ is any real place of $F$, then $r|G_{F_v}$ is odd.
 \item \label{taylor} If $r$ is ramified at a prime~$x$, then
 $r|I_x$ is unipotent. Moreover, if, furthermore,  $\rbar$ is unramified at~$x$, 
 then $N(x) \equiv 1 \mod p$.
 \item The determinant of $r$ is  $\eps^{n(n-1)/2}$.
\item \label{conditionserre} Either:
\begin{enumerate}
\item \label{charzero} There exists a cuspidal
automorphic representation $\pi_0$ of
$\GL_n(\A_F)$ such that: $\pi_{0,v}$ has trivial infinitesimal
character for all $v|\infty$, good reduction at all $v|p$, and the $p$-adic Galois representation
$r_p(\pi)$ both satisfies condition~\ref{taylor} and the identity
$\rbar_p(\pi) = \rbar$.
\item  \label{charp} $\rbar$ is Serre modular of minimal level $N(\rbar)$, and $r$ is ramified only at primes
which ramify in $\rbar$.
\end{enumerate}
\end{enumerate}
Then $r$ is modular, that is, there exists a regular algebraic cusp form $\pi$ for $\GL_n(\A_F)$
with trivial infinitesimal character 
such that $L(r,s) = L(\pi,s)$.
\end{theorem}

This theorem will follow immediately from Theorem~\ref{thm:mod-lifting}, proved below.
As in Theorem~\ref{theorem:modularityimaginary}, condition~\ref{charp} is only a statement about the
existence of a mod-$p$ cohomology class of level $N(\rbar)$, not the existence of a
characteristic zero lift; this condition is the natural  generalization of Serre's
conjecture.   On the other hand, the usual strategy for proving potential modularity usually
proceeds by producing characteristic zero lifts which are not minimal, and thus
condition~\ref{charzero} will be useful for applications.
If conditions $1$, $2$, and $3$  are satisfied, then conditions~\ref{odd} and~\ref{taylor} are satisfied after a solvable extension
which is unramified at $p$.
Moreover, \emph{if} $\rbar$ admits an automorphic lift  with trivial infinitesimal character
and good reduction at $p$, then condition~\ref{charzero} is
also satisfied after a solvable extension which is unramified at $p$.
Condition~\ref{charp}, however, is not obviously preserved under cyclic base change.

\medskip

Note that it will be obvious to the expert that our methods will allow for  (conditional) generalizations of these theorems
to other contexts (for example, varying the weight) but we have contented ourselves with the simplest possible statements necessary to
deduce Theorem~\ref{theorem:ST}. We caution, however, that several techniques are
\emph{not} available in this case, in particular, the lifting techniques of Ramakrishna and Khare--Wintenberger require
that $l_0 = 0$.

\section{Some Commutative Algebra II}
\label{section:commutetwo}

The general difficulty 
in proving that $R_{\infty} = \T_{\infty}$
is to show that there are  \emph{enough}
modular Galois representations. 
If the cohomology we are interested in occurs in a range of degrees of length $l_0$,
then we would like to show that in at least one of these degrees that
the associated  modules $H_N$ (which are both Hecke modules and
modules for the group rings $S_N:=\CO[(\Z/p^N\Z)^q]$)
compile, in a Taylor--Wiles patching process, to form a module of
codimension $l_0$ over the completed group ring $S_\infty:=
\CO[[(\Z_p)^q]]$. The problem then becomes to find a find a suitable
notion of ``codimension $l_0$'' for modules over a local ring that
\begin{enumerate}
\item is well behaved for non-reduced quotients of power series rings
over $\CO$ (like $S_N$),
\item can be established for the spaces $H_N$ in question,
\item compiles well in a Taylor--Wiles system.
\end{enumerate}

It turns out to be more effective to patch together a series of complexes $D_N$ of length $l_0$ whose cohomology
computes the cohomology of $\Gamma_1(Q_N)$ localized at $\m$. The limit of these patched complexes will
then turn out to be a length $l_0$ resolution of an associated patched module.

\medskip

It will be useful to prove the following lemmas.

\begin{lemma} Let $S$ \label{lemma:bound}
be a Noetherian local ring. 
If $N$ is an $S$-module with depth $n$, and $0 \ne M \subseteq N$, then $\dim(M) \ge n$.
\end{lemma}

\begin{proof}
Let $\p$ be an associated prime of $M$ (and hence of $N$). Then $\p$ is the annihilator of some $0 \ne m \in M$,
and it suffices to prove the result for $M$ replaced by $mS \subset M$. On the other hand, for a Noetherian local ring,
one has the inequality (see~\cite{matsumura}, Theorem 17.2)
$$n = \depth(N) \le \min_{\Ass(N)} \dim S/\p \le \dim(M).$$
\end{proof}

We deduce from this the following:

\begin{lemma} \label{lemma:deduce}
Let $l_0\geq 0$ be an integer and let $S$ be a 
Noetherian regular local ring of dimension $n \ge l_0$.
Let $P$ be a perfect complex of $S$-modules which is concentrated in degrees $0, \ldots, l_0$.
Then  $\codim(H^*(P)) \le l_0$, and moreover, if equality occurs, then:
\begin{enumerate}
\item $P$ is a projective
resolution of
$H^{l_0}(P)$,
\item $H^{l_0}(P)$ has depth $n - l_0$ and has projective dimension $l_0$.
\end{enumerate}
\end{lemma}

\begin{proof}
Let $\delta^i : P^i \to P^{i+1}$ denote the differential and let $m \le l_0$ denote the smallest integer such that $H^m(P) \ne 0$.
Consider the complex:
$$P^0 \rightarrow P^1 \rightarrow \ldots \rightarrow P^m.$$
By assumption, this complex is exact until the final term, and thus it
is a projective resolution
of the $S$-module $K^m := P^m/\im(\delta^{m-1})$. It follows that the projective dimension
of $K^m$ is $\le m$.
On the other hand, we see that
$$H^m(P) = \ker(\delta^m)/\im(\delta^{m-1}) \subseteq K^m,$$
and thus
$$\codim(H^m(P)) = n - \dim(H^m(P))  \le n - \depth(K^m) = \projdim(K^m) \le m,$$
where the central inequality is Lemma~\ref{lemma:bound}, and the second equality is the
Auslander-Buchsbaum formula.

Suppose that $\codim(H^*(P)) \ge l_0$. Then it follows from the argument above that the smallest
$m$ for which $H^m(P)$ is non-zero is $m = l_0$, 
that $\codim(H^{l_0}(P)) = l_0$,  that $P$ is a resolution of $H^{l_0}(P)$, and that
$\projdim(H^{l_0}(P)) = l_0$, completing the argument.
\end{proof}

\subsection{Patching}
\label{sec:patching}

We establish in this section an abstract Taylor--Wiles style patching
result which may be viewed as an analogue of Theorem 2.1
of~\cite{DiamondMult} and Proposition~\ref{prop:patchingimaginary}, but also including refinements
due to Kisin.

\begin{theorem}
\label{prop:patching}
Let $q$ and $j$ be non-negative integers with $q+j\geq l_0$, and let
$S_\infty=\CO[[(\Z_p)^q]]$. For each integer $N\geq 0$, let
$S_N:=\CO[\Delta_N]$ with $\Delta_N:=(\Z/p^N\Z)^q$. For each $M\geq N
\geq 0$ and each ideal $I$ of $\CO$, we regard $S_N/I$ (and in
particular, $\CO/I = S_0/I$) as a quotient of $S_M$ via the quotient map
$\Delta_M\onto \Delta_N$ and reduction modulo $I$. 
\begin{enumerate}[\ \ (1)]
\item\label{Rinfty}  Let  $R_\infty$ be an object of $\CC_{\CO}$ of Krull dimension $1+j+q-l_0$. 
\item Let $R$ be an object of $\CC_{\CO}$, and let $H$ be an $R$-module.
\item\label{complex-T} Let $T$ be a complex of finite-dimensional $k$-vector spaces
  concentrated in degrees $0,\dots,l_0$ together with a differential
  $d = 0$ and an isomorphism $H^{l_0}(T) \iso H/\varpi$ of $k$-modules. 
\end{enumerate} 
Let $\CO^\square = \CO[[z_1,\dots,z_j]]$ and for each $\CO$-module or $\CO$-algebra
$M$, we let $M^\square:=M\otimes_{\CO}\CO^\square$. For any
$\CO$-algebra $A$, we regard $A$ as a quotient of $A^{\square}$ via
the map sending each $z_i$ to $0$. 

Suppose that, for
each integer $N \geq 1$, $D_N$ is a perfect complex of $S_N/\varpi^N$-modules
with the following properties:
\begin{enumerate}[\ \ (a)]
\item \label{bounded} There is an isomorphism $D_N \otimes_{S_N} S_N/\m_{S_N} \simeq T$.
\item \label{galois-action} For each $M\geq N\geq 0$ with $M\geq 1$
  and each $n\geq 1$, there is an action of $R_\infty$ on the on the cohomology of the complex $D_M^\square \otimes_{S_M}
 S_N/\varpi^n$  that commutes with that of $S_M^{\square}$.
If, in addition, $N \geq N' \geq 0$ and $n\geq n'\geq 1$, then the natural map
$H^*(D_M^{\square}\otimes_{S_M}S_N/\varpi^n)\to
H^*(D_M^{\square}\otimes_{S_M}S_{N'}/\varpi^{n'})$ is compatible with the $R_\infty$-actions.
  \item \label{top-deg} For each $N\geq 1$, there is a surjective map $\phi_N: R_{\infty} \rightarrow R$, 
 and for each $n\geq 1$ we are given an isomorphism
 $$H^{l_0}(D_N^{\square} \otimes_{S_N^{\square}} \OL/\varpi^n) =
 H^{l_0}(D_N \otimes_{S_N} \OL/\varpi^n)
   \simeq H/\varpi^n$$
    of $R_{\infty}$-modules where $R_{\infty}$ acts on $H/\varpi^n$
    via $\phi_N$. Moreover, these isomorphisms are compatible for fixed $N$ and
    varying $n$.
\item \label{image-S} For $M,N$ and $n$ as above, the image of $S_M^\square$ in $\End_\CO(H^*(D_M^{\square} \otimes_{S_M} S_N/\varpi^n))$ is contained in the image
   of $R_\infty$ and moreover, the image of the augmentation ideal of
   $S_N^{\square}$ (that is, the kernel of $S_N^{\square}\to \CO$) is contained in the image of $\ker(\phi_N)$.

\end{enumerate}
Let $\a \subset S_\infty^\square$ denote the kernel of the map
$S_\infty^\square\to\CO$ sending each element of $(\Z_p)^q$ to $1$ and
each $z_i$ to $0$. Then the following holds: there is a perfect complex $P_\infty$ of
finitely generated $S_\infty$-modules concentrated in degrees
$0,\dots,l_0$ such that
\begin{enumerate}[\ \ (i)]
\item\label{proj-res} The complex $P_\infty^\square$ is a projective resolution, of minimal length, of its top
  degree cohomology $H^{l_0}(P_\infty^\square)$.
\item There is an action of $R_\infty\widehat\otimes_{\CO}S_\infty^\square$ on
  $H^{l_0}(P_\infty^\square)$ extending the action of $S_\infty^\square$ and such that
  $H^{l_0}(P_\infty^\square)$ is a finite $R_\infty$-module.
\item The $R_\infty$-depth of $H^{l_0}(P_\infty)$ is equal to
  $1+j+q-l_0(=\dim R_\infty$).
\item\label{top-deg-patchted} There is a surjection $\phi_\infty : R_\infty\onto R$ and an
  isomorphism $\psi_\infty : H^{l_0}(P_\infty^\square)/\a \iso H$ of
  $R_\infty$-modules where $R_\infty$ acts on $H$ via
  $\phi_\infty$. Moreover, the image of $\a$ in
  $\End(H^{l_0}(P_\infty^\square))$ is contained in that of $\ker (\phi_\infty)$.
\end{enumerate}
\end{theorem}

\begin{proof}
  For each $N\geq 1$, let
  $\a_N$ denote the kernel of the natural
  surjection $S_\infty \onto S_N$ and let $\b_N$ denote the open ideal
  of $S_\infty^\square$ generated by $\varpi^N$, $\a_N$ and $(z_1^N,\dots,z_j^N)$. 
   Choose a
  sequence of open ideals $(\d_N)_{N\ge 1}$ of $R$ such that
  \begin{itemize}
  \item $\d_N\supset \d_{N+1}$ for all $N\ge 1$;
  \item $\cap_{N\ge 1}\d_N = (0)$;
  \item $\varpi^NR \subset \d_N \subset \varpi^NR+\Ann_{R}(H)$ for all $N$.
  \end{itemize}
(As in the proof of Theorem~\ref{prop:patchingimaginary}, one can take $\d_N$ to be the ideal generated by
$\varpi^N$ and $\Ann_R(H)^N$.)

Define a \emph{patching datum of level $N$} to be a $3$-tuple $(\phi,\psi,P)$
where
\begin{itemize}
\item $\phi : R_\infty \onto R/\d_N$ is a surjection in $\CC_\CO$;
\item $P$ is a perfect complex
of  $S_\infty/(\a_N+\varpi^N)$-modules such that
$P \otimes S_{\infty}/\m_{S_{\infty}} \simeq T$;
\item For each $N\geq N'\geq 0$, each $N\geq n'\geq 1$ and each ideal $I$ of $\CO^\square$
 with  $(z_1^N,\dots,z_j^N)\subset I \subset
 (z_1,\dots,z_j)$, the cohomology groups $H^i(P^{\square}
 \otimes_{S_\infty^{\square}} S_{N'}^{\square}/(I+\varpi^{n'}))$
 carry an action of
 $R_\infty$ that commutes with the action of
     $S_{\infty}^\square$
     and these $R_\infty$-actions are compatible for varying $N'$,
      $n'$ and $I$;
\item $\psi : H^{l_0}(P^{\square} \otimes_{S_\infty^\square} S_\infty^\square/(\a+\varpi^N))  \iso H/\varpi^NH$ is an isomorphism of $R_\infty$
  modules (where $R_\infty$ acts on $H/\varpi^N H$ via $\phi$). (Note
  that $\psi$ then gives rise to an isomorphism of $R_\infty$-modules
  between $H^{l_0}(P^{\square} \otimes_{S_\infty^\square} S_\infty^\square/(\a+\varpi^{n'}))$ and
  $H/\varpi^{n'}H$ for each $N\geq n' \geq 1$.)
\end{itemize}
We say that two such $3$-tuples $(\phi,\psi,P)$ and
$(\phi',\psi',P')$ are isomorphic if 
$\phi = \phi'$ and
there is an isomorphism of complexes $P \iso P'$ of $S_\infty$-modules
inducing isomorphisms of
$R_\infty\widehat{\otimes}_{\CO}S_\infty$-modules on cohomology which
are compatible with $\psi$ and $\psi'$ in degree $l_0$.
We note that, up to isomorphism, there are finitely many patching
data of level $N$. (This follows from the fact that $R_\infty$ and
$S_\infty$ are topologically finitely generated, and that $T$ is finite.) If $D$ is a patching
datum of level $N$ and $1\le N' \le N$, then $D$ gives rise to a
patching datum of level $N'$ in an obvious fashion. We denote this
datum by $D \bmod N'$.

For each pair of integers $(M,N)$ with $M\ge N\ge 1$, we define a
patching datum $D_{M,N}$ of level $N$ as follows: the statement of the
proposition gives a homomorphism $\phi_M:R_\infty \onto R$ and an
$S_M/\varpi^M$-complex $D_M$. We take
\begin{itemize}
\item $\phi$ to be the composition $R_\infty\stackrel{\phi_M}{\onto} R \onto R/\d_N$;
\item $P$ to be $D_M \otimes_{S_\infty}  S_\infty/(\a_N+\varpi^N) = D_M\otimes_{S_M}S_N/\varpi^N$;
\item $\psi : H^{l^0}(P^{\square} \otimes_{S_\infty^\square} S_\infty^\square/ (\a+\varpi^N))=H^{l_0}(D_M^{\square}
  \otimes_{S_M^{\square}} \CO/\varpi^N) \iso
  H/\varpi^N$ to be the given isomorphism.
\end{itemize}
To see that the third condition in the definition of a patching datum
is satisfied, let $I$ be an ideal of $\CO^\square$ with  $(z_1^N,\dots,z_j^N)\subset I \subset
 (z_1,\dots,z_j)$, and let $1\leq n' \leq N$, $0\leq N' \leq N$. Then
 we have
\[ H^i(P^{\square} \otimes_{S_\infty^\square}
S_{N'}^{\square}/(I+\varpi^{n'})) = H^i(D^{\square}_M
\otimes_{S_M}
S_{N'}/\varpi^{n'})\otimes_{\CO^\square}\CO^{\square}/I, \]
and hence,  by assumption
\eqref{galois-action}, this space carries an $R_\infty$-action that
commutes with the $S_\infty^\square$-action and is compatible for
varying $I$, $N'$ and $n'$. Thus, $D_{M,N}$ is indeed a patching datum of
level $N$.

Since there are finitely many patching data of each level $N\ge 1$, up
to isomorphism, we can find a sequence of pairs $(M_i,N_i)_{i\geq 1}$
such that
\begin{itemize}
\item $M_i \ge N_i$, $M_{i+1}\ge M_i$, and $N_{i+1}\ge N_i$ for all $i$;
\item $D_{M_{i+1},N_{i+1}} \bmod N_i$ is isomorphic to
  $D_{M_i,N_i}$ for all $i\ge 1$.
\end{itemize}
For each $i\ge 1$, we write $D_{M_i,N_i}=(\phi_i,\psi_i,P_i)$ and
we fix an isomorphism  between $D_{M_{i+1},N_{i+1}}\bmod N_i $ and $D_{M_i,N_i}$. 
We define
\begin{itemize}
\item $\phi_\infty : R_\infty \onto R$ to be the inverse limit of the $\phi_i$;
\item $P_\infty := \varprojlim_i P_i$ where each transition map is the
  composite of $P_{i+1}\onto P_{i+1}/(\varpi^{N_i}+\a_{N_{i}})$ with the
  isomorphism $P_{i+1}/(\varpi^{N_i}+\a_{N_{i}}) \iso P_i$ coming from the
  chosen isomorphism between $D_{M_{i+1},N_{i+1}}\bmod N_i$ and
  $D_{M_i,N_i}$.

\item $\psi_\infty$ to be the isomorphism of $R_\infty$-modules
  $H^{l_0}(P_{\infty}^{\square})/\a = H^{l_0}(P_\infty^{\square}/\a) \iso H$
  (where $R_\infty$ acts on $H$ via $\phi_\infty$) arising from the isomorphisms $\psi_i$.
\end{itemize}

Then $P_\infty$ is a perfect complex of $S_\infty$-modules
concentrated in degrees 
$0,\ldots,l_0$ such that $H^*(P_\infty^\square)$ carries an action of
$R_\infty\widehat\otimes_\CO S_\infty^\square$ (extending the action of $S_\infty^\square$).
The image of $S_\infty^\square$ in $\End_\CO(H^*(P_{\infty}^\square))$ is contained in
the image of $R_\infty$. (Use assumption~\eqref{image-S}, and the
fact that the image of $R_\infty$ is closed in $\End_\CO(H^*(P_{\infty}^\square))$
(with its profinite topology).) It follows that $H^i(P_{\infty}^\square)$ is a finite
$R_\infty$-module for each $i$. Moreover, since $S_\infty^\square$ is formally smooth over $\CO$, we
can and do choose a homomorphism $\imath:S_\infty^\square \to R_\infty$ in $\CC_\CO$,
compatible with the actions of $S_\infty^\square$ and $R_\infty$ on $H^*(P_{\infty}^\square)$.

Since $\dim_{S_\infty^\square}(H^*(P_\infty^\square)) =
\dim_{R_\infty}(H^*(P_\infty^\square))$ and
$\dim R_\infty = \dim S_\infty^\square - l_0$, we deduce that
$H^*(P_{\infty}^\square)$ has codimension at \emph{least} $l_0$ as an $S_{\infty}^\square$-module.
By Lemma~\ref{lemma:deduce} (with $S = S_{\infty}^\square$ and $P = P_{\infty}^\square$)
we deduce that $P_{\infty}^\square$ is a resolution of minimal length of $H^{l_0}(P_{\infty}^\square)$
and that 
$$\depth_{S_\infty^\square}(H^{l_0}(P_{\infty}^\square)) =
\dim(S_\infty^\square)-l_0=1+j+q-l_0.$$
Finally note that the image of of $\a$ in
  $\End(H^{l_0}(P_\infty^\square))$ is contained in that of $\ker
  (\phi_\infty)$ by assumption~\eqref{image-S}.
\end{proof}

\begin{theorem}
  \label{thm:faithfulness-statements}
Keep the notation of the previous theorem and suppose in addition that
$R_\infty$ is $p$-torsion free.
\begin{enumerate}
\item If $R_\infty$ is formally smooth over $\OL$, so $R \simeq  \CO[[x_1,\dots,x_{q+j-l_0}]]$, then $H$ is a free $R$-module.
\item\label{irred-generic} If $R_{\infty}[1/p]$ is irreducible, then $H$ is a nearly faithful 
 $R$-module (in the terminology of~\cite{Taylor}). 
\item\label{all-components} More generally, $H$ is nearly faithful as an $R$-module
  providing that every irreducible component of
  $\Spec(R_{\infty}[1/p])$ is in the support of $H^{l_0}(P_{\infty})[1/p]$.
\end{enumerate}
\end{theorem}

\begin{proof}
 Suppose first of all that $R_{\infty} \simeq \OL[[x_1,\ldots,x_{j+q - l_0}]]$.
 Since
$\depth_{R_\infty}(H^{l_0}(P_{\infty}^\square))=\dim R_\infty$, applying the
Auslander--Buchsbaum formula again, we deduce that $H^{l_0}(P_{\infty}^\square)$ is free
over $R_\infty$. Let $\imath : S_\infty^\square \to R_\infty$ be as in
the proof of Theorem~\ref{prop:patching}. Then, the existence of the isomorphism $\psi_\infty: H^{l_0}(P_{\infty}^\square)/\a H^{l_0}(P_{\infty}^\square)
 \iso H$ tells us that $R_\infty/\imath(\a)R_\infty$ acts freely on
 $H$ and hence $\ker(\phi_\infty)\subset \imath(\a)R_\infty$. On the
 other hand, the freeness of $H^{l_0}(P_{\infty}^\square)$ over
 $R_\infty$ and the fact that the image of $\a$ in
 $\End(H^{l_0}(P_\infty^\square))$ is contained in that of
 $\ker(\phi_\infty)$ imply that $\imath(\a)R_\infty \subset
 \ker(\phi_\infty)$. We deduce that
 $\imath(\a)R_\infty=\ker(\phi_\infty)$ and that $R$ acts freely on
 $H$, as required.

For the remaining cases, to show that $H$ is nearly faithful as an
$R$-module, it suffices to show that $H^{l_0}(P_\infty^\square)$ is
nearly faithful as an $R_\infty$-module. To see this, suppose  $H^{l_0}(P_\infty^\square)$ is
nearly faithful as an $R_\infty$-module.  Then
$H^{l_0}(P_\infty^\square)/\a \cong H$ is nearly faithful as an
$R_\infty/\imath(\a)R_\infty$-module. The action of $R_\infty$ on this
module also factors through $R_\infty/\ker(\phi_\infty)=R$. Thus we
see that $R^{\red} \onto (R_\infty/\imath(\a)R_\infty)^{\red}$
 and it suffices to show this
map is an isomorphism. However, the fact that  $H^{l_0}(P_\infty^\square)$ is
nearly faithful as an $R_\infty$-module together with the fact that
the image of $\a$ in
  $\End(H^{l_0}(P_\infty^\square))$ is contained in that of $\ker
  (\phi_\infty)$ imply that 
\[ \imath(\a) + N \subset \ker(\phi_\infty) + N \]
where $N$ is the ideal of nilpotent elements in $R_\infty$. From this
it follows immediately that $(R_\infty/\imath(\a))^{\red} \onto
R^{\red}$, as required.

 Since $R_\infty$
is $p$-torsion free, all its minimal primes have characteristic
0. Thus $H^{l_0}(P_\infty^\square)$ is nearly faithful as an $R_\infty$-module if and
only if each irreducible component of $\Spec (R_\infty[1/p])$ lies in the support of
$H^{l_0}(P_\infty^\square)[1/p]$. Part \eqref{all-components} follows immediately.
For part \eqref{irred-generic}, note that since
$\depth_{R_\infty}(H^{l_0}(P_{\infty}^\square))=\dim R_\infty$, the support of
$H^{l_0}(P_\infty^\square)$ is a union of irreducible components of $\Spec
(R_\infty)$ of maximal dimension. Since $H^{l_0}(P_\infty^\square)\neq \{0\}$,
the result follows.
\end{proof}

\begin{remark}
  It follows from the proof of the previous theorem that for
  $H^{l_0}(P_\infty^\square)$ to be nearly faithful as an $R_\infty$-module,
  it is necessary that $R_\infty$ be equidimensional.
\end{remark}

To implement the level-changing techniques of \cite{Taylor}, we will
need the following refinement of Theorem \ref{prop:patching}. 

\begin{prop}
  \label{prop:simult-patching}
  Let $S_N$ and $\CO^\square$ be as in Theorem \ref{prop:patching}. Suppose
  we are given two sets of data
  $(R_\infty^i,R^i,H^i,T^i,(D_N^i)_{N\geq 1},(\phi^i_N)_{N\geq 1})_{i=1,2}$
  satisfying assumptions \eqref{Rinfty}--\eqref{complex-T} and
  \eqref{bounded}--\eqref{image-S} of Theorem \ref{prop:patching}, for
  $i=1,2$. Suppose also that we are given:
  \begin{itemize}
  \item  isomorphisms of $k$-algebras
\begin{eqnarray*} 
R_\infty^1/\varpi &\liso& R_\infty^2/\varpi \\
R^1/\varpi &\liso & R^2/\varpi,
\end{eqnarray*}

\item an isomorphism of $R^1/\varpi\iso R^2/\varpi$-modules
\[ H^1/\varpi  \liso H^2/\varpi, \]
\item an isomorphism, for each $M\geq N\geq 0$, of $S_N/\varpi$-modules
\[ H^{l_0}(D_M^1\otimes_{S_M}S_N/\varpi) \liso
H^{l_0}(D_M^2\otimes_{S_M}S_N/\varpi)  \]
which induces (after tensoring over $\CO$ with
$\CO^\square$) an isomorphism of
$R^1_\infty\otimes_{\CO} S_N^\square/\varpi\iso
  R^2_\infty\otimes_{\CO} S_N^\square/\varpi$-modules
\[ H^{l_0}((D_M^1)^\square\otimes_{S_M}S_N/\varpi) \liso
H^{l_0}((D_M^2)^\square\otimes_{S_M}S_N/\varpi)  \]
  \end{itemize}
 such that for each $M\geq 1$ the square
\[\begin{CD}
  H^{l_0}(D^1_M\otimes_{S_M}\CO/\varpi) @>>>
  H^{l_0}(D^2_M\otimes_{S_M}\CO/\varpi)\\
@VVV @VVV \\
H^1/\varpi @>>> H^2/\varpi
\end{CD}\]
commutes. Then we can find complexes $P_\infty^{i,\square}$ for $i=1,2$
satisfying conclusions~\eqref{proj-res}--\eqref{top-deg-patchted} of
Theorem \ref{prop:patching} as well as the following additional
property:
\begin{itemize}
\item There is an isomorphism of $R^1_\infty/\varpi \iso R^2_\infty/\varpi$-modules
\[ H^{l_0}(P^{1,\square}_\infty)/\varpi \liso
H^{l_0}(P^{2,\square}_\infty)/\varpi \]
such that the square
\[\begin{CD}
  H^{l_0}(P^{1,\square}_\infty)/(\a+\varpi) @>>>
  H^{l_0}(P^{2,\square}_\infty)/(\a+\varpi)\\
@VVV @VVV \\
H^1/\varpi @>>> H^2/\varpi
\end{CD}\]
commutes.
\end{itemize}
\end{prop}

\begin{proof}
  This can be proved in much the same way as Theorem
  \ref{prop:patching}; we omit the details.
\end{proof}

In
practice, we will apply Prop~\ref{prop:simult-patching} in a situation
where we are primarily interested in the collection of data indexed by
$i=1$. For the data indexed by $i=2$, the ring $R^2_\infty[1/p]$
will be irreducible and hence $H^2$ will be a nearly faithful
$R^2$-module by
Theorem~\ref{thm:faithfulness-statements}. Proposition~\ref{prop:simult-patching}
will then allow us to deduce that $H^1$ is a nearly faithful
$R^1$-module, following the arguments of~\cite{Taylor}.

\section{Existence of complexes}
\label{sec:complexes}

In this section, we prove the existence of the appropriate perfect
complexes of length $l_0$ which are required for patching.  In both
cases --- the Betti case or the coherent case --- the setting is similar: we have a covering space $X_{\Delta}(Q)
\rightarrow X_0(Q)$ of manifolds or an \'etale map $X_{\Delta}(Q)
\rightarrow X_0(Q)$ of schemes over $\CO$, each with covering group
$\Delta$, which is a finite abelian group of the form
$(\Z/p^N\Z)^{q}$. In both cases, the cohomology localized at a maximal
ideal $\m$ of the corresponding Hecke algebra $\T$ is assumed to
vanish outside a range of length $l_0$. The key point is thus to
construct complexes of the appropriate length whose size is bounded
(in the sense of condition~\eqref{bounded} of Theorem
\ref{prop:patching}) independently of $Q$, so that one may apply our
patching result.

\subsection{The Betti Case}
\label{sec:betti-case}
We put ourselves in the following somewhat general situation. Let
$X_0(Q)$ denote the locally symmetric space associated to a reductive
group $G$ over some number field $F$ and a compact open subgroup
$K_0(Q)$ of $G(\A_F^\infty)$. Similarly let $X_{\Delta}(Q)$ be
associated to a normal subgroup $K_{\Delta}(Q)\subset K_0(Q)$ with
quotient $\Delta$, a finite abelian group. In practice, $Q$ will
represent a set of Taylor--Wiles primes. See
Section~\ref{sec:hom-arithm-quot} for the specific compact open
subgroups that we will choose when $G = \PGL(n)$.
Let $R = \CO[\Delta]$ and let $\a$ denote the augmentation ideal of
$R$. Recall that by a \emph{perfect
  complex} of $R$-modules, we mean a bounded complex of finite free
$R$-modules. If $Y$ is a topological space, we let $C(Y)$ denote the
complex of $\CO$-valued singular chains on $Y$.

\begin{lemma}
  \label{lem:perfect-complex}
There exists a perfect complex of $R$-modules $C$ together with a
quasi-isomorphism 
\[C \to C(X_{\Delta}(Q))\]
of complexes of $R$-modules. In particular, we have isomorphisms:
\begin{align*}
  H_*(C\otimes_{\CO}\CO/\varpi^N) &= H_*(X_{\Delta}(Q),\CO/\varpi^N) \\
H_*(C \otimes_{R} R/(\a+\varpi^N))& =
  H_*(X_0(Q),\CO/\varpi^N)
\end{align*}
for all integers $N$.
\end{lemma}

\begin{proof}
By~\cite[\S II.5 Lemma 1]{Mumford}, there exists a perfect complex of
$R$-modules $C$ together with a quasi-isomorphism $C
\to C(X_{\Delta}(Q))$ of complexes of $R$-modules. (Mumford only guarantees that the final term of the complex is flat, but
since $R$ is local, this final term is also free.) Since $C$ and
$C(X_{\Delta}(Q))$ are bounded complexes of flat $R$-modules, we have
\[ H_*(C\otimes_{R} A) \liso H_*(C(X_{\Delta}(Q))\otimes_{R}A) \]
for every $R$-algebra $A$ by~\cite[\S II.5 Lemma 2]{Mumford}. Taking $A=R/\varpi^N$ gives the first
isomorphism. For the second isomorphism, the fact that
$X_{\Delta}(Q)\to X_0(Q)$ is a covering map with group $\Delta$
implies that
\[ C(X_{\Delta}(Q))\otimes_R R/\a \iso C(X_{0}(Q)). \]
Thus, taking $A=R/(\a+\varpi^N)$ gives the second isomorphism. 
\end{proof}

Let $\gamma \in G(\A^{\infty,p})$ with associated Hecke operator
$T_{\gamma}  : H_*(X_{\Delta}(Q),\CO/\varpi^N) \to
H_*(X_{\Delta}(Q),\CO/\varpi^N)$ for $N\leq \infty$ (where we define $\CO/\varpi^{\infty}:=\CO$).

\begin{lemma}
  \label{lem:hecke-op-complex}
Let $C$ be as in Lemma~\ref{lem:perfect-complex}. Then the action of $T_{\gamma}$ on homology may be
lifted to a map
$$T_{\gamma}: \Cc \rightarrow \Cc$$
of complexes of $R$-modules.
\end{lemma}

\begin{proof}
  Let $A=\CO/\varpi^N$ for $N\leq \infty$ and let $K\subset G(\A^\infty_F)$ be the compact
  open subgroup with $X_K = X_{\Delta}(Q)$ (this was called
  $K_{\Delta}(Q)$ above). Let $K' =
  \gamma K\gamma^{-1}\cap K$ and $K'' = K \cap \gamma^{-1} K
  \gamma$. Note that right multiplication by $\gamma$ gives an
  isomorphism: $\gamma : X_{K'} \to X_{K''}$.
The operator $T_{\gamma}$ is equal, up to an invertible scalar which
we may ignore, to
  the composition:
\[ H_i(X_K, A)\stackrel{\Tr}{\to} H_i(X_{K'},A) \stackrel{\gamma_*}{\ra}
H_ii(X_{K''},A)\to H_i(X_K,A),\]
where the first map is the transfer (or corestriction) map. 
Thus $T_{\gamma}$ is induced by the corresponding composition of morphisms of complexes
\[ C(X_K) \stackrel{\Tr}{\to} C(X_{K''}) \stackrel{\gamma_*}{\to} C(X_{K'}) \to C(X_K), \]
(after tensoring over $\CO$ with $A$). Denote this composition $\wt{T}_{\gamma}$. Restricting to
$C$ (by means of the quasi-isomorphism $C \to C(X_K)$ of
Lemma~\ref{lem:perfect-complex}), we obtain $\wt{T}_{\gamma}: C\to
C(X_K)$ which also gives rise to $T_\gamma$ on homology. We thus have a
diagram
\[ \begin{CD}
  C @>\wt{T}_{\gamma}>> C(X_K) \\
@. @AAA \\
@. C
\end{CD} \]
of complexes of $R$-modules with the vertical morphism being a
quasi-isomorphism. Since $C$ is perfect, the morphism
$\wt{T}_{\gamma}$ can be lifted to a morphism $T_{\gamma} : C\to C$
making the diagram commute. (See
\cite[\href{http://stacks.math.columbia.edu/tag/08FQ}{Tag
  08FQ}]{stacks-project} for example.)
\end{proof}

Let $\T$ denote a Hecke algebra generated over $\CO$ by a collection of
operators $T_\gamma$. Then, for any $T\in \T$, we we can express $T$
as a polynomial in the operators $T_{\gamma}$ and thus lift the action
of $T$ on homology to an endomorphism $T:C \to C$. 

\begin{lemma}  For $T\in \T$, let
  $\displaystyle{\Cc_{T}:=\lim_{\rightarrow} T^n \Cc}$. Then $\Cc_{T}$
  is a perfect complex of $R$-modules whose
homology is 
$$\displaystyle{\lim_{\rightarrow} T^n H_*(\Cc)}.$$
\end{lemma}

\begin{proof} Since $R$ is complete, the functor $M \rightarrow
  \displaystyle{\lim_{\rightarrow} T^n} M$ on finitely generated
  $R$-modules is exact and 
and $\displaystyle{\lim_{\rightarrow} T^n} M$ is in fact a direct summand of $M$ (the other factor being the submodule of $M$ on which $T$ is topologically nilpotent).
A direct summand of a projective module is projective, and the
equality of homology follows from the exactness of the functor. 
 \end{proof}

We now assume that $\Delta$ is of $p$-power order. Thus
$R=\CO[\Delta]$ is local and we let $\m_R$ be the maximal ideal of $R$.

\begin{lemma}[Nakayama's Lemma for perfect complexes]
\label{lem:nak-complex}
Let $T\in \T$ be as above. Then
there exists a perfect complex of $R$-modules $\Dc$ which is
quasi-isomorphic to $C_T$ and such that
$$\dim D_n/\m_R = \dim H_n(\Cc_{T} \otimes R/\m_R)=\dim
\displaystyle{\lim_{\rightarrow}} \ T^m H_n(X_{0}(Q),k)$$
for all $n$. Moreover,
the length of $\Dc$ is at most $l_0$, where $l_0$ is the range of cohomology groups such that 
$\displaystyle{\lim_{\rightarrow}  \ T^n H^*(\Cc)}$ is non-zero.
\end{lemma}

\begin{proof}
By Lemma~1 of Mumford (\cite{Mumford}, Ch.II.5) again, one may find a
perfect complex $\Kc$
quasi-isomorphic to $\Cc_{T}$ and such that $K$ is
bounded of length $l_0$. Assume that the
differential $d$ on $D$ is non-zero modulo $\m_R$ from
degree $n+1$ to $n$. Then by Nakayama's Lemma, there exists a direct
sum decomposition of perfect complexes of $R$-modules
$$\Kc \simeq \Kcp \oplus J, $$
 where $J_i$ is zero for $i\neq n+1,n$ and $d:J_{n+1}\to J_n$ is an
isomorphism of free rank 1 $R$-modules. Thus, $\Kcp$ is also a perfect
complex of $R$-modules which is quasi-isomorphic to $\Kc$. Replacing $\Kc$ by $\Kcp$ and using
induction, we eventually arrive at a complex $\Dc$ so that $d$ is zero modulo $\m_R$, from which the equality of dimensions
follows by Nakayama's Lemma.
\end{proof}

In practice, the Hecke algebra $\T$ will be of the form $\Tan[U_x:x\in
Q]$ where $\Tan$ is the subalgebra generated by good Hecke operators
away from $Q$ and $p$. We say that two maximal ideals of $\T$ give rise to the
same Galois representation if they contract to the same ideal of
$\Tan$. In practice, we will be interested in localizing the homology
groups $H_*(X_{\Delta}(Q),\CO/\varpi^N)$ at a particular maximal ideal
$\m$ of $\T$. The residue field of $\m$ will be equal to $k$.
In order to apply the above lemmas, we take the
Hecke operator $T$ to be
$$\prod_{x\in Q} P_{x} \circ
\prod_{i\in \Omega} P_i,$$ where $\Omega$, $P_x$ and
$P_i$ are chosen as follows: for each of the finitely many maximal ideals $\n$ of $\T$ which occurs in
$H_*(X_{\Delta}(Q),k)$,  choose an $\CO$-algebra
homomorphism $\phi_{\n} : \T \to \overline{k}$ with kernel $\n$.  
We let $\Omega$ index the
collection of such maximal ideals $\n$ of $\T$ which give rise to a Galois representation
distinct from $\m$. This is equivalent to $\phi_{\n}$ and $\phi_{\m}$
differing on $\Tan$.  Thus, for $i\in \Omega$ corresponding to a
maximal ideal $\n$, there
exists a good Hecke operator $T_i$ such
that $\phi_{\n}(T_i)\neq \phi_{\m}(T_i)$. Let $F_i(T)$ denote 
the minimal polynomial over~$E$ of the Teichmuller lift of~$\phi_{\n}(T_i)$ to~$\overline{E}$.
By Hensel's Lemma, the element~$\phi_{\m}(T_i) \in k$ is not a root modulo~$\varpi$ of~$F_i(T)$.
Hence~$P_i = F_i(T_i)$ is an element
of $\n$ but not of $\m$.
For the maximal ideals $\n$ with the same
Galois representation as $\m$, for all $x|Q$ by construction there
will be a projector $P_x$ which commutes with the action of the
diamond operators and cuts out the localization at $\m$.  (For
example, if $G = \GL(2)$, then $P_x$ can be taken to be $\lim (U_x -
\beta_x)^{n!}$ for $x|Q$, where $U_x - \alpha_x \in \m$ and
${\alpha_x,\beta_x}$ are arbitrary lifts of the (distinct) eigenvalues of
$\rbar(\Frob_x)$ to~$\OL$). In particular, we have
$$\lim_{\rightarrow} T^n H_*(\Gamma_{\Delta}(Q_N),\OL/\varpi^N)
= H_*(\Gamma_{\Delta}(Q_N),\OL/\varpi^N)_{\m}.$$
  Thus, letting $D$ be as in Lemma~\ref{lem:nak-complex} for this
  choice of $T$, we have that $D$ is a perfect complex of $R$-modules
  such that:
  \begin{itemize}
  \item $D_n \neq 0$ if and only if $H_n(X_0(Q),k)_{\m}\neq \{ 0\}$,
  \item $H_n(D\otimes_{\CO} \CO/\varpi^N) \cong
    H_n(X_{\Delta}(Q),\CO/\varpi^N)_{\m}$, for all $n, N$, and
  \item $H_n(D\otimes_R R/(\a+\varpi^N))\cong
    H_n(X_0(Q),\CO/\varpi^N)_{\m}$ for all $n, N$.
  \end{itemize}

\subsection{The Coherent Case} \label{section:nakajima} 
We now explain how to prove the existence of appropriate complexes in
the setting of coherent cohomology. The setting will be as
follows.
We
will have an \'{e}tale map $Y=X_{\Delta}(Q) \to X = X_0(Q)$ with
Galois group $\Delta$, a finite abelian $p$-group. The spaces $X$ and
$Y$ will be proper and smooth over $\Spec (\CO)$; they will arise as 
integral models of the Shimura varieties associated to some reductive
group $G$ over a number field $F$ and some compact open subgroups $K_{\Delta}(Q)\subset K_0(Q)
\subset G(\A^{\infty}_F)$. In the case that these Shimura varieties
are not compact, $X$ and $Y$ will be arithmetic toroidal
compactifications, as constructed in \cite{lan-thesis}.
We will be given an automorphic
vector bundle $\EE$ on $Y$ (in the case of a toroidal
compactification, this will either be a canonical or subcanonical
extension) which pulls back to a bundle also denoted
by $\EE$ on $X$. We will be interested in producing a perfect
complex of $R/\varpi^n$-modules computing
\[ H_i(Y,\EE\otimes \CO/\varpi^n)_{\m} \]
where $\m$ is a maximal ideal of the Hecke algebra generated by `good'
Hecke operators at the unramified primes together with certain 
operators at the primes in the set of auxiliary Taylor--Wiles primes
$Q$; here the homology group is defined as
\[ H_i(Y,\EE\otimes \CO/\varpi^n)
:= H^i(Y,\EE^*\otimes_{\CO_Y} \omega_Y
\otimes_{\CO} \CO/\varpi^n)^{\vee}\]
where $\omega_Y$ is the determinant of
$\Omega^1_{Y/\CO}$.
The reader may wonder why we introduce here the non-standard concept of ``coherent homology.''
The reason is a mixture of both the practical and the psychological. In both 
the {\bf Betti\rm} and {\bf Coherent\rm} case, the modules which patch are obtained by taking 
the Pontryagin duals of the non-zero cohomology
group in \emph{lowest} degree with coefficients in~$E/\OL$. In Betti cohomology, the groups~$H^{q_0}(X,\E/\OL)^{\vee}_{\m}$
may be identified with the non-zero homology group of highest degree with coefficients in~$\OL$, namely~$H_{q_0 + l_0}(X,\OL)_{\m}$.
This identification is essentially a consequence of Poincar\'{e} duality (the manifolds~$X$ have boundary, but the ideals~$\m$ are chosen
specifically so that the cohomology of the boundary becomes trivial after localization at~$\m$). In the coherent case, our use of the
terminology ``homology'' is thus to preserve the analogy, but also as a convenient shorthand for the necessary operation of taking
coefficients of our sheaves in~$\E/\OL$ and then taking Pontryagin duals. One could also make these analogies more precise by
comparing Poincar\'{e} duality to Serre and Verdier duality.

We begin with some commutative algebra.

\begin{lemma}  
\label{lemma:free}
Let $P$ be an $\OL$-module  such that $P$ is $\varpi$-torsion free. Then $P/\varpi^n$ is free over $\OL/\varpi^n$ for each $n$.
\end{lemma}

\begin{proof}
If $n = 1$, then $P/\varpi$ is a module over a field $k = \OL/\varpi$, and hence admits a basis 
$\{\overline{x}_{\alpha}\}$. Let $\{x_{\alpha}\}$ denote any lift of this generating set to $P$.
Let $Q \subset P$ denote the $\OL$-submodule generated by the $x_{\alpha}$.
We claim that $Q$ surjects onto $P/\varpi^n$ for all $n$, and moreover that
$P/\varpi^n$ is free on the images of the generators $x_{\alpha}$ of $Q$.
 We prove this by induction. It is true for $n = 1$ by construction.
Suppose that $Q \rightarrow P/\varpi^n$ is surjective. Let $x \in P$, and consider the image
of $x$ in $P/\varpi^{n+1}$. After subtracting a suitable element of $Q$, we may assume that the image
of $x$ in $P/\varpi^n$ is trivial. Hence  we may write $x = \varpi^n y$ for some $y \in P$. The image of $y$ in
$P/\varpi$ can be written as the image of an element of $Q$, and so $y = z + \varpi w$ for $z \in Q$ and $w \in P$.
It follows that $x = \varpi^n z \mod \varpi^{n+1}$, and thus the image of $x$
 in $P/\varpi^{n+1}$ is  contained in the image of $Q$. It follows that $Q \rightarrow P/\varpi^{n+1}$ is surjective.
Let us now show that that $P/\varpi^{n+1}$ is free over $\OL/\varpi^{n+1}$.
Assume otherwise. Then there exists a relation of the form
$$\sum r_{\alpha} x_{\alpha} \equiv 0 \mod \varpi^{n+1} P.$$
Reducing this equation modulo $\varpi$, we deduce by construction that $r_{\alpha}$ is divisible by $\varpi$ for all $\alpha$.
Yet, since $P$ is $\varpi$-torsion free, any equality $\varpi x = \varpi y$ in $P$ implies the equality $x = y$.
Hence if we write $r_{\alpha} = \varpi s_{\alpha}$,  we obtain a relation
$$\sum s_{\alpha} x_{\alpha} \equiv 0 \mod \varpi^{n} P.$$
By induction, we deduce that $s_{\alpha}$ is $0$ in $\OL/\varpi^{n}$, and hence $r_{\alpha}$ is trivial
in $\OL/\varpi^{n+1}$. In particular, $P/\varpi^{n+1}$ is freely generated by the images of the $x_{\alpha}$, completing
the induction.
\end{proof}

(As pointed out by the referee, Lemma~\ref{lemma:free} is also an immediate consequence of the fact that~$P/\varpi^m$ is flat over~$\OL/\varpi^m$
and so automatically free because~$\OL/\varpi^m$ is an Artinian local ring.)

In the following lemma, $\Delta$ may be any finite abelian group.

\begin{lemma} \label{lemma:cross} Let 
  $R = \OL[\Delta]$.  If $M$ is an $R$-module which is  free  over $\OL/\varpi^n$ and
$M/\varpi$ is free over $R/\varpi$, then $M$ is free over $R/\varpi^n$.
\end{lemma}

\begin{proof} Let $\{\overline{y}_{\alpha}\}$ be an $R/\varpi = k[\Delta]$ basis for $M/\varpi$. Since $R$
is a free $\OL$-module, we may choose a finite basis
 $\{z_i\}$ for $R$ over $\OL$. Note that
 $\{\overline{z}_i\}$ is a basis for $k[\Delta]$ over $k$, and $\overline{z}_{i} \cdot \overline{y}_{\alpha}$ is a basis
 for $M/\varpi$ over $k$. Lifting the elements $\overline{y}_{\alpha}$ to elements $y_{\alpha}$ of $M$, we see,
 as in the proof of the Lemma~\ref{lemma:free}, that $z_{i} \cdot y_{\alpha}$ is a free basis for $M$ as an
 $\OL/\varpi^n$ module. We claim that $y_{\alpha}$ is a free basis for $M$ as an $R/\varpi^n$-module.
 Assume otherwise, so that there is a relation
 $$\sum r_{\alpha} y_{\alpha} = 0$$
 with $r_{\alpha} \in R/\varpi^n$. We may uniquely write $r_{\alpha} = \sum s_{\alpha,i} z_i$ with $s_{\alpha,i}$
 in $\OL/\varpi^n$. We then deduce from the freeness of $M$ over $\OL/\varpi^n$ with $z_i y_{\alpha}$ as a basis
 that $s_{\alpha,i} = 0$ for all $\alpha$ and all $i$, and hence $r_{\alpha} = 0$.
  \end{proof}

  We now return to the situation described at the beginning of this
  section: $f: Y \to X$ is an \'{e}tale map of smooth proper
  $\CO$-schemes with Galois group $\Delta$ abelian of $p$-power order.
  Let $\Ff$ be a coherent sheaf of $\OL_X$-modules which is
  $\varpi$-torsion free.  Following Nakajima~\cite{Nakajima}, we take
  an affine covering of $X$ by affine schemes $\{U_{\alpha}\}$ (which
  are necessarily flat over $\Spec(\OL)$).  We thus obtain a \u{C}ech
  complex $D$ of $\CO$-flat $\OL[\Delta]$-modules computing
  $H^i(Y,f^*(\mathscr{F}))$. More precisely, the terms of $D$ are direct sums
  of modules of the form $N \otimes_A B$, where $\Spec(A)\subset X$ is
  an intersection of $U_{\alpha}$'s with preimage $\Spec(B) \subset Y$
  and $N = \Gamma(\Spec (A),\CF)$. 
 Moreover, the complex $D \otimes
  \CO/\varpi^n$ computes $H^i(Y,f^*(\mathscr{F}) \otimes
  \OL/\varpi^n)$ for every $n$.

\begin{lemma} Let $\Spec(B) \rightarrow \Spec(A)$ be a
finite Galois \'{e}tale morphism of flat $\OL$-algebras with Galois
group $\Delta$. Let $R$ denote the local ring $\CO[\Delta]$. Then, for any 
$A$-module $N$ which is flat over $\OL$,
$(N \otimes_A B)/\varpi^n$ is a free $R/\varpi^n$-module for each $n$.
\end{lemma}

\begin{proof} Let $M = N \otimes_A B$. Then
$$M/\varpi = (N \otimes_A B)/\varpi = N/\varpi \otimes_{A/\varpi} B/\varpi,$$
and thus $M/\varpi$ is a projective $R/\varpi = k[\Delta]$-module by  Lemma~1 of~\cite{Nakajima}, and is 
thus free. (A theorem of Kaplansky implies that a projective module $P$ over a local ring $R$
is always free (see~\cite{Kaplansky}).)
Note that $M$ is flat over $\CO$ since $N$ and $A$ are flat over $\CO$
and $B$ is flat over $A$. It follows from Lemma~\ref{lemma:free} that $M/\varpi^n$ is free over $\OL/\varpi^n$.
By Lemma~\ref{lemma:cross}, we deduce that $M/\varpi^n$ is free over $R/\varpi^n$.
\end{proof}

It follows that $D\otimes \CO/\varpi^n$ is a bounded complex of
projective $R/\varpi^n$-modules computing the cohomology groups
$H^i(Y,f^*(\mathscr{F})/\varpi^n)$. We apply this as follows. Let 
$\EE$ denote the automorphic vector bundle introduced above and take
\[ \mathscr{F}=\EE^* \otimes_{\CO_X} \omega_X .\]
Applying Lemma~1 of \cite{Mumford}, Ch.II.5 again, we
may replace $D\otimes \CO/ \varpi^n$ by a quasi-isomorphic complex of
$R/\varpi^n$-modules $C_n^{\vee}$ which is perfect. Then dualizing this latter
complex, we obtain a (chain) complex $C_n$ whose homology groups are 
\[ H_i(Y,\EE\otimes \CO/\varpi^n).\]
It remains to define an action of the Hecke algebra and cut out the localization at $\m$.

The Hecke algebra will be generated by operators $T_{\gamma}$ where
$\gamma \in G(\A^{\infty,p}_F)$. Let $L = K_{\Delta}(Q)\subset
G(\A^{\infty})$ be the compact open subgroup corresponding to
$Y=X_{\Delta}(Q)$. Let $L^\gamma = \gamma L \gamma^{-1}\cap L$. Then
by choosing suitable polyhedral cone decomposition data, there exists
an arithmetic toroidal compactification $Y^\gamma$, proper and smooth
over $\CO$, of the Shimura variety of level $L^{\gamma}$ together with
maps $\pi_1,\pi_2 : Y^\gamma \to Y$ where $\pi_1$ is associated to the
inclusion $L^\gamma \to L$ and $\pi_2$ is associated to right
multiplication by $\gamma$ on complex points. (See \cite[\S
6.4.3]{lan-thesis}.) By \cite[Thm.\ 2.15(4)(c)]{lan-kuga} and the fact
that all automorphic vector bundles are constructed from the Hodge
bundle (see \cite[\S 4.2]{LanSuh}), there is an
isomorphism
\[ \phi : \pi_2^* \EE \liso \pi_1^* \EE \]
of sheaves on $Y^\gamma$. To define the Hecke operator $T_{\gamma}$, we follow
the approach of \cite[p.\ 256]{chai-faltings} which avoids having to
define the trace of $\pi_1$ on cohomology. 

Let $A=\CO/\varpi^n$ and let $\Ff = \EE^* \otimes
\omega_X$ be as above. If $M$ is an $\CO$-module or a sheaf of
$\CO$-modules on some space, we denote by $M_A$ the tensor product
$M\otimes_{\CO}A$. By Verdier duality, the group $H^i(Y,(f^*\Ff)_A)$
is Pontryagin dual to \[ H^{d-i}(Y,\mathcal Hom_{\CO_{Y}}(f^*\Ff,\omega_{Y})_A)=H^{d-i}(Y,(f^*\EE)_A)\] where $d$
is the dimension of $Y$. Thus, to
define an operator $T_\gamma$ on $H^i(Y,(f^*\Ff)_A)$, it
suffices to define a pairing
\[ H^i(Y,(f^*\Ff)_A)\otimes_A H^{d-i}(Y,(f^*\EE)_
A)\to A .\]
We define the pairing by sending $x \otimes y $ to 
\[ \Tr( \pi_1^*(x) \cup \phi \pi_2^*(y)) \]
where
\begin{itemize}
\item $\pi_1^*(x) \cup \phi \pi_2^*(y)$ denotes the image of
  the cup product of $\pi_1^*(x)$ and $\phi \pi_2^*(y)$ under the
  natural map
\[ H^d(Y^{\gamma}, (\EE^*\otimes \omega_{Y^\gamma} \otimes \EE)_A) \to H^d(Y^{\gamma}, (\omega_{Y^\gamma})_A),\]
and
\item $\Tr$ denotes the trace isomorphism
\[   H^d(Y^{\gamma}, (\omega_{Y^\gamma})_A) \to A.\]
\end{itemize}
This defines the action of $T_\gamma$ on cohomology. We now want to
lift the action of $T_\gamma$ to an endomorphism the complex $C_n$ introduced above.

Let $Y_A = Y\times_{\CO}A$ and let $\pi : Y_A \to \Spec A$ be the
structural morphism. Then, since $\GG:=(f^*\Ff)_A$ is a $\Delta$-equivariant sheaf
on $Y_A$, we may regard $R\pi_*(\GG)=R\Hom(\CO_{Y_A},\GG)$ as an object of the bounded
derived category of $R=A[\Delta]$-modules $D^{\flatt}(R)$. By Verdier
duality, we have
\[ R\Hom(R\Hom(\CO_{Y_A},\GG),A[0])  = R\Hom(\GG,\omega_{Y_A}[d]) \]
Thus, we have an equality (of $\Hom$'s in the category $D^\flatt(R)$):
\[  \Hom( R\Hom(\CO_{Y_A},\GG),R\Hom(\CO_{Y_A},\GG)) =
 \Hom(R\Hom(\CO_{Y_A},\GG)\stackrel{L}{\otimes} R\Hom(\GG,\omega_{Y_A}[d]),
 A[0]),\]
and we define an element of $\wt{T}_{\gamma}$ of the right hand side by composing
\begin{itemize}
\item the pullback under $\pi_1^*\otimes (\phi\circ \pi_2^*)$,
\[  R\Hom(\CO_{Y_A},\GG)\stackrel{L}{\otimes}
R\Hom(\GG,\omega_{Y_A}[d])\longrightarrow
R\Hom(\CO_{Y^{\gamma}_A},\pi_1^*\GG)\stackrel{L}{\otimes} R\Hom(\pi_1^*\GG,\omega_{Y^{\gamma}_A}[d])
\]
\item the composition morphism
\[ R\Hom(\CO_{Y^{\gamma}_A},\pi_1^*\GG) \stackrel{L}{\otimes} R\Hom(\pi_1^*\GG,\omega_{Y^{\gamma}_A}[d])
\to R\Hom(\CO_{Y^{\gamma}_A},\omega_{Y^{\gamma}_A}[d]), \]
and
\item the duality isomorphism
\[ R\Hom(\CO_{Y^{\gamma}_A},\omega_{Y^{\gamma}_A}[d]) \liso A[0].\]
\end{itemize}
Then $\wt{T}_{\gamma}$ induces the Hecke operator 
\[ T_{\gamma} : H^i(Y_A, \GG) \to H^i(Y_A,\GG) \]
defined above. 
If $C_n$ is the complex introduced above, then it comes
equipped with an isomorphism
\[ C_n^{\vee} \liso R\pi_*(\GG)  = R\Hom(\CO_{Y_A},\GG) \] 
in $D^{\flatt}(R)$ and we have a diagram:
\[ \begin{CD}
  C_n^{\vee} @>>> R\Hom(\CO_{Y_A},\GG) @>\wt{T}_{\gamma}>> R\Hom(\CO_{Y_A},\GG) \\
@. @. @AAA \\
@. @. C_n^{\vee}
\end{CD} \]
Since $C_n^{\vee}$ is perfect, we can apply \cite[\href{http://stacks.math.columbia.edu/tag/08FQ}{Tag
  08FQ}]{stacks-project} once again to lift $\wt{T}_{\gamma}$ to a
morphism of complexes $T_{\gamma} : C_n^{\vee} \to C_n^{\vee}$ that
induces the operator $T_{\gamma}$ on cohomology.

Now that we can lift a given Hecke operator $T$ to the complex
$C_n^{\vee}$ (and hence to its dual $C_n$), we can show, exactly
as in the previous section, that given a maximal ideal $\m$ of the Hecke
algebra, then for a judicious choice of Hecke
operator $T$, the complex 
\[ \varinjlim_m T^m C_n \]
is a perfect complex of $R$-modules with homology equal to 
\[ H_i(Y,\EE_A)_{\m}. \]

\section{Galois deformations}
\label{sec:galois-deformations}

\subsection{Outline of what remains to be done to prove Theorem~\ref{theorem:modularity}}
Given the patching result (Theorem~\ref{prop:patching}) and the existence of complexes
satisfying an appropriate boundedness condition (condition~\eqref{bounded} of \emph{ibid}.), to complete
the argument consists of the following steps. First, consider the \emph{minimal} case where there
are no primes $x \nmid p$ such that $\rho | D_x$ is ramified when $\rbar | D_x$ is unramified. In this case, it suffices to
 construct a sequence $Q_N$ of collections of $q$ primes
$N(x) \equiv 1 \mod p^N$ such that the corresponding Hecke rings $\T_{Q_N}$ are all
quotients of a fixed patched global deformation 
ring $R_{\infty}$ such that $R_{\infty}[1/p]$ is irreducible of the appropriate dimension.
The usual Taylor--Wiles--Kisin method (with some modifications due to Thorne) exactly produces the desired
sets $Q_N$, and the computation of $R_{\infty}$ (which is naturally a power series over $\Rloc$) 
is computed in the usual manner, except now its dimension is $l_0$ less than in the classical case.
If one \emph{assumes} vanishing of cohomology outside
the expected range and also \emph{assumes} that the Hecke action of cohomology in that range comes from Galois
representations with the expected properties, then one may construct a series of complexes
(as in the last section) which all have actions by  $R_{\infty}$, and then using 
Theorem~\ref{prop:patching} the desired conclusions follow as expected.

\medskip

This leaves the case when there exist primes such that~$\rho |I_x$ is unipotent, $N(x) \equiv 1 \mod p$,  and
yet $\rbar|D_x$ is unramified. Here one uses Taylor's trick~\cite{Taylor} to avoid Ihara's lemma.
 The key calculation
in Taylor's paper requires only that one has control over the depth of
the patched module on which 
the ring $R_\infty$ acts, as
well as the structure mod $\varpi$ of various local deformation rings. The required information concerning depth is
exactly what one deduces in the proof of Theorem~\ref{prop:patching}. The only difference in this setting is that
the relevant dimension of $R_{\infty}$ is $l_0$ less than the classical case, whereas the corresponding depth of 
$H^{l_0}(P_{\infty})$ is also exactly $l_0$ less than the classical case --- this means that the argument goes through as expected.
We begin, however, by explaining our  method in the case of one-dimensional representations,
in order to demonstrate the method.

\subsection{Modularity of one-dimensional representations}
\label{section:GL1}

In this section, 
we apply our method to one-dimensional representations,
that is, to the case when $G = \GL(1)/F$ for
an arbitrary number field~$F$. We will need to assume that $F$ does
not contain $\zeta_p$. 
The arguments here are (ultimately) somewhat circular, but they exhibit
all the various aspects of the general method. The invariant value of $\ell_0$ for
a field $F$ of signature $(r_1,r_2)$ will be $r_1 + r_2 - 1$.

\medskip

Up to twist, there is only one residual Galois representation, namely, the trivial
representation
$$\rbar: G_F \rightarrow \F^{\times}_p.$$
Minimal deformations of $\rbar$ consist of
(everywhere) unramified representations.
Note that $\ad(\rbar) = \F_p$ and $\ad(\rbar)(1) = \mu_p$. 
Let us consider the dimensions of the associated Selmer group
$H^1_{L}(F,\F_p)$ and the dual Selmer group $H^1_{L^*}(F,\mu_p)$.
Recall from the Greenberg--Wiles formula that:
$$ \frac{|H^1_{L}(F,\F_p)|}{|H^1_{L^*}(F,\F_p)|} = 
\frac{|H^0(F,\F_p)|}{|H^0(F,\mu_p)|}
\prod_{v} \frac{|L_v|}{|H^0(F_v,\F_p)|}.$$
The possible local  contributions come from $v|p$, $v|\infty$, and the $H^0$ term.
We assume that $\zeta_p \notin F$.
\begin{enumerate}
\item The contribution from
$\displaystyle{\frac{|H^0(F,\F_p)|}{|H^0(F,\mu_p)|}}$
is~$p$, because $\zeta_p \notin F$. Thus the contribution to  $\dim
H^1_{L}(F,\F_p) - \dim H^1_{L^*}(F,\mu_p)$  is $1$.
\item  Let $v|\infty$, so $F_v = \R$ or $F_v = \C$.
The groups $H^0(\R,\F_p)$ and $H^0(\C,\F_p)$ are both one-dimensional.
Hence,  the contribution to $\dim H^1_{L}(F,\F_p) - \dim H^1_{L^*}(F,\mu_p)$ 
at the infinite places is
$-r_1 - r_2$.
\item Let $v|p$. Let $k_v$ be the residue field of $F_v$.   The group $H^0(F_v,\F_p)$ is always $1$-dimensional.
The Selmer condition $L_v \subset H^1(F_v,\F_p)$ is defined
to be the classes that are unramified when restricted to inertia.
 By inflation--restriction, there is a map:
 $$ 0 \rightarrow H^1(\Gal(\overline{k_v}/k_v),\F_p) \rightarrow 
 H^1(F_v,\F_p) \rightarrow H^1(I_v,\F_p).$$
 Since $H^1(\Gal(\overline{k}_v/k_v),\F_p)$ is clearly one-dimensional,
 the contribution of these terms is $1 - 1 = 0$ for each $v|p$.
 \end{enumerate}
 It follows that
 $$\dim H^1_{L}(F,\F_p) - \dim H^1_{L^*}(F,\mu_p) = -(r_1 + r_2 - 1).$$
 We can, in fact, deduce this equality directly by computing both terms via
 class field theory. The first group has dimension
 $\dim \Cl(F)/p$. For the second, recall that for $v|p$ the group $L^*_v$ is
 one-dimensional and is dual to the unramified classes in $H^1(F_v,\F_p)$. This dual
 consists of classes which are finite flat.
 So we are interested in $H^1_{\fppf}(\OL_F,\mu_p)$, which sits in the exact sequence:
 $$0 \rightarrow \OL^{\times}_F/\OL^{\times p}_F \rightarrow
 H^1_{\fppf}(\OL_F,\mu_p) \rightarrow \Cl(F)[p] \rightarrow 0$$
 by Hilbert's theorem~$90$~(Prop.~III 4.9 of~\cite{milne}). Hence the difference in ranks is
 $$\dim \OL^{\times}_F/\OL^{\times p}_F = r_1 + r_2 - 1,$$
 as follows from Dirichlet's unit theorem and the assumption that $\zeta_p \notin F$.
 
 \medskip

Now let us consider  the corresponding symmetric space. The natural space
to consider is
$$Y:=J_F/F^{\times} = F^{\times} \backslash \A^{\times}_F / U K_{\infty}$$
where $U$ is the maximal compact subgroup of the finite adeles, and
$K_{\infty}$ is the connected component of the identity of the maximal compact
subgroup of $F^{\times} \otimes \R$. For a set $Q$ of auxiliary primes,
we would also like to consider the space
$$Y_{Q} = F^{\times} \backslash \A^{\times}_F / U_Q$$
where , for $v \in Q$ with $N(v) \equiv 1 \mod p^n$,
one replaces $\OL^{\times}_v$ by $\OL^{\times p^n}_v$, the unique subgroup of
index~$p^n$. In
the corresponding notation for~$\GL(n)$, we have $Y = Y_0(Q)$ and
$Y_Q = Y_1(Q)$.
It is slightly more aesthetically pleasing  to work with the compact part of
this space:
$$X_{Q} :=  F^{\times} \backslash \A^{\times}_F / U_Q A^0_{\infty},$$
where $A^0_{\infty}$ is the identity component of the $\R$-points of the maximal $\Q$-split
torus in the centre. Note that $Y_Q$ is an $\R$-bundle over $X_{Q}$, and so, 
from the cohomological viewpoint,
the extra factor of $\R$ does not change any  of the cohomology groups. 
What, geometrically, is $X_{Q}$? The component group is
the maximal quotient of the ray class group of conductor $Q$ and exponent~$p^n$.
The fibres are coming from the infinite primes. The group $K_{\infty}$
is $r_2$ copies of $S^1$ coming from the complex places.
The fibres of $Y_{Q}$ are then exactly the connected components of the cokernel of the map
$$\OL^{\times}_F \rightarrow \R^{\times r_1} \times \C^{\times r_2}/K_{\infty},$$
which is  $(S^1)^{r_1 + r_2-1} \times \R$. Passing to~$X_{Q}$ excises the~factor of~$\R$.
Hence the components of $X_{Q}$ consist simply of a product of circles.
Let us examine the cohomology of~$X_{Q}$. In  degree zero, the cohomology is
$$ \Z_p[\RCl(Q)/p^n],$$
by class field theory.  Let $\pi_v \in \OL_v$ be a uniformizer, and suppose that $(v,Q) =1$.
Then the Hecke operator $T_{\pi}$ acts on the component groups via 
the image of $[\pi_v] \in \RCl(Q)$. 
For $v | Q$, we also have diamond operators for $\alpha \in \OL^{\times}_v$.
There is a corresponding Galois representation
$$\rho_Q: G_F \rightarrow \T^{\times}_Q,$$
which is exactly the Galois representation coming from class field theory; the diamond operators
correspond, via the local Artin map, to the representation of the inertia subgroups at $v|Q$.
If we localize at the maximal ideal~$\m$ of $\T_Q$ corresponding to $\rhobar$, we obtain a
deformation of $\rhobar$ which is ramified only at~$Q$. On the other hand, the action
of~$\T_Q$ on the higher cohomology groups simply propagates from $H^0$ using the K\"{u}nneth formula,
and so one obtains the same action on all cohomology groups. If $R_{Q}$ denotes the
deformation ring of $\rhobar$ unramified outside~$Q$, we obtain a surjection
$$R_Q \rightarrow \T_{Q,\m},$$
moreover, both rings are naturally modules over the ring of diamond operators
$\Z_p[\Delta_{Q}] = \Z_p[U/U_Q]$, acting on the left via Yoneda's lemma and local class field theory,
and on the right via Hecke operators; and this action is the same. 

\medskip

In this setting, a Taylor--Wiles prime~$v$ is a prime $N(v) \equiv 1 \mod p^n$ which
satisfies a Galois condition and an automorphic condition.  The automorphic condition
is that there is no extra cohomology when passing from $X$ to $X_0(Q)$. Since $X = X_0(Q)$,
this is tautologically true. The Galois condition is that we have to be able
 to choose $|Q| = \dim H^1_{L*}(F,\mu_p) = \dim H^1_{L}(F,\F_p) + \ell_0$ primes which
exactly annihilate the dual Selmer group, which is  $H^1_{L^*}(F,\mu_p) = H^1_{\fppf}(F,\mu_p)$.
This group sits inside $H^1(F,\mu_p) = F^{\times}/F^{\times p}$, by Hilbert's Theorem~$90$ in
the classical version. By Kummer theory, the corresponding elements
give rise to extensions
$F(\alpha^{1/p},\mu_p)$. What does it mean to annihilate this class by allowing
ramification at a
prime $N(v) \equiv 1 \mod p^n$? Allowing ramification at~$v$ in the Selmer group corresponds to
assuming that the classes in the dual Selmer group split completely at~$v$. 
That is, we want to choose primes~$v$ so that the class is non-trivial under the map
$$H^1(F,\mu_p) \rightarrow H^1(F_v,\mu_p) = H^1(F_v,\F_p).$$
Note that $\mu_p = \F_p$ as a $\Gal(\overline{F}_v/F_v)$-module if $N(v) \equiv 1 \mod p$. This amounts to asking that the element
$\Frob_v$ in $\Gal(F(\alpha^{1/p},\zeta_p)/F)$ is non-trivial, and that $N(v) \equiv 1 \mod p^n$.
By the Chebotarev density theorem,
this is possible unless $F(\alpha^{1/p}) \subset F(\zeta_{p^n})$. 
Note that $\alpha \in F$. 
 This can happen if $\alpha = \zeta_p$. So we have
to assume that $\zeta_p \notin F$, although we have already made this assumption previously.
More generally, if $\alpha^{1/p} \in F(\zeta_{p^n})$, then, since $F(\zeta_{p^n})$ is 
Galois an abelian over~$F$, it must be the case that $F(\alpha^{1/p})$ is also
Galois an abelian over~$F$. This implies that $F(\zeta_p) \subset F(\alpha^{1/p})$, and
a consideration of degrees implies that $\zeta_p \in F$. Hence, if $\zeta_p \notin F$,
we may choose suitable primes to annihilate the dual Selmer group.

\medskip

The invariant $\ell_0 = r_1 + r_2 - 1$ exactly matches the dimensions of cohomology
which are supported on~$\m$ and the number of elements in the dual Selmer group which
have to be annihilated (the dimension of $X_{Q}$ is equal to~$\ell_0$, so the cohomology
vanishing results in degrees $> \ell_0$ are automatic). Our patching result then allows us to deduce that there is an isomorphism
$$\Ru \simeq \Z_p[\Cl(\OL_F) \otimes \Z_p],$$
and hence an  isomorphism
between the Galois group of the maximal unramified abelian~$p$-power
extension of $F$ and the class group of~$F$.

\begin{remark} \emph{The circularity of this argument comes from the
application of Greenberg--Wiles, which requires the full use of class field theory.
Even though we only apply this theorem in the seemingly innocuous case of
$M = \F_p$, in fact the general proof of the Euler characteristic formula reduces
exactly to this case by inflation. 
}
\end{remark}

\subsection{Higher dimensional Galois representations}

In this section, we apply our methods to Galois representations of
regular weight over number fields. When the relevant local deformation rings
are all smooth, the argument is
similar to the corresponding result for imaginary quadratic fields
 in~\S~\ref{sec:imag-quadr-fields} that corresponds to the case  $n = 2$ and $l_0 = 1$. However, in order to
 prove non-minimal modularity theorems, it is necessary to consider non-smooth rings, following Kisin.
 The main reference for many of the local computations below is the paper~\cite{CHT}.

 \subsection{The invariant \texorpdfstring{$l_0$}{l0}}
\label{sec:l0}
 Let $F$ be a number field of signature $(r_1,r_2)$. 
 The invariants $l_0$  and $q_0$ may be defined explicitly by the following formula, where
``rank'' denotes rank over $\R$.
$$\begin{aligned}
l_0 := & \  r_1 \left(\rank(\SL_n(\R)) - \rank(\SO_{n}(\R)) \right)
 + r_2 \left(\rank(\SL_n(\C)) - \rank(\SU_{n}(\C)) \right) \\
   = &   \ \begin{cases} \displaystyle{r_1 \left(\frac{n-1}{2}\right) + r_2 (n-1)}, & n \ \text{odd}, \\
   \displaystyle{r_1 \left(\frac{n-2}{2}\right) + r_2 (n-1)}, & n \ \text{even}. \end{cases} \end{aligned}$$
    $$\begin{aligned}
2 q_0 + l_0  := & \  r_1 \left(\dim(\SL_n(\R)) - \dim(\SO_{n}(\R)) \right)
 + r_2 \left(\dim(\SL_n(\C)) - \dim(\SU_n(\C)) \right) \\
   = & \   r_1 \left(n^2 - 1 - \frac{n(n-1)}{2}\right) + r_2 \left(2 (n^2 - 1) - (n^2 - 1)\right). \end{aligned} $$
   The invariants $l_0$ and $2q_0 + l_0$ arise as follows: $2 q_0 + l_0$ is the
   (real) dimension  of the locally symmetric space associated to
   $\G:=\Res_{F/\Q}(\PGL(n))$, and $[q_0,\ldots,q_0 + l_0]$ is the range such that 
   cuspidal automorphic $\pi$ for $\G$ which are tempered at $\infty$ contribute to cuspidal
   cohomology (see Theorem~6.7, VII, p.226 of~\cite{BW}). (In particular, $q_0$ is an integer.)
   
   \medskip

 Let $V$ be a representation of $G_F$ of dimension $n$ over a field of 
 characteristic different from $2$, and assume that  the action of $G_{F_v}$ is \emph{odd}
 for each $v| \infty$. Explicitly, this is the trivial condition for complex places, and
 for real places $v|\infty$ says 
 that the action of complex conjugation $c_v \in G_{F_v}$
 satisfies $\Trace(c_v) \in \{-1,0,1\}$. Then, via an elementary calculation,
 one has:
 \begin{eqnarray*} 
 \sum_{v|\infty} \dim H^0(F_v,\Ad^0(V)) = \begin{cases}
 \displaystyle{r_1 \left(\frac{n^2+1}{2} - 1 \right) + r_2 (n^2 - 1),} & n \ \text{odd}, \\
 \displaystyle{r_1 \left(\frac{n^2}{2} - 1 \right) + r_2 (n^2 - 1),} & n \ \text{even}. \end{cases}
  \end{eqnarray*}
Thus, in both cases we see that:
\begin{eqnarray}
 \label{infinity}
 \sum_{v|\infty} \dim H^0(F_v,\Ad^0(V)) = [F:\Q]\frac{n(n-1)}{2} +
l_0 . 
\end{eqnarray}

\subsection{Deformations of Galois Representations}
\label{sec:deform-galo-repr}
Let $p>n$ be a prime that is unramified in $F$ and assume that $\Frac W(k)$
contains the image of every embedding $F \into \overline{\Q}_p$. Fix a continuous absolutely irreducible representation:
$$\rbar: G_{F} \rightarrow \GL_n(k).$$
We assume that:
\begin{itemize}
\item For each $v|p$, $\rbar|_{G_{v}}$ is Fontaine--Laffaille
with weights $[0,1,\ldots,n-1]$ for each $\tau : \CO_F \to k$ factoring
through $\CO_{F_v}$.
\item For each $v\nmid p$, $\rbar|_{G_v}$ has at worst unipotent
  ramification and $\mathbf{N}_{F/\Q}(v)\equiv 1 \mod p$ if $\rbar$ is
  ramified at $v$.
\item The restriction $\rbar|_{G_v}$ is odd for each $v|\infty$.
\end{itemize}
We also fix a continuous character
\[ \xi : G_F \to \CO^\times \]
lifting $\det(\rbar)$ and with 
\begin{itemize}
\item $\xi|_{I_v} = \eps^{n(n-1)/2}$ for all $v|p$, and
\item $\xi|_{I_v} = 1$ for all $v\nmid p$.
\end{itemize}
For example, we can simply take  $\xi =  \eps^{n(n-1)/2}$.

\medskip

Let $S_p$ denote the set of primes of $F$ lying above $p$. Let $R$
denote a finite set of primes of $F$ disjoint from $S_p$ that
contains all primes at which $\rbar$ ramifies and is such that
$\mathbf{N}_{F/\Q}(v)\equiv 1\mod p$ for each $v\in R$. Let  
$Q$
denote a finite set of primes of $F$ disjoint from
$R\cup S_p$. Finally, let $S=S_p\coprod R$ 
and $S_Q =
S\coprod Q$. In what follows, $R$ will consist of primes away from $p$
where we allow ramification and $Q$ will consist of Taylor--Wiles primes.

\subsubsection{Local deformation rings}
\label{sec:local-deform-rings}

For $v|p$, let $R_v$ denote the framed Fontaine--Laffaille $\CO$-deformation ring
with determinant $\xi|_{G_{F_v}}$ and $\tau$-weights equal to
$[0,1,\dots,n-1]$ for each $\tau : \CO_F \into W(k)$. By~\cite{CHT}  Prop.~2.4.3, $R_v$ is formally smooth over
$\OL$ of relative dimension $n^2-1+[F_v:\Q_p]n(n-1)/2$. 

\medskip

For each $v\in R$, choose a tuple $\chi_v =
(\chi_{v,1},\dots,\chi_{v,n})$ of distinct characters \[\chi_{v,i}:I_v
\lra 1+\m_{\CO}\subset \CO^\times\]
such that $\prod_i\chi_{v,i}$ is trivial.
We introduce the following framed deformation rings for each $v\in R$:
\begin{itemize}
\item Let $R_v^1$ denote the universal framed $\CO$-deformation ring of
$\rbar|_{G_v}$ corresponding to lifts of determinant $\xi$ and with
the property that each element $\sigma\in I_v$ has characteristic
polynomial $(X-1)^n$.
\item Let $R_v^{\chi_v}$ denote the universal framed $\CO$-deformation ring of
$\rbar|_{G_v}$ corresponding to lifts of determinant $\xi$ and with
the property that each element $\sigma\in I_v$ has characteristic
polynomial $\prod_i(X-\chi_{v,i}(\sigma))$.
\end{itemize}

\medskip

We let
\begin{eqnarray*}
\Rloc^1 & := \left(\widehat\bigotimes_{v\in S_p} R_v \right)
\widehat\bigotimes \left(\widehat\bigotimes_{v\in R} R_v^1 \right) \\
\Rloc^\chi & := \left(\widehat\bigotimes_{v\in S_p} R_v \right)
\widehat\bigotimes \left(\widehat\bigotimes_{v\in R} R_v^{\chi_v} \right)
\end{eqnarray*}

\begin{lemma}
\label{lem:Rloc-props}
The rings $\Rloc^1$ and $\Rloc^\chi$ have the following properties:
\begin{enumerate}
\item Each of $\Rloc^1$ and $\Rloc^\chi$ is $p$-torsion free and
equidimensional of dimension
\[ 1+|S_p\cup R|(n^2-1) + [F:\Q]\frac{n(n-1)}{2}.\]
\item We have a natural isomorphism:
\[ \Rloc^1/\varpi \liso \Rloc^{\chi}/\varpi.\]
\item The topological space $\Spec \Rloc^{\chi}$ is irreducible.
\item Every irreducible component of $\Spec \Rloc^1/\varpi$ is
  contained in a unique irreducible component of $\Spec \Rloc^1$.
\end{enumerate}
\end{lemma}

\begin{proof}
  This follows from \cite[Lemma 3.3]{blght} using \cite[Prop.\
  3.1]{Taylor} and the properties of the Fontaine--Laffaille rings
  $R_v$ recalled above.
\end{proof}

For each $v\in Q$, we assume that:
\begin{itemize}
\item $\rbar|_{G_v} \cong \sbar_v \oplus \psibar_v$ where $\psibar_v$ is
a generalized eigenspace of Frobenius of dimension 1.
\end{itemize}
Moreover, we let $\DP_v$ denote the deformation problem (in the sense
of \cite[Defn.\ 2.2.2]{CHT}) consisting of lifts $r$ of $\rbar|_{G_v}$
of determinant $\xi$ and of the form $\rho\cong s_v\oplus
  \psi_v$ where $s_v$ (resp.\ $\psi_v$) lifts $\sbar_v$ (resp.\
  $\psibar_v$) and $I_{v}$ acts via (possibly different) scalars on $s_v$ and
  $\psi_v$.
Let 
\[ L_v \subset H^1(G_v,\ad^0(\rbar)) \]
denote the Selmer condition determined by all deformations of
$\rbar|G_{v}$ to $k[\epsilon]/\epsilon^2$ of type $\DP_v$. Then
\[ \dim_k L_v - h^0(G_v,\ad^0(\rbar))=1.\]

\subsubsection{Global deformation rings}
\label{sec:glob-deform-rings}

We now consider the following global deformation data:
\begin{eqnarray*}
  \CS_Q &= (\rbar,\CO,S_Q,\xi,(\DP_v)_{v\in S_p\cup Q},(\DP^1_v)_{v\in R}) \\
  \CS_Q^\chi &= (\rbar,\CO,S_Q,\xi,(\DP_v)_{v\in S_p\cup Q},(\DP^\chi_v)_{v\in R}),
\end{eqnarray*}
where $\DP_v$, $\DP^1_v$ and $\DP^\chi_v$ are the local deformation
problems determined by the rings $R_v$, $R_v^1$ and $R_v^{\chi_v}$ for
$v\in S_p$ or $v\in R$. A deformation of $\rbar$ to an object of
$\CC_{\CO}$ is said to be of type $\CS_Q$ (resp.\ $\CS_Q^\chi$) if:
\begin{enumerate}
\item it is unramified outside $S_Q$;
\item it is of determinant $\xi$;
\item for each $v\in S_p\cup Q$, it restricts to a lifting of type
  $\DP_v$ and for $v\in R$, to a lifting of type $\DP^1_v$ (resp.\ $\DP^\chi_v$).
\end{enumerate}
If $Q=\emptyset$, we will denote $\CS_Q$ and $\CS_Q^\chi$ by $\CS$ and
$\CS^{\chi}$. The functor from $\CC_{\CO}$ to Sets sending $R$ to the
set of deformations of type $\CS_Q$ (resp.\ $\CS_Q^\chi$) is
represented by an object $R_{\CS_Q}$ (resp.\ $R_{\CS_Q^\chi}$).

We will also need to introduce framing. To this end, let 
\[ T = S_p \cup R.\] Let $R^{\square_T}_{\CS_Q}$ (resp.\
$R^{\square_T}_{\CS_Q^\chi}$) denote the object of
$\CC_\CO$ representing the functor sending $R$ in $\CC_\CO$ to the set
of deformations of $\rbar$ of type $\CS_Q$ (resp.\ $\CS_Q^\chi$) framed at each $v\in
T$. (We refer to Definitions 2.2.1 and 2.2.7 of \cite{CHT} for the
notion of a framed deformation of a given type, replacing the group
$\mathcal{G}_n$ of \emph{op.\ cit.}\ with $\GL_n$ where appropriate.)

We have natural maps
\begin{eqnarray*}
  R_{\CS_Q} &\lra R^{\square_T}_{\CS_Q} \\
  \Rloc^1 &\lra R^{\square_T}_{\CS_Q}
\end{eqnarray*}
coming from the obvious forgetful maps on deformation
functors. Similar maps exist for the `$\chi$-versions' of these
rings. The following lemma is immediate:

\begin{lemma}
  The map 
\[ R_{\CS_Q} \lra R^{\square_T}_{\CS_Q} \]
is formally smooth of relative dimension $n^2|T| -1$. The same
statement holds for the corresponding rings of type $\CS_Q^\chi$.
\end{lemma}

\medskip

We now consider the map   $\Rloc^1 \ra R^{\square_T}_{\CS_Q}$. For
this, we will need to consider the following Selmer groups:
{\small
\begin{eqnarray*} H^1_{\CL(Q),T}(G_{F},\ad^0\rbar) &:= &  \ker \left(H^1(G_{F,S_Q},\ad^0\rbar) \to \bigoplus_{x\in
  T}H^1(G_{x},\ad^0\rbar)\bigoplus \bigoplus_{x\in
  Q} 
   H^1(G_{x},\ad^0\rbar)/L_x \right)\\
  H^1_{\CL(Q)^\perp,T}(G_{F},\ad^0\rbar(1)) &:=& \ker \left(H^1(G_{F,S_Q},\ad^0\rbar(1)) \to  \bigoplus_{x\in
  Q} 
  H^1(G_{x},\ad^0\rbar(1))/L_x^\perp\right).\end{eqnarray*}
  }

\begin{prop}
\label{prop:number-generators}
  \begin{enumerate}
  \item The ring $R^{\square_T}_{\CS_Q}$ (resp.\ $R^{\square_T}_{\CS_Q^\chi}$) is a quotient of a power series
ring over $\Rloc^1$ (resp.\ $\Rloc^\chi$)
in 
\[ h^1_{\CL,T}(G_{F},\ad^0\rbar) + \sum_{v\in T}h^0(G_v,\ad\rbar) - h^0(G_F,\ad\rbar) \ \text{variables.} \]
\item We have
  \begin{eqnarray*} h^1_{\CL,T}(G_{F},\ad^0\rbar)  &=&
    h^1_{\CL^{\perp},T}(G_{F},\ad^0\rbar(1))+h^0(G_F,\ad^0\rbar)-h^0(G_F,\ad^0\rbar(1))
    \\
    & &+ \sum_{v\in 
    Q} (\dim_k L_v - h^0(G_{v},\ad^0\rbar)) - \sum_{v\in
      T\cup\{v|\infty\}}h^0(G_v,\ad^0\rbar).
    \end{eqnarray*}  
  \end{enumerate}
\end{prop}

\begin{proof}
  The first part follows from the argument of \cite[Lemma
  3.2.2]{kisin-moduli} while the second follows from Poitou-Tate
  duality and the global Euler characteristic formula (c.f.\ the proof
  of \cite[Prop.\ 3.2.5]{kisin-moduli}).
\end{proof}

\subsection{The numerical coincidence}

By choosing a set of Taylor--Wiles primes $Q$ to kill the dual Selmer group, one
deduces the following.

\begin{prop}
\label{prop:tw-primes}
  Assume that $\rbar(G_{F(\zeta_p)})$ is big and let  $q \geq h^1_{\CL^{\perp},T}(G_{F},\ad^0\rbar(1))$ be an
integer. Then for any $N\geq 1$, we can find a tuple
$(Q,(\psibar_v)_{v\in Q})$ where 
\begin{enumerate}
\item $Q$ is a finite set of primes of $F$ disjoint from $S$ with $|Q|=q$.
\item For each $v\in Q$, we have
$\rbar|_{G_{v}} \cong \sbar_v \oplus \psibar_v$ where $\psibar_v$ is
a generalized eigenspace of Frobenius of dimension 1.
\item For each $v\in Q$, we have $\mathbf{N}_{F/\Q}(v)\equiv 1 \mod p^N$.
\item\label{generators} The ring $R^{\square_T}_{\CS_Q}$ (resp.\ $R^{\square_T}_{\CS_Q^\chi}$) is a quotient of a
  power series ring over $\Rloc^1$ (resp.\ $\Rloc^\chi$) in
\[ q + |T| -1 - [F:\Q]\frac{n(n-1)}{2} - l_0. \]
variables. 
\end{enumerate}
\end{prop}

\begin{proof}
Suppose given a tuple $(Q,(\psibar_v)_{v\in Q})$ satisfying the first
three properties. 
Let $e_v\in \ad\rbar$ denote the $G_v$-equivariant projection
onto $\psibar_v$. Then, as in \cite[Prop.\ 2.5.9]{CHT} (although we
work here with a slightly different deformation problem at each $v\in Q$) we have
\[ 0 \lra H^1_{\CL(Q)^{\perp},T}(G_F,\ad^0\rbar(1))\lra
H^1_{\CL^{\perp},T}(G_F,\ad^0\rbar(1)) \lra \oplus_{v\in Q}k \]
where the last map is given by
$[\phi]\mapsto(\Tr(e_v\phi(\Frob_v)))_v$.
The argument of \cite[Prop.\ 2.5.9]{CHT} can then be applied to deduce
that one
may choose a tuple $(Q,(\psibar_x)_{x\in Q})$ satisfying the first
three required properties and such that
\[ H^1_{\CL(Q)^{\perp},T}(G_F,\ad^0\rbar(1)) = \{ 0\}.\]
The last property then follows from Prop.\ \ref{prop:number-generators},
equation \eqref{infinity} and the fact that
$$\dim_k L_v - h^0(G_v,\ad^0\rbar) =  1 \ \text{\ if
    $v\in Q$}.$$ 

\end{proof}

\section{Homology of Arithmetic Quotients}
\label{sec:hom-arithm-quot}

 Let $\A$ denote the adeles of $\Q$, and
$\A^\infty$ the finite  adeles. Similarly, let $\A_F$ and
$\A_F^\infty$ denote the adeles and finite adeles of $F$. Let $\G = \mathrm{Res}_{F/\Q} \PGL(n)$, and write $G_{\infty} = 
\G(\R) = \PGL_n(\R)^{r_1} \times \PGL_n(\C)^{r_2}$.
Let $K_{\infty}$ denote a maximal compact of $G_{\infty}$ with
connected component $K^{0}_{\infty}$.
For any 
 compact open subgroup $K$ of
$\G(\A^\infty)$, we may define an arithmetic orbifold $Y(K)$ as follows:
$$Y(K):= \G(\Q) \backslash \G(\A)/K^0_{\infty} K.$$
It has dimension $2 q_0 + l_0$ in the notation above.
We will specifically be interested in the following $K$.  Let
$S=S_p\cup R$ and $S_Q=S\cup Q$ be as in Section \ref{sec:galois-deformations}.
We follow the convention that the cohomology of $Y(K)$ is the orbifold cohomology
in the sense of Remark~\ref{remark:orbifold}.

\subsubsection{Arithmetic Quotients} \label{section:ar1-imaginary}
Let $K_{Q}=\prod_{v}K_{Q,v}$ and
$L_{Q}=\prod_vL_{Q,v}$ denote the open compact
subgroups of $\G(\A)$ such that:
\begin{enumerate}
\item  If $v \in Q$,  $K_{Q,v}$ is the image in $\PGL(\OL_v)$ of $\left\{ g \in \GL_n(\OL_{v}) \ | \
g \ \text{stabilizes $\ell$}   \mod \pi_{v} \right\}$ where $\ell$ is
a fixed
line in $k_v^n$.
\item  If $v \in Q$,  $L_{Q,v} \subset K_{Q,v}$ is the normal subgroup with quotient group $k^{\times}_v$.
Explicitly,
$$K_{Q,v} = \left( \begin{matrix} 1 & * \\ 0 & \GL_{N-1}(\OL_v) \end{matrix} \right) \mod \pi_v,$$
$$L_{Q,v} = \left( \begin{matrix} 1 & * \\ 0 & \SL_{N-1}(\OL_v) \end{matrix} \right) \mod \pi_v.$$
\item If $v\not \in S_Q$ 
$K_{Q,v} = L_{Q,v} = \PGL_n(\OL_v)$.
\item  If $v \in R$,  $K_{Q,v} = L_{Q,v}=  \text{the Iwahori $\Iw(v)$ subgroup of
    $\PGL_n(\OL_v)$}$ associated to the upper triangular unipotent subgroup.
\end{enumerate}

When $Q=\emptyset$, we let $Y=Y(K_\emptyset)$. Otherwise, we define
the arithmetic quotients $Y_0(Q)$ and $Y_1(Q)$ to be $Y(K_{Q})$ and
$Y(L_{Q})$ respectively. They are the analogues of the modular curves
corresponding to the congruence subgroups consisting of $\Gamma_0(Q)$
and $\Gamma_1(Q)$. (Rather, they are the appropriate analogues in the $\PGL$-context.)

For each $v\in R$, let $\Iw_1(v)\subset \Iw(v)$ denote the pro-$v$
Iwahori. We fix a character
\[ \psi_v=\psi_{v,1}\times\dots\times\psi_{v,n} : \Iw(v)/\Iw_1(v)\cong(k_v^\times)^n/(k_v^{\times}) \lra 1+\m_{\CO}\subset \CO^\times. \]
The collection of characters $\psi=(\psi_v)_{v\in R}$ allows us to
define a local system of free rank 1 $\CO$-modules $\CO(\psi)$ on
$Y_i(Q)$ for $i=1,2$: let $\pi:\tY_i(Q)\to Y_i(Q)$ denote the arithmetic quotient
obtained by replacing the subgroup $\Iw(v)$ with $\Iw_1(v)$ for each
$v\in R$. Then, a section of $\CO(\psi)$ over an open subset
$U\subset Y_i(Q)$ is a locally constant function $f: \pi^{-1}(U)\ra
\CO$ such that $f(\gamma u) = \psi(\gamma)f(u)$ for all $\gamma \in
\Iw(v)/\Iw_1(v)$. We let $H^i_\psi(Y_i(Q),\CO)$ and
$H_{i,\psi}(Y_i(Q),\CO)$ denote $H^i(Y_i(Q),\CO(\psi))$ and
$H_{i}(Y_i(Q),\CO(\psi))$. Note that if $\psi=1$ is the collection of
trivial characters, then $\CO(\psi)\cong \CO$ and hence
$H^i_{\psi}(Y_i(Q),\CO)\cong H^i(Y_i(Q),\CO)$.

\subsubsection{Hecke Operators}
We recall the construction of the Hecke operators. \label{subsection:heckeoperators}
Let $g \in \G(\A^\infty)$ be an invertible matrix trivial at each
place $v\in R$.
For $K\subset \G(\A^\infty)$ a compact open subgroup of the form $K_Q$
or $L_Q$, the Hecke
operator $T(g)$ is defined on the homology modules
$H_{\bullet,\psi}(Y(K),\OL)$ by considering the composition:
$$H_{\bullet,\psi}(Y(K),\OL) \rightarrow
H_{\bullet,\psi}(Y(g K g^{-1} \cap K_{Q}),\OL)\ra H_{\bullet,\psi}(Y(K\cap g^{-1}Kg),\OL)  \rightarrow
H_{\bullet,\psi}(Y(K),\OL),$$
the first map coming from the corestriction map, the second coming
from the map $Y(gKg^{-1}\cap K,\OL) \ra Y(K\cap g^{-1} Kg,\OL)$
induced by right multiplication by $g$ on $\G(\A)$ and the third
coming from the natural map on homology. The maps on cohomology $\H^{\bullet}(Y(K),\OL)$ are defined
similarly. (Since, conjecturally, the cohomology of the boundary will vanish after localizing at the relevant $\m$,
we may work either with cohomology or homology, by duality.)
The Hecke operators act on $\H^{\bullet}(Y(K),\OL)$ but
do not preserve the homology of the connected components. 
For $\alpha \in \A_F^{\infty,\times}$, we define the
Hecke operator $T_{\alpha,k}$ by taking 
$$g =  \diag(\alpha,\alpha, \ldots, 1, \ldots,1)$$
consisting of $k$ copies of $\alpha$ and $n-k$ copies of $1$.
We now define the Hecke algebra.
\begin{df}
Let $\Tan_{Q,\psi}$ denote the subring of 
$$ \End \bigoplus_{k,n} \H^{k}(Y_1(Q),\CO/\varpi^n)$$
generated by Hecke endomorphisms 
$T_{\alpha,k}$ for 
all $k \le n$ and
all $\alpha$ which
are units at primes
in $S_Q$.
Let $\T_{Q,\psi}$ denote the $\CO$-algebra generated by
the same operators  with $T_{\alpha}$ for
${\alpha}$ non-trivial at places in $Q$.
If $Q = \emptyset$, we write $\T_\psi$ for $\T_{Q,\psi}$.
\end{df}
If $\a \subseteq \OL_F$ is an ideal prime to the level, we may define the Hecke operator
$T_{\a,k}$ as $(1/\mathbf{N}_{F/Q}(\a)^k)T_{\alpha,k}$ where
$\alpha \in \A^{\times,\infty}_F$ is any element which represents the ideal $\a$ and such that
$\alpha$ is $1$ for each component dividing the level. In particular, if  $\a = x$ is prime, then $T_x$
is uniquely defined when $x$ is prime to the level but not when $x$ divides the level.

\begin{remark} \emph{It would be more typical to define $\T_{Q,\psi}$ as the subring of endomorphisms
of 
$$\End \bigoplus_{k} \H^{k}(Y_1(Q),\CO),$$
 except that it would not be obvious from this definition
that $\T_{Q,\psi}$  acts on $\H^k(Y_1(Q),\CO/\varpi^n)$ for any $n$. It may well be true (for the
$\m$ we consider) that $\T_{Q,\psi}$ acts \emph{faithfully} on the module $\H^{q_0 + l_0}(Y_1(Q),\CO)$ ---
and indeed (at least for $Q = \emptyset$)
this   (conjecturally) follows when Theorem~\ref{theorem:modularity} applies
and (in addition) $\Rloc$ is smooth. Whether one can prove this directly is an interesting question.
(The claim is obvious when the cohomology occurs in a range of length $l_0 = 0$, and also follows
in the case $l_0 = 1$ given known facts about the action of $\T_Q$ on cohomology with $K$ coefficients.)
}
\end{remark}

\subsection{Conjectures on Existence of Galois Representations}
\label{conjectures}

Let $\m$ denote a maximal ideal of $\Tan_{Q,\psi}$, and let $\Tan_{Q,\psi,\m}$ denote the
completion. It is a local ring which is finite (but not necessarily flat) over
$\CO$.

\begin{conjectureA} 
\label{conj:AA} 
There exists a semisimple continuous Galois representation
$\rbar_\m:G_F \rightarrow \GL_n(\Tan_{Q,\psi}/\m)$ with the following property:
 if $\lambda\not\in S_Q$ is a prime of $F$, then $\rbar_\m$ is unramified
at $\lambda$, and the characteristic polynomial of $\rbar_\m(\Frob_{\lambda})$ is
$$X^n - T_{\lambda,1} X^{n-1} + \ldots +
(-1)^i\nmf(\lambda)^{i(i-1)/2}T_{\lambda,i}X^{n-i}+\ldots+
(-1)^n\nmf(\lambda)^{n(n-1)/2}T_{\lambda,n},$$
in $\Tan_{Q,\psi}/\m[X]$. Note that this property determines $\rbar_\m$ uniquely by the
Chebotarev density theorem. If $\rbar_\m$ is absolutely irreducible,
we say that $\m$ is \emph{non-Eisenstein}. In this case we further predict
that there exists a deformation $\r_\m:G_F\rightarrow
\GL_n(\Tan_{Q,\psi,\m})$ of $\rbar_\m$ unramified outside $S_Q$ and such that
the characteristic polynomial of
$\r_\m(\Frob_\lambda)$ is given by the same formula as above. 
In
addition, suppose that $\r_{\m}\cong \rbar$ (where $\rbar$ is the representation
introduced in Section \ref{sec:deform-galo-repr}).
 Suppose also that the set of primes $Q$ consists
of a set of Taylor--Wiles primes, that is, a set of primes 
as constructed in Proposition~\ref{prop:tw-primes};  this is an empty condition when
$Q = \emptyset$. 
We conjecture that $\r_{\m}$ enjoys the following properties:
\begin{enumerate}
\item If $v|p$, then $\r_\m|_{G_v}$ is Fontaine-Laffaille with all
  weights equal to $[0,1,\dots,n-1]$.
\item \label{condition:two} If $v\in Q$, then $\r_\m|_{G_v}$ is a lifting of type $\DP_v$
  where $\DP_v$ is the local deformation problem specified in Section \ref{sec:local-deform-rings}.
\item If $v\in R$, then the characteristic polynomial of
  $\r_\m(\sigma)$ for each
  $\sigma\in I_v$ is $(X-\psi_{v,1}(\Art_v^{-1}(\sigma)))\dots(X-\psi_{v,n}(\Art_v^{-1}(\sigma)))$.
 \item  \label{condition:vanishingconjA} 
 \begin{enumerate}
 \item The localizations $\H^i(Y_1(Q),\OL/\varpi^m)_{\m}$ vanish unless $i \in [q_0,\ldots,q_0 + l_0]$.
 \item The localizations $\H^i(\partial Y_1(Q),\OL/\varpi^m)_{\m}$ vanish for all $i$, where 
 $\partial Y_1(Q)$ is the boundary of the Borel--Serre compactification of $Y_1(Q)$.
 \end{enumerate}
\item\label{condition:aux}   
For $x\in Q$, let $P_x(X) = (X - \alpha_x)Q_x(X)$ denote the characteristic polynomial
of $\rbar(\Frob_x)$ where $\alpha_x=\psibar_x(\Frob_x)$. Let $\m_Q$
denote the maximal ideal of $\T_{Q,\psi}$ containing $\m$ and $V_{x} - \alpha_x$ for all $x|Q$, where
$V_{x} = T_{x,1}$, which is well defined modulo $\m$. Then there is an isomorphism
$$ \lim_{k \rightarrow \infty} \prod_{x\in Q} Q_x(V_{{x}})^{k!}:\H^*(Y,\CO/\varpi^m)_{\m} \iso \H^*(Y_0(Q),\CO/\varpi^m)_{\m_Q},$$
  \end{enumerate}
It follows that $\r_\m$ is a deformation of $\rbar$ of type $\CS_Q$
(resp.\ $\CS^\chi_Q$) if each $\psi_v$ is the trivial character
(resp.\ $\psi_v=\chi_v$ for each $v\in R$). In this case, we obtain a
surjection $R_{\CS_Q}\onto \Tan_{Q,1,\m}$ (resp.\ $R_{\CS_Q^\chi}\onto \Tan_{Q,\chi,\m}$).
\end{conjectureA}

Some form of this conjecture has been suspected to be true at least as far back
as the investigations of F.~Grunewald in the early 70's (see~\cite{Grunewald,GHM}). Related conjectures about the existence
of $\rbar_{\m}$ 
were made for $\GL(n)/\Q$ 
by Ash~\cite{Ash}, and for $\GL(2)/F$ by Figueiredo~\cite{Fig}.
 One aspect of this conjecture is that
it implies that the local properties of the (possibly torsion) Galois representations are captured by the
characteristic zero local deformation rings $\Rbox_v$ for primes $v$. One might hope that such a conjecture is true in maximal
generality, but we feel comfortable making the conjecture in this case because the relevant local deformation rings (including the
Fontaine--Laffaille deformation rings) reflect an honest integral theory, which is not necessarily true of all the
local deformation rings constructed by Kisin,
(although the work of Snowden~\cite{Snowden} gives hope that at least in the ordinary case that local deformation rings 
may capture all integral phenomena).
By d\'{e}vissage, conditions~\ref{condition:vanishingconjA} and~\ref{condition:aux} are satisfied if and only if they are satisfied for $n = 1$,
e.g.,  with coefficients in the residue field $k = \OL/\varpi$.

\medskip

\begin{remark} \emph{
Part~(\ref{condition:vanishingconjA})(a) of Conjecture~\ref{conj:AA} may be verified directly in a number of small rank cases,
in particular for~$\GL(2)/F$ when~$F$ is CM field of degree either~$2$ or~$4$, or for~$\GL(3)/\Q$. In the latter two cases
(where~$(q_0,l_0) = (2,2)$ and~$(2,1)$ respectively), the key point is that the lattices in question satisfy the 
congruence subgroup property~\cite{CSP}, which yields vanishing for both~$(H^1)_{\m}$ and (by duality, considering both~$\m$ and~$\m^*$)~$(H^{q_0 - 1}_c)_{\m}$ 
where~$\m$ is a non-Eisenstein maximal ideal.
On the other hand, the vanishing of~$(H^{q_0 - 1}_c)_{\m}$ also implies the vanishing of~$(H^{q_0 - 1})_{\m}$ after localization at~$\m$,
since the cohomology of the boundary vanishes after localization at~$\m$. (In these cases, we are implicitly using the fact that we know
enough about the boundary of the locally symmetric varieties in question to also resolve Part~(\ref{condition:vanishingconjA})(b) of Conjecture~\ref{conj:AA}.)}
\end{remark}

\begin{remark} \emph{In stating Conjecture~\ref{conj:AA}, we have assumed that $Q$ is divisible
only by Taylor--Wiles primes.
To modify the conjecture appropriately for more general $Q$,
one would have to modify condition~\ref{condition:two} to allow for more general
quotients of the appropriate local deformation ring (which would involve a mix
of tamely ramified principal series and  unipotent representations) and one
would also drop condition~\ref{condition:aux}.}
 \end{remark}

\medskip

\begin{remark}  \emph{Condition~\ref{condition:vanishingconjA} of Conjecture~\ref{conj:AA} says that we could
also have formulated our conjecture for compactly supported cohomology, or equivalently for homology.
The complexes we eventually patch are computing 
$$H^*(Y,\OL/\varpi^n)^{\vee}_{\m} =  H^*_c(Y,\OL/\varpi^n)^{\vee}_{\m} =  H_*(Y,\OL/\varpi^n)_{\m^*}$$
for the dual maximal ideal $\m$,
so it may have made more sense to work with homology. Indeed, in the homological
formulation, we wouldn't need to assume anything about the vanishing of the homology of the
boundary localized at $\m$.  However, for historical reasons, we
continue to work in the present setting, with the understanding that the real difficulty in Conjecture~\ref{conj:AA}
lies (after Scholze~\cite{Scholze}) with proving the non-vanishing of (co-)homology groups in the required range after
localization at $\m$ and proving local-global compatibility, especially at $v|p$. 
}
\end{remark}

\medskip

The reason for  condition~\ref{condition:aux} of Conjecture~\ref{conj:AA}   is that the arguments of~\S3 of~\cite{CHT}
(in particular,  Lemma~3.2.2 of \emph{ibid}.) often require that the $\GL_n(F_x)$-modules $M$
in question are $\OL$-flat. However, it may be possible to remove this condition, we hope to return to
this point later (it is also true that slightly weaker hypotheses are sufficient for our arguments).
On the other hand, we have the following:
\begin{lemma} \label{lemma:TWtwo}  If $n = 2$, then 
 condition~\ref{condition:aux} of Conjecture~\ref{conj:AA}  
 holds for all Taylor--Wiles primes.
\end{lemma}

\begin{proof} By induction, it suffices to prove the result when
  $Q=\{x\}$ consists of a
  single such prime. For simplicity, we drop $\psi$ from the notation.
The two natural degeneracy maps induce maps:
$$\begin{aligned}
\pssi: & \ H^*(Y,\CO/\varpi^n)^2 \rightarrow H^*(Y_0(x),\CO/\varpi^n), \\
\pssi^{\vee}: & \ H^*(Y_0(x),\CO/\varpi^n)
 \rightarrow H^*(Y,\CO/\varpi^n)^2 \end{aligned}$$
such that the composition $\pssi^{\vee} \circ \pssi$ is the matrix
$$\left( \begin{matrix} N(x) + 1 & T_x \\ T_x & N(x) + 1 \end{matrix} \right),$$
which has determinant $T^2_x  - (1 + N(x))^2$. If the eigenvalues
of $\rhobar(\Frob_x)$ are $\alpha_x$ and $\beta_x$, then
 $\alpha_x \beta_x \equiv N(x) \equiv 1 \mod p$.
 If $x$ is a Taylor--Wiles prime, then by assumption, $\alpha_x$ is distinct from $\beta_x$, or equivalently,
$\alpha_x \not\equiv \pm 1 \mod p$. It follows that $T^2_x - (1 + N(x))^2 \not\in \m$, and hence 
$\pssi^{\vee} \circ \pssi$ is invertible after localizing at $\m$. In particular, the maps $\pssi$ and $\pssi^{\vee}$
induce a splitting
$$\begin{aligned}
H^*(Y_0(x),\CO/\varpi^n)_{\m}  \simeq & \ H^*(Y,\CO/\varpi^n)^2_{\m} \oplus W, \\
H^*(Y_0(x),\CO/\varpi^n)_{\wtm} \simeq & \ H^*(Y,\CO/\varpi^n)_{\wtm} \oplus W_{\wtm}
\end{aligned}
$$
for some $\Tan_{Q,\m}$-module $W \subset H^*(Y_0(x),\CO/\varpi^n)$, and
$\wtm = (\m,U_x - \alpha_x)$. (Here, by abuse of notation, $\alpha_x$ denotes any
lift of $\alpha_x \in \OL/\varpi$ to $\OL/\varpi^n$.) It suffices to prove that $W$ is trivial.
One approach is to try to show that some $w \in W$ generates 
either  the Steinberg representation  $\Sp$ or $\Sp \otimes \chi$
for the quadratic unramified character $\chi$ of $F^{\times}_x$, and then deduce that
the action of $U_{x}$ on $w$ is via $\pm 1$, contradicting the assumption on $\m$.
However, 
the map 
$$H^{*}(Y_0(x),\CO/\varpi^n) \rightarrow  \left(\lim_{\rightarrow} H^*(Y(x^m),\CO/\varpi^n)\right)^{U_0(x)}$$
is \emph{a priori} neither  surjective nor injective, which causes some complications with this approach.
Hence we proceed somewhat differently
(although see Remark~\ref{remark:lowest} below).
The following argument is implicit in
the discussion of Ihara's Lemma in Chapter~4 of~\cite{CV}.

\medskip

Recall that there exists a decomposition
$$Y \simeq \coprod \Gamma_i \backslash \mathcal{H},$$
where the $\Gamma_i$ are congruence subgroups commensurable with  $\PGL_2(\OL_F)$ and
$\mathcal{H}$ denotes the corresponding locally symmetric space. The finitely many connected
 components  of $Y$ will naturally be a torsor over a ray
 class group corresponding to the  level of $Y$.
 Let $\Gamma$ be one such subgroup. 
Denote by  $\Gamma^1$ the intersection
$\Gamma \cap \PSL_2(\OL_F)$. By construction,  $\Gamma/\Gamma^1$ is an elementary two group, which
we denote by $\Phi$. Recall that we are assuming that
 $k = \OL/\varpi$ has odd characteristic~$p$. Then, by Hochschild--Serre, there is an isomorphism
$$H^*(\Gamma,\OL/\varpi^n) \simeq H^*(\Gamma^1,\OL/\varpi^n)^{\Phi}.$$

\medskip

By construction, for a Taylor--Wiles prime $x$, the
level structure of $Y$ at $x$ is maximal. For convenience, let us also assume that $x$ is
trivial in the ray class group corresponding to the component group of $Y$ (this is equivalent to
imposing a further congruence condition on $x$, but is imposed only for notational convenience
in the argument below).
For such a prime $x$, we may form the amalgam 
$$G:= \Gamma^1 *_{\Gamma^1_0(x)} \Gamma^1$$
of $\Gamma^1$ with itself along the subgroup
$\Gamma^1_0(x):=\Gamma_0(x)\cap \Gamma^1$. Then $G$
will be a congruence subgroup of the $S$-arithmetic group
$\PSL_2(\OL_F[1/x])$ with the same level structure of $\Gamma^1$ at primes away from~$x$.
(Without the extra assumption on $x$, one would have to amalgamate \emph{different} pairs
of lattices $\Gamma_i$ occurring in the decomposition of $Y$ according to the action of the
ray class group, cf.~\S~4.1.4 of~\cite{CV}. The argument would then proceed quite similarly, but it would require
more notation) 
The long exact sequence of Lyndon for an amalgam
(See~\cite{SerreTrees}, p.169) gives rise to the following exact sequence:
$$\ldots \rightarrow H^{i-1}(G,\OL/\varpi^n) \rightarrow H^{i}(\Gamma^1,\OL/\varpi^n)^2 \rightarrow H^i(\Gamma^1_0(x),\OL/\varpi^n) \rightarrow \ldots$$
We claim that this sequence is equivariant with respect to the Hecke operator $U_{x^2}$, which acts
on $H^*(G,\OL/\varpi^n)$ by $1$.
First, recall the definition of $U_{x^2}$. It is defined to be $1/N(x)^2$ times the operator induced
by taking the double coset operator for $\Gamma^1_0(x)$ corresponding to the matrix:
$$\left(\begin{matrix} \pi^2_x & 0 \\ 0 & 1 \end{matrix} \right).$$
Since this operator only depends on the matrix up to scalar, one may equally take the matrix to be
$$g:=\left(\begin{matrix} \pi_x & 0 \\ 0 & \pi^{-1}_x \end{matrix} \right).$$
With this normalization, the matrix $g$ lies in $G$. 
In particular, the corresponding map on $H^*(G,\OL/\varpi^n)$ is given by multiplication by the degree
of this operator on $\Gamma^1_0(x)$, which is $N(x)^2$, and thus (after normalizing) it follows that $U_{x^2}$ acts by~$1$.
It follows that if we localize the sequence at any ideal $\m$ such that $U_{x^2} - 1$ is invertible, then
there is an isomorphism
 $$H^{i}(\Gamma^1,\OL/\varpi^n)^2_{\m} \simeq H^i(\Gamma^1_0(x),\OL/\varpi^n)_{\m}.$$
 To recover the isomorphism for $Y$, it suffices to repeat this argument for each lattice $\Gamma_i$.
 On $H^*(Y_0(x),\OL/\varpi^n)$, however, the operator $U_x$ satisfies $U^2_x = U_{x^2}$.
 In particular, since neither $\alpha_x$ nor $\beta_x$ is equal to~$\pm 1$, we deduce (for the maximal ideal $\m$ of
 interest) that there is an isomorphism
 $$H^{i}(Y,\OL/\varpi^n)^2_{\m} \simeq H^i(Y_0(x),\OL/\varpi^n)_{\m}.$$
 Taking $\wtm = (\m,U_x - \alpha_x)$ and
 applying the projections $\displaystyle{\lim_{n \rightarrow \infty} (U_x - \beta_x)^{n}}$ and 
 $\displaystyle{\lim_{n \rightarrow \infty} (U_x - \alpha_x)^{n}}$
 gives the necessary isomorphism.
 \end{proof}
 
 \begin{remark}  \emph{ As noted in~\cite{CV}, the group $\PSL_2(F_x)$ decomposes as an
 amalgam whereas $\PGL_2(F_x)$ does not --- this is the reason for the reduction
 step to the $\PSL_2$ case above. One \emph{could} proceed above with $\PGL_2$, but then the amalgams
 would more naturally be subgroups of $\PGL_2(\OL_F[1/x])^{(\mathrm{ev})} \subset \PGL_2(\OL_F[1/x])$
 consisting of matrices whose determinant has even valuation at $x$ (cf. Chapter~4 of~\cite{CV}).
 In either case, one deduces as above that $U^2_{x}$ acts by $+1$ on $H^*(G,\OL/\varpi^m)$.
 }
\end{remark}

\begin{remark} \label{remark:lowest}   \emph{ The proof of this lemma is related to the proof
of Lemmas~\ref{lemma:matt-w1} and~\ref{lemma:matt}. However, in those cases, it was only necessary  to prove equality in the lowest
degree, which is more elementary. 
Indeed, if $q_0$ denotes the lowest degree in which $H^{q_0}(Y,\CO/\varpi)_{\m}$ is nonzero, then,
by Hochschild--Serre, the kernel of the map 
$$H^{q_0}(X_0(x),\OL/\varpi^n) \rightarrow H^{q_0}(X_1(x),\OL/\varpi^n)$$
has a filtration by terms of the form $H^i(\Delta, H^j(X_1(x),\OL/\varpi^n))$ for $j < q_0$. Since (by assumption)
the coefficients of this expression are trivial after localization at $\m$ for $j < q_0$,  the kernel vanishes and the map above is injective. Hence the argument above using the representation $\Pi$ generated by $w \in W$ applies in this case.
Analysis of this spectral sequence suggests, however, that the map localized at $\m$ will not (in general)
be injective for $i > q_0$ when
$l_0 > 0$. 
}
\end{remark}

\subsection{An approach to Conjecture~\ref{conj:AA} part~5}
In this section, we present an informal approach to proving part~$5$ of Conjecture~\ref{conj:AA} under a
stronger assumption that $\rbar$ has enormous image, at least in the analogous case of $\GL(n)$ (from which
it should be easy to deduce the corresponding claim for $\PGL$, since manifolds for the former are circle bundles over
manifolds for the latter, and so have highly related Hecke actions). Here, by enormous image, we require (in
addition to bigness) the existence
of suitable Taylor--Wiles primes $x$ such that $\rbar(\Frob_x)$ has \emph{distinct} eigenvalues.  We thank David Helm
for some helpful remarks concerning the deformation theory of  unramified principal series.

\subsubsection{Local Preliminaries}

Let $F/\Q$ be a number field, let $x$ be a prime in $F$ such that $N(x) \equiv 1 \mod p$.
Let $G = \GL_n(F_x)$, and let $D \subset B \subset P \subset G$ denote the Borel subgroup  $B$ and a
parabolic subgroup $P$ with Levi factor $L:=\GL_{n-1}(F_x) \times F^{\times}_x$, and
$D \simeq (F^{\times}_x)^n$ the Levi of $B$.
Let $G(\OL_x) = \GL_n(\OL_x)$, let $U(x) \subset G(\OL_x)$ denote the full congruence subgroup
of level $x$, and  let $U_0(x) \subset G(\OL_x)$ denote
the largest subgroup containing $U(x)$ whose image in $\GL_n(\OL_x/\varpi_x)$ stabilizes a line, chosen
compatibly with respect to $P$.
Suppose that $N(x) \equiv 1 \mod p$.

\medskip

Let $\rhobar: G_x \rightarrow \GL_n(k)$ be a  continuous semi-simple representation.
 We say that an
irreducible admissible mod-$p$ representation $\pi$ is \emph{associated} to $\rhobar$ if
$$\mathrm{rec}(\pi) = \WD(\rhobar)$$
under the semi-simple local Langlands correspondence of Vign{\'e}ras~\cite{VL}.
The following is well known.

\begin{lemma} \label{lemma:firstcase} Let $\rhobar:G_x \rightarrow \GL_n(k)$ be  unramified with
distinct eigenvalues. Then $\mathrm{rec}(\pi) = \rhobar$ if and only if
$\pi$ is the irreducible unramified mod-$p$ principal series:
$$\pi = \nind^{G}_{B}(\chi),$$
where $\chi: (F^{\times}_x)^n \rightarrow k^{\times}$ factors
through $(F^{\times}_x/\OL^{\times}_x)^n$ and sends each uniformizer
to a distinct eigenvalue of $\rhobar(\Frob_x)$.
\end{lemma}

In addition, we have the following:

\begin{lemma} \label{lemma:ext} Let $\pi$ be the unramified principal series
in Lemma~\ref{lemma:firstcase}, and let $\pi'$ denote any irreducible admissible
representation of $G$ such that $\pi' \not\simeq \pi$. Then
$\Ext^1(\pi,\pi') = \Ext^1(\pi',\pi) = 0$.
\end{lemma}

\begin{proof} The supercuspidal support of $\pi$ consists of the distinct characters
$\chi_i$. If either extension group is non-zero, then, by
Theorem~3.2.13 of~\cite{EH}, it follows that $\pi'$ has the same supercuspidal
support as $\pi$. But this implies that $\pi'$ is a quotient of $\pi$, and hence
is isomorphic to $\pi$.
\end{proof}

\medskip

\begin{df} Let $\Cat$ denote the category of locally admissible $G$-modules over 
$\displaystyle{
 A:=\OL/\varpi^k}$ such that every irreducible subquotient of
$M \in \Cat$ is associated to $\pi$. 
\end{df}

Under our assumptions on $\rhobar$, we may give a quite precise
description of the finite length elements $M \in \Cat$.

\begin{lemma} Suppose that $M \in \Cat$ has finite length as a $G$-module.
Then there exists a finite length $A$-module $M_B$ and a
character
$$\wchi: B \rightarrow \Aut(M_B)$$
whose irreducible constituents correspond to the character $\chi$, and such that
$M \simeq \nind^{G}_{B}(M_B)$.
\end{lemma}

\begin{proof}
The irreducible constituents of the parabolic restriction $\res^{B}_{G}(\pi)$ consist of
the characters $\chi^{w}$ for~$w$ in the Weyl group $W$ of $G$. By assumption,
all the eigenvalues of $\rhobar(\Frob_x)$ are distinct, and hence all the characters $\chi^w$ are distinct. In particular, 
$\Ext^i(\chi^v,\chi^{w}) = 0$ for all~$i$ if $v \ne w \in W$. It follows that $\res^{B}_{G}(M)$ admits a decomposition
$$\res^{B}_{G}(M) \simeq \bigoplus_{W} M^{w}_{B},$$
where   the irreducible constituents of
$M^{w}_{B}$ are $\chi^{w}$ for
$w \in W$. Moreover, 
 $\length_{A}(M^{w}_{B}) = \length_{G}(M)$ is finite for any $w \in W$.
 Let $M_{B}:=M^{\mathrm{id}}_B$.  There is a natural map
 $$M \rightarrow \nind^{G}_{B} \res^{B}_{G}(M) = \bigoplus_{W} \nind^{G}_{B} M^{w}_{B}
 \rightarrow  \nind^{G}_{B} M_B.$$
 Note that $M$ and $ \nind^{G}_{B} M_B$ are elements of $\Cat$ of the same length, and
  all the irreducible constituents of $ \nind^{G}_{B} M^w_B$ for $w \ne \mathrm{id}$ are
 distinct from $\pi$. Thus, by comparing lengths, to
 prove that the composite of these maps is an isomorphism it suffices to prove that the first map is injective.
 If $K$ denotes the kernel, then $\res^{B}_{G}(K) = 0$. Yet this contradicts the assumption that
 $M$ (and hence $K$) lies in $\Cat$, since $\res^{B}_{G}(\pi) \ne 0$.
\end{proof}

Recall that $D \subset B$ denotes the Levi of $B$, which is $(F^{\times})^n$.
Since $\chi$ is trivial on $D(\OL_x)$, Any finite deformation $\wchi$
of $D(F)$  which deforms $\chi$ has pro-$p$ image after restriction to $D(\OL_x)$, and
thus factors through
$$\left(\lim_{\leftarrow} F^{\times}/F^{\times p^m}\right)^n \simeq
\left( \Z_p \oplus \lim_{\leftarrow} k^{\times}/k^{\times p^m} \right).$$
The universal deformation of this group can be given quite explicitly:

\begin{corr} Let $q = |k^{\times} \otimes \Z_p|$. Then $\nind^{G}_B$ induces an equivalence of
categories between the category of direct limits of finite length modules over the ring $R$ below and and $\Cat$:
$$R:=  \bigotimes_{i=1}^{n} A[T]/(T^q - 1) \otimes_{\OL} A[[X]].$$
\end{corr}

Using this description of $\Cat$, we may prove the following:

 \begin{lemma}  \label{lemma:inject} The category $\Cat$ has enough injectives.
 The functor $M \rightarrow M^{U(x)}$ from $\Cat$ to $G(k):=G(\OL_x)/U(x) \simeq \GL_n(k)$-modules
 takes injectives to acyclic modules.
\end{lemma}

\begin{proof}  
One may explicitly
observe that the appropriate category of $R$-modules has enough injectives. The composite functor from $R$-modules
to $G(k)$-modules can be described explicitly as follows. Given a  deformation $\wchi$, recall
that $\wchi$ factors through $(k^{\times} \times \Z)^n$. Hence the restriction $\wchi |_{D(k)}$ to $(k^{\times})^n
= D(k) \subset B(k)$  is well defined, and one has
$$\left( \nind^{G}_{B} (\wchi)\right)^{U(x)} \simeq \Ind^{G(k)}_{B(k)} \left(\wchi |_{D(k)}\right).$$
Since finitely generated injective $A[T]/(T^q - 1)$-modules are free, it follows that the image of an injective module has
a filtration whose pieces are isomorphic to
$\Phi:=\Ind^{G(k)}_{B(k)}(\psi)$,
where $\psi$ is the universal deformation over $k$ of $D(k)$ to
$k[k^{\times}/k^{\times q}]$.
Yet $\Phi$ is a direct summand of $\Ind^{G(k)}_{B(k)} k[D(k)]$ and thus of  $k[G(k)]$; hence it is  injective and acyclic.
\end{proof}

\begin{remark} \label{remark:abstract}
\emph{
Since the group $U(x)$ is pro-$p$, the higher cohomology of $U(x)$ vanishes. Hence, for any
subgroup $U(x) \subset \Gamma \subset G(\OL_x)$,  by Hochschild--Serre there are identifications
$$H^i(\Gamma,M) \simeq H^i(\Gamma/U(x),M^{U(x)}).$$
Any injective $G(k)$-module is also injective as a $\Gamma/U(x) \subset G(k)$-module.
Hence, by Lemma~\ref{lemma:inject},  the derived functors of $M \mapsto M^{\Gamma}$ are
well defined, and  they coincide with $H^i(\Gamma,M)$.
}
\end{remark}

\medskip

Suppose that the roots of $P(T) = (T - \alpha)Q(T)$ are the Satake parameters of~$\pi$.
The Hecke algebra of $U_0(x)$ contains the operator   
$V= V_{\varpi_x}$  corresponding to the double coset
of the diagonal matrix with $n-1$ entry $1$ and the final entry $\varpi_x$.  
The operator $P(V)$ is zero on $\pi^{U_0(x)}$ and thus acts nilpotently on $M^{U_0(x)}$.
In particular, if
$$e_{\alpha}:=\lim_{\rightarrow} Q(V)^{n!},$$
then $e_{\alpha}$ is a projection of $M^{U_0(x)}$
onto the \emph{localization} of $M^{U_0(x)}$ at the ideal
$(V - \alpha)$ for any lift of $\alpha$ to $\OL$.

\medskip

We shall now define two functors $\Ff$ and $\Gf$ on the category $\Cat$, defined on objects by
$$\Ff(M):=M^{G(\OL_x)}, \qquad \Gf(M):= e_{\alpha} M^{U_0(x)}.$$
There is a natural transformation:
$\iota: \Ff \rightarrow \Gf$
defined by the composition of the obvious inclusion $M^{G(\OL_x)} \hookrightarrow M^{U_0(x)}$ with
$e_{\alpha}$. Note that $\Ff$ and $\Gf$ are left exact. Since $\Cat$ has enough injectives (Lemma~\ref{lemma:inject}),
we have associated right derived functors $R^k \Ff$ and $R^k \Gf$ respectively. 
Hence, since $e_{\alpha}$ is exact, we may (see Remark~\ref{remark:abstract}) identify these right derived functors
with the following cohomology groups:
$$R^k \Ff(M) \simeq H^k(G(\OL_x),M), \qquad R^k \Gf(M) \simeq e_{\alpha} H^k(U_0(x),M).$$

\begin{theorem} 
\label{theorem:cht} The natural transformation $\iota:\Ff \rightarrow \Gf$ is an isomorphism
of functors. In particular,  there is an isomorphism
$$\iota_*: H^k(G(\OL_x),M) \rightarrow e_{\alpha} H^k(U_0(x),M).$$
\end{theorem}

\begin{remark} \emph{ One should compare Theorem~\ref{theorem:cht} to Lemma~3.2.2 of~\cite{CHT}, which
implies  that $\iota_{*}$ is an isomorphism for $k = 0$ and modules $M \in \Cat$ of the form $N \otimes \OL/\varpi^n$
where $N$ is admissible and flat over $\OL$ and $N \otimes_{\OL} \overline{K}$ is semi-simple. 
The (implicit) assumptions on $\rhobar$ in~\cite{CHT}  are, however, somewhat weaker;
one only need assume that the particular eigenvalue $\alpha$ has multiplicity one, and moreover the
assumption that $M$ is an element of $\Cat$ is relaxed (although, for the module $M$ in~\S3 of~\cite{CHT}
for which Lemma~3.2.2 is applied, one may deduce from local-global compatibility
that $M \in \Cat$).
We expect that Theorem~\ref{theorem:cht} is true under these weaker assumptions as well, and
possibly even under the generalization of Lemma~3.2.2 of~\cite{CHT} due to Thorne (Proposition~5.9 of~\cite{thorne}),
see Remark~\ref{remark:possible} following the proof.
}
\end{remark}

\begin{proof}  For $F = \Ff$ or $\Gf$, one has  $F(M) = \displaystyle{\lim_{\rightarrow} F(M_i)}$ as the limit runs over all
finite length submodules $M_i$, hence it suffices to prove the isomorphism for $M$
of finite length. 
In particular, we may assume that $M = \nind^{G}_{B} M_B$ for some finite deformation $\wchi$ of $\chi$.
Then we have an isomorphism
$$\Ff(M) \simeq (M_B)^{D(\OL_x)}.$$
Let $\wchi(m)$ denote the restriction of $\wchi$ to  $F^{\times}_x$ whose irreducible
constituents correspond to the unramified character which takes the value $\alpha_m$
on a uniformizer for some eigenvalue $\alpha_m$ of $\rhobar(\Frob_x)$.
Let  $\wchi(\mhat)$ denote the restriction of $\wchi$ to
$(F^{\times}_x)^{n-1}$ corresponding to the other $n-1$ eigenvalues.
 Then there is an isomorphism
$$\begin{aligned}
M^{U_0(x)} \simeq & \ \bigoplus_{m=1}^{n} \left(\wchi(m) \otimes  \nind(\wchi(\mhat) \right))^{L(\OL_x)} \\
\simeq & \ \bigoplus_{m=1}^{n} (M_B)^{D(\OL_x)}. \end{aligned}
$$
Moreover, the action of $V$ on each factor is given by
the coset corresponding to $\varpi_x \times \mathrm{Id} \in L(F_x)$, which
acts via $\wchi(m)(\varpi_x)$, and the normalized sum of
the invariants is equal to the image of $M^{G(\OL_x)}$. In particular, the operator $e_{\alpha}$
projects onto the $m$th factor such that $\alpha = \alpha_m$, which is an isomorphism.
\end{proof}

\begin{remark} \label{remark:possible} \emph{ The proof of this result, is, to some extent, ``by explicit computation.'' 
Here is a different approach which may work under the weaker assumption
that $\alpha$ has multiplicity one but there is no other assumption on the eigenvalues of
$\rhobar(\Frob_x)$.
First one establishes, for irreducible $\pi \in \Cat$, that there is an isomorphism
$\Ff(\pi) \simeq \Gf(\pi)$. This is essentially already done in~\S3 of~\cite{CHT}.
Now proceed by induction on the length of $M$. 
Suppose now that the claim is true for  modules of length $< \length(M)$. By assumption, there
is an inclusion $\pi \subset M$, let $N$ denote the quotient. By induction, there is a long
exact sequence as follows:
$$
\begin{diagram}
0 & \rTo & \Ff(\pi) & \rTo & \Ff(M) & \rTo & \Ff(N) & \rTo & R^1 \Ff(\pi) \\
 & & \dEquals & & \dTo & & \dEquals & & \dTo \\
 0 & \rTo & \Gf(\pi) & \rTo & \Gf(M) & \rTo & \Gf(N) & \rTo & R^1 \Gf(\pi) \\
  \end{diagram}
$$
By the five lemma, it suffices to show that $R^1 \Ff(\pi) \rightarrow R^1 \Gf(\pi)$ is injective.
By Hochschild--Serre, one has isomorphisms
$$R^1 \Ff(\pi) \simeq H^1(\GL_n(k),\pi^{U(x)}), \qquad
R^1 \Gf(\pi) \simeq e_{\alpha} H^1(U_0(k),\pi^{U(x)}),$$
which (in principle) one might be able to compute explicitly for the relevant $\pi$.
}
\end{remark}

\subsubsection{Applications to Taylor--Wiles primes}
	
We now define the modules $M_j$ as follows:
$$M_j:=\lim_{m \rightarrow \infty} H^j(X(x^m),\OL/\varpi^n)_{\m}.$$
The modules $M_j$ are filtered (as $G = \GL_n(F_x)$-modules) by the the admissible module $M_j[\m]$.
By assumption, any representation $\pi \subset M_j[\m]$ lies in $\Cat$, and hence $M_j \in \Cat$
by Lemma~\ref{lemma:ext}.
By Hochschild--Serre, we have two spectral sequences, namely,
$$H^i(G(\OL_x),M_j) \Rightarrow H^{i+j}(X,\OL/\varpi^n)_{\mE},$$
$$
e_{\alpha} H^i(U_0(x),M_j) \Rightarrow e_{\alpha} H^{i+j}(X_0(x) ,\OL/\varpi^n)_{\mE} = H^{i+j}(X_0(x) ,\OL/\varpi^n)_{\m}.$$
There is a natural map between these spectral sequences given by $\iota_{*}$. 
By Theorem~\ref{theorem:cht}, these maps are isomorphisms, and hence we deduce that the map:
$$e_{\alpha}: H^*(X,\OL/\varpi^n)_{\mE} \simeq H^*(X_0(x),\OL/\varpi^n)_{\m},$$
is an isomorphism, as required.

\subsection{Modularity Lifting}
\label{sec:modularity-lifting}

In this section we prove our main theorem on modularity lifting. We
note that Theorem~\ref{theorem:modularity} follows from it immediately as a
corollary.

\medskip

We assume the existence of a maximal ideal
$\m$ of $\T:=\T_{\emptyset,1}$ with $\rbar_\m\cong \rbar$. We assume
also that $\rbar(G_{F(\zeta_p)})$ is big. 

\medskip

For each integer $N\geq 1$, let $Q_N$ be a set of primes satisfying the conclusions of
Proposition~\ref{prop:tw-primes}. We also assume that
Conjecture~\ref{conj:AA} holds for each of the sets $Q_N$.

\medskip

For each $N$, there is a natural covering map $Y_1(Q_N) \rightarrow Y_0(Q_N)$
with Galois group
$$\wt\Delta := \prod_{x \in Q} (\OL_F/x)^{\times}.$$
Choose a surjection $\wt\Delta \onto \Delta_N:=(\Z/p^N \Z)^q$ and let
$Y_{\Delta_N}(Q_N) \to Y_0(Q_N)$ be the corresponding sub-cover. For each $0\leq M \leq N$, we regard $\Delta_M$ as a quotient of
$\Delta_N$ in the natural fashion. This gives rise to further
sub-covers $Y_{\Delta_M}(Q_N)\to Y_0(Q_N)$.

By Conjecture~\ref{conj:AA} and the results of
Section~\ref{sec:betti-case}, there exists a perfect
complex $\wt D_N$ of free $S_N = \OL[\Delta_N]$-modules such that
\begin{itemize}
\item $\wt D_N$ is concentrated in degrees $q_0,\dots,q_0+l_0$,
\item the complex $\wt D_N\otimes_{S_N} S_N/\m_{S_N}$ has trivial differentials,
\item for each $i$, $n\geq 1$ and $0\leq M\leq N$, we have an isomorphism of $S_N$-modules
\[  H_i(\wt D_N\otimes_{S_N}S_M/\varpi^n) \cong \left( \lim_{n \rightarrow \infty} \prod_{x\in Q}
Q_x(V_{{x}})^{n!}\right) H_{i}(Y_{\Delta_M}(Q_N),\CO/\varpi^n)_{\m^*} . \]
\end{itemize}
Note that we have
$$H^*(Y_1(Q_N),\OL/\varpi^{n})^{\vee}_{\m} \simeq H_{*}(Y_1(Q_N),\OL/\varpi^{n})_{\m^*},$$
where the equivalence comes from the fact that we are assuming the cohomology of the boundary vanishes
after localization at $\m$.  

Similarly, working with the local system associated to our choice of
characters $\chi = (\chi_v)_{v \in R}$, there exists a perfect complex
$\wt D_N^{\chi}$ of free $S_N$-modules satisfying the first two properties above
as well as:
\begin{itemize}
\item for each $i$, $n\geq 1$ and $0\leq M\leq N$, we have an isomorphism of $S_N$-modules
\[  H_i(\wt D_N^{\chi}\otimes_{S_N}S_M/\varpi^n) \cong \left( \lim_{n \rightarrow \infty} \prod_{x\in Q}
Q_x(V_{{x}})^{n!}\right) H_{i,\chi}(Y_{\Delta_M}(Q_N),\CO/\varpi^n)_{\m^*} . \]
\end{itemize}
Note that we have
$$H^*_{\chi}(Y_1(Q_N),\OL/\varpi^{n})^{\vee}_{\m} \simeq H_{*,\chi}(Y_1(Q_N),\OL/\varpi^{n})_{\m^*},$$
again by Conjecture~\ref{conj:AA} part~\eqref{condition:vanishingconjA}.)

Since
\[ H^*(Y,\CO/\varpi)\cong H^*_{\chi}(Y,\CO/\varpi),\] the ideal $\m$
induces a maximal ideal of $\T_\chi:=\T_{\emptyset,\chi}$, which we
also denote by $\m$ in a slight abuse of notation. By Conjecture
\ref{conj:AA}, we have surjections $R_{\CS}\onto \T_\m$ and
$R_{\CS_\chi}\onto \T_{\chi,\m}$. 

\begin{theorem}
  \label{thm:mod-lifting}
If we regard $H^{q_0}(Y,K/\CO)^{\vee}_\m$ as an $R_{\CS}$-module via the
map $R_{\CS}\onto \T_\m$, then it is a nearly faithful $R_{\CS}$-module.
\end{theorem}

\begin{proof}
  We will apply the results of Section~\ref{sec:patching}. For each
  $N\geq 1$, we have chosen a set of Taylor--Wiles primes $Q_N$ satisfying the
  assumptions of Proposition~\ref{prop:tw-primes} (for some fixed
  choice of $q$). Let 
\[ g = q+|T|-1-[F:\Q]\frac{n(n-1)}{2}-l_0 \]
be the integer appearing in part~\eqref{generators} of this
proposition. We will apply Proposition \ref{prop:simult-patching} with
the following:
\begin{itemize}
\item Let $S_\infty=\CO[[(\Z_p)^q]]$ and $S_N = \CO[\Delta_N]$ as in
  the statement of Proposition \ref{prop:patching}.
\item Let $j=n^2|T|-1$ and $\CO^\square=\CO[[z_1,\dots,z_j]]$.
\item Let 
  \begin{eqnarray*}
    R^1_\infty &=& \Rloc^1[[x_1,\dots,x_g]] \\
    R^2_\infty &=& \Rloc^\chi[[x_1,\dots,x_g]].
  \end{eqnarray*}
  Note each $R^i_\infty$ is $p$-torsion free and equidimensional of
  dimension $1+q+j-l_0$ by Lemma \ref{lem:Rloc-props}. In addition, we
  have a natural isomorphism $R^1_\infty/\varpi \iso
  R^2_\infty/\varpi$.
\item Let $(R^1,H^1)=(R_{\CS},H^{q_0}(Y,K/\CO)^{\vee}_\m)$ and
  $(R^2,H^2)=(R_{\CS^\chi},H^{q_0}_\chi(Y,K/\CO)^{\vee}_\m)$. Note that we have
  natural compatible isomorphisms $R^1/\varpi\iso R^2/\varpi$ and $H^1/\varpi
  \iso H^2/\varpi$.
\item Let $T=T^1=T^2$ be the complex with
  $T^i=H^{i-l_0}(Y,\CO/\varpi)^{\vee}_\m$ and with all
  differentials $d : T^i \to T^{i+1}$ equal to $0$.
\item For each $N\geq 1$, let $Y_{\Delta_N}(Q_N)\to Y_0(Q_N)$ denote
  the subcover of $Y_1(Q_N)\to Y_0(Q_N)$ with Galois group $\Delta_N =
  (\Z/p^N)^q$. We introduced above perfect (homological) complexes of $S_N$-modules $\wt
  D_N$ and $\wt D^{\chi}_N$ above; they are concentrated in degrees $q_0,\dots,q_0+l_0$.  We now regard these as
  cohomological complexes concentrated in degrees $0,\dots,l_0$. Then
  we let $D^1_N$ (resp.\ $D^2_N$) denote the perfect
  complex of $S_N$-modules $\wt D_N / \varpi^N$ (resp.\ $\wt D^{\chi}_N/\varpi^N$).
   Note that the cohomology of $D^1_N$ (resp.\ $D^2_N$) computes
   $H^*(Y_{\Delta_N}(Q_N),\CO/\varpi^N)^{\vee}_\m$ (resp.\
   $H^*_{\chi}(Y_{\Delta_N}(Q_N),\CO/\varpi^N)^{\vee}_\m$), after a
   shift in degree by $q_0$.
  We can and do assume that
  $D^i_N\otimes S_N/\m_{S_N}\cong T$ for $i=1,2$.
\item Choose representatives for the universal deformations of type
  $\CS_{Q_N}$ and $\CS^\chi_{Q_N}$ which agree modulo $\varpi$. This gives rise to isomorphisms 
  \begin{eqnarray*}
     R^{\square_T}_{\CS_{Q_N}} &\liso& R_{\CS_{Q_N}}[[z_1,\dots,z_j]]  \\
 R^{\square_T}_{\CS^\chi_{Q_N}} &\liso & R_{\CS^{\chi}_{Q_N}}[[z_1,\dots,z_j]] .
  \end{eqnarray*}
  In the notation of Proposition \ref{prop:patching}, the rings on the
  right hand side can be written $R_{\CS_{Q_N}}^\square$ and
  $R_{\CS^\chi_{Q_N}}^{\square}$. By Proposition \ref{prop:tw-primes},
  we can and do choose surjections $R_\infty^1 \onto
  R_{\CS_{Q_N}}^\square$ and $R_\infty^2 \onto
  R_{\CS_{Q_N}^\chi}^\square$. Composing these with the natural maps
  $R^{\square}_{\CS_{Q_N}}\onto R_{\CS_{Q_N}}\onto R_{\CS}=R^1$ and
  $R^{\square}_{\CS_{Q_N}^\chi}\onto R_{\CS_{Q_N}^\chi}\onto
  R_{\CS^\chi}=R^2$, we obtain surjections $\phi_N^1 : R^1_\infty \onto
  R^1$ and $\phi_N^2:R^2_\infty\onto R^2$.
\end{itemize}
We have now introduced all the necessary input data to Proposition
\ref{prop:simult-patching}. We now check that they satisfy the
required conditions. 
\begin{itemize}
\item For each $M\geq N\geq 0$ with $M\geq 1$ and each $n\geq 1$, we
  have an action of $R_{\CS_{Q_N}}$ (resp.\ $R_{\CS^\chi_{Q_N}}$) on
  the cohomology $H^*(Y_{\Delta_N}(Q_M),\CO/\varpi^n)_\m$ (resp.\
  $H^*_\chi(Y_{\Delta_N}(Q_M),\CO/\varpi^n)_\m$) by Conjecture \ref{conj:AA}. Applying the
  functor, $X\mapsto X^\square$, and using the surjection
  $R^1_\infty\onto R_{\CS_{Q_N}}^\square$ (resp.\ $R^2_\infty\onto
  R_{\CS^\chi_{Q_N}}$), we obtain an action of $R^1_\infty$
  (resp. $R^2_\infty$) on
  $H^*(D_M^{1,\square}\otimes_{S_M}S_N/\varpi^n)$ (resp.\ $H^*(D_M^{2,\square}\otimes_{S_M}S_N/\varpi^n)$).
\end{itemize}
Thus condition~\eqref{galois-action} of Proposition~\ref{prop:patching}
is satisfied for both sets of patching data. Condition~\eqref{image-S}
follows from Conjecture~\ref{conj:AA}, while condition~\eqref{top-deg}
is clear. Finally, we note that we have isomorphisms
{\small
\[
H^{l_0}((D_M^1)^\square\otimes_{S_M}S_N/\varpi)=H^{l_0}(Y_{\Delta_N}(Q_M),\CO/\varpi)^\square
\iso H^{l_0}_\chi(Y_{\Delta_N}(Q_M),\CO/\varpi)^\square=
H^{l_0}((D_M^2)^\square\otimes_{S_M}S_N/\varpi) \]
}
for all $M\geq N\geq 0$ with $M\geq 1$. These isomorphisms are
compatible with the actions of $R^i_\infty$ and give rise to the
commutative square required by
Proposition~\ref{prop:simult-patching}. 

We have now satisfied all the
requirements of Proposition~\ref{prop:simult-patching} and hence we
obtain two complexes $P^{1,\square}_\infty$ and
$P^{2,\square}_\infty$. By Lemma~\ref{lem:Rloc-props}, $\Spec
R_\infty^2$ is irreducible, and hence by
Theorem~\ref{thm:faithfulness-statements}
$H^{l_0}(P^{2,\square}_{\infty})$ is nearly faithful as an
$R^2_\infty$-module. Thus
\[ H^{l_0}(P^{1,\square}_\infty)/\varpi \cong
H^{l_0}(P^{2,\square}_\infty)/\varpi \] 
is nearly faithful over $R^1_\infty/\varpi \iso R^2_\infty/\varpi$. By
Lemma~\ref{lem:Rloc-props} and 
\cite[Lemma 2.2]{Taylor}, it follows that
$H^{l_0}(P^{1,\square}_\infty)$ is nearly faithful over $R^1_\infty$
providing that $H^{l_0}(P^{1,\square}_\infty)$ is $p$-torsion
free. However, each associated prime of
$H^{l_0}(P^{1,\square}_\infty)$ is a minimal prime of $R^1_\infty$ and
by Lemma~\ref{lem:Rloc-props}, all such primes have characteristic
0. Thus $p$ cannot be a zero divisor on
$H^{l_0}(P^{1,\square}_\infty)$ and the result of Taylor applies. By
conclusion~\eqref{top-deg-patchted} of
Proposition~\ref{prop:patching} we deduce that $H^1 =
H^{q_0}(Y,K/\CO)^{\vee}_\m$ is nearly faithful over $R^1=R_{\CS}$, as required.
\end{proof}

\section{Proof of Theorem~\ref{theorem:ST}}
\label{sec:proof-of-ST}

In this section, we prove Theorem~\ref{theorem:ST}  

\begin{proof}
Let $A$ be an elliptic curve over a number field $K$. If $A$ has CM, then the result is well known, so
we may assume that $\End_{\C}(A) = \Z$.
 Let 
 $$r = \Sym^{2n-1} \rho: G_{K} \rightarrow \GL_{2n}(\Q_p)$$
 denote the representation corresponding to the $(2n-1)$th symmetric power of the Tate module of $A$.
To prove Theorem~\ref{theorem:ST},
 it suffices (following, for example, the proof of Theorem~4.2 of~\cite{HSBT}) to prove that
  for each $n$, there exists a $p$ such that $r$ is potentially modular.
 We follow the proof of Theorem~6.4 of~\cite{blght}.
 (The reason for following the proof of Theorem~6.4 instead of Theorem~6.3 of \emph{ibid}. is that the
 latter theorem proceeds via compatible families arising from the Dwork family such that
 $V[\lambda]_t$ is ordinary but not crystalline, which would necessitate a  different
 version of Theorem~\ref{theorem:modularity}.)
 In particular, we make the following extra hypothesis:
 \begin{itemize}
\item There exists a  prime $p$ which is totally split in $K$,
and such that $p+1$ is divisible by an integer $N_2$ which is greater than $n$
and prime to the conductor of $A$.
Moreover,
 the mod-$p$ representation $\rhobar_A:G_{K} \rightarrow \GL_2(\F_p)$ associated to $A[p]$ is surjective, 
 $A$ has good reduction at all $v|p$, and
  for all primes $v|p$ we have
 $$\rhobar_A:G_{\Q_p} \simeq \Ind^{\Q_p}_{\Q_{p^2}} \omega_2.$$
\end{itemize}
 (This is a non-trivial condition on $A$, we consider the general case below.)
 It  then  suffices to find sufficiently large primes $p$ and $l$,   a finite extension $L/K$,
 an integer $N_2$ with $N_2 > n+1$ and $p+1 = N_1 N_2$ 
 (as in the statement of Theorem~6.4 of~\S~4 of~\cite{blght})
 and primes $\lambda$, $\lambda'$ of
 $\Q(\zeta_N)^{+}$ (with $\lambda$ dividing $p$ and $\lambda'$ dividing $l$) and
 a point $t \in T_0(L)$ on the Dwork family such that:
 \begin{enumerate}
 \item
 $V[\lambda]_{t} \simeq \rbar |_{G_{L}}$,
 \item $V[\lambda']_{t} \simeq \rbar' |_{G_{L}}$, where
 $r'$ is an ordinary weight $0$ representation which
 induced from $G_{L M}$ for some suitable CM field $M/\Q$ of degree $2n$.
 \item $p$ splits completely in $L$.
 \item $A$ and $V$ are semistable over $L$.
\item $ \rbar |_{G_{L}}$ and $\rbar'|_{G_{L}}$ satisfy
all the hypotheses
of Theorem~\ref{theorem:modularity} with the possible exception of residual modularity.
\end{enumerate}
This can be deduced (as in the proof of Theorem~6.4 of~\cite{blght})
using the theorem of Moret--Bailly (in the form of Prop.~6.2 of  \emph{ibid}) and
via character building.
By construction, the modularity of $r$  follows from
two applications of Theorem~\ref{theorem:modularity}, once applied
to the $\lambda'$-adic representation associated to $V$ (using $\rbar'$ and
the residual modularity coming from the induction of a Grossencharacter) and once to 
the $\lambda$-adic representation associated to $\Sym^{2n-1}(A)$, using the residual modularity
coming from $V$.

\medskip

For a general elliptic curve $E$,  we reduce to the previous case as follows.
It suffices to find a second elliptic curve $A$, a number field $L/K$, 
and primes
$p$ and $q$ such that:
\begin{enumerate}
\item The mod-$p$ representation $ \rbar  = (\Sym^{2n-1} \rhobar_E)|_{G_{L}}$ satisfies
all the hypotheses
of Theorem~\ref{theorem:modularity} with the possible exception of residual modularity.
\item $A$ and $E$ are semistable over $L$ and have good reduction at all primes dividing~$p$ and~$q$.
\item $p$ and $q$ split completely in $L$.
\item $p+1$ is divisible by an integer $N_2 > n + 1$ which is prime to  the conductor of $A$.
\item $E[q] \simeq A[q]$ as $G_L$-modules, and the corresponding mod-$p$ representation is surjective.
\item The mod-$p$ representation $\rhobar_A:G_{L} \rightarrow \GL_2(\F_p)$ associated to $A$ is surjective,
and $\rhobar_{A}|_{G_{\Q_p}} \simeq \Ind^{\Q_p}_{\Q_{p^2}} \omega_2$.
\end{enumerate}
This lemma also follows easily from Prop.~6.2 of~\cite{blght}, now applied to twists of a
modular curve.
We deduce as above (using the mod-$q$ representation)
that $\Sym^{2n-1}(A)$ is potentially modular over some extension which
is unramified at $p$, and then use Theorem~\ref{theorem:modularity} once more now at the prime at $p$ to
deduce that $\Sym^{2n-1}(E)$ is modular.
\end{proof}

\begin{remark}
\emph{
It is no doubt possible to  also deal with  even symmetric powers
using the tensor product idea of Harris~\cite{Trick}.
}
\end{remark}

\section{Acknowledgements}

We would like to thank Matthew Emerton, David Helm, Toby Gee, Andrew
Snowden, Richard Taylor, David Treumann, Akshay Venkatesh, and Andrew Wiles
 for useful conversations. We also thank Toby Gee for detailed comments on several drafts of
 this paper.
Many of the ideas of this paper were discovered when both the authors were members of the Institute for
Advanced Study in 2010--2011 --- the authors would like to thank the IAS for the pleasant and productive environment
provided. Finally, the authors would like to thank each of the referees for their help in improving both the accuracy
and readability of this manuscript.

\bibliographystyle{amsalpha}
\bibliography{CalegariGeraghty}
 
 \end{document}